\documentclass[10pt]{article}

\usepackage[left=1in, right=1in, bottom=1in, top=1in, footskip=0.3in, 
]{geometry}

\usepackage{xcolor}
\definecolor{LinkColor}{rgb}{0,0,1}
\definecolor{LinkColor2}{rgb}{1,0,0}
\definecolor{lbcolor}{rgb}{0.85,0.85,0.85}
\definecolor{FrameColor}{rgb}{0.85,0.85,0.85}

\usepackage[ngerman, english]{babel}
\usepackage[utf8]{inputenc}
\usepackage{enumerate}
\usepackage[normalem]{ulem}
\usepackage{setspace}

\def\pskip{\\[-3mm]}

\usepackage{multicol}
\usepackage[babel,french=guillemets,german=swiss]{csquotes}

\usepackage{array}
\usepackage{booktabs}
\usepackage{framed}
\usepackage{rotating}
\usepackage{longtable}
\usepackage{multirow}
\usepackage{tabularx}
\newcolumntype{L}[1]{>{\raggedright\arraybackslash}p{#1}} 
\newcolumntype{C}[1]{>{\centering\arraybackslash}p{#1}} 
\newcolumntype{R}[1]{>{\raggedleft\arraybackslash}p{#1}} 

\usepackage{graphicx,import}
\usepackage{caption}

\usepackage{amssymb}
\usepackage{dsfont}
\usepackage{nicefrac}
\usepackage{empheq}
\allowdisplaybreaks

\usepackage[%
pdftitle={Titel},%
pdfauthor={Autor},%
pdfcreator={LaTeX, LaTeX with hyperref and KOMA-Script},
pdfsubject={Betreff}, 
pdfkeywords={Keywords}
]{hyperref} 

\hypersetup{%
	colorlinks=true,
	linkcolor=LinkColor,%
	anchorcolor=LinkColor,%
	citecolor=LinkColor2,%
	filecolor=LinkColor,%
	menucolor=LinkColor,%
	urlcolor=LinkColor,%
}

\usepackage[savemem]{listings}
\usepackage{paralist}
{\begin{list}{$\diamondsuit$}{}}%
	{\end{list}}

\usepackage{amsthm}
\usepackage{upgreek}

\newtheoremstyle{tstyle}
{15pt}	
{5pt}	
{\itshape}	
{}	
{\bfseries}	
{.}	
{0.5em}	
{}	

\theoremstyle{tstyle}

\newtheorem{theorem}{Theorem}[section]
\newtheorem{lemma}[theorem]{Lemma}

\newtheorem{corollary}[theorem]{Corollary}
\newtheorem{definition}[theorem]{Definition}
\newtheorem{remark}[theorem]{Remark}
\newtheorem{assumptions}[theorem]{Assumptions}

\newtheoremstyle{cstyle}
{15pt}	
{5pt}	
{}	
{}	
{\bfseries}	
{}	
{0.2222em}	
{}	

\theoremstyle{cstyle}

\makeatletter
\g@addto@macro{\thm@space@setup}{\thm@headpunct{}}
\renewenvironment{proof}[1][\proofname.]{\par
	\pushQED{\qed}%
	\normalfont \topsep6\p@\@plus6\p@\relax
	\trivlist
	\item[\hskip\labelsep
	\bfseries
	#1\@addpunct{\,}]\ignorespaces
}{%
	\popQED\endtrivlist\@endpefalse
}
\makeatother

\makeatletter
\g@addto@macro{\thm@space@setup}{\thm@headpunct{}}
\newenvironment{sketch-proof}[1][Sketch of the proof]{\par
	\pushQED{\qed}%
	\normalfont \topsep6\p@\@plus6\p@\relax
	\trivlist
	\item[\hskip\labelsep
	\bfseries
	#1\@addpunct{\,}]\ignorespaces
}{%
	\popQED\endtrivlist\@endpefalse
}
\makeatother

\newcounter{subeq}
\renewcommand{\thesubeq}{\theequation\alph{subeq}}
\newcommand{\newsubeqblock}{\setcounter{subeq}{0}\refstepcounter{equation}}
\newcommand{\mysubeq}{\refstepcounter{subeq}\tag{\thesubeq}}

\usepackage{mathtools}  
\usepackage{amsmath}
\usepackage{fancyhdr}
\usepackage{tikz}
\usepackage{pdflscape}
\usepackage{subfig}
\usepackage{float}
\usepackage{hyperref}
\usepackage{tikz}
\usepackage{pgfplots}
\pgfplotsset{compat=1.16} 
\usepackage{algorithmic}

\usepackage[section]{placeins}         

\usepackage[backend=biber,style=numeric,maxbibnames=99]{biblatex} 

\numberwithin{equation}{section}

\makeatletter
\newcommand{\mylabel}[2]{%
   \protected@write \@auxout {}{\string \newlabel {#1}{{#2}{\thepage}{#2}{#1}{}}}%
   \hypertarget{#1}{#2}%
}
\newcommand{\mylabelHIDE}[2]{%
   \protected@write \@auxout {}{\string \newlabel {#1}{{#2}{\thepage}{#2}{#1}{}}}%
   \hypertarget{#1}{}%
}
\makeatother

\newcommand{\J}{\pmb{\mathrm{J}}}

\newcommand{\pmbn}{\pmb{\mathrm{n}}}
\newcommand{\pmbP}{\pmb{\mathrm{P}}}

\newcommand{\pmbv}{\pmb{\mathrm{v}}}
\newcommand{\pmbw}{\pmb{\mathrm{w}}}

\renewcommand{\rho}{\varrho}
\renewcommand{\phi}{\varphi}

\newcommand{\A}{\ensuremath{\mathbb{A}}}
\newcommand{\B}{\ensuremath{\mathbb{B}}} 
\newcommand{\C}{\ensuremath{\mathbb{C}}} 
\newcommand{\D}{\ensuremath{\mathbb{D}}}   
\newcommand{\F}{\ensuremath{\mathbb{F}}}   
   
\newcommand{\I}{\ensuremath{\mathbb{I}}}  
\renewcommand{\L}{\ensuremath{\mathbb{L}}}    
\newcommand{\N}{\ensuremath{\mathbb{N}}}   
\newcommand{\R}{\ensuremath{\mathbb{R}}} 
\renewcommand{\S}{\ensuremath{\mathbb{S}}}   
\newcommand{\T}{\ensuremath{\mathbb{T}}}

\newcommand{\mycal}[1]{\ensuremath{\mathcal{#1}}}
\newcommand{\calA}{\mycal{A}}
\newcommand{\calB}{\mycal{B}}
\newcommand{\calC}{\mycal{C}}
\newcommand{\calD}{\ensuremath{\mathcal{D}}}

\newcommand{\calF}{\ensuremath{\mathcal{F}}}
\newcommand{\calG}{\ensuremath{\mathcal{G}}}
\newcommand{\calH}{\ensuremath{\mathcal{H}}}
\newcommand{\calI}{\mycal{I}}

\newcommand{\calQ}{\mycal{Q}}
\newcommand{\calR}{\mycal{R}}
\newcommand{\calP}{\ensuremath{\mathcal{P}}}
\newcommand{\calS}{\ensuremath{\mathcal{S}}}
\newcommand{\calT}{\ensuremath{\mathcal{T}}}

\newcommand{\calV}{\ensuremath{\mathcal{V}}}
\newcommand{\calW}{\ensuremath{\mathcal{W}}}

\newcommand{\divergenz}[1]{\operatorname{div} \left({#1}\right)}

\DeclareMathOperator{\trace}{Tr}

\newcommand{\abs}[1]{\left\lvert{#1}\right\rvert}
\newcommand{\norm}[1]{\|{#1}\|}
\newcommand{\nnorm}[1]{\left\|{#1}\right\|}
\newcommand{\dualp}[2]{\left<{#1, #2}\right>}
\newcommand{\skp}[2]{\left({#1, #2}\right)}

\newcommand{\dv}[1]{\,{\mathrm d}#1}
\newcommand{\dx}{\dv{x}}

\newcommand{\dt}{\dv{t}}
\newcommand{\dr}{\dv{r}}

\newcommand{\dH}{\ \mathrm d \calH}

\newcommand{\ddv}[2]{ \frac{{\mathrm d}#1}{{\mathrm d}#2} }
\newcommand{\fracdel}[2]{ \frac{\partial #1}{\partial #2} }
\newcommand{\ddt}{\ddv{}{t}}

\newcommand{\dtv}{\partial_t^\bullet}

\bibliography{main}

\begin{document}

%
%
	
\begin{center}	
	\LARGE{Viscoelastic Cahn--Hilliard models for tumour growth}
\end{center}
\bigskip

\begin{center}	
	\normalsize{Harald Garcke}\\[1mm]
	\textit{Fakult\"at f\"ur Mathematik, Universit\"at Regensburg, 93053 Regensburg, Germany}\\[1mm]
	\texttt{Harald.Garcke@ur.de}
\end{center}

\begin{center}	
	\normalsize{Bal{\'a}zs Kov{\'a}cs}\\[1mm]
	\textit{Fakult\"at f\"ur Mathematik, Universit\"at Regensburg, 93053 Regensburg, Germany}\\[1mm]
	\texttt{Balazs.Kovacs@ur.de}
\end{center}

\begin{center}	
	\normalsize{Dennis Trautwein}\\[1mm]
	\textit{Fakult\"at f\"ur Mathematik, Universit\"at Regensburg, 93053 Regensburg, Germany}\\[1mm]
	\texttt{Dennis.Trautwein@ur.de}
\end{center}

\bigskip

\begin{abstract}
\footnotesize
We introduce a new phase field model for tumour growth where viscoelastic effects are taken into account.
The model is derived from basic thermodynamical principles and consists of a convected Cahn--Hilliard equation with source terms for the tumour cells and a convected reaction-diffusion equation with boundary supply for the nutrient. Chemotactic terms, which are essential for the invasive behaviour of tumours, are taken into account.
The model is completed by a viscoelastic system constisting of the Navier--Stokes equation for the hydrodynamic quantities, and a general constitutive equation with stress relaxation for the left Cauchy--Green tensor associated with the elastic part of the total mechanical response of the viscoelastic material. 

For a specific choice of the elastic energy density and with an additional dissipative term accounting for stress diffusion, we prove existence of global-in-time weak solutions of the viscoelastic model for tumour growth in two space dimensions $d=2$ by the passage to the limit in a fully-discrete finite element scheme where a CFL condition, i.e. $\Delta t\leq Ch^2$, is required.

Moreover, in arbitrary dimensions $d\in\{2,3\}$, we show stability and existence of solutions for the fully-discrete finite element scheme, where positive definiteness of the discrete Cauchy--Green tensor is proved with a regularization technique that was first introduced by Barrett and Boyaval \cite{barrett_boyaval_2009}. After that, we improve the regularity results in arbitrary dimensions $d\in\{2,3\}$ and in two dimensions $d=2$, where a CFL condition is required. Then, in two dimensions $d=2$, we pass to the limit in the discretization parameters and show that subsequences of discrete solutions converge to a global-in-time weak solution.

Finally, we present numerical results in two dimensions $d=2$.
\pskip

\pskip
	
\noindent\textit{MSC Classification:} 
\end{abstract}



\section{Introduction}
\label{sec:introduction}

In the past few years, the study of mathematical models for tumour growth has become a popular topic of research. 
Even though many biological processes with regard to tissue growth are very complicated and still not fully understood, mathematical models try to give an insight into the qualitative behaviour of the most significant processes. 
Yet, the main difficulty is to choose the model in a way such that the individual properties of the respective biological material are described as good as possible.

Here, material laws play a decisive role and several different approaches have been proposed in the literature. 
Detailed comparisons with \textit{in vivo} experiments indicate that neglecting the elastic effects completely would be too restrictive, as mechanical stresses have a noticeable impact on the growth behaviour \cite{lima_2016}. Hence, living tissues are sometimes modelled as an elastic solid where the behaviour is described with linear or nonlinear elasticity. 
Moreover, there are models that refer to very short time scales for stress relaxation and thus propose viscous approaches, as they allow to consider the random and directional movement of the cells qualitatively, which is a well-known behaviour of tumour cells \cite{chemotaxis_in_cancer_2011}. 
On the other hand, the behaviour of tumour cells within the extracellular matrix resembles granular material for which usually Darcy's law is prescribed \cite{ambrosi_2009}.

A popular ansatz in the literature is to combine multiple material laws at once.
For example, Brinkman's law is used to describe material featuring properties of granular material and viscous fluids \cite{ebenbeck_garcke_nurnberg_2020}. 
To account for viscous and elastic properties, viscoelastic approaches are very helpful and they are mostly studied in the context of polymeric fluids \cite{barrett_boyaval_2009, barrett_lu_sueli_2017, Lukacova_2017}. Although there exist viscoelastic models for tumour growth \cite{ambrosi_2009, bresch_2009, lowengrub_2021_viscoelastic}, there is still a huge gap in the literature concerning the derivation and mathematical analysis, especially for phase field approaches. 
A Cahn--Hilliard model coupled to viscoelasticity with a Neo-Hookean finite elasticity, which is different to the one in the present paper, has been derived and analyzed in \cite{garcke_2022_viscoelastic}.

The goal of this work is to introduce and study a new mathematical model for tumour growth where viscoelasticity is taken into account. 
The general mathematical model of our interest is given by the following nonlinear system of partial differential equations.

\subsubsection*{Problem \ref{P}:}
\mylabelHIDE{P}{$(\pmbP)$} Find $\phi,\mu,\sigma,p: \Omega\times (0,T) \to \R$, $\pmbv:\Omega\times(0,T)\to \R^d$, $\B:\Omega\times(0,T)\to\R^{d\times d}$ such that in $\Omega\times (0,T)$:
\begin{subequations}
\begin{align}
    \label{eq:phi}
    \partial_t \phi + \divergenz{\phi \pmbv} 
    &= \divergenz{m(\phi) \nabla\mu} + \Gamma_\phi(\phi,\mu,\sigma,\B),
    \\
    \label{eq:mu}
    \mu &= A \psi'(\phi) - B \Delta\phi 
    + N_{,\phi}(\phi,\sigma) 
    + W_{,\phi}(\phi,\B),
    \\
    \label{eq:sigma}
    \partial_t \sigma + \divergenz{\sigma \pmbv} 
    &= \divergenz{n(\phi) \nabla N_{,\sigma}(\phi,\sigma)} 
    - \Gamma_\sigma(\phi,\mu,\sigma),
    \\
    \label{eq:div}
    \divergenz{\pmbv} &= \Gamma_{\pmbv}(\phi,\mu,\sigma,\B),
    \\
    \label{eq:v}
    \rho \partial_t \pmbv + \rho (\pmbv\cdot \nabla)\pmbv 
    &=  \divergenz{\T(\phi,\pmbv,p,\B)} 
    + \big(\mu  - W_{,\phi}(\phi,\B)
    \big) \nabla\phi
    + N_{,\sigma}(\phi,\sigma) \nabla\sigma,
    \\
    \label{eq:B}
    \partial_t \B + (\pmbv\cdot\nabla)\B
    + \frac{1}{\tau(\phi)} \T_{\mathrm{el}}(\phi,\B)
    &= \nabla\pmbv \B + \B (\nabla\pmbv)^T,
\end{align}
\end{subequations}
where $\Omega\subset \R^d$, $d\in\{2,3\}$, is a bounded domain and $T>0$ is a fixed time. Here, the viscoelastic stress tensor is given by 
\begin{align}
    \label{eq:T_viscoelastic}
    \T(\phi,\pmbv, p, \B)
    \coloneqq 
    \T_{\mathrm{visc}}(\phi,\pmbv,p) + \T_{\mathrm{el}}(\phi,\B),
\end{align}
where the viscous and the elastic parts of the stress tensor are defined as
\begin{align}
    \T_{\mathrm{visc}}(\phi,\pmbv,p)
    &\coloneqq 
    \eta(\phi) \big( \nabla\pmbv + (\nabla\pmbv)^T \big) 
    + \lambda(\phi) \divergenz{\pmbv}\I - p\I,
    \\
    \T_{\mathrm{el}}(\phi,\B) &\coloneqq 2W_{,\B}(\phi,\B) \B.
\end{align}

The above system is composed of a convected Cahn--Hilliard system \eqref{eq:phi}--\eqref{eq:mu} for the order parameter $\phi\in[-1,1]$ denoting the difference of volume fractions, with $\{\phi=1\}$ representing the unmixed tumour tissue and $\{\phi=-1\}$ representing the surrounding healthy tissue, and the chemical potential $\mu$ related to the phase field variable $\phi$. This system is coupled to a convected parabolic diffusion equation \eqref{eq:sigma} where $\sigma$ denotes the concentration of an unknown species serving as a nutrient for the tumour. 
We include hydrodynamic effects through the viscoelastic system \eqref{eq:div}--\eqref{eq:v} with constant mass density $\rho$ for the volume-averaged velocity $\pmbv$, the pressure $p$ and the viscoelastic stress tensor $\T$. Here, $\B$ denotes the left Cauchy--Green tensor associated with the elastic part of the total mechanical response of the viscoelastic fluid and it is given by the constitutive equation \eqref{eq:B} of Oldroyd-B type \cite{Oldroyd_1950}, but other constitutive equations for $\B$ are also possible, e.g., a constitutive equation of Giesekus type \cite{giesekus_1982}, i.e.
\begin{align}
    \label{eq:giesekus}
    \partial_t \B + (\pmbv\cdot\nabla)\B
    + \frac{1}{\tau(\phi)} \B \T_{\mathrm{el}}(\phi,\B)
    &= \nabla\pmbv \B + \B (\nabla\pmbv)^T.
\end{align}

By $N_{,\phi}(\phi,\sigma)$ and $N_{,\sigma}(\phi,\sigma)$, we denote the variations of a general nutrient energy density $N(\phi,\sigma)$ with respect to $\phi$ and $\sigma$, respectively. Similarly, by $W_{,\phi}(\phi,\B)$ and $W_{,\B}(\phi,\B)$, we denote the variations of a general elastic energy density $W(\phi,\B)$ with respect to $\phi$ and $\B$. These energies will be specified later.

The positive constants $A,B$ usually have the form $A = \frac{\beta}{\epsilon}$ and $B=\beta\epsilon$, where $\epsilon$ is proportional to the thickness of the diffuse interface and $\beta$ represents the surface tension. By $m(\cdot)$ and $n(\cdot)$, we denote the non-negative mobilities for the order parameter $\phi$ and the nutrient $\sigma$, respectively, and $\psi(\cdot)$ is a non-negative potential with two equal minima at $\pm 1$.
Biological effects like proliferation, apoptosis and nutrient consumption are taken into account through the source and sink terms $\Gamma_\phi$ and $\Gamma_\sigma$. Moreover, $\Gamma_{\pmbv}$ denotes a source for the velocity divergence and is often related to $\Gamma_\phi$. 
The non-negative functions $\eta(\cdot)$ and $\lambda(\cdot)$ denote the shear and bulk viscosities, respectively. The non-negative function $\tau(\cdot)$ is the viscoelastic relaxation time accounting for dissipation. 

We now present the outline of this work. We end Section \ref{sec:introduction} by introducing our notation.
Then, in Section \ref{sec:derivation}, we present the derivation of the general viscoelastic model for tumour growth \ref{P} using basic thermodynamical principles and we give several examples of constitutive laws. Moreover, we highlight further important aspects of modelling like a dissipation law for a general energy of the system, reformulations of the pressure leading to variants of the velocity equation \eqref{eq:v}, and initial and boundary conditions. We also give relevant examples for the source functions $\Gamma_\phi,\Gamma_\sigma,\Gamma_{\pmbv}$. 
Further, we specify the nutrient energy density $N(\phi,\sigma)$ and the elastic energy density $W(\phi,\B)$.
Besides, we handle the case of possible source or sink terms due to growth in the equation for $\B$, and we present several limit cases of our model \ref{P} which were introduced for other models in the literature.

In Section \ref{sec:Section3}, we consider a special variant of the problem \ref{P} which is additionally regularized with a dissipative term $\alpha\Delta\B$ in the Oldroyd-B equation, and the regularized problem is denoted by \ref{P_alpha}. This regularization improves the mathematical properties of the governing equations while it has a minor impact on the dynamical behaviour of the model, supposed that the viscoelastic diffusion constant $\alpha$ is small. 
For \ref{P_alpha}, we give the definition of weak solutions and provide an existence result in two spatial dimensions in Section \ref{sec:weak_solution}. To highlight the difficulties and to better understand the techniques in the proof of the existence result, we present the derivation of formal \textit{a priori} estimates in Section \ref{sec:formal_bounds}, the need for the restriction to two dimensions in Section \ref{sec:formal_bounds_2d} and a regularization strategy from Barrett and Boyaval \cite{barrett_boyaval_2009} in Section \ref{sec:regularization} which is needed to show positive definiteness of the Cauchy--Green tensor $\B$.
The existence result itself will be proved in Section \ref{sec:fem} by the limit passing in a fully-discrete finite element scheme in two dimensions where a CFL condition, i.e.~$\Delta t \leq C h^2$, is required, whereby Section \ref{sec:fem} is organized as follows. First, a regularized fully-discrete finite element scheme is introduced in arbitrary dimensions $d\in\{2,3\}$, where the regularization strategy from Barrett and Boyaval \cite{barrett_boyaval_2009} is applied on the fully-discrete level. For the regularized discrete scheme, stability and existence in arbitrary dimensions $d\in\{2,3\}$ are shown in Sections \ref{sec:stability} and \ref{sec:existence}, respectively. Then, in Section \ref{sec:delta_to_zero}, the regularization parameter is sent to zero which guarantees that the discrete Cauchy--Green tensor is positive definite. After that, the regularity of the discrete solutions is improved in arbitrary dimensions $d\in\{2,3\}$ in Section \ref{sec:regularity} and in also in two dimensions $d=2$ in Section \ref{sec:regularity_2D}, where a CFL condition, i.e.~$\Delta t\leq C h^2$, is needed. Then, Section \ref{sec:convergence} is devoted to the limit process $(\Delta t,h)\to (0,0)$ in two space dimensions, where existence of global-in-time weak solutions is provided by converging subsequences of the discrete solutions. Finally, in Section \ref{sec:numeric}, we present numerical results for the fully-discrete tumour model from Section \ref{sec:fem} in two spatial dimensions.

\subsection{Notation}
In this work, vector or matrix valued quantities are represented with a bold or blackboard bold font, respectively.
For $d\in\{2,3\}$, we define the scalar product of two vectors $\pmbv,\pmbw\in\R^d$ by $\pmbv\cdot\pmbw \coloneqq \pmbv^T \pmbw = \pmbw^T \pmbv$, and the scalar product of two matrices $\A,\B\in\R^{d\times d}$ by $\A:\B \coloneqq \trace(\A^T \B) = \trace(\B^T \A)$, where $\trace(\A)$ denotes the trace of a matrix $\A\in\R^{d\times d}$.
Moreover, $\R^{d\times d}_{\mathrm{S}}$ and $\R^{d\times d}_{\mathrm{SPD}}$ are the sets of symmetric $\R^{d\times d}$ and symmetric positive definite $\R^{d\times d}$ matrices, respectively.
For a vector or matrix valued quantity, we denote the induced norm by $\abs{\cdot}$, and for a scalar quantity, we denote by $\abs{\cdot}$ the Euclidean norm.
For a real Banach space $X$, we denote by $\norm{\cdot}_{X}$ its norm, by $X'$ its dual space, and by $\dualp{\cdot}{\cdot}_{X}$ the duality pairing between $X$ and $X'$.
For $p\in[1,\infty]$, an integer $m\geq 0$ and a bounded domain $\Omega\subset\R^d$, $d\in\{2,3\}$, we use the standard notation from, e.g., \cite{alt_2016}, and we write $L^p\coloneqq L^p(\Omega)$, $W^{m,p}\coloneqq W^{m,p}(\Omega)$ and $H^m\coloneqq H^m(\Omega) \coloneqq W^{m,2}(\Omega)$, where $W^{0,p}\coloneqq L^p$ in the case $m=0$. We also define $L^2_0 \coloneqq L^2_0(\Omega) \coloneqq \{q \in L^2 \mid \int_\Omega q \dx = 0\}$ and $H^1_0 \coloneqq H^1_0(\Omega) \coloneqq \{q \in H^1 \mid q|_{\partial\Omega}=0$ a.e.~on $\partial\Omega\}$, where $q|_{\partial\Omega}$ should be interpreted in the sense of the trace theorem. 
We sometimes use the same notation for vector valued or matrix valued spaces. 
For instance, $L^p$ can mean $L^p(\Omega)$, $L^p(\Omega;\R^d)$ or $L^p(\Omega;\R^{d\times d})$, which of course depends on the context. 
Moreover, we define $\mathbf{H} \coloneqq \left\{\pmbw \in L^2(\Omega;\R^d) \mid \divergenz{\pmbw} = 0 \text{ a.e.~in }\Omega, \ \pmbw\cdot\pmbn = 0 \text{ on } \partial\Omega \right\}$ and  $\mathbf{V}\coloneqq \left\{ \pmbw \in H^1_0(\Omega;\R^d) \mid \divergenz{\pmbw} = 0 \text{ a.e.~in }\Omega\right\}$, where $\pmbn$ denotes the outer unit normal on $\partial\Omega$.
The norms and seminorms of the Sobolev spaces are denoted by $\norm{\cdot}_{W^{m,p}}$ and $\abs{\cdot}_{W^{m,p}}$, respectively, and similarly for the spaces $L^p$ and $H^m$.
We denote the inner product of the spaces $L^2$ and $L^2({\partial\Omega})$ by $\skp{\cdot}{\cdot}_{L^2}$ and $\skp{\cdot}{\cdot}_{L^2({\partial\Omega})}$, respectively.
For $\alpha\in[0,1]$, we write $C^{0,\alpha}(\overline\Omega)$ for the H{\"o}lder spaces.
For a real Banach space $X$, a real number $p\in[1,\infty]$ and an integer $m\geq 0$, we denote the Bochner spaces by $L^p(0,T;X)$ and $W^{m,p}(0,T;X)$ and they are equipped with the norms $\norm{\cdot}_{L^p(0,T;X)}$ and $\norm{\cdot}_{W^{m,p}(0,T;X)}$. For $p=2$, we will also write $H^m(0,T;X)\coloneqq W^{m,2}(0,T;X)$ and $\norm{\cdot}_{H^m(0,T;X)} \coloneqq \norm{\cdot}_{W^{m,2}(0,T;X)}$. Sometimes, $L^p(0,T;L^p)$ will be identified with $L^p(\Omega_T)$, where $\Omega_T\coloneqq \Omega\times (0,T)$ with $\Omega\subset\R^d$, $d\in\{2,3\}$, and $T>0$.

\section{Derivation and modelling aspects}
\label{sec:derivation}

In this section, we present the derivation of the general viscoelastic model for tumour growth \ref{P}.
The outline of this section is as follows. 
We first present basic balance laws, before we use an energy inequality, a Lagrange multiplier approach and several constitutive assumptions to derive the general viscoelastic model. 
We then reformulate the pressure and derive a general energy identity, before we specify the initial and boundary conditions. Then, we give the most relevant examples for the source terms and specify the nutrient and elastic energy densities. After that, we present several variants and limit cases of the model.

\subsection{Conservation laws}

\subsubsection{Balance law of mass}

We consider a mixture consisting of healthy and tumour cells. We denote their difference of volume fractions by $\phi$, with $\{\phi=1\}$ representing the unmixed tumour tissue and $\{\phi=-1\}$ representing the surrounding healthy tissue. 
We assume the existence of an unspecified species acting as a nutrient for the tumour whose concentration is denoted by $\sigma$. 
Moreover, we assume that $\phi$ and $\sigma$ are transported by a volume-averaged velocity $\pmbv$ and some diffusive fluxes $\J_\phi$ and $\J_\sigma$, respectively.
Based on these assumptions, the balance laws of mass read
\begin{align}
    \label{eq:balance_phi}
    \partial_t \phi + \divergenz{\phi \pmbv} + \divergenz{\J_\phi} 
    &= \Gamma_\phi,
    \\
    \label{eq:balance_sigma}
    \partial_t \sigma + \divergenz{\sigma \pmbv} + \divergenz{\J_\sigma} 
    &= - \Gamma_\sigma,
\end{align}
where $\Gamma_\phi$ and $\Gamma_\sigma$ denote the source and sink terms of the phase field variable and the nutrient. 
Moreover, mass exchange in terms of the divergence of $\pmbv$ is explicitly given by
\begin{align}
    \label{eq:balance_div_v}
    \divergenz{\pmbv} &= \Gamma_{\pmbv}.
\end{align}
The specific motivation for \eqref{eq:balance_phi} and \eqref{eq:balance_div_v} is based on mass balance laws for the two components of the mixture, i.e.~the tumour and healthy cells, and we refer the reader to \cite{ebenbeck_garcke_nurnberg_2020} for more details.
In the general case, the mass densities of the tumour cells $\bar{\rho}_{1}$ and the healthy cells $\bar{\rho}_{-1}$ can differ, which yields for the mass density $\rho$ of the mixture,
\begin{align}
    \label{eq:mass_density}
    \rho = \hat\rho(\phi) = 
    \tfrac{1}{2} \bar{\rho}_{1} (1+\phi)
    + \tfrac{1}{2} \bar{\rho}_{-1} (1-\phi).
\end{align} 
For simplicity reasons, we consider matching mass densities of the pure components in this work, which results in $\rho = \bar{\rho}_{1} = \bar{\rho}_{-1}$.

\subsubsection{Balance law of linear momentum}
Motivated by, e.g., \cite{AbelsGG_2012, ebenbeck_garcke_nurnberg_2020}, we assume that the mixture is a single viscoelastic fluid that fulfills the balance law of linear momentum of continuum mechanics. 
We further neglect any gravity effects or body forces and suppose that contact forces are represented by a viscoelastic stress tensor $\overline\T$. 
Moreover, we assume that the viscoelastic stress tensor is symmetric, isotropic and can depend on $\nabla\pmbv, \phi,\mu,\sigma,\nabla\phi$ and $\B_e$, where $\B_e$ is the left Cauchy--Green tensor associated with the elastic part of the total mechanical response which will be defined later (see \eqref{eq:B_e}).
With these assumptions, the balance law of linear momentum is given by
\begin{align}
    \label{eq:balance_momentum}
    \rho \partial_t \pmbv 
    + \rho (\pmbv\cdot\nabla)\pmbv 
    = \divergenz{\overline\T},
\end{align}
where $\overline\T$ has to specified by constitutive assumptions.


\subsubsection{Concept of viscoelasticity}
\label{sec:concept_viscoelasticity}

In the literature \cite{HuLinLiu_2018, liu_2008_viscoelastic_incompr, LinLiuZhang_2005}, a popular approach for viscoelasticity is in terms of the deformation gradient $\F: \Omega \times (0,T) \to \R^{d\times d}$ between the initial configuration and the current configuration of a viscoelastic body. Writing $\F$ in Eulerian coordinates, we obtain the hyperbolic evolution equation
\begin{align}
\label{eq:F}
    \partial_t \F + (\pmbv\cdot\nabla) \F = \nabla\pmbv \F.
\end{align}
Hence, it is easily deduced from \eqref{eq:F} that the left Cauchy--Green tensor $\tilde \B \coloneqq \F\F^T$ satisfies the evolution equation 
\begin{align}
\label{eq:B0_infinite_weissenberg}
    \partial_t \tilde\B + (\pmbv\cdot\nabla) \tilde\B= \nabla\pmbv\tilde\B + \tilde\B(\nabla\pmbv)^T.
\end{align}
This is the so-called Oldroyd-B equation with infinite Weissenberg number which is a common way to describe viscoelastic materials of Kelvin--Voigt type, where stress relaxation is neglected. 
However, stress relaxation is a typical behaviour of living tissues \cite{ambrosi_2009}.
For this reason, we follow the approach of M{\'a}lek and Pr{\r u}{\v s}a \cite{malek_prusa_2018} in order to derive a viscoelastic approach that accounts for stress relaxation. 
We assume a \textit{virtual} framework consisting of three configurations:~the initial configuration, the current configuration at time $t>0$ and the natural configuration which would be taken by the considered body at time $t>0$ after immediate relaxation, see Figure \ref{fig:configurations}.
Therefore, we assume a \textit{virtual} multiplicative decomposition of the full deformation gradient $\F$ by 
\begin{align}
    \label{eq:F_decomposition}
    \F =  \F_{e} \F_{d},
\end{align}
where $\F_{d}$ describes the deformation gradient between the initial and the natural configuration, taking into account only the dissipative processes of the fluid, which, in the biological context, can arise from, e.g., cell reorganizations, birth and death of cells \cite{ambrosi_2009}. 
Besides, $\F_{e}$ measures only the elastic part of the deformation, which is the deformation gradient between the natural and the current configuration. 
Then, the sought measure of our main interest is the left Cauchy--Green tensor $\B_{e} \coloneqq \F_{e} \F_{e}^T$ associated with the elastic part of the deformation.

\begin{figure}[ht]
\small
\centering
\includegraphics[width=0.7\linewidth]{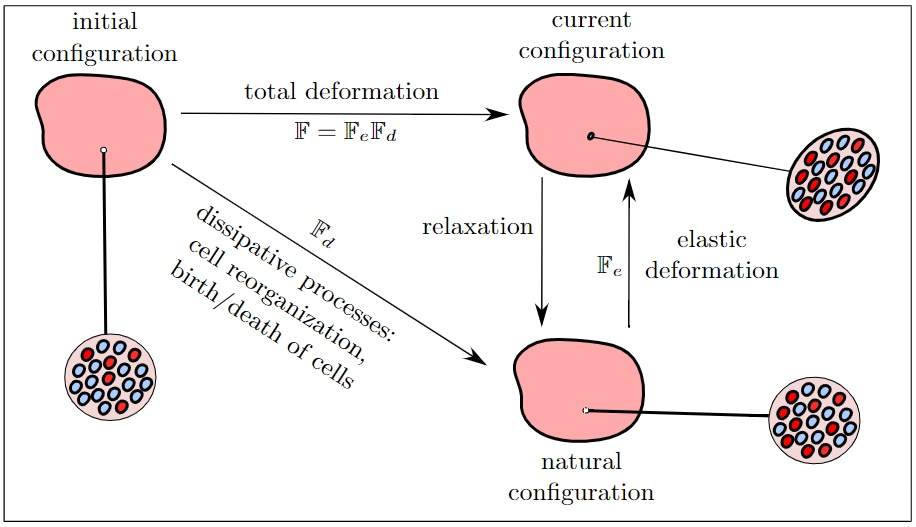}
\caption{Configurations of a viscoelastic cell-mixture within a fixed domain, described by a virtual decomposition of the total deformation. Adapted from \cite{malek_prusa_2018}.}
\label{fig:configurations}
\end{figure}



Following the ideas of M{\'a}lek and Pr{\r u}{\v s}a \cite{malek_prusa_2018}, we define $\L \coloneqq \nabla\pmbv$ and its symmetric part $\D\coloneqq \frac{1}{2}(\L+\L^T)$. From \eqref{eq:F} we see that $\L = (\dtv \F) \F^{-1}$, where the material derivative of $\F$ is defined by $\dtv \F \coloneqq \partial_t \F + (\pmbv \cdot \nabla) \F$.
This motivates to introduce the tensorial quantity $\L_{d}$ by
\begin{align}
    \label{eq:L_e}
    \L_{d} \coloneqq (\dtv \F_{d}) \F_{d}^{-1},
\end{align}
and its symmetric part by $\D_{d} \coloneqq \frac{1}{2} ( \L_{d} + \L_{d}^T)$.
Together with the formula $\dtv (\F_{d}^{-1}) = -\F_{d}^{-1} (\dtv\F_{d}) \F_{d}^{-1}$ we see that the material derivative of the relative deformation gradient $\F_{e}$ is given by
\begin{align}
    \label{eq:F_e}
    \dtv \F_{e}
    = \dtv ( \F \F_{d}^{-1} )
    = (\dtv \F) \F_{d}^{-1} + \F \dtv (\F_{d}^{-1})
    = \L\F\F_{d}^{-1} - \F\F_{d}^{-1}(\dtv\F_{d})\F_{d}^{-1}
    = \L\F_{e} - \F_{e} \L_{d},
\end{align}
which implies, as $\B_{e} = \F_{e} \F_{e}^T$, that
\begin{align}
    \label{eq:B_e}
    \dtv \B_{e} = \L \B_{e} + \B_{e} \L^T
    - 2 \F_{e} \D_{d} \F_{e}^T.
\end{align}
This is the sought formula for the evolution of the left Cauchy--Green tensor $\B_{e}$. The right-hand side of \eqref{eq:B_e} depends on the quantities $\F_{e}$, $\D_{d}$ and $\L=\nabla\pmbv$. Later, the dependency on the tensor $\D_{d}$ will be removed by constitutive assumptions.



\subsection{Energy inequality and the Lagrange multiplier method}
In order to derive the system \ref{P} from thermodynamical principles, we apply the Lagrange multiplier method by Liu and M{\"u}ller developed in \cite{liu_1972}. We remark that the mass density $\rho$ is assumed to be constant. In the case of a non-constant mass density given by formula \eqref{eq:mass_density}, the derivation of a system of equations can be performed with methods from Abels, Garcke and Gr{\"u}n \cite{AbelsGG_2012}.

We postulate a general energy density of the form
\begin{align}
    e = \hat e(\phi,\nabla\phi,\sigma, \B_{e}) + \frac{1}{2} \rho \abs{\pmbv}^2,
\end{align}
where $\hat e$ denotes the free energy density of the system which can depend on $\phi,\nabla\phi,\sigma, \B_{e}$, whereas $\frac{1}{2} \rho \abs{\pmbv}^2$ denotes the kinetic energy density.
Let $V(t)\subset \Omega$ be an arbitrary volume which is transported with the fluid velocity. We now consider the following energy inequality resulting from the second law of thermodynamics:
\begin{align}
\label{eq:energy_law}
    \begin{split}
    \underbrace{\ddt \int_{V(t)} 
    e(\phi,\nabla\phi,\sigma, \pmbv, \B_{e}) \dx}_{\text{ change of energy}}
    &\leq 
    \underbrace{- \int_{\partial V(t)} \J_e \cdot \pmbn \dH^{d-1}}_{ \substack{\text{ energy flux across} \\ \text{ the boundary}}}
    + \underbrace{\int_{\partial V(t)} (\overline\T\pmbn)\cdot\pmbv \dH^{d-1}}_{ \substack{\text{ work due to } \\ \text{ macroscopic stresses}}}
    \\
    &\quad 
    + \underbrace{\int_{V(t)} c_{\pmbv} \Gamma_{\pmbv} + c_\phi \Gamma_\phi + c_\sigma (-\Gamma_\sigma) \dx}_{\text{\normalfont supply of energy}},
    \end{split}
\end{align}
where $\pmbn$ is the outer unit normal to $\partial V(t)$, $\J_e$ is an energy flux yet to be determined and $c_{\pmbv}$, $c_\phi$ and $c_\sigma$ are unknown multipliers which have to be specified. 

Following the arguments in, e.g., \cite{AbelsGG_2012, ebenbeck_garcke_nurnberg_2020, GarckeLSS_2016}, we introduce Lagrange multipliers $\lambda_{\pmbv}$, $\lambda_\phi$ and $\lambda_\sigma$ for \eqref{eq:balance_div_v}, \eqref{eq:balance_phi} and \eqref{eq:balance_sigma}, respectively.
Using the momentum balance equation \eqref{eq:balance_momentum} and the Gauss theorem, we can reformulate the boundary integral describing work due to macroscopic stresses by
\begin{align*}
    - \int_{\partial V(t)} (\overline\T\pmbn)\cdot\pmbv \dH^{d-1} 
    = - \int_{V(t)} \divergenz{\overline\T} \cdot \pmbv + \overline\T : \nabla\pmbv \dx 
    = - \int_{V(t)} \rho \dtv \pmbv \cdot \pmbv
    + \overline\T : \nabla\pmbv \dx.
\end{align*}
Therefore, using Reynold's transport theorem \cite{eck-garcke-knabner} and the fact that $V(t)$ is arbitrary,
we obtain the following local dissipation inequality 
\begin{align}
\begin{split}
\label{eq:dissipation_1}
    -\calD 
    &\coloneqq
    \dtv e + e \divergenz{\pmbv} + \divergenz{\J_e} - \overline\T:\nabla\pmbv
    - \rho \dtv\pmbv \cdot \pmbv
    - c_{\pmbv}\Gamma_{\pmbv} - c_\phi\Gamma_\phi + c_\sigma\Gamma_\sigma 
    \\
    &\quad - \lambda_{\pmbv} (\divergenz{\pmbv} - \Gamma_{\pmbv})
    - \lambda_{\phi} (\dtv\phi + \phi\divergenz{\pmbv} + \divergenz{\J_\phi}- \Gamma_{\phi})
    - \lambda_{\sigma} (\dtv\sigma + \sigma\divergenz{\pmbv} + \divergenz{\J_\sigma}+ \Gamma_{\sigma})
    \\
    &\leq 0,
\end{split}
\end{align}
which has to hold for arbitrary values of $\phi$, $\sigma$, $\nabla\phi$, $\nabla\sigma$, $\rho$, $\pmbv$, $\B_e$, $\Gamma_{\pmbv}$, $\Gamma_\phi$, $\Gamma_\sigma$, $\dtv \phi$, $\dtv \sigma$ and $\dtv \pmbv$.
By the chain rule, we then have
\begin{align*}
    \dtv e &= \fracdel{e}{\phi} \dtv\phi + \fracdel{e}{\nabla\phi} \cdot \dtv(\nabla\phi) + \fracdel{e}{\sigma} \dtv\sigma 
    + \fracdel{e}{\pmbv} \cdot \dtv \pmbv
    + \fracdel{e}{\B_{e}} : \dtv \B_{e}.
\end{align*}
Therefore, on noting \eqref{eq:B_e}, we obtain
\begin{align}
\begin{split}
\label{eq:dissipation_2}
    - \calD 
    &=
    \divergenz{\J_e - \lambda_\phi \J_\phi - \lambda_\sigma \J_\sigma  } 
    + \nabla\lambda_\phi \cdot \J_\phi + \nabla\lambda_\sigma \cdot \J_\sigma 
    + \fracdel{e}{\nabla\phi} \cdot \dtv(\nabla\phi)
    \\
    &\quad 
    + \dtv\phi \Big( \fracdel{e}{\phi}-\lambda_\phi \Big)
    + \dtv\sigma \Big( \fracdel{e}{\sigma}-\lambda_\sigma \Big)
    + \dtv\pmbv \cdot \Big( \fracdel{e}{\pmbv}- \rho\pmbv \Big)
    \\
    &\quad
    - \T:\nabla\pmbv  
    + \fracdel{e}{\B_{e}} : (\L \B_{e} + \B_{e} \L^T
    - 2 \F_{e} \D_{d} \F_{e}^T)
    \\
    &\quad 
    + (c_\sigma - \lambda_\sigma) \Gamma_\sigma
    + (\lambda_{\pmbv} - c_{\pmbv}) \Gamma_{\pmbv}
    + (\lambda_\phi - c_\phi) \Gamma_\phi
    + (e - \lambda_\phi \phi - \lambda_\sigma \sigma - \lambda_{\pmbv}) \divergenz{\pmbv}
    \\
    &\leq 0.
\end{split}
\end{align}
Together with $\partial_{x_j} (\dtv\phi) 
= \partial_t \partial_{x_j} \phi 
+ \pmbv \cdot \nabla(\partial_{x_j}\phi)
+ \partial_{x_j} \pmbv \cdot \nabla\phi
= \dtv(\partial_{x_j}\phi) + \partial_{x_j}\pmbv \cdot \nabla\phi$,
we calculate
\begin{align*}
    \divergenz{\dtv\phi \fracdel{e}{\nabla\phi}} 
    = \dtv \phi \divergenz{\fracdel{e}{\nabla\phi}} 
    + \dtv(\nabla\phi) \cdot \fracdel{e}{\nabla\phi}
    + \nabla\pmbv : \Big( \nabla\phi \otimes \fracdel{e}{\nabla\phi} \Big).
\end{align*}
Then, using $\fracdel{e}{\pmbv} = \rho\pmbv$, we can reformulate \eqref{eq:dissipation_2} as
\begin{align}
\begin{split}
\label{eq:dissipation_3}
    - \calD 
    &=
    \divergenz{\J_e - \lambda_\phi \J_\phi - \lambda_\sigma \J_\sigma + \dtv\phi \fracdel{e}{\nabla\phi}} 
    + \nabla\lambda_\phi \cdot \J_\phi 
    + \nabla\lambda_\sigma \cdot \J_\sigma
    \\
    &\quad 
    + \dtv\phi \Big( \fracdel{e}{\phi} - \divergenz{\fracdel{e}{\nabla\phi}} - \lambda_\phi \Big)
    + \dtv\sigma \Big( \fracdel{e}{\sigma}-\lambda_\sigma \Big)
    \\
    &\quad
    - \Big( \overline\T+ \Big( \nabla\phi \otimes \fracdel{e}{\nabla\phi} \Big)  \Big) : \nabla\pmbv 
    + \fracdel{e}{\B_{e}} : (\L \B_{e} + \B_{e} \L^T
    - 2 \F_{e} \D_{d} \F_{e}^T)
    \\
    &\quad 
    + (c_\sigma - \lambda_\sigma) \Gamma_\sigma
    + (\lambda_{\pmbv} - c_{\pmbv}) \Gamma_{\pmbv}
    + (\lambda_\phi - c_\phi) \Gamma_\phi
    + (e - \lambda_\phi \phi - \lambda_\sigma \sigma - \lambda_{\pmbv}) \divergenz{\pmbv}
    \\
    &\leq 0.
\end{split}
\end{align}
In the following we denote the chemical potential of the order parameter $\phi$ as
\begin{align*}
    \mu \coloneqq \fracdel{e}{\phi} - \divergenz{\fracdel{e}{\nabla\phi}}.
\end{align*}


\subsection{Constitutive relations}
To fulfill the last inequality for $-\calD$, we can argue similarly to \cite{AbelsGG_2012, ebenbeck_garcke_nurnberg_2020, GarckeLSS_2016} and we make the following constitutive assumptions for the fluxes $\J_e$, $\J_\sigma$, $\J_\phi$, for the constants $c_{\pmbv}$, $c_\phi$, $c_\sigma$ and for the Lagrange multipliers $\lambda_{\pmbv}$, $\lambda_\phi$, $\lambda_\sigma$:
\begin{align}
    \label{eq:constitutive_1}
    \begin{split}
    \J_e &= \lambda_\sigma \J_\sigma + \lambda_\phi \J_\phi 
    - \dtv\phi \fracdel{e}{\nabla\phi} ,
    \quad
    \J_\phi = - m(\phi) \nabla\lambda_\phi,
    \quad
    \J_\sigma = -n(\phi) \nabla\lambda_\sigma,
    \\
    c_{\pmbv} &= \lambda_{\pmbv},
    \quad
    c_\phi = \lambda_\phi = \fracdel{e}{\phi} - \divergenz{\fracdel{e}{\nabla\phi}} = \mu,
    \quad
    c_\sigma = \lambda_\sigma = \fracdel{e}{\sigma},
    \end{split}
\end{align}
where $m(\phi)$ and $n(\phi)$ are non-negative mobilities corresponding to a generalised Fick's law (see \cite{AbelsGG_2012}). 

Actually, $m(\phi)$ and $n(\phi)$ could also depend on the chemical potential $\mu$ and the nutrient $\sigma$. As pointed out in \cite{malek_prusa_2018}, one could also consider cross effects by assuming that the relations between the fluxes $\J_\phi$, $\J_\sigma$ and the gradients of the Lagrange parameters $\nabla\lambda_\phi$, $\nabla\lambda_\sigma$ take the form 
$\J_\alpha = - \sum_{\beta\in\{\phi,\sigma\}} \mathrm{M}_{\alpha,\beta} \nabla\lambda_\beta$ for $\alpha\in\{\phi,\sigma\}$, 
where $(\mathrm{M}_{\alpha,\beta})$ is a symmetric positive definite matrix which could depend on, e.g., $\phi$, $\mu$ and $\sigma$.
This would of course lead to a more complex system of equations, and is hence not explored here.

Under these constitutive assumptions, the inequality for the local dissipation \eqref{eq:dissipation_3} holds if
\begin{align}
    \label{eq:constitutive_3a}
    \begin{split}
    &- \Big( \overline\T+ \Big( \nabla\phi \otimes \fracdel{e}{\nabla\phi} \Big)  \Big) : \nabla\pmbv 
    + \fracdel{e}{\B_{e}} : (\nabla\pmbv \B_{e} + \B_{e} (\nabla\pmbv)^T
    - 2 \F_{e} \D_{d} \F_{e}^T)
    \\
    &\quad + (e - \lambda_\phi \phi - \lambda_\sigma \sigma - \lambda_{\pmbv}) \divergenz{\pmbv}
    \\
    &\leq 0.
    \end{split}
\end{align}
Using the properties of the trace and the symmetry 
of $\B_{e}$ and  $\fracdel{e}{\B_{e}}$, 
we calculate
\begin{align*}
    &\fracdel{e}{\B_{e}} : (\nabla\pmbv \B_{e})
    = \trace\Big( \fracdel{e}{\B_{e}} 
    \B_{e}^T (\nabla\pmbv)^T \Big)
    = \Big( \fracdel{e}{\B_{e}} \B_{e}\Big) : \nabla\pmbv,
    \\
    &\fracdel{e}{\B_{e}} : \big(\B_{e}(\nabla\pmbv)^T \big)
    = \trace\Big( \Big(\fracdel{e}{\B_{e}}\Big)^T
    \B_{e} (\nabla\pmbv)^T \Big)
    = \Big( \fracdel{e}{\B_{e}} \B_{e}\Big) : \nabla\pmbv,
\end{align*}
and
\begin{align*}
    & \fracdel{e}{\B_{e}} : ( \F_{e} \D_{d} \F_{e}^T)
    = \trace\Big(\fracdel{e}{\B_{e}} \F_{e} \D_{d}^T \F_{e}^T \Big)
    = \trace\Big( \F_{e}^T \fracdel{e}{\B_{e}} \F_{e} \D_{d}^T  \Big)
    = \Big( \F_{e}^T \fracdel{e}{\B_{e}} \F_{e} \Big) : \D_{d}.
\end{align*}
Therefore, we can reformulate \eqref{eq:constitutive_3a} as
\begin{align}
\label{eq:constitutive_3b}
    \begin{split}
    & \Big( - \overline\T- \Big( \nabla\phi \otimes \fracdel{e}{\nabla\phi}\Big) 
    + 2 \fracdel{e}{\B_{e}} \B_{e}  \Big) : \nabla\pmbv 
    - 2 \Big( \F_{e}^T \fracdel{e}{\B_{e}} \F_{e} \Big) : \D_{d}
    + (e - \lambda_\phi \phi - \lambda_\sigma \sigma - \lambda_{\pmbv}) \divergenz{\pmbv}
    \leq 0.
    \end{split}
\end{align}
Introducing the unknown pressure $p$, we can rewrite the stress tensor $\T$ as follows:
\begin{align}
    \label{eq:constitutive_3}
    \overline\T= \S - p \I, \quad \text{ i.e. } \quad \S = \overline\T+ p\I.
\end{align}
Similar arguments as in \cite{ebenbeck_garcke_nurnberg_2020} imply that
\begin{align}
    \label{eq:constitutive_4}
    \fracdel{e}{\nabla\phi}
    = a(\phi,\nabla\phi,\sigma,\B_{e}) \nabla\phi,
\end{align}
where $a(\phi,\nabla\phi,\sigma,\B_{e})$ is a real valued function. Since $\S$ is symmetric, we have
\begin{align*}
    \S:\nabla\pmbv = \S : \frac{1}{2}(\nabla\pmbv + (\nabla\pmbv)^T) 
    + \S : \frac{1}{2} (\nabla\pmbv - (\nabla\pmbv)^T)
    = \S : \D,
\end{align*}
and similarly, as $\fracdel{e}{\B_{e}}$ and $ \B_{e}$ are symmetric, we also obtain
\begin{align*}
    \Big( \fracdel{e}{\B_{e}} \B_{e}  \Big) : \nabla\pmbv 
    = \Big( \fracdel{e}{\B_{e}} \B_{e}  \Big) : \D.
\end{align*}
Using the identity $\I : \nabla\pmbv = \I:\D = \divergenz{\pmbv}$, we have
\begin{align*}
    \overline\T: \nabla\pmbv = (\S - p \I) : \nabla\pmbv 
    = \S : \D - p \divergenz{\pmbv}.
\end{align*}
This yields
\begin{align}
    \label{eq:constitutive_5a}
    \begin{split}
    &  \Big( - \S -   a(\phi,\nabla\phi,\sigma,\B_{e}) (\nabla\phi \otimes  \nabla\phi) + 2 \fracdel{e}{\B_{e}} \B_{e}  \Big) : \D  
    \\
    &\quad
    - 2 \Big( \F_{e}^T \fracdel{e}{\B_{e}} \F_{e} \Big) : \D_{d}
    + (e - \lambda_\phi \phi - \lambda_\sigma \sigma + p - \lambda_{\pmbv}) \divergenz{\pmbv}
    \leq 0.
    \end{split}
\end{align}
The quantities $\D$ and $\divergenz{\pmbv}$ appear in the first and in the last term on the left-hand side of \eqref{eq:constitutive_5a} as it holds $\trace\D = \divergenz{\pmbv}$. Therefore, these two quantities are not independent. As pointed out in \cite{malek_prusa_2018}, it would actually be necessary to split $\D$ to mutually independent quantities consisting of the traceless part of $\D$ and the part containing $\divergenz{\pmbv}$, which requires one to split the quantities in the first brackets in \eqref{eq:constitutive_5a} in a similar manner.
However, following the strategy of \cite{ebenbeck_garcke_nurnberg_2020, GarckeLSS_2016}, we choose the constitutive assumption
\begin{align}
    \label{eq:constitutive_5}
    \lambda_{\pmbv} \coloneqq e - \lambda_\phi \phi - \lambda_\sigma \sigma + p, 
\end{align}
for the Lagrange parameter $\lambda_{\pmbv}$ in order to control the mass exchange term even though $\D$ and $\divergenz{\pmbv}$ are not independent quantities. The reason for this is that we can reformulate the unknown pressure $p$ and therefore adapt the constitutive assumption for the Lagrange parameter $\lambda_{\pmbv}$ afterwards.
Hence, it remains to fulfill the inequality
\begin{align}
\label{eq:constitutive_6a}
    \begin{split}
    &  - \Big( \S +  a(\phi,\nabla\phi,\sigma,\B_{e}) (\nabla\phi \otimes  \nabla\phi) 
    - 2 \fracdel{e}{\B_{e}} \B_{e}  \Big) : \D  
    - 2 \Big( \F_{e}^T \fracdel{e}{\B_{e}} \F_{e} \Big) : \D_{d}
    \leq 0.
    \end{split}
\end{align}
At this point we make the constitutive assumptions 
for $\S$ and $\D_{d}$ as follows:
\begin{align}
    \label{eq:constitutive_6}
    \S 
    &= 2 \eta(\phi)\D + \lambda(\phi)\divergenz{\pmbv}\I
    + 2 \fracdel{e}{\B_{e}} \B_{e} 
    - a(\phi,\nabla\phi,\sigma,\B_{e}) (\nabla\phi \otimes  \nabla\phi),
    \\
    \label{eq:constitutive_7}
    \D_{d} &= \frac{1}{\tau(\phi)} \Big( \F_{e}^{-1} \fracdel{e}{\B_{e}} \F_{e} \Big),
\end{align}
with non-negative viscosities $\eta(\cdot)$, $\lambda(\cdot)$ and a non-negative viscoelastic relaxation function $\tau(\cdot)$, which could also depend on $\mu$ and $\sigma$. 

Noting that $\lambda(\phi) \divergenz{\pmbv} \I : \D = \lambda(\phi) \divergenz{\pmbv}^2$ and 
\begin{align*}
    \Big( \F_{e}^T \fracdel{e}{\B_{e}} \F_{e} \Big) : 
    \Big( \F_{e}^{-1} \fracdel{e}{\B_{e}} \F_{e} \Big)
    = \trace\Big( \F_{e}^T \fracdel{e}{\B_{e}} \F_{e} 
    \F_{e}^T \Big(\fracdel{e}{\B_{e}}\Big)^T \F_{e}^{-T} \Big)
    = \abs{\fracdel{e}{\B_{e}} \F_{e}}^2,
\end{align*}
we obtain that the local dissipation inequality is fulfilled, i.e.
\begin{align}
\label{eq:constitutive_10}
    \begin{split}
    \calD 
    &= 2\eta(\phi) \abs{\D}^2 + \lambda(\phi) (\divergenz{\pmbv})^2
    + m(\phi) \abs{\nabla\mu}^2
    + n(\phi) \abs{\nabla \fracdel{e}{\sigma}}^2
    + \frac{2}{\tau(\phi)} \abs{\fracdel{e}{\B_{e}} \F_{e}}^2
    \geq 0.
    \end{split}
\end{align}
So, dissipation can be divided into the following processes: viscosity effects on the velocity (i.e.~$2\eta(\phi) \abs{\D}^2$),
changes in volume (i.e.~$\lambda(\phi) (\divergenz{\pmbv})^2$),
transport along $\nabla\mu$ and $\nabla \fracdel{e}{\sigma}$, and dissipation caused by viscoelastic relaxation (i.e.~$\frac{2}{\tau(\phi)} \abs{\fracdel{e}{\B_{e}} \F_{e}}^2$). 

We remark that multiplying \eqref{eq:constitutive_7} with $\F_{e}$ from the left and with $\F_{e}^T$ from the right yields a formula for $\F_{e} \D_{d} \F_{e}^T$, i.e.
\begin{align}
    \label{eq:constitutive_14}
    & \F_{e} \D_{d} \F^T_{e}
    = \frac{1}{\tau(\phi)} \fracdel{e}{\B_{e}} \F_{e} \F^T_{e}
    = \frac{1}{\tau(\phi)} \fracdel{e}{\B_{e}} \B_{e}.
\end{align}
Combining \eqref{eq:constitutive_14} and \eqref{eq:B_e} leads to the following constitutive equation for the left Cauchy--Green tensor:
\begin{align}
    \dtv \B_{e} 
    + \frac{1}{\tau(\phi)} \fracdel{e}{\B_{e}} \B_{e}
    = \nabla\pmbv \B_{e} 
    + \B_{e} (\nabla\pmbv)^T.
\end{align} 
This can be seen as a generalized viscoelastic model of Oldroyd-B type \cite{Oldroyd_1950}.

Instead of the constitutive assumption \eqref{eq:constitutive_7}, it is also possible to assume
\begin{subequations}
\begin{alignat}{2}
    \label{eq:constitutive_7b}
    &\D_{d} = \frac{1}{\tau(\phi)} \Big( \F_{e}^T \fracdel{e}{\B_{e}} \F_{e} \Big),
\end{alignat}
\end{subequations}
which, after multiplication with $\F_{e}$ from the left and with $\F_{e}^T$ from the right, using $\B_{e} = \F_{e} \F_{e}^T$ and noting \eqref{eq:B_e}, leads to
\begin{align}
    \dtv \B_{e} 
    + \frac{2}{\tau(\phi)} \B_{e} \fracdel{e}{\B_{e}} \B_{e}
    = \nabla\pmbv \B_{e} 
    + \B_{e} (\nabla\pmbv)^T.
\end{align}
This can be seen as a generalized version of the viscoelastic model of Giesekus \cite{giesekus_1982}.
In this case, the local dissipation is given by
\begin{align}
\label{eq:constitutive_10b}
    \begin{split}
    \calD 
    &= 2\eta(\phi) \abs{\D}^2 + \lambda(\phi) (\divergenz{\pmbv})^2
    + m(\phi) \abs{\nabla\mu}^2
    + n(\phi) \abs{\nabla \fracdel{e}{\sigma}}^2
    + \frac{2}{\tau(\phi)} \abs{\F_{e}^T \fracdel{e}{\B_{e}} \F_{e}}^2
    \geq 0,
    \end{split}
\end{align}
instead of \eqref{eq:constitutive_10}.
However, the focus of this work lies on the constitutive relation \eqref{eq:constitutive_7} leading to the viscoelastic model of Oldroyd-B type.

We now summarize all the constitutive assumptions from this section:
\begin{align}
    \label{eq:constitutive_9}
    \begin{split}
    &\J_e = \lambda_\sigma \J_\sigma + \lambda_\phi \J_\phi 
    - \dtv\phi \fracdel{e}{\nabla\phi} ,
    \quad\quad
    \J_\phi = - m(\phi) \nabla\lambda_\phi,
    \quad\quad
    \J_\sigma = -n(\phi) \nabla\lambda_\sigma,
    \\
    &c_{\pmbv} = \lambda_{\pmbv} = e - \lambda_\phi \phi - \lambda_\sigma \sigma + p,
    \quad\quad
    c_\phi = \lambda_\phi = \fracdel{e}{\phi} - \divergenz{\fracdel{e}{\nabla\phi}} = \mu,
    \quad\quad
    c_\sigma = \lambda_\sigma = \fracdel{e}{\sigma},
    \\
    &\Big( \S +   a(\phi,\nabla\phi,\sigma,\B_{e}) (\nabla\phi \otimes  \nabla\phi)   
    - 2 \fracdel{e}{\B_{e}} \B_{e} \Big)
    = 2 \eta(\phi)\D + \lambda(\phi)\divergenz{\pmbv}\I,
    \\
    & \D_{d} = \frac{1}{\tau(\phi)} \Big( \F_{e}^{-1} \fracdel{e}{\B_{e}} \F_{e} \Big),
    \quad \text{or} \quad
    \D_{d} = \frac{1}{\tau(\phi)} \Big( \F_{e}^T \fracdel{e}{\B_{e}} \F_{e} \Big).
    \end{split}
\end{align}


\subsection{Further aspects of modelling}
\subsubsection{The model equations}
From now on we suppress the index of $\B_{e}$, i.e.~we write $\B$ instead of $\B_e$, and we also write $\D(\pmbv)$ instead of $\D$ to point out the dependency on $\pmbv$. 
In the following, we assume a general energy density of the form
\begin{align}
\label{eq:constitutive_energy}
    e(\phi,\nabla\phi,\sigma,\pmbv,\B) 
    = f(\phi,\nabla\phi) 
    + N(\phi,\sigma)
    + W(\phi,\B)
    + \frac{1}{2} \rho \abs{\pmbv}^2.
\end{align}
The first term $f(\phi,\nabla\phi)$ in \eqref{eq:constitutive_energy} accounts for interfacial energy of the diffuse interface \cite{cahn_hilliard_1958}
which we assume to be of Ginzburg--Landau type:
\begin{align}
    f(\phi,\nabla\phi) = A \psi(\phi) + \frac{B}{2} \abs{\nabla\phi}^2,
\end{align}
where $\psi(\cdot)$ is a non-negative potential with equal minima at $\phi=\pm1$, and $A,B>0$ are constants. Usually we set $A=\frac{\beta}{\epsilon}$ and $B=\beta\epsilon$, where the constants $\beta, \epsilon>0$ are related to the surface tension and the interfacial thickness, respectively. 

The second term $N(\phi,\sigma)$ in \eqref{eq:constitutive_energy} describes the energy contribution due to the presence of the nutrient and the interaction between the tumour tissues and the nutrients, also see \cite{GarckeLSS_2016}. The third term $W(\phi,\B)$ in \eqref{eq:constitutive_energy} represents the elastic part of the energy which we additionally assume to depend on the type of material and hence on $\phi$. 
For the moment, both the nutrient and the elastic energy density are kept in a general form, but later, possible choices are given.
The last term in \eqref{eq:constitutive_energy} is the kinetic part of the energy.

With these choices we calculate
\begin{align}
    \begin{split}
    &\fracdel{e}{\phi} = A \psi'(\phi) + N_{,\phi}
    + W_{,\phi},
    \quad
    \fracdel{e}{\nabla\phi} = B \nabla\phi,
    \quad
    \fracdel{e}{\sigma} = N_{,\sigma},
    \quad
    \fracdel{e}{\B} 
    = W_{,\B},
    \quad
    a(\phi,\nabla\phi,\sigma,\B) = B,
    \end{split}
\end{align}
where $N_{,\phi}$, and $N_{,\sigma}$ denote the partial derivatives of $N(\phi,\sigma)$ with respect to $\phi$ and $\sigma$. For better readability, note that we sometimes suppress the arguments of $N_{,\phi}(\phi,\sigma), N_{,\sigma}(\phi,\sigma)$ and we write $N_{,\phi}, N_{,\sigma}$ instead. Similarly, we adopt the notation for $W$.
Next, we specify the constitutive relation for the full stress tensor $\overline\T= - p\I + \S$:
\begin{align}
    \label{eq:constitutive_11}
    \begin{split}
    \overline\T&= -p\I + 2\eta(\phi)\D(\pmbv) + \lambda(\phi) \divergenz{\pmbv}\I 
    + 2 W_{,\B}(\phi,\B) \B
    - B \nabla\phi \otimes\nabla\phi.
    \end{split}
\end{align}

Collecting all equations from above, the general viscoelastic model of Oldroyd-B type reads:
\begin{subequations}
\begin{align}
    \label{eq:phi0}
    \partial_t \phi + \divergenz{\phi \pmbv} 
    &= \divergenz{m(\phi) \nabla\mu} + \Gamma_\phi,
    \\
    \label{eq:mu0}
    \mu &= A \psi'(\phi) - B \Delta\phi + N_{,\phi}(\phi,\sigma)
    + W_{,\phi}(\phi,\B),
    \\
    \label{eq:sigma0}
    \partial_t \sigma + \divergenz{\sigma \pmbv} 
    &= \divergenz{n(\phi) \nabla N_{,\sigma}(\phi,\sigma)} - \Gamma_\sigma,
    \\
    \label{eq:div_v0}
    \divergenz{\pmbv} &= \Gamma_{\pmbv},
    \\
    \label{eq:v0}
    \rho \partial_t \pmbv + \rho (\pmbv\cdot \nabla)\pmbv 
    &= \divergenz{\T(\phi,\pmbv, p, \B)}
    - \divergenz{B \nabla\phi \otimes \nabla\phi},
    \\
    \label{eq:B0}
    \partial_t \B + (\pmbv\cdot\nabla)\B
    + \frac{1}{\tau(\phi)} \T_{\mathrm{el}}(\phi,\B) 
    &= \nabla\pmbv \B + \B (\nabla\pmbv)^T.
\end{align}
\end{subequations}

For future reference, the full viscoelastic stress tensor is denoted by 
\begin{align}
    \label{eq:T_viscoelastic0}
    \T(\phi,\pmbv, p, \B)
    \coloneqq \T_{\mathrm{visc}}(\phi,\pmbv,p) + \T_{\mathrm{el}}(\phi,\B),
\end{align}
where the viscous and the elastic parts of the stress tensor are defined as
\begin{align}
    \T_{\mathrm{visc}}(\phi,\pmbv,p)
    &\coloneqq 
    \eta(\phi) \big( \nabla\pmbv + (\nabla\pmbv)^T \big) 
    + \lambda(\phi) \divergenz{\pmbv}\I - p\I,
    \\
    \label{eq:T_elastic0}
    \T_{\mathrm{el}}(\phi,\B) &\coloneqq 2W_{,\B}(\phi,\B) \B.
\end{align}
Note that $\T(\phi,\pmbv, p, \B)$ corresponds to $\overline\T$ without the last term in \eqref{eq:constitutive_11}.

As remarked in the derivation, a viscoelastic description of Giesekus type is also possible, which then leads to the system of equations \eqref{eq:phi0}--\eqref{eq:v0} together with the constitutive equation
\begin{align}
    \label{eq:giesekus0}
    \partial_t \B + (\pmbv\cdot\nabla)\B
    + \frac{1}{\tau(\phi)} \B \T_{\mathrm{el}}(\phi,\B)
    &= \nabla\pmbv \B + \B (\nabla\pmbv)^T.
\end{align}
However, the focus in this work lies on the viscoelastic description of Oldroyd-B type.


\subsubsection{Reformulations of the pressure}
We consider the following two reformulations of the pressure leading to a variant of \eqref{eq:v0}. For more examples, see \cite{ebenbeck_garcke_nurnberg_2020, GarckeLSS_2016}. 

\begin{itemize}
\item
Using the fact that $\nabla\big( \frac{B}{2} \abs{\nabla\phi}^2 \big) = \divergenz{B \nabla\phi \otimes \nabla\phi} - B \Delta\phi\nabla\phi$ and defining $q \coloneqq p + f(\phi,\nabla\phi)+N(\phi,\sigma)$ yields
\begin{align*}
    \nabla q 
    = \nabla p + (\mu - W_{,\phi}) \nabla\phi 
    + N_{,\sigma}\nabla\sigma 
    + \divergenz{B\nabla\phi\otimes\nabla\phi}. 
\end{align*}
We can hence write \eqref{eq:v0} as
\begin{align}
    \label{eq:v0b}
    \rho \partial_t \pmbv + \rho (\pmbv\cdot\nabla)\pmbv 
    =  \divergenz{\T(\phi,\pmbv,q,\B)} 
    + (\mu-W_{,\phi}) \nabla\phi + N_{,\sigma}\nabla\sigma.
\end{align}
Let us mention that the system \eqref{eq:phi0}--\eqref{eq:B0} with \eqref{eq:v0} replaced by \eqref{eq:v0b} matches with the general viscoelastic model \ref{P}. Moreover, replacing \eqref{eq:v0} by \eqref{eq:v0b} makes it possible that the convection terms in \eqref{eq:phi0} and \eqref{eq:sigma0} cancel out within specific testing procedures.

\item The following reformulation is of great importance when dealing with quasi-static nutrient equations, see \cite{ebenbeck_garcke_nurnberg_2020}.
Setting $q \coloneqq p + f(\phi,\nabla\phi)$ yields
\begin{align*}
    \nabla q 
    = \nabla p + (\mu-N_{,\phi} - W_{,\phi}) \nabla\phi 
    + \divergenz{B\nabla\phi\otimes\nabla\phi},
\end{align*}
so that \eqref{eq:v0} becomes 
\begin{align}
\begin{split}
    \label{eq:v0a}
    \rho \partial_t \pmbv + \rho (\pmbv\cdot\nabla)\pmbv 
    =  \divergenz{\T(\phi,\pmbv,q,\B)} 
    + (\mu - N_{,\phi} - W_{,\phi})\nabla\phi.
\end{split}
\end{align}

\end{itemize}



\subsubsection{A general energy identity}
In the following, we derive a general energy identity for the viscoelastic model of Oldroyd-B type \eqref{eq:phi0}--\eqref{eq:B0}, 
where we write \eqref{eq:v0b} instead of \eqref{eq:v0} and we write $p$ instead of $q$, such that the convection terms in \eqref{eq:phi0} and \eqref{eq:sigma0} cancel out within the following testing procedure. 

Let us temporarily assume that there exists a sufficiently smooth solution of the above system. We multiply \eqref{eq:phi0} with $\mu$, \eqref{eq:mu0} with $-\partial_t\phi$ and \eqref{eq:sigma0} with $N_{,\sigma}$, integrate over $\Omega$ and use Green's formula. We then obtain:
\begin{align}
    \label{eq:gen_eq_1}
    0&=\int_\Omega \partial_t\phi \mu 
    + \mu\nabla\phi\cdot\pmbv + \phi\Gamma_{\pmbv}\mu 
    + m(\phi) \abs{\nabla\mu}^2
    - \Gamma_\phi \mu \dx
    - \int_{\partial\Omega} m(\phi) \mu \nabla\mu\cdot\pmbn \dH^{d-1} ,
    \\
    \label{eq:gen_eq_2}
    0&= - \int_\Omega \partial_t\phi(\mu-N_{,\phi}-W_{,\phi}) \dx
    + \ddt \int_\Omega A\psi(\phi) + \frac{B}{2}\abs{\nabla\phi}^2 \dx
    - \int_{\partial\Omega} B\partial_t\phi \nabla\phi\cdot\pmbn \dH^{d-1},
    \\
    \label{eq:gen_eq_3}
    \begin{split}
    0&= \int_\Omega \partial_t\sigma N_{,\sigma} 
    + N_{,\sigma} \nabla\sigma \cdot \pmbv
    + N_{,\sigma} \sigma \Gamma_{\pmbv} 
    + n(\phi) \abs{\nabla N_{,\sigma}}^2
    + \Gamma_\sigma N_{,\sigma} \dx
    \\
    &\quad - \int_{\partial\Omega} n(\phi) N_{,\sigma} \nabla N_{,\sigma} \cdot\pmbn \dH^{d-1} .
    \end{split}
\end{align}
Next, we multiply \eqref{eq:v0b} with $\pmbv$ and integrate over $\Omega$ and use Green's formula so that we have
\begin{align}
    \label{eq:gen_eq_4}
    \begin{split}
    0 &=
    \int_\Omega \ddt\left(\frac{1}{2} \rho \abs{\pmbv}^2 \right)
    + \rho (\pmbv\cdot\nabla)\pmbv \cdot \pmbv
    + 2\eta(\phi) \abs{\D(\pmbv)}^2 + \lambda(\phi) (\divergenz{\pmbv})^2
    - p\Gamma_{\pmbv} \dx
    \\
    &\quad + \int_\Omega(- (\mu-W_{,\phi})\nabla\phi - N_{,\sigma} \nabla\sigma) \cdot\pmbv
    + 2 (W_{,\B} \B) : \nabla\pmbv \dx
    - \int_{\partial\Omega} \big(\T(\phi,\pmbv,p,\B)\pmbn \big)\cdot \pmbv \dH^{d-1}.
    \end{split}
\end{align}
Here we used that $\D(\pmbv) :\nabla\pmbv = \D(\pmbv) :\D(\pmbv)$ and $\divergenz{\pmbv}\I : \nabla\pmbv = (\divergenz{\pmbv})^2$.
After that, we multiply \eqref{eq:B0} with $W_{,\B}$ and we integrate over $\Omega$ and apply Green's formula. This yields
\begin{align}
    \label{eq:gen_eq_5}
    0 &= \int_\Omega \Big( \partial_t\B 
    + (\pmbv\cdot\nabla) \B
    - \nabla\pmbv\B - \B(\nabla\pmbv)^T 
    + \frac{2}{\tau(\phi)} W_{,\B} \B \Big) : W_{,\B} \dx.
\end{align}
For the reader's convenience, we now note some useful identities concerning the velocity and the Cauchy--Green tensor:
\begin{align*}
    \int_\Omega \rho (\pmbv\cdot\nabla)\pmbv \cdot \pmbv
    &= 
    - \int_\Omega \Gamma_{\pmbv} \Big(\frac{1}{2}\rho\abs{\pmbv}^2\Big) \dx
    + \int_{\partial\Omega} \pmbv\cdot\pmbn \Big(\frac{1}{2}\rho\abs{\pmbv}^2\Big)\dx
    \\
    \big(\nabla\pmbv\B + \B(\nabla\pmbv)^T\big) : W_{,\B}
    &= 2 (W_{,\B} \B) : \nabla\pmbv,
    \\
    \dtv W(\phi,\B) 
    &= \dtv \phi W_{,\phi}
    + \dtv \B : W_{,\B},
    \\
    \int_\Omega (\pmbv\cdot\nabla) W(\phi,\B) \dx
    &= - \int_\Omega \Gamma_{\pmbv} W(\phi,\B) \dx
    + \int_{\partial\Omega} \pmbv\cdot\pmbn W(\phi,\B)\dx,
    \\
    (W_{,\B} \B) : W_{,\B} & \geq 0, \quad \text{ if } \B \text{ is positive definite.}
\end{align*}

Collecting all equations \eqref{eq:gen_eq_1}--\eqref{eq:gen_eq_5}, we obtain the general energy identity for the viscoelastic model for tumour growth:
\begin{align}
    \label{eq:energy0}
    \begin{split}
    0 &=
    \ddt \Big( \int_\Omega A\psi(\phi) + \frac{B}{2} \abs{\nabla\phi}^2 + N(\phi,\sigma) + W(\phi,\B) 
    + \frac{1}{2} \rho \abs{\pmbv}^2 \dx \Big)
    \\
    &\quad
    + \int_\Omega m(\phi) \abs{\nabla\mu}^2 
    + n(\phi)\abs{\nabla N_{,\sigma}}^2 
    + 2 \eta(\phi) \abs{\D(\pmbv)}^2 
    + \lambda(\phi)(\divergenz{\pmbv})^2
    + \frac{2}{\tau(\phi)} (W_{,\B} \B):W_{,\B} \dx 
    \\
    &\quad
    + \int_\Omega - \Gamma_\phi \mu + \Gamma_\sigma N_{,\sigma} 
    + \Big(\mu\phi + N_{,\sigma}\sigma - p - W(\phi,\B) 
    - \frac{1}{2} \rho \abs{\pmbv}^2 \Big) \Gamma_{\pmbv} \dx
    \\
    &\quad
    - \int_{\partial\Omega} m(\phi) \mu\nabla\mu \cdot\pmbn 
    + n(\phi) N_{,\sigma} \nabla N_{,\sigma} \cdot\pmbn 
    + B \partial_t\phi \nabla\phi\cdot\pmbn \dH^{d-1}
    \\
    &\quad
    + \int_{\partial\Omega} \pmbn\cdot\pmbv \Big( 
    W(\phi,\B) + \frac{1}{2} \rho \abs{\pmbv}^2 \Big)
    - \big(\T(\phi,\pmbv,p,\B)\pmbn \big) \cdot\pmbv \dH^{d-1}.
    \end{split}
\end{align}
Note that in order to study existence theory, there are several difficulties that arise from this general identity and heavily depend on the choices for the potential $\psi(\phi)$, the energy densities $N(\phi,\sigma), W(\phi,\B)$, the source terms $\Gamma_\phi, \Gamma_\sigma, \Gamma_{\pmbv}$, the functions $m(\phi), n(\phi), \eta(\phi), \lambda(\phi), \tau(\phi)$ and the initial and boundary conditions.



\subsubsection{Initial and boundary conditions}
For $\phi$, $\sigma$, $\pmbv$ and $\B$, we impose the initial conditions
\begin{align}
    \phi(\cdot,0) = \phi_0, \quad 
    \sigma(\cdot,0) = \sigma_0, \quad 
    \pmbv(\cdot,0) = \pmbv_0, \quad
    \B(\cdot,0) = \B_0 \quad \text{ a.e.~in } \Omega.
\end{align}
We prescribe homogeneous Neumann boundary conditions on $\partial\Omega$ for the phase field variable and the chemical potential, i.e.
\begin{align}
    \nabla\phi\cdot\pmbn = \nabla\mu\cdot\pmbn = 0
    \quad \text{ a.e.~on } \partial\Omega \times (0,T).
\end{align}
For the nutrient we prescribe Robin-type boundary conditions
\begin{align}
    n(\phi) \nabla N_{,\sigma} \cdot \pmbn = K(\sigma_\infty - \sigma)
    \quad \text{ a.e.~on } \partial\Omega \times (0,T),
\end{align}
where the constant $K\geq0$ is referred to as the boundary permeability and $\sigma_\infty$ denotes a given nutrient supply at the boundary.

The last boundary condition is depending on the choice of $\Gamma_{\pmbv}$. In this work, we consider $\Gamma_{\pmbv}=0$. Hence, we prescribe no-slip (homogeneous Dirichlet) boundary conditions for the velocity, i.e.
\begin{align}
    \pmbv = \pmb 0
    \quad \text{ a.e.~on } \partial\Omega \times (0,T).
\end{align}
In the case $\Gamma_{\pmbv}\not=0$, we prescribe the following boundary condition for $\T$,
\begin{align}
    \Big(\T(\phi,\pmbv,p,\B) - W(\phi,\B) \I - \frac{1}{2} \rho \abs{\pmbv}^2 \I \Big) \pmbn = \pmb 0
    \quad \text{ a.e.~on } \partial\Omega \times (0,T),
\end{align}
so that the last line in the general energy identity \eqref{eq:energy0} vanishes.

In the case of no-slip boundary conditions for the velocity, we recall that no further boundary conditions for the Cauchy--Green tensor $\B$ are needed, as the evolution equation \eqref{eq:B0} is a hyperbolic partial differential equation of first order and has no incoming characteristics at the boundary.


\subsubsection{Specific choices for the source terms}
Now, we explain possible specifications for the source terms $\Gamma_\phi,\Gamma_\sigma,\Gamma_{\pmbv}$.

\begin{itemize}
\item Usually the source terms $\Gamma_\phi$ and $\Gamma_{\pmbv}$ are closely related. In particular, 
\begin{align}
    \Gamma_\phi \coloneqq \frac{1}{\bar\rho_1}\Gamma_{1} - \frac{1}{\bar\rho_{-1}}\Gamma_{-1},
    \qquad 
    \Gamma_{\pmbv} \coloneqq \frac{1}{\bar\rho_1}\Gamma_{1} + \frac{1}{\bar\rho_{-1}}\Gamma_{-1},
\end{align}
where $\bar\rho_1,\bar\rho_{-1}$ are the mass densities of the tumour cells and healthy cells, respectively, and $\Gamma_{1}, \Gamma_{-1}$ are source or sink terms in the mass balance laws for the single components of the mixture, see \cite{ebenbeck_garcke_nurnberg_2020, GarckeLSS_2016}. By the assumption of matching mass densities, we have $\rho= \bar\rho_1 = \bar\rho_{-1}$ and hence
\begin{align}
    \Gamma_\phi = \frac{1}{\rho} (\Gamma_{1} - \Gamma_{-1}), \qquad \Gamma_{\pmbv} = \frac{1}{\rho}( \Gamma_{1} + \Gamma_{-1}).
\end{align}
A common choice of $\Gamma_{\pmbv}$ is obtained by assuming no gain or loss of mass locally, i.e.~$\Gamma_1 \coloneqq - \Gamma_{-1}$, which implicates 
\begin{align}
    \label{source_1a}
    \Gamma_\phi = \frac{2}{\rho} \Gamma_1, \qquad \Gamma_{\pmbv} = 0.
\end{align}
On the other hand, setting $\Gamma_{-1} \coloneqq 0$ yields
\begin{align}
    \label{source_1b}
    \Gamma_\phi = \Gamma_{\pmbv}  = \frac{1}{\rho} \Gamma_1.
\end{align}

\item
Motivated by linear kinetics, Garcke and co-authors \cite{GarckeLSS_2016} suggested
\begin{align}
\label{eq:source_2}
    &\Gamma_\phi(\phi,\sigma) 
    \coloneqq ( \calP \sigma - \calA ) h(\phi),
    \qquad
    \Gamma_\sigma(\phi,\sigma) 
    \coloneqq \calC \sigma  h(\phi)
    + \calB (\sigma_B - \sigma), 
\end{align}
where $\calP, \calA, \calC$ denote the proliferation rate, apoptosis rate and consumption rate.
Moreover, $h(\cdot)$ is an interpolation function with $h(-1) = 0$ and $h(1) = 1$ which ensures that proliferation, apoptosis and nutrient consumption only take place in the tumour phase. The simplest example ist $h(\phi) = \tfrac{1}{2} (1+\phi)$.
Besides, $\calB (\sigma_B - \sigma)$ models the nutrient supply from an existing vasculature.

\item To account for the influence of mechanical stresses on tissue growth, the authors of \cite{garcke_lam_signori_2021_optimal_control} proposed to scale the proliferation term $ \calP \sigma  h(\phi)$ in \eqref{eq:source_2} with $\tilde f(\T_{\mathrm{el}}) \coloneqq (1 + \abs{\T_{\mathrm{el}}}^2)^{-1/2}$,
which decreases when elastic stresses increase.
This motivates to introduce the choice 
\begin{align}
\label{eq:source_3}
    \Gamma_\phi(\phi,\sigma,\B) \coloneqq \big( \calP \sigma f(\phi,\B) -\calA \big) h(\phi) , 
    \quad \text{where} \quad
    f(\phi,\B) \coloneqq \left(1 + \abs{\T_{\mathrm{el}}(\phi,\B)}^2\right)^{-\frac{1}{2}},
\end{align}
where $h(\cdot), \calP, \calA$ are as in \eqref{eq:source_2}.

\item
Based on linear phenomenological laws for chemical reactions, the authors of \cite{hawkins_2012} proposed to take 
\begin{align}
\label{eq:source_4}
    &\Gamma_\phi(\phi,\mu,\sigma) 
    = \Gamma_\sigma(\phi,\mu,\sigma) 
    \coloneqq P(\phi) \big( N_{,\sigma}(\phi,\sigma) - \mu \big),
\end{align}
with a non-negative proliferation function $P(\cdot)$, e.g., $P(\phi) = \max\left\{0, \delta P_0 (1+\phi)\right\}$, where $\delta, P_0$ are positive constants. 
For a different example of $P(\cdot)$, we refer to, e.g., \cite{ebenbeck_garcke_nurnberg_2020}.

\end{itemize}


\subsubsection{Specific choices for the nutrient and elastic energy density}
In the following, we specify the nutrient energy density and give several examples for the elastic energy density.

\begin{itemize}
\item 
In the literature, the nutrient energy density usually takes the form 
\begin{align}
    \label{eq:energy_nutrient}
    N(\phi,\sigma) \coloneqq \frac{\chi_\sigma}{2} \abs{\sigma}^2 
    + \chi_\phi \sigma(1-\phi), 
\end{align}
with
\begin{align*}
    N_{,\phi}(\phi,\sigma) = -\chi_\phi \sigma,
    \qquad 
    N_{,\sigma}(\phi,\sigma) = \chi_\sigma \sigma - \chi_\phi \phi.
\end{align*}
The first term $\frac{\chi_\sigma}{2} \abs{\sigma}^2$ increases the energy in the presence of nutrients, where $\chi_\sigma > 0$ denotes the diffusivity of the nutrient. 
The second term $\chi_\phi \sigma(1-\phi)$ can be regarded as chemotaxis energy which accounts for interactions between the tumour and the nutrient. Here, the constant $\chi_\phi \geq 0$ can be seen as a sensitivity parameter for chemotaxis and active uptake mechanisms which favours unstable tumour growth \cite{GarckeLSS_2016}. Let us point out that the nutrient energy density can have a negative sign in general if $\chi_\phi \not= 0$, which is one difficulty in the derivation of suitable \textit{a priori} estimates from the general energy identity \eqref{eq:energy0}.

\item 
Physically motivated by the theory of constitutive relations for isotropic compressible elastic materials, where large elastic stresses are penalized, the elastic energy density is supposed to satisfy
\begin{align}
    W(\phi,\B) \to +\infty, \quad \text{if} \quad 
    \begin{cases}
    \abs{\B} \to +\infty, \\
    \det(\B) \to 0.
    \end{cases}
\end{align}
Hence, an infinite amount of energy is required such that the material can be expanded to infinite volume or compressed to a single point \cite{horgan_2004_constitutive}.
Note that $\B\coloneqq \B_e = \F_e \F_e^T$ is always symmetric and positive semi-definite by definition while the elastic part of the deformation gradient $\F_e$ is not symmetric in general.

\item 
An example for the elastic energy density we have in mind is
\begin{align}
\label{eq:energy_oldroyd}
    W(\phi,\B)
    \coloneqq \frac{1}{2} \kappa(\phi) \trace (\B)  
    - \frac{1}{2} \kappa_0(\phi)\ln(\det\B),
\end{align}
with
\begin{align*}
    W_{,\phi}(\phi,\B) = \frac{1}{2} \kappa'(\phi) \trace (\B)  - \frac{1}{2} \kappa_0'(\phi) \ln(\det\B),
    \qquad
    W_{,\B}(\phi,\B) = \frac{1}{2} \kappa(\phi) \I - \frac{1}{2} \kappa_0(\phi) \B^{-1},
\end{align*}
where $\kappa(\phi), \kappa_0(\phi) > 0$ denote elasticity parameter functions
depending on the material and $\I\in\R^{d\times d}$ denotes the identity matrix.
Hence, the elastic stress tensor is $\T_{\mathrm{el}}(\phi,\B) = \kappa(\phi) \B - \kappa_0(\phi) \I $ and \eqref{eq:B0} is specified by
\begin{align}
\label{eq:oldroyd}
    \partial_t \B + (\pmbv\cdot\nabla)\B 
    +\frac{1}{\tau(\phi)} \left( \kappa(\phi) \B - \kappa_0(\phi) \I \right)
    &= \nabla\pmbv\B + \B (\nabla\pmbv)^T.
\end{align}

In the case $\kappa(\phi)=\kappa_0(\phi)$ and for fixed $\phi$, \eqref{eq:oldroyd} is exactly the classical viscoelastic Oldroyd-B equation \cite{barrett_boyaval_2009, malek_prusa_2018} for the left Cauchy--Green tensor $\B$.

\item 
For a given $b\gg1$, the elastic energy density for the viscoelastic FENE-P model \cite{barrett_2018_fene-p} reads
\begin{align}
\label{eq:energy_fene-p}
    W(\B) \coloneqq  - \frac{b}{2} \ln\left( 1-\frac{\trace\B}{b}\right) 
    - \frac{1}{2} \trace (\ln\B),
    \qquad \text{with} 
    \qquad W_{,\B}(\B) = \frac{1}{2} \left( 1 - \frac{\trace\B}{b} \right)^{-1} \I - \frac{1}{2} \B^{-1}.
\end{align}
Note that $W(\B) \to +\infty$ if $\trace(\B) \to b$ or if $\B$ becomes singular.
Here, the corresponding elastic stress tensor is $\T_{\mathrm{el}}(\B) = \left( 1 - \frac{\trace\B}{b} \right)^{-1} \B - \I$ and the constitutive law \eqref{eq:B0} reads
\begin{align}
\label{eq:fene-p}
    \partial_t \B + (\pmbv\cdot\nabla)\B 
    +\frac{1}{\tau(\phi)} \left( \left( 1 - \frac{\trace\B}{b} \right)^{-1} \B - \I \right)
    &= \nabla\pmbv\B + \B (\nabla\pmbv)^T.
\end{align}
Moreover, the Oldroyd-B equation \eqref{eq:oldroyd} with $\kappa(\phi)=\kappa_0(\phi) = 1$ can be recovered by formally sending $b\to+\infty$.

\item The authors of \cite{Brunk_Lukacova_2021} study a generalized viscoelastic Peterlin model for phase separation which is based on the elastic energy density
\begin{align}
\label{eq:energy_gen_peterlin}
    W(\B) \coloneqq \frac{1}{4} \trace(\B)^2 - \frac{1}{2} \trace(\ln\B)
    \qquad \text{with} \qquad
    W_{,\B}(\B) = \frac{1}{2} \trace(\B) \I - \frac{1}{2}  \B^{-1},
\end{align}
and the elastic stress tensor $\T_{\mathrm{el}}(\B) = \trace(\B) \B - \I$. Moreover, the tensor $\B$ satisfies a generalized evolution equation of the form
\begin{align}
\label{eq:gen_peterlin}
\partial_t \B + (\pmbv\cdot\nabla)\B 
    + f(\phi) g(\trace\B) \left( \trace(\B) \B - \I \right)
    &= \nabla\pmbv\B + \B (\nabla\pmbv)^T,
\end{align}
where $(f(\phi)g(\trace\B))^{-1}$ denotes a generalized relaxation time depending on the phase field variable $\phi$ and the trace of $\B$.
Besides, a more generalized approach has been studied in \cite{Lukacova_2017} which includes several viscoelastic models.

\end{itemize}

The elastic energy densities \eqref{eq:energy_oldroyd}, \eqref{eq:energy_fene-p}, \eqref{eq:energy_gen_peterlin} in the one dimensional case are visualized in Figure \ref{fig:elastic_energy}.

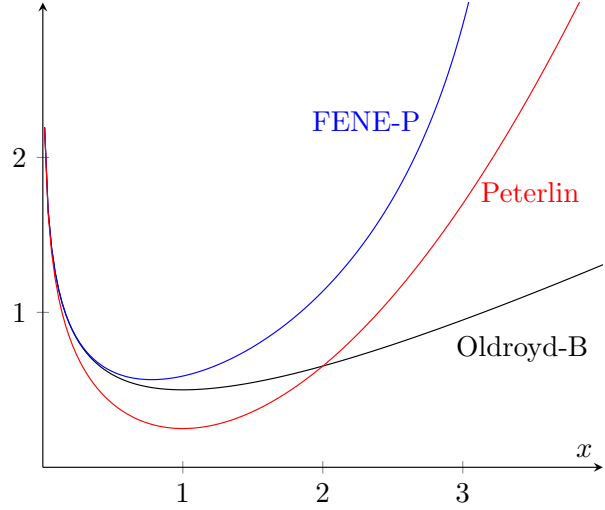
\begin{figure}[ht]
\centering
\begin{tikzpicture}[scale=\textwidth/15.2cm]
\begin{axis}[
axis x line=middle, 
axis y line=middle,
xmin=0,xmax=4,
ymin=0,ymax=3,
xlabel=$x$, 
xtick={1, 2, 3},
xticklabels={1, 2, 3},
ytick={1, 2},
yticklabels={1, 2},
]
\addplot[samples=400, smooth, sharp plot]
{0.5*x - 0.5*ln(x)}
node [pos=0.65, below right] {Oldroyd-B};
\addplot[samples=400, smooth, blue, sharp plot]
{ - 0.5*3.5* ln(1-x/3.5) - 0.5*ln(x)}
node [pos=0.34, left] {FENE-P};
\addplot[samples=400, smooth, red, sharp plot]
{0.25*x*x - 0.5*ln(x)}
node [pos=0.55, right] {Peterlin};
\end{axis}
\end{tikzpicture}
\caption{The elastic energy densities for the Oldroyd-B model \eqref{eq:energy_oldroyd} with $\kappa=\kappa_0=1$ (black), the FENE-P model \eqref{eq:energy_fene-p} with $b=3.5$ (blue) and the generalized Peterlin model \eqref{eq:energy_gen_peterlin} (red) in the one dimensional case.}
\label{fig:elastic_energy}
\end{figure}


\subsubsection{Including growth in the equation for \texorpdfstring{$\B$}{B}}
Mechanical stresses increase when tumour cells proliferate \cite{ambrosi_2009}. Therefore, instead of \eqref{eq:B0}, we may consider
\begin{align}
    \label{eq:B_growth}
    \partial_t \B + (\pmbv\cdot\nabla)\B
    + \frac{1}{\tau(\phi)} \T_{\mathrm{el}}(\phi,\B)
    &= \nabla\pmbv \B + \B (\nabla\pmbv)^T
    - \gamma(\phi,\mu,\sigma,\B) \B,
\end{align}
where the scalar function $\gamma(\phi,\mu,\sigma,\B)$ acts as a source or sink term for the left Cauchy--Green tensor and can depend on $\phi$, $\mu$, $\sigma$ and $\B$.

This source term can be derived from the multiplicative decomposition
\begin{align}
    \F = \F_{e} \F_{d} \F_{g},
\end{align}
where $\F_{g}$ describes deformation by growth \cite{ambrosi_2009}. Assuming spherical growth, i.e.~$\F_{g} = g \I$, then, analogously to \eqref{eq:F_e}, we obtain
\begin{align}
    \dtv \F_{e} = \L \F_{e} - \F_{e} \L_{d} - \F_{e} (\dtv g) \frac{1}{g},
\end{align}
and, as $\B_{e}=\F_{e}\F_{e}^T$, we have
\begin{align}
    \dtv \B_{e} = \L \B_{e} + \B_{e} \L^T 
    - 2 \F_{e} \D_{d} \F_{e}^T 
    - 2 \B_{e} (\dtv g) \frac{1}{g} ,
\end{align}
which coincides with \cite[eq.~(3.29)]{ambrosi_2009}. Then, \eqref{eq:B_growth} is recovered with
\begin{align}
    \D_{d} = \frac{1}{\tau(\phi)} \Big( \F_{e}^{-1} \fracdel{e}{\B_{e}} \F_{e} \Big), 
    \qquad
    \dtv g = \frac{1}{2} g \gamma(\phi,\mu,\sigma,\B_{e}).
\end{align}

A different way to obtain \eqref{eq:B_growth} can be motivated by the constitutive choice
\begin{align}
\label{eq:constitutive_growth}
    \D_{d} &= \frac{1}{\tau(\phi)} \Big( \F_{e}^{-1} \fracdel{e}{\B_{e}} \F_{e} \Big) 
    + \frac{1}{2} \gamma(\phi,\mu,\sigma,\B_{e}) \I, 
\end{align}
instead of \eqref{eq:constitutive_7}, and by inserting this into \eqref{eq:B_e}. 
On the right-hand side of \eqref{eq:constitutive_growth}, the first term accounts for stress relaxation while the second term is responsible for growth induced stresses.

In \cite{ambrosi_2009}, 
it has been suggested that $\gamma$ is proportional to the source function $\Gamma_\phi$. 
Therefore, we propose the choice
\begin{align}
    \gamma(\phi,\mu,\sigma,\B) = c \Gamma_\phi(\phi,\mu,\sigma,\B),
\end{align}
where $\Gamma_\phi$ can be given by \eqref{eq:source_2}, \eqref{eq:source_3} or \eqref{eq:source_4}, and $c\in\R$. 



\subsection{Variants of the model}
Now, we present several variants of the model \eqref{eq:phi0}--\eqref{eq:B0} and exemplify the motivation of these variants with strategies of related works in the literature.

\subsubsection{Limit of a small Reynolds number}
In biological processes, the Reynolds number is often very small. Then, a non-dimensionalization argument motivates to neglect the terms $\rho \partial_t\pmbv + \rho (\pmbv\cdot\nabla)\pmbv$ in the momentum equation. Hence, we introduce the viscoelastic model with quasi-static momentum equation, which is the system \eqref{eq:phi0}--\eqref{eq:B0} with \eqref{eq:v0} replaced by 
\begin{align}
    \label{eq:v0_Brinkman}
    - \divergenz{2\eta(\phi) \D(\pmbv) 
    + \lambda(\phi) \divergenz{\pmbv}\I}
    + \nabla p
    =  \divergenz{\T_{\mathrm{el}}(\phi,\B)} 
    + (\mu-W_{,\phi}) \nabla\phi + N_{,\sigma}\nabla\sigma.
\end{align}
In absence of the elastic effects, this model corresponds to a special case of the Cahn--Hilliard--Brinkman model for tumour growth which has been extensively studied in, e.g., \cite{ebenbeck_2019_analysis, ebenbeck_2019_singular_limit, ebenbeck_garcke_nurnberg_2020, ebenbeck_2020_optimal_control, ebenbeck_2019_medication, ebenbeck_2021_singular_potentials}.

\subsubsection{Limit of a short nutrient diffusion timescale}

From the modelling point of view, sometimes a quasi-static nutrient equation instead of \eqref{eq:sigma0} seems realistic since the timescale of nutrient diffusion can be quite small compared to the tumour doubling timescale. 
Such approaches have been introduced for related models in the literature, e.g., for the Cahn--Hilliard--Brinkman model \cite{ebenbeck_garcke_nurnberg_2020} or the Cahn--Hilliard--Darcy model \cite{GarckeLSS_2016}. 

Hence, we introduce the viscoelastic model with quasi-static nutrients which corresponds to \eqref{eq:phi0}--\eqref{eq:B0} with \eqref{eq:sigma0} replaced by
\begin{align}
    \label{eq:sigma0_quasistatic}
    0&= \divergenz{n(\phi) \nabla N_{,\sigma}(\phi,\sigma)} - \Gamma_\sigma(\phi,\sigma).
\end{align}

\subsubsection{Interpolation between different rheologies}
The main concept of viscoelastic models is that both viscous and elastic effects are taken into account. In the context of tumour growth, Bresch and co-authors \cite{bresch_2009} proposed a viscoelastic multiphase tumour model of Oldroyd-B type in presence of healthy cells, tumour cells and extracellular liquids, where the material parameters depend on the phases. For example, healthy cells are considered to be more elastic, extracellular liquids are supposed to be fully viscous and tumour cells are assumed to combine both elastic and viscous properties.

We now illustrate the idea of different material laws on the basis of the model \eqref{eq:phi0}--\eqref{eq:B0} with the help of suitable choices of the viscosities $\eta(\phi), \lambda(\phi)$ and the relaxation time $\tau(\phi)$. We can account for a Newtonian fluid without elastic stresses by sending the relaxation time $\tau(\phi)$ to zero, which leads to \eqref{eq:phi0}--\eqref{eq:B0} with \eqref{eq:v0}--\eqref{eq:B0} replaced by
\begin{align}
    \rho \partial_t \pmbv + \rho (\pmbv\cdot \nabla)\pmbv 
    &= \divergenz{\T_{\mathrm{visc}}(\phi,\pmbv, p)}
    - \divergenz{B \nabla\phi \otimes \nabla\phi},
    \\
    \T_{\mathrm{el}}(\phi,\B) 
    &= 0.
\end{align}
Besides, we can allow a viscoelastic description of Maxwell type by neglecting the viscosities $\eta(\phi),\lambda(\phi)$. Hence, the model corresponds to \eqref{eq:phi0}--\eqref{eq:B0} with \eqref{eq:v0}--\eqref{eq:B0} replaced by
\begin{align}
    \rho \partial_t \pmbv + \rho (\pmbv\cdot \nabla)\pmbv + \nabla p
    &= \divergenz{\T_{\mathrm{el}}(\phi, \B)}
    - \divergenz{B \nabla\phi \otimes \nabla\phi},
    \\
    \partial_t \B + (\pmbv\cdot\nabla)\B
    + \frac{1}{\tau(\phi)} \T_{\mathrm{el}}(\phi,\B) 
    &= \nabla\pmbv \B + \B (\nabla\pmbv)^T.
\end{align}
Moreover, by sending the relaxation time $\tau(\phi)$ to infinity, we obtain the viscoelastic material law of Kelvin--Voigt type and hence recover the Oldroyd-B equation with infinite Weissenberg number \eqref{eq:B0_infinite_weissenberg} from \eqref{eq:B0} in the limit $\tau(\phi)\to\infty$, i.e.~the model corresponds to \eqref{eq:phi0}--\eqref{eq:div_v0} with
\begin{align}
    \rho \partial_t \pmbv + \rho (\pmbv\cdot \nabla)\pmbv 
    &= \divergenz{\T(\phi,\pmbv, p, \B)}
    - \divergenz{B \nabla\phi \otimes \nabla\phi},
    \\
    \partial_t \B + (\pmbv\cdot\nabla)\B
    &= \nabla\pmbv \B + \B (\nabla\pmbv)^T.
\end{align}
Further, the material law for an elastic solid can be obtained by neglecting the viscosities and assuming an infinite relaxation time. Hence, the model reads \eqref{eq:phi0}--\eqref{eq:div_v0} combined with
\begin{align}
    \rho \partial_t \pmbv + \rho (\pmbv\cdot \nabla)\pmbv + \nabla p
    &= \divergenz{\T_{\mathrm{el}}(\phi,\B)}
    - \divergenz{B \nabla\phi \otimes \nabla\phi},
    \\
    \partial_t \B + (\pmbv\cdot\nabla)\B
    &= \nabla\pmbv \B + \B (\nabla\pmbv)^T.
\end{align}

Of course, we can handle different material laws for the respective phases $\phi=1$ and $\phi=-1$ at once by specifying the viscosities and the relaxation time for the respective phases.
An overview can be found in Table \ref{tab:material_parameters}, which has been adapted from \cite{mokbel_abels_aland_2018}.

\begin{table}[ht]
\centering
    \begin{tabular}{|c|c|c|c|}
    viscosities
    & relaxation time 
    & material law
    & stress tensor
    \\
    $\eta(\phi),\lambda(\phi)$
    & $\tau(\phi)$
    & 
    & $\T(\phi,\pmbv,p,\B)$
    \\ \hline
    $*$  & 0 
    & Newtonian fluid 
    & $\T_{\mathrm{visc}}(\phi,\pmbv,p)$
    \\
    0 &  $*$ 
    & Maxwell (viscoelastic)
    & $- p \I + \T_{\mathrm{el}}(\phi,\B)$
    \\
    $*$ &  $+\infty$ 
    & Kelvin--Voigt (viscoelastic)
    & $\T_{\mathrm{visc}}(\phi,\pmbv,p) + \T_{\mathrm{el}}(\phi,\B)$
    \\
    0 &  $+\infty$ 
    & elastic solid 
    & $- p \I + \T_{\mathrm{el}}(\phi,\B)$
    \end{tabular}
\caption{Different material laws can be obtained by a different choice of the viscosities and the relaxation time, where `$*$' marks parameters that are given by the physical problem itself; adapted from \cite{mokbel_abels_aland_2018}.}
\label{tab:material_parameters}
\end{table}



\subsubsection{Evolution of the elastic stress tensor}
In the literature, viscoelastic models related to the works of Oldroyd \cite{Oldroyd_1950} or Giesekus \cite{giesekus_1982} are sometimes stated in terms of the elastic stress tensor $\T_{\mathrm{el}}(\phi,\B) = 2 W_{,\B}(\phi,\B)\B$ instead of the left Cauchy--Green tensor $\B$.
Therefore, we shortly explain how the evolution of the elastic stress tensor is resulting from the evolution equation of the left Cauchy--Green tensor $\B$ for the case $W(\B) = \frac{1}{2} \kappa \trace(\B - \ln\B)$, where, for simplicity, $\kappa$ is constant.  
Then, in the Oldroyd-B model, the evolution equation \eqref{eq:B0} for the Cauchy--Green tensor $\B$ is equivalent to the following evolution equation for the elastic stress tensor $\T_{\mathrm{el}} = \kappa(\B-\I)$:
\begin{subequations}
\begin{align}
    \label{eq:B0c}
    \partial_t \T_{\mathrm{el}} + (\pmbv\cdot\nabla) \T_{\mathrm{el}}
    + \frac{\kappa}{\tau(\phi)} \T_{\mathrm{el}} - 2\kappa\D(\pmbv)
    &= \nabla\pmbv \T_{\mathrm{el}} + \T_{\mathrm{el}} (\nabla\pmbv)^T,
\end{align}
while, in the Giesekus model, \eqref{eq:giesekus0} is equivalent to the following evolution equation for $\T_{\mathrm{el}}$:
\begin{align}
    \label{eq:giesekus_c}
    \partial_t \T_{\mathrm{el}} + (\pmbv\cdot\nabla)\T_{\mathrm{el}} 
    + \frac{1}{\tau(\phi)} \T_{\mathrm{el}}^2
    + \frac{\kappa}{\tau(\phi)} \T_{\mathrm{el}} - 2\kappa\D(\pmbv)
    &= \nabla\pmbv \T_{\mathrm{el}} + \T_{\mathrm{el}} (\nabla\pmbv)^T.
\end{align}
\end{subequations}
For more details concerning the calculation, we refer to \cite[eq.~(205)]{malek_prusa_2018} for the Oldroyd-B model and to \cite[eq.~(187)]{malek_prusa_2018} for the Giesekus model.

\section{A viscoelastic tumour model with stress diffusion}
\label{sec:Section3}

In this section, we consider a variant of the system \ref{P}, where we fix the mass density of the mixture as $\rho\coloneqq 1$ and we neglect local exchange of mass, i.e.~$\Gamma_{\pmbv}\coloneqq0$, see \eqref{source_1a}. 
From the modelling point of view, the term $\Gamma_\phi$ usually describes biological effects like proliferation or apoptosis of the tumour, whereas the term $\Gamma_\sigma$ models nutrient consumption of the tumour \cite{GarckeLSS_2016}. Moreover, elastic stresses are supposed to influence growth.
Hence, it makes sense to assume $\Gamma_\phi$ to depend on $\phi,\sigma,\B$, and $\Gamma_\sigma$ to depend on $\phi,\sigma$, respectively. 
Moreover, we choose the nutrient energy density \eqref{eq:energy_nutrient} and assume that the elastic energy is hence given by $W(\B) = \frac{1}{2} \kappa \trace\big( \B - \ln\B \big)$, which corresponds to \eqref{eq:energy_oldroyd} with the elasticity parameters $\kappa=\kappa_0$ not depending on $\phi$. 
Besides, we assume small stress diffusion effects, i.e., we add the dissipative term $+\alpha\Delta \B$
to the right-hand side of the Oldroyd-B equation, which improves the mathematical properties of the system. This mathematical regularization can physically be motivated from a nonlocal energy storage mechanism or a nonlocal entropy production mechanism, see, e.g., \cite{malek_2018_Oldroyd_diffusive}. 

Then, the mathematical system of our interest reads:

\medskip

\subsubsection*{Problem \ref{P_alpha}:}
\mylabelHIDE{P_alpha}{$(\pmbP_\alpha)$} 
For a given constant $\alpha>0$, consider the system in $\Omega\times (0,T)$
\begin{subequations}
\label{eq:system2}
\begin{align}
    \label{eq:phi2} 
    \partial_t \phi + \pmbv \cdot\nabla\phi 
    &= \divergenz{m(\phi) \nabla\mu} + \Gamma_\phi(\phi,\sigma,\B), 
    \\
    \label{eq:mu2}
    \mu &= A \psi'(\phi) - B \Delta\phi - \chi_\phi \sigma, 
    \\
    \label{eq:sigma2}
    \partial_t \sigma + \pmbv \cdot\nabla\sigma 
    &= \divergenz{n(\phi) \nabla (\chi_\sigma \sigma - \chi_\phi \phi)} - \Gamma_\sigma(\phi,\sigma), 
    \\
    \label{eq:div_v2}
    \divergenz{\pmbv} &= 0,
    \\
    \label{eq:v2}
    \partial_t \pmbv + (\pmbv\cdot\nabla) \pmbv
    -\divergenz{2\eta(\phi) \D(\pmbv) }  
    + \nabla p
    &= \divergenz{\kappa(\B-\I)} 
    + \mu\nabla\phi 
    + (\chi_\sigma \sigma - \chi_\phi \phi) \nabla\sigma, 
    \\ 
    \label{eq:B2}
    \partial_t \B + (\pmbv \cdot\nabla)\B 
    + \frac{\kappa}{\tau(\phi)}\big( \B-\I \big)
    &= \nabla\pmbv \B
    + \B (\nabla\pmbv)^T
    + \alpha \Delta\B, 
\end{align}
together with the boundary conditions on $\partial\Omega \times (0,T)$
\begin{align}
    \label{eq:bc_phimu2}
    \nabla\phi\cdot\pmbn &= \nabla\mu\cdot\pmbn = 0, 
    \\
    \label{eq:bc_sigma2}
    \chi_\sigma n(\phi) \nabla \sigma \cdot \pmbn &= K(\sigma_\infty - \sigma), 
    \\
    \label{eq:bc_v2}
    \pmbv&= 0, 
    \\
    \label{eq:bc_B2}
    (\pmbn\cdot\nabla)\B &= 0,
\end{align}
\end{subequations}
and the initial data $\phi(0) = \phi_0$, $\sigma(0) = \sigma_0$, $\pmbv(0)=\pmbv_0$ and $\B(0) = \B_0$.

\subsection{Assumptions and existence of weak solutions}
\label{sec:weak_solution}

In this section, we state the definition of a weak solution of \ref{P_alpha} and provide an existence result in two space dimensions. First, we state our assumptions.

\begin{assumptions}
\label{as:weak_solution}~
\begin{itemize}
\item[\mylabel{A1}{$(\mathrm{A}1)$}] Let $T>0$ and suppose that $\Omega \subset\R^d$, $d\in\{2,3\}$, is a convex, polygonal domain with boundary $\partial\Omega$. 
	    
\item[\mylabel{A2}{$(\mathrm{A}2)$}] 
For $d\in\{2,3\}$, the source functions $\Gamma_\phi: \R\times\R\times\R^{d\times d} \to \R$ and $\Gamma_\sigma: \R\times\R\to\R$ are continuous and there exists a constant $R_0>0$ such that, for all $\phi,\sigma\in\R$, and $\B\in \R^{d\times d}$,
\begin{align*}
    \abs{\Gamma_\phi(\phi,\sigma,\B)} + \abs{\Gamma_\sigma(\phi,\sigma)} \leq R_0 (1+\abs{\phi}+\abs{\sigma}).
\end{align*}

\item[\mylabel{A3}{$(\mathrm{A}3)$}]
Let $\chi_\phi\geq 0$ and $\chi_\sigma, A, B, K, \kappa, \alpha>0$ be constants. Moreover, let $m, n, \eta, \tau \in C^0(\R)$ and suppose there exist constants $m_0, m_1, n_0, n_1, \eta_0, \eta_1, \tau_0, \tau_1>0$ such that, for all $t\in\R$,
\begin{align*}
	m_0 &\leq m(t) \leq m_1, 
	\quad 
	n_0 \leq n(t) \leq n_1,
	\quad
	\eta_0 \leq \eta(t) \leq \eta_1, 
    \quad
    \tau_0 \leq \tau(t) \leq \tau_1.
\end{align*}
	    
\item[\mylabel{A4}{$(\mathrm{A}4)$}] The potential $\psi$ is non-negative and belongs to $C^{1}(\R)$ with
\begin{align}
	\label{A4_1}
	\psi(t) &\geq R_1 \abs{t}^2 - R_2 \qquad \forall \  t\in\R,
\end{align}
where $R_1,R_2>0$. Additionally, the potential can be decomposed as $\psi=\psi_1 + \psi_2$ with $\psi_1$ convex and $\psi_2$ concave such that
\begin{align} 
	\label{A4_2}
	\abs{\psi_i^\prime(t)} &\leq R_3 (1+\abs t) \qquad \forall \  t\in\R,
\end{align}
where $i=1,2$ and $R_3 >0$. Moreover, we assume
\begin{align}
	\label{A4_3}
	A > \frac{4\chi_\phi^2}{\chi_\sigma R_1}.
\end{align}
	
\item[\mylabel{A5}{$(\mathrm{A}5)$}] 
For the initial and boundary data, assume
\begin{align*}
    &\phi_0 \in H^2_{\mathrm{N}}(\Omega) \coloneqq \{q \in H^2(\Omega) \mid \nabla q\cdot \pmbn = 0 \text{ on }\partial\Omega\}, 
    \quad  \sigma_0 \in L^2(\Omega), 
    \quad \pmbv_0 \in \mathbf{H},
    \quad \sigma_\infty \in L^2(0,T;H^1(\Omega)),
    \\
    & \B_0 \in L^\infty(\Omega;\R^{d\times d}_{\mathrm{SPD}})
    \quad
    \text{with} \quad
    b^0_{\min} \abs{\pmb\xi}^2 
    \leq \pmb\xi^T \B_0(x) \pmb\xi
    \leq b^0_{\max} \abs{\pmb\xi}^2
    \quad \forall \  \pmb\xi\in \R^d
    \quad \text{for a.e. } x\in\Omega,
\end{align*}
where $d\in\{2,3\}$, $b^0_{\min},b^0_{\max}\in\R$ with $0<b^0_{\min}\leq b^0_{\max}$ and $\pmbn$ denotes the outer unit normal on ${\partial\Omega}$.

\item[\mylabel{A6}{$(\mathrm{A}6)$}] The spatial dimension is restricted to $d=2$. 

\item[\mylabel{A7}{$(\mathrm{A}7)$}] The source functions $\Gamma_\phi,\Gamma_\sigma$ from \ref{A2} are Lipschitz continuous.

\end{itemize}
\end{assumptions}


Unter these assumptions, we provide an existence result for weak solutions to \ref{P_alpha}. 
However, note that \ref{A1}--\ref{A5} are stated for arbitrary dimensions $d\in\{2,3\}$. This is sufficient when studying stability and existence of discrete solutions to a fully-discrete finite element approximation of \ref{P_alpha} in Section \ref{sec:fem}. Later, also \ref{A6} is needed to improve the regularity of discrete solutions. Moreover, \ref{A7} is needed for the limit passing in the discrete scheme in presence of mass lumping but can be dropped if the terms containing $\Gamma_\phi, \Gamma_\sigma$ are integrated exactly.

Possible choices for the source functions which fulfill the assumptions can be constructed as follows. Let
\begin{align}
    \label{eq:source_terms}
    \Gamma_\phi(\phi,\sigma,\B) \coloneqq h(\phi) \big( \calP g(\sigma) f(\B) -\calA \big) ,
    \qquad \Gamma_\sigma(\phi,\sigma) \coloneqq \calC h(\phi) g(\sigma),
\end{align}
where $\calP,\calA, \calC\geq 0$ are non-negative constants accounting for proliferation, apoptosis and nutrient consumption, and 
\begin{subequations}
\begin{alignat}{3}
    \label{eq:h_phi}
    h(\phi) &\coloneqq \max\left\{ 0, \ \min\left\{ \tfrac{1}{2} (1+\phi), 1 \right\} \right\} 
    \qquad && \forall \  \phi\in\R,
    \\
    g(\sigma) &\coloneqq \max\left\{ 0, \ \min\left\{ \sigma , 1 \right\} \right\} 
    \qquad && \forall \  \sigma\in\R,
    \\
    \label{eq:f_B}
    f(\B) &\coloneqq \left(1 + \abs{\kappa  (\B-\I)}^2 \right)^{-1/2}
    \qquad && \forall \  \ \B\in\R^{d\times d}.
\end{alignat}
\end{subequations}
Then, as $f,g,h$ are non-negative, Lipschitz-continuous and bounded, the source functions $\Gamma_\phi,\Gamma_\sigma$ satisfy the assumptions from above, i.e.~\ref{A2} and \ref{A7}. 
%
%
The mobility functions $m,n$, the viscosity $\eta$ and the relaxation time $\tau$ can be defined with similar cut-offs outside of the interval $[-1,1]$ such that \ref{A3} holds.

In practice, the polynomial double-well potential $\tilde \psi(t) = \tfrac{1}{4} (1- t ^2)^2$ is a common choice. However, in order to fulfill \ref{A4}, the growth of the polynomial double-well potential shall be restricted to be at most quadratic for, e.g., $t\not\in[-1,1]$, i.e.
\begin{align}
    \label{eq:psi_modified}
    \psi(t) =
    \begin{cases} 
    t^2 - 2 t + 1
    &\mbox{if } t>1,
    \\
    \frac{1}{4} (1-t^2)^2  
    &\mbox{if } t \in [-1, 1],
    \\
    t^2 + 2 t + 1 
    & \mbox{if } t < -1,
    \end{cases}
    \qquad 
    \psi'(t) =
    \begin{cases} 
    2 t - 2 
    &\mbox{if } t>1,
    \\
    t^3-t  
    &\mbox{if } t \in [-1, 1],
    \\
    2 t + 2 
    & \mbox{if } t < -1.
    \end{cases}
\end{align}

Besides, the parameter $A$ is often chosen as $A=\frac{\beta}{\epsilon}$ with $\beta>0$ and a small constant $\epsilon>0$ relating to the thickness of the diffuse interface.
Therefore, \eqref{A4_3} is not a severe constraint.

\begin{definition}[Weak solution]
\label{def:weak_solution}
Under the assumptions \ref{A1}--\ref{A7}, the quintuple $(\phi,\mu,\sigma,\pmbv,\B)$ for $d=2$ is called a weak solution of \ref{P_alpha} if
\begin{subequations}
\begin{align}
    \label{eq:conv_phi}
    \phi &\in L^\infty(0,T;H^1) 
    \cap L^2(0,T;H^2)
    \cap H^1(0,T; (H^1)'),
    \\
    \label{eq:conv_mu}
    \mu &\in L^2(0,T; H^1),
    \\
    \label{eq:conv_sigma}
    \sigma &\in L^\infty(0,T; L^2) 
    \cap L^2(0,T; H^1) 
    \cap H^1(0,T; (H^1)'),
    \\
    \label{eq:conv_v}
    \pmbv &\in L^\infty(0,T; \mathbf{H})
    \cap L^2(0,T; \mathbf{V})
    \cap W^{1,\frac{4}{3}}(0,T; \mathbf{V}'),
    \\
    \label{eq:conv_B}
    \B &\in L^\infty\left(0,T; L^2(\Omega;\R^{2\times2}_{\mathrm{SPD}})\right)
    \cap L^2\left(0,T; H^1(\Omega;\R^{2\times2}_{\mathrm{S}})\right)
    \cap W^{1,\frac{4}{3}}\left(0,T; (H^1(\Omega;\R^{2\times2}_{\mathrm{S}}))'\right),
\end{align}
\end{subequations}
such that 
\begin{align}
\begin{split}
\label{eq:weak_init}
    &\phi(0) = \phi_0 \text{ in } L^2(\Omega), 
    \quad
    \sigma(0) = \sigma_0 \text{ in } L^2(\Omega),
    \quad 
    \pmbv(0) = \pmbv_0 \text{ in } \mathbf{H},
    \\
    &\B(0) = \B_0 \text{ in } L^2(\Omega;\R^{2\times 2}_{\mathrm{S}})
    \quad \text{and} \quad \B \text{ positive definite a.e.~in } \Omega\times(0,T),
\end{split}
\end{align}
and
\begin{subequations}
\begin{align}
    \label{eq:phi_weak}
    0 &= \int_0^T \dualp{\partial_t\phi}{\zeta}_{H^1} \dt
    + \int_{\Omega_T} m(\phi)\nabla\mu\cdot\nabla\zeta 
    - \Gamma_\phi(\phi,\sigma,\B) \zeta 
    - \phi \pmbv\cdot \nabla\zeta
    \dx\dt , 
    \\
    \label{eq:mu_weak}
    0 &= \int_{\Omega_T} \mu\rho 
    - A \psi^\prime(\phi)\rho 
    - B\nabla\phi\cdot\nabla\rho 
    + \chi_\phi \sigma \rho \dx \dt, 
    \\
    \nonumber
    \label{eq:sigma_weak}
    0 &= \int_0^T \dualp{\partial_t\sigma}{\xi}_{H^1} \dt
    + \int_{\Omega_T} n(\phi) \nabla (\chi_\sigma\sigma - \chi_\phi\phi) 
    \cdot\nabla\xi 
    + \Gamma_\sigma(\phi,\sigma)\xi 
    - \sigma \pmbv \cdot \nabla\xi \dx \dt
    \\
    &\qquad + \int_0^T \int_{\partial\Omega} K(\sigma-\sigma_\infty)\xi \dH^{d-1} \dt,
    \\
    \label{eq:v_weak}
    \nonumber
    0 &= \int_0^T \dualp{\partial_t\pmbv}{\pmbw}_{\mathbf{V}} \dt
    + \int_{\Omega_T} (\pmbv\cdot\nabla)\pmbv \cdot \pmbw 
    + 2 \eta(\phi) \D(\pmbv) : \D(\pmbw) 
    + \kappa (\B-\I) : \D(\pmbw)  \dx\dt 
    \\
    &\qquad + \int_{\Omega_T} 
    \big(\phi \nabla\mu + \sigma \nabla(\chi_\sigma\sigma - \chi_\phi\phi) \big) \cdot \pmbw \dx\dt,
    \\
    \label{eq:B_weak}
    0 &= \int_0^T \dualp{\partial_t\B}{\C}_{H^1} \dt
    + \int_{\Omega_T} \frac{\kappa}{\tau(\phi)} (\B-\I): \C 
    - 2 (\nabla\pmbv \B) : \C 
    + \alpha \nabla\B : \nabla\C 
    - \B : (\pmbv\cdot\nabla)\C  \dx\dt,
\end{align}
\end{subequations}
for all $\zeta,\rho,\xi \in L^2(0,T;H^1)$, $\pmbw\in L^4(0,T;\mathbf{V})$ and $\C \in L^4\left(0,T;H^1(\Omega;\R^{2 \times 2}_{\mathrm{S}})\right)$.
\end{definition}

\begin{theorem}[Existence of weak solutions]
\label{theorem:weak_solution}
Let \ref{A1}--\ref{A7} hold. Then, there exists a weak solution $(\phi,\mu,\sigma,\pmbv,\B)$ of \ref{P_alpha} in the sense of Definition \ref{def:weak_solution}. Moreover, there exist positive constants $C_1(T), C_2(T,\alpha^{-1})$, both depending exponentially on $T$ and $C_2(T,\alpha^{-1})$ depending additionally on $\alpha^{-1}$, such that
\begin{subequations}
\begin{align}
    \nonumber
    &\nnorm{\phi}_{L^\infty(0,T;H^1)} 
    + \nnorm{\phi}_{L^2(0,T;H^2)} 
    + \nnorm{\partial_t\phi}_{L^2(0,T;(H^1)')}
    + \nnorm{\mu}_{L^2(0,T;H^1)}
    \\
    &\quad
    + \nnorm{\sigma}_{L^\infty(0,T;L^2)}
    + \nnorm{\sigma}_{L^2(0,T;H^1)} 
    + \nnorm{\partial_t\sigma}_{L^2(0,T;(H^1)')}
    + \nnorm{\pmbv}_{L^\infty(0,T;L^2)}
    + \nnorm{\pmbv}_{L^2(0,T;H^1)} \leq C_1(T),
    \\[1ex]
    & \nnorm{\partial_t\pmbv}_{L^{4/3}(0,T;\mathbf{V}')}
    + \nnorm{\B}_{L^\infty(0,T;L^2)} + \nnorm{\B}_{L^2(0,T;H^1)} + \nnorm{\partial_t \B}_{L^{4/3}(0,T;(H^1)')}  
    \leq C_2(T,\alpha^{-1}).
\end{align}
\end{subequations}
\end{theorem}

\begin{remark}~
\label{remark:weak_solution}
\begin{enumerate}[(i)]
\item
This existence result will be proved in Section \ref{sec:fem} by the passage to the limit in a fully-discrete finite element scheme in two dimensions, where a CFL condition is necessary, i.e.~$\Delta t\leq C h^2$ with a possibly very small positive constant $C$, see Theorem \ref{theorem:convergence}. Further, the additional Lipschitz assumption \ref{A7} on the source terms is needed for the limit passing in presence of mass lumping, but it can be dropped if the integrals containing $\Gamma_\phi,\Gamma_\sigma$ are evaluated exactly.

\item As $\pmbv\in L^2(0,T;\mathbf{V})$, integration by parts over $\Omega$ in the convection term in \eqref{eq:phi_weak} leads to \eqref{eq:phi_weak} with $-\phi\pmbv\cdot\nabla\zeta$ replaced by $\pmbv\cdot\nabla\phi \zeta$, which is consistent with \eqref{eq:phi2}. One can argue similarly for the convection terms in \eqref{eq:sigma_weak}, \eqref{eq:B_weak} and the last term in \eqref{eq:v_weak}.

\item Moreover, one can obtain \eqref{eq:B_weak} with $2\nabla\pmbv\B$ replaced by $\nabla\pmbv\B + \B (\nabla\pmbv)^T$, which is consistent with \eqref{eq:B2}, by choosing the test function $\C = \frac{1}{2} (\mathbb{G} + \mathbb{G}^T)$, where $\mathbb{G} \in L^4(0,T;H^1(\Omega;\R^{2\times 2}))$, and using the symmetry of $\B$.

\item This existence result still holds true if source or sink terms for $\B$ are included like in \eqref{eq:B_growth}, i.e.~if \eqref{eq:B2} is replaced by
\begin{align}
    \partial_t \B + (\pmbv \cdot\nabla)\B 
    + \frac{\kappa}{\tau(\phi)}\big( \B-\I \big)
    &= \nabla\pmbv \B
    + \B (\nabla\pmbv)^T
    - \gamma(\phi,\sigma)\B
    + \alpha \Delta\B, 
\end{align}
where $\gamma\colon \R^2\to\R$ is continuous and bounded. The additional term $\gamma(\phi,\sigma)\B$ on the right-hand side can be controlled with a Gronwall argument similarly to \eqref{eq:formal_4}.
\end{enumerate}
\end{remark}
\subsection{Formal a priori estimates}
\label{sec:formal_bounds}

To better understand the strategy for the proof of Theorem \ref{theorem:weak_solution}, we temporarily assume that \ref{A1}--\ref{A5} hold and that $(\phi,\mu,\sigma,p,\pmbv,\B)$ is a sufficiently smooth solution of \ref{P_alpha} with $\B$ positive definite in $\Omega_T \coloneqq \Omega\times (0,T)$. The first step is to provide the formal derivation of \textit{a priori} estimates based on the energy 
\begin{align}
    \calF(\phi,\sigma,\pmbv,\B) \coloneqq \int_\Omega
    A\psi(\phi) + \frac{B}{2} \abs{\nabla\phi}^2
    + \frac{\chi_\sigma}{2} \abs{\sigma}^2 
    + \chi_\phi \sigma (1-\phi)
    + \frac{1}{2} \abs{\pmbv}^2
    + \frac{\kappa}{2}\trace(\B- \ln\B) \dx.
\end{align}
Note that this energy is not finite if $\B$ is not positive definite, which is due to the logarithmic term.
Later, with the help of suitable regularization techniques, it turns out that the positive definiteness of the left Cauchy--Green tensor $\B$ is preserved for all $t>0$ if $\B(t=0)$ is positive definite. This has also been observed for other viscoelastic systems in the literature,
see, e.g., \cite[Lem.~2.1]{Hu_Lelievre_2007} or \cite[Rem.~3.4]{Lukacova_2017}.

Moreover, the energy $\calF(\phi,\sigma,\pmbv,\B)$ can become negative due to the term $\sigma(1-\phi)$. This is one reason why the derivation of reasonable \textit{a priori} estimates must be performed carefully.

From the general energy identity \eqref{eq:energy0}, we have
\begin{align}
\begin{split}
\label{eq:formal_1}
    &\ddt \calF(\phi, \sigma, \pmbv, \B) 
    + \int_\Omega m(\phi) \abs{\nabla\mu}^2 
    + n(\phi) \abs{\nabla (\chi_\sigma \sigma - \chi_\phi\phi)}^2
    + 2 \eta(\phi) \abs{\D(\pmbv)}^2 \dx
    \\
    &\quad 
    + \int_{\partial\Omega} K \chi_\sigma \abs{\sigma}^2 \dH^{d-1}
    + \int_\Omega \frac{\kappa^2}{2\tau(\phi)} \trace\big(\B + \B^{-1}-2 \I\big) 
    - \frac{\alpha\kappa}{2} \nabla \B : \nabla \B^{-1} \dx
    \\
    &=
    \int_\Omega \mu \Gamma_\phi(\phi,\sigma,\B)
    - (\chi_\sigma \sigma + \chi_\phi (1-\phi)) \Gamma_\sigma(\phi,\sigma) \dx
    + \int_{\partial\Omega} K (\chi_\sigma \sigma + \chi_\phi(1-\phi)) \sigma_\infty
    - K \chi_\phi(1-\phi) \sigma \dH^{d-1}.
\end{split}
\end{align}
We recall that this is obtained by formally multiplying \eqref{eq:phi2} with $\mu$, \eqref{eq:mu2} with $\partial_t \phi$, \eqref{eq:sigma2} with $\chi_\sigma \sigma + \chi_\phi (1-\phi)$, \eqref{eq:v2} with $\pmbv$ and \eqref{eq:B2} with $\frac{\kappa}{2} (\I - \B^{-1})$, integrating over the domain $\Omega$, using Green's formula and then summing up the resulting equations.

Apart form the energy, the terms on the left-hand side of \eqref{eq:formal_1} are non-negative as $\B$ is positive definite and as the functions $m(\cdot), n(\cdot), \eta(\cdot), \tau(\cdot)$ are continuous, uniformly positive and bounded due to \ref{A3}. As $\B$ is symmetric positive definite, we note that it holds
\begin{align*}
    - \int_\Omega \nabla \B : \nabla \B^{-1} \dx 
    &\geq 
    \int_\Omega \frac{1}{d} \abs{\nabla \trace(\ln\B)}^2 \dx,
    \quad
    \text{ and }
    \quad
    \trace\big(\B + \B^{-1}-2 \I\big)
    = \abs{ (\I - \B^{-1}) \sqrt{\B}}^2
    \geq 0,
\end{align*}
see \cite[Lem.~3.1]{barrett_lu_sueli_2017} for the first inequality. 

Now we estimate the terms on the right-hand side of \eqref{eq:formal_1}. First, for the terms involving the boundary integrals on the right-hand side of \eqref{eq:formal_1}, we apply H{\"o}lder's and Young's inequalities and the trace theorem to obtain
\begin{align}
\begin{split}
\label{eq:formal_boundary_term}
    &\abs{ \int_{\partial\Omega} K (\chi_\sigma \sigma + \chi_\phi(1-\phi)) \sigma_\infty
    - K \chi_\phi(1-\phi) \sigma \dH^{d-1} }
    \\
    &\leq 
    \frac{3}{4} K \chi_\sigma  \norm{\sigma}_{L^2(\partial\Omega)}^2 
    + C(K, C_{\mathrm{tr}}, \chi_\phi, \chi_\sigma)  \norm{\phi}_{H^1}^2
    + C(K, \chi_\phi, \chi_\sigma) \big( \abs{\partial\Omega}
    + \norm{\sigma_\infty}_{L^2(\partial\Omega)}^2 \big) .
\end{split}
\end{align}
The terms on the right-hand side of \eqref{eq:formal_1} involving the source terms $\Gamma_\phi$ and $\Gamma_\sigma$ are bounded as follows,
\begin{align}
\begin{split}
\label{eq:formal_source_terms}
    &\abs{ \int_\Omega 
    \mu \Gamma_\phi(\phi,\sigma,\B)
    - \big(\chi_\sigma \sigma + \chi_\phi (1-\phi)\big) \Gamma_\sigma(\phi,\sigma)
    \dx }
    \\
    &\leq 
    \frac{1}{2} \norm{\mu}_{L^2}^2 
    + C(R_0, \chi_\sigma, \chi_\phi) \Big( 
    \norm{\phi}_{L^2}^2
    +\norm{\sigma}_{L^2}^2 \Big)
    + C(R_0, \chi_\phi, \Omega),
\end{split}
\end{align}
where we used \ref{A2} such as Hölder's and Young's inequalities. 
However, one now needs an $L^2(\Omega)$-bound for the chemical potential in order to control the source terms.
This is obtained by multiplying \eqref{eq:mu2} with $\mu$ and integrating over the domain $\Omega$ and applying Green's formula, which yields
\begin{align*}
    & \norm{\mu}_{L^2}^2
    =
    \int_\Omega \big( 
    A\psi'(\phi) 
    - \chi_\phi \sigma \big) \mu 
    + B \nabla\phi \cdot \nabla\mu \dx.
\end{align*}
At this point, we also need that the elasticity parameter $\kappa$ is independent of $\phi$, otherwise we would have to control additional $\B$-dependent terms, see \eqref{eq:mu0}, 
for which we do not have any \textit{a priori} knowledge.

In absence of any \textit{a priori} estimate for $\phi$, we can control $\norm{\mu}_{L^2}^2$ only if $\psi'(\cdot)$ has at most linear growth.
Hence, one obtains with H{\"o}lder's and Young's inequalities that
\begin{align}
\label{eq:formal_mu}
    \norm{\mu}_{L^2}^2 
    \leq 
    \frac{m_0}{4} \norm{\nabla\mu}_{L^2}^2
    + C(A, B, R_3, \chi_\phi, m_0) \Big( 
    1 + \norm{\phi}_{L^2}^2 
    + \norm{\nabla\phi}_{L^2}^2 
    + \norm{\sigma}_{L^2}^2 \Big).
\end{align}

Moreover, we use
\begin{align}
\begin{split}
\label{eq:formal_nabla_sigma}
    \chi_\sigma^2 \norm{\nabla \sigma}_{L^2}^2
    &\leq 2 \nnorm{\chi_\sigma \nabla \sigma - \chi_\phi \nabla\phi}_{L^2}^2
    + 2 \chi_\phi^2 \norm{\nabla \phi}_{L^2}^2,
\end{split}
\end{align}
such that \eqref{eq:formal_1} becomes
\begin{align}
\begin{split}
\label{eq:formal_2}
    &\ddt \calF(\phi, \sigma, \pmbv, \B) 
    + C \Big( 
    \norm{\mu}_{H^1}^2 
    + \norm{\nabla \sigma}_{L^2}^2
    + \norm{\D(\pmbv)}_{L^2}^2 
    + \norm{\sigma}_{L^2(\partial\Omega)}^2 \Big)
    \\
    &\quad
    + C\Big( \norm{\trace\big( \B + \B^{-1}-2 \I \big)}_{L^1}
    + \alpha \norm{\nabla \trace(\ln\B)}_{L^2}^2 \Big)
    \\
    &\leq
    C\Big( 1 +  \norm{\phi}_{H^1}^2 
    + \norm{\sigma}_{L^2}^2
    + \norm{\sigma_\infty}_{L^2(\partial\Omega)}^2 \Big).
\end{split}
\end{align}

As the term $\chi_\phi \sigma (1-\phi)$ in the energy can have a negative sign, the next step is to absorb it with the help of the non-negative terms in the energy. In particular, we first apply Hölder's and Young's inequalities
\begin{align*}
    \abs{\int_\Omega \chi_\phi \sigma (1-\phi) \dx }
    \leq 
    \frac{\chi_\sigma}{4} \norm{\sigma}_{L^2}^2 
    + \frac{2 \chi_\phi^2}{\chi_\sigma} \norm{\phi}_{L^2}^2,
\end{align*}
so that we obtain by integrating over $t\in(0,s)$, where $s\in(0,T)$,
\begin{align}
\begin{split}
\label{eq:formal_3}
    & A \norm{\psi(\phi(s))}_{L^1} 
    + \frac{B}{2} \norm{\nabla\phi(s)}_{L^2}^2
    + \frac{\chi_\sigma}{4} \norm{\sigma(s)}_{L^2}^2 
    + \frac{1}{2} \norm{\pmbv(s)}_{L^2}^2
    + \frac{\kappa}{2} \norm{\trace(\B(s) - \ln\B(s))}_{L^1}
    \\
    &\quad
    + C \Big( 
    \norm{\mu}_{L^2(0,s;H^1)}^2 
    + \norm{\nabla \sigma}_{L^2(0,s;L^2)}^2
    + \norm{\D(\pmbv)}_{L^2(0,s;L^2)}^2 
    + \norm{\sigma}_{L^2(0,s;L^2(\partial\Omega))}^2 \Big)
    \\
    &\quad
    + C\Big( \norm{\trace\big( \B + \B^{-1}-2 \I \big)}_{L^1(0,s;L^1)}
    + \alpha \norm{\nabla \trace(\ln\B)}_{L^2(0,s;L^2)}^2 \Big)
    \\
    &\leq
    \frac{2 \chi_\phi^2}{\chi_\sigma} \norm{\phi(s)}_{L^2}^2
    + C\Big( 1 + \abs{\calF(\phi_0, \sigma_0, \pmbv_0, \B_0)}
    + \norm{\phi}_{L^2(0,s;H^1)}^2 
    + \norm{\sigma}_{L^2(0,s;L^2)}^2
    + \norm{\sigma_\infty}_{L^2(0,s;L^2(\partial\Omega))}^2 \Big).
\end{split}
\end{align}
Then, by \eqref{A4_1}, we have
\begin{align}
\begin{split}
\label{eq:formal_4}
    & \left( A R_1 - \frac{2 \chi_\phi^2}{\chi_\sigma} \right) \norm{\phi(s)}_{L^2}^2
    + \frac{B}{2} \norm{\nabla\phi(s)}_{L^2}^2
    + \frac{\chi_\sigma}{4} \norm{\sigma(s)}_{L^2}^2 
    + \frac{1}{2} \norm{\pmbv(s)}_{L^2}^2
    + \frac{\kappa}{2} \norm{\trace(\B(s) - \ln\B(s))}_{L^1}
    \\
    &\quad
    + C \Big( 
    \norm{\mu}_{L^2(0,s;H^1)}^2 
    + \norm{\nabla \sigma}_{L^2(0,s;L^2)}^2
    + \norm{\D(\pmbv)}_{L^2(0,s;L^2)}^2 
    + \norm{\sigma}_{L^2(0,s;L^2(\partial\Omega))}^2 \Big)
    \\
    &\quad
    + C\Big( \norm{\trace\big( \B + \B^{-1}-2 \I \big)}_{L^1(0,s;L^1)}
    + \alpha \norm{\nabla \trace(\ln\B)}_{L^2(0,s;L^2)}^2 \Big)
    \\
    &\leq
    C\Big( 1 + \abs{\calF(\phi_0, \sigma_0, \pmbv_0, \B_0)}
    + \norm{\phi}_{L^2(0,s;H^1)}^2 
    + \norm{\sigma}_{L^2(0,s;L^2)}^2
    + \norm{\sigma_\infty}_{L^2(0,s;L^2(\partial\Omega))}^2 \Big).
\end{split}
\end{align}
Note that $AR_1 - \frac{2 \chi_\phi^2}{\chi_\sigma}$ is positive due to \eqref{A4_3}, which is not a severe constraint in practice, as $A = \frac{\beta}{\epsilon}$ with a small $\epsilon>0$. Hence, we apply a Gronwall argument (see below for Lemma \ref{lemma:gronwall}), to obtain the inequality
\begin{align}
\begin{split}
\label{eq:formal_5}
    &  \norm{\phi(s)}_{H^1}^2
    +  \norm{\sigma(s)}_{L^2}^2 
    +  \norm{\pmbv(s)}_{L^2}^2
    +  \norm{\trace(\B(s) - \ln\B(s))}_{L^1} 
    \\
    &\quad
    + 
    \norm{\mu}_{L^2(0,s;H^1)}^2 
    + \norm{\nabla \sigma}_{L^2(0,s;L^2)}^2
    + \norm{\D(\pmbv)}_{L^2(0,s;L^2)}^2 
    + \norm{\sigma}_{L^2(0,s;L^2(\partial\Omega))}^2 
    \\
    &\quad
    + \norm{\trace\big( \B + \B^{-1}-2 \I \big)}_{L^1(0,s;L^1)}
    + \alpha \norm{\nabla \trace(\ln\B)}_{L^2(0,s;L^2)}^2 
    \\
    &\leq
    C\Big( 
    1 + \abs{\calF(\phi_0, \sigma_0, \pmbv_0, \B_0)}
    + \norm{\sigma_\infty}_{L^2(0,T;L^2(\partial\Omega))}^2 \Big),
\end{split}
\end{align}
for almost all $s\in(0,T)$. This leads to formal \textit{a priori} estimates as the right-hand side of \eqref{eq:formal_5} is bounded due to \ref{A5}.

For completeness, we recall the following Gronwall inequality from \cite[Lem.~3.1]{garcke_lam_2017}.
\begin{lemma}
\label{lemma:gronwall}
Let $\alpha, \beta, u$ and v be real-valued functions defined on $I=[0,T]$. Assume that $\alpha$ is integrable, $\beta$ is non-negative and continuous, $u$ is continuous, $v$ is non-negative and continuous. Suppose $u$ and $v$ satisfy the integral inequality
\begin{align*}
    u(s) + \int_0^s v(t) \dt \leq \alpha(s) + \int_0^s \beta(t) u(t) \dt \quad \forall \  s\in I.
\end{align*}
Then it follows
\begin{align*}
    u(s) + \int_0^s v(t) \dt 
    \leq \alpha(s) + \int_0^s \alpha(t)\beta(t) \exp\Big(\int_t^s \beta(r) \dr \Big) \dt.
\end{align*}
\end{lemma}

\subsection{Stronger bounds in two spatial dimensions}
\label{sec:formal_bounds_2d}

The bounds on $\B$ are not sufficiently strong to establish existence of a solution. However, we get an estimate in a stronger norm if we restrict to two spatial dimensions. Hence, suppose that \ref{A6} holds true in addition to \ref{A1}--\ref{A5}. 
To derive higher order estimates for $\B$, we formally multiply \eqref{eq:B2} with $\B$, integrate over $\Omega$ and apply Green's formula to obtain
\begin{align}
\label{eq:formal_B_1}
    \ddt \frac{1}{2} \norm{\B}_{L^2}^2
    + \frac{\kappa}{\tau(\phi)} \norm{\B}_{L^2}^2 
    + \alpha \norm{\nabla\B}_{L^2}^2 \dx
    &= 
    \int_\Omega \frac{\kappa}{\tau(\phi)} \trace\B 
    + 2 \nabla\pmbv : \B^2 
    - (\pmbv \cdot\nabla)\B : \B \dx.
\end{align}
On noting \eqref{eq:div_v2}, the last term in \eqref{eq:formal_B_1} vanishes by integration by parts. 
Then, with H{\"o}lder's and Young's inequalities and a Gagliardo--Nirenberg interpolation inequality (see, e.g., \cite{barrett_boyaval_2009}) for $d=2$, it holds
\begin{align}
\label{eq:formal_B_2}
    &\ddt \norm{\B}_{L^2}^2
    +  C \norm{\B}_{L^2}^2 
    + \alpha \norm{\nabla\B}_{L^2}^2 
    \leq 
    C(\alpha^{-1}) \big( 1 + \norm{\B}_{L^2}^2 
    + \norm{\nabla\pmbv}_{L^2}^2  \norm{\B}_{L^2}^2  \big),
\end{align}
where $C(\alpha^{-1})$ denotes a constant that depends on the inverse of the viscoelastic diffusion parameter $\alpha$. 
It follows with integration in time and Lemma \ref{lemma:gronwall} that 
\begin{align}
\label{eq:formal_B_3}
    &\norm{\B(s)}_{L^2}^2
    + \norm{\B}_{L^2(0,s;L^2)}^2 
    + \alpha \norm{\nabla\B}_{L^2(0,s;L^2)}^2
    \leq C \big(\alpha^{-1}, \norm{\nabla\pmbv}_{L^2(0,T;L^2)} \big) \norm{\B_0}_{L^2}^2,
\end{align}
for almost all $s\in(0,T)$. The right-hand side of \eqref{eq:formal_B_3} is bounded due to \eqref{eq:formal_5} and \ref{A5}.


\subsection{Formal estimates for a regularized problem}
\label{sec:regularization}

Showing the positive definiteness of the left Cauchy--Green tensor $\B$ is one of the main difficulties we have to deal with.
Here, we apply a regularization strategy of Barrett and Boyaval \cite{barrett_boyaval_2009} and introduce a regularized problem with a cut-off on the left Cauchy--Green tensor on certain terms in the system \ref{P_alpha}.

First, we introduce the following concave regularized approximations of the logarithm function $G(s)=\ln(s)$ and of the identity $\beta(s) = [G'(s)]^{-1} = s$ for all $s>0$ similarly to \cite[Sec.~2.1]{barrett_boyaval_2009}: 
\begin{alignat}{5}
    &G_\delta(s) &&= 
    \begin{cases}
    \frac{s}{\delta} + \ln(\delta) - 1, 
    & s < \delta,
    \\
    \ln(s), 
    & s \geq \delta,
    \end{cases}
    \quad\quad\quad
    &&\beta_\delta(s) &&= 
    \big[G_\delta'(s)\big]^{-1} 
    &&= \max\{s,\delta\}
    \quad \forall \  s\in\R,
    \\
    &G^L(s) &&= 
    \begin{cases}
    \ln(s), 
    & s\in(0,L),
    \\
    \frac{s}{L} + \ln(L) - 1, 
    & s\geq L,
    \end{cases}
    \quad\quad\quad
    &&\beta^L(s) &&= 
    \big[G^{L'}(s)\big]^{-1}
    &&= \min\{s,L\}
    \quad \forall \  s>0,
\end{alignat}
where $0 < \delta < 1 < L$, also see Figure \ref{fig:regularizations}.  
We also define the concave $C^1(\R)$ function
\begin{align}
    H_\delta(s) \coloneqq G^{\delta^{-1}}(s)
    \quad \forall \  s\in\R_{>0}.
\end{align}


\begin{figure}[ht]
\centering
\subfloat
{
\begin{tikzpicture}
\begin{axis}[
axis x line=middle, 
axis y line=middle,
xlabel=$s$, 
xtick={0.4, 1, 3},
xticklabels={$\delta$, 1, $L$},
legend style={at={(axis cs:3,-3)},anchor=south west},
ymajorticks=false,
]
\addplot[samples=400, smooth, sharp plot, domain=0:7] 
{ln(x)};
\addplot[samples=5, dashdotted, sharp plot, domain=3:7, mark=square, every mark/.append style={solid}]
{x/3+ln(3)-1};
\addplot[samples=3, dotted, sharp plot, domain=-0.8:0.4, mark=*]
{x/0.4 + ln(0.4) - 1};
\addplot[samples=3, only marks, domain=0.1:2, mark=square, every mark/.append style={solid}]
{ln(x)};
\addplot[samples=7, only marks, domain=0.4:7, mark=*]
{ln(x)};
\addplot[dashed, sharp plot] coordinates {(0.4,0) (0.4, -1 )};
\addplot[dashed, sharp plot] coordinates {(3,0) (3, 1.2 )};
\addlegendentry{$G(s)$}
\addlegendentry{$G^L(s)$}
\addlegendentry{$G_\delta(s)$}
\end{axis}
\end{tikzpicture}
}
\hspace{0.5cm}
\subfloat
{
\begin{tikzpicture}
\begin{axis}[
axis x line=middle, 
axis y line=middle,
xlabel=$s$, 
xtick={0.5,1.2,3},
xticklabels={$\delta$,1,$L$},
ymajorticks=false,
legend style={at={(axis cs:0.5,3)},anchor=south west},
]
\addplot[samples=100, smooth, sharp plot, domain=-1:4] 
{x};
\addplot[samples=2, dashdotted, sharp plot, domain=3:4, mark=square, every mark/.append style={solid, fill=gray}]
{min(x,3)};
\addplot[samples=3, dotted, sharp plot, domain=-1:0.5, mark=*]
{max(x,0.5)};
\addplot[samples=5, only marks, domain=0:3, mark=square, every mark/.append style={solid}]
{min(x,3)};
\addplot[samples=5, only marks, domain=0.5:4, mark=*]
{max(x,0.5)};
\addplot[dashed, sharp plot] coordinates {(0.5,0) (0.5, 0.5 )};
\addplot[dashed, sharp plot] coordinates {(3,0) (3, 3 )};
\addlegendentry{$\beta(s)$}
\addlegendentry{$\beta^L(s)$}
\addlegendentry{$\beta_\delta(s)$}
\end{axis}
\end{tikzpicture}
}
\caption{The functions $G$ (left) and $\beta$ (right) and their regularizations.}
\label{fig:regularizations}
\end{figure}
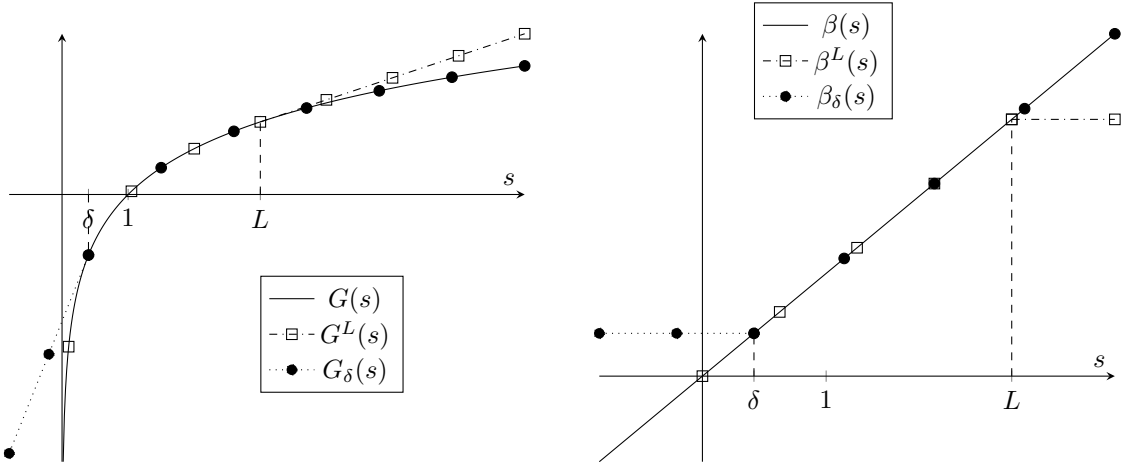


We recall the following result from \cite[Lem.~2.1]{barrett_boyaval_2009}. Let us note that the domain of definition of scalar functions is naturally extended to symmetric matrices in terms of the eigenvalues. 

\begin{lemma}
For all $\Phi,\Psi\in\R^{d\times d}_{\mathrm{S}}$ and for any $\delta \in (0,1)$, it holds
\begin{subequations}
\begin{align}
    \label{eq:lemma_reg1a}
    \beta_\delta(\Phi) G_\delta'(\Phi) 
    &= G_\delta'(\Phi) \beta_\delta(\Phi)
    = \I,
    \\
    \label{eq:lemma_reg1b}
    \trace\big( \beta_\delta(\Phi) 
    + [\beta_\delta(\Phi)]^{-1} - 2\I \big) 
    &\geq 0,
    \\
    \label{eq:lemma_reg1c}
    \trace\big( \Phi - G_\delta(\Phi)  - \I \big)  
    &\geq 0,
    \\
    \label{eq:lemma_reg1d}
    \big( \Phi - \beta_\delta(\Phi)\big) 
    : \big( \I - G_\delta'(\Phi) \big) 
    &\geq 0,
    \\
    \label{eq:lemma_reg1e}
    (\Phi-\Psi) : G_\delta'(\Psi) 
    &\geq \trace\big( G_\delta(\Phi) - G_\delta(\Psi) \big),
    \\
    \label{eq:lemma_reg1f}
    - (\Phi-\Psi) : \big( G_\delta'(\Phi) - G_\delta'(\Psi) \big)
    &\geq \delta^2 \abs{G_\delta'(\Phi) - G_\delta'(\Psi)}^2.
\end{align}
In addition, if $\delta\in(0,\frac{1}{2}]$, it holds
\begin{align}
    \label{eq:lemma_reg1g}
    \trace\big( \Phi - G_\delta(\Phi) \big)  
    &\geq
    \begin{cases}
    \frac{1}{2} \abs{\Phi},
    \\
    \frac{1}{2\delta} \abs{ [\Phi]_{-} },
    \end{cases}
    \\
    \label{eq:lemma_reg1h}
    \Phi : \big( \I - G_\delta'(\Phi) \big) 
    &\geq \frac{1}{2}\abs{\Phi} - d,
\end{align}
\end{subequations}
where $[\cdot]_{-}$ denotes the negative part function defined by $[s]_{-} \coloneqq \min\{s,0\}$ $\forall \  s\in\R$.
\end{lemma}



Let us return to the problem \ref{P_alpha} and introduce the reguralized problem with a cut-off on the left Cauchy--Green tensor on certain terms in the system.

\subsubsection*{Problem \ref{P_alpha_delta}:}
\mylabelHIDE{P_alpha_delta}{$(\pmbP_{\alpha,\delta})$}
Let $\delta\in(0,\frac{1}{2}]$. The regularized problem \ref{P_alpha_delta} corresponds to \ref{P_alpha} with \eqref{eq:v2}--\eqref{eq:B2} replaced by
\begin{align}
    \label{eq:v2_delta}
    \partial_t \pmbv + (\pmbv\cdot\nabla) \pmbv 
    -\divergenz{2\eta(\phi) \D(\pmbv) }  
    + \nabla p
    &= \divergenz{\kappa(\beta_\delta(\B)-\I)} 
    + \mu\nabla\phi 
    + (\chi_\sigma \sigma - \chi_\phi \phi) \nabla\sigma, 
    \\ 
    \label{eq:B2_delta}
    \partial_t \B + (\pmbv \cdot\nabla)\beta_\delta(\B) 
    + \frac{\kappa}{\tau(\phi)}\big( \B-\I \big)
    &= \nabla\pmbv \beta_\delta(\B) 
    + \beta_\delta(\B) (\nabla\pmbv)^T
    + \alpha \Delta\B.
\end{align}

Now again, we temporarily assume that \ref{A1}--\ref{A5} hold and that $(\phi,\mu,\sigma,p,\pmbv,\B)$ is a sufficiently smooth solution of \ref{P_alpha_delta} for a given $\delta \in (0,\frac{1}{2}]$. The first step is again to provide the formal derivation of \textit{a priori} estimates based on the regularized energy 
\begin{align}
    \calF_\delta(\phi,\sigma,\pmbv,\B) \coloneqq \int_\Omega
    A\psi(\phi) + \frac{B}{2} \abs{\nabla\phi}^2
    + \frac{\chi_\sigma}{2} \abs{\sigma}^2 
    + \chi_\phi \sigma (1-\phi)
    + \frac{1}{2} \abs{\pmbv}^2
    + \frac{\kappa}{2}\trace(\B- G_\delta(\B)) \dx.
\end{align}
Let us remark that $\B$ does not necessarily have to be positive definite in presence of the regularization parameter $\delta$ as the term $\trace(\B-G_\delta(\B))$ is well-defined even if $\B$ is not positive definite.

We perform a similar testing procedure as in \eqref{eq:formal_1}. More concretely, we formally multiply \eqref{eq:phi2} with $\mu$, \eqref{eq:mu2} with $\partial_t \phi$, \eqref{eq:sigma2} with $\chi_\sigma \sigma + \chi_\phi (1-\phi)$, \eqref{eq:v2_delta} with $\pmbv$ and \eqref{eq:B2_delta} with $\frac{\kappa}{2} (\I - G_\delta'(\B))$, integrate over the domain $\Omega$, apply Green's formula, and then sum up the resulting equations, so that
\begin{align}
\begin{split}
\label{eq:formal_delta_1}
    &\ddt \calF_\delta(\phi, \sigma, \pmbv, \B) 
    + \int_\Omega m(\phi) \abs{\nabla\mu}^2 
    + n(\phi) \abs{\nabla (\chi_\sigma \sigma - \chi_\phi\phi)}^2
    + 2 \eta(\phi) \abs{\D(\pmbv)}^2 \dx
    \\
    &\quad 
    + \int_{\partial\Omega} K \chi_\sigma \abs{\sigma}^2 \dH^{d-1}
    + \int_\Omega \frac{\kappa^2}{2\tau(\phi)} (\B-\I) : (\I - G_\delta'(\B)) 
    - \frac{\alpha\kappa}{2} \nabla \B : \nabla G_\delta'(\B)) \dx
    \\
    &=
    \int_\Omega \mu \Gamma_\phi(\phi,\sigma,\B)
    - (\chi_\sigma \sigma + \chi_\phi (1-\phi)) \Gamma_\sigma(\phi,\sigma) \dx
    + \int_{\partial\Omega} K (\chi_\sigma \sigma + \chi_\phi(1-\phi)) \sigma_\infty
    - K \chi_\phi(1-\phi) \sigma \dH^{d-1}.
\end{split}
\end{align}
On noting \eqref{eq:lemma_reg1a} and \eqref{eq:lemma_reg1d} we have
\begin{align}
\label{eq:delta_2}
    \int_\Omega (\B - \I) : \big(\I - G_\delta'(\B) \big) \dx
    &\geq \int_\Omega \trace\big(\beta_\delta(\B) + [\beta_\delta(\B)]^{-1}-2\I \big) \dx \geq 0,
\end{align}
and similarly to \eqref{eq:lemma_reg1f}, see \cite[Sec.~4.2]{barrett_boyaval_2009}, it holds
\begin{align}
\label{eq:delta_3}
    - \int_\Omega \nabla\B : \nabla G_\delta'(\B) \dx
    \geq \delta^2 \int_\Omega \abs{\nabla G_\delta'(\B)}^2 \dx.
\end{align}
Then, with arguments that are similar to \eqref{eq:formal_5}, the following inequality can be derived, for almost all $s\in(0,T)$,
\begin{align}
\begin{split}
\label{eq:delta_4}
    &  \norm{\phi(s)}_{H^1}^2
    +  \norm{\sigma(s)}_{L^2}^2 
    +  \norm{\pmbv(s)}_{L^2}^2
    +  \norm{\trace(\B(s) - G_\delta(\B(s))}_{L^1} 
    \\
    &\quad
    + 
    \norm{\mu}_{L^2(0,s;H^1)}^2 
    + \norm{\nabla \sigma}_{L^2(0,s;L^2)}^2
    + \norm{\D(\pmbv)}_{L^2(0,s;L^2)}^2 
    + \norm{\sigma}_{L^2(0,s;L^2(\partial\Omega))}^2 
    \\
    &\quad
    + \nnorm{\trace\big(\beta_\delta(\B) + [\beta_\delta(\B)]^{-1}-2\I \big)}_{L^1(0,s;L^1)}
    + \alpha \delta^2 \norm{\nabla G_\delta'(\B)}_{L^2(0,s;L^2)}^2 
    \\
    &\leq
    C\Big( 
    1 + \calF_\delta(\phi_0, \sigma_0, \pmbv_0, \B_0)
    + \norm{\sigma_\infty}_{L^2(0,T;L^2(\partial\Omega))}^2 \Big),
\end{split}
\end{align}
which holds uniformly in $\delta\in(0,\frac{1}{2}]$. Moreover, with \eqref{eq:lemma_reg1g}, it additionally holds
\begin{align}
\label{eq:delta_5}
    \norm{\B(s)}_{L^1} 
    + \frac{1}{\delta} \nnorm{ [\B(s)]_- }_{L^1} 
    \leq \norm{\trace(\B(s) - G_\delta(\B(s))}_{L^1},
\end{align}
for almost all $s\in(0,T)$, which, together with \eqref{eq:delta_4}, makes sure that the eigenvalues of $\B$ are positive in the formal limit $\delta\to 0$.

\section{Finite element approximation of the model with stress diffusion}
\label{sec:fem}

In this section, we provide a proof for Theorem \ref{theorem:weak_solution} with the following strategy.
First, we attend some ideas of \cite[Sec.~5]{barrett_boyaval_2009} and introduce a finite element approximation of the problem \ref{P_alpha_delta}, which helps us to mimic the inequality \eqref{eq:delta_4} on the fully discrete level (see Section \ref{sec:stability}) and to show that there exist stable $\delta$-regularized discrete solutions in abritrary dimensions $d\in\{2,3\}$, see Section \ref{sec:existence}. 
In Section \ref{sec:delta_to_zero}, we pass to the limit $\delta\to 0$ and obtain the existence of discrete functions for $d\in\{2,3\}$, including a positive definite discrete left Cauchy--Green tensor, which solve a finite element approximation of the problem \ref{P_alpha}. 
After that, we first improve the regularity of discrete solutions in arbitrary dimensions $d\in\{2,3\}$ in Section \ref{sec:regularity} and then restrict to $d=2$ to improve the regularity of the discrete Cauchy--Green tensor and the discrete velocity in Section \ref{sec:regularity_2D}.
Finally, in Section \ref{sec:convergence}, we send the discretization parameters to zero in order to obtain existence of a global-in-time weak solution to the problem \ref{P_alpha} in two dimensions.


Let us introduce the notation for the fully-discrete finite element approximation. 
From now on, we throughout assume that \ref{A1} holds, i.e., suppose that $T>0$ and $\Omega \subset\R^d$, $d\in\{2,3\}$, is a convex, polygonal domain with boundary $\partial\Omega$. We split the time interval $[0,T)$ into intervals $[t^{n-1},t^n)$ with $t^n = n \Delta t$ and $t^{N_T}=T$, where $\Delta t>0$ and $n=0,...,N_T$. 
We require $\{\calT_h\}_{h>0}$ to be a quasi-uniform family of conforming triangulations with mesh parameter $h>0$ (in the sense of \cite{bartels_2016}). 
We also require that the family of meshes $\{\calT_h\}_{h>0}$ consists only of non-obtuse simplices. 
For a given partitioning of meshes $\calT_h$, we denote the simplices by $K_k$ with $k\in\{1,...,N_K\}$. 
The set of internal edges of triangles ($d=2$) in the mesh $\calT_h$ or facets of tetrahedra ($d=3$) is denoted by $\partial\calT_h = \{E_j\}_{j=1}^{N_E}$. The set of all the vertices of $\calT_h$ is denoted by $\{P_p\}_{p=1}^{N_p}$.

Let us consider the problem \ref{P_alpha_delta}.
We approximate the scalar variables $\phi$, $\mu$ and $\sigma$ and the matrix valued quantity $\B$ with continuous and piecewise linear functions. 
Hence, we define the following scalar $\calP_1$-finite element space
\begin{subequations}
\begin{align}
    \calS_h &\coloneqq \left\{ q_h \in C(\overline\Omega) \mid q_h|_{K} \in \calP_1(K) \ \forall \  K\in\calT_h \right\} 
    \subset H^1(\Omega),
\end{align}
and the matrix valued $\calP_1$-finite element space 
\begin{align}
    \calW_h &\coloneqq \left\{ \B_h \in C(\overline\Omega;\R^{d\times d}_{\mathrm{S}}) \mid \B_h|_{K} \in \calP_1(K; \R^{d\times d}_{\mathrm{S}}) \ \forall \  K\in\calT_h \right\} 
    \subset H^1(\Omega;\R^{d\times d}_{\mathrm{S}}).
\end{align}
Moreover, we define
\begin{align}
    \calW_{h,\mathrm{PD}} &\coloneqq \left\{ \B_h \in \calW_h \mid \B_h(P_p) \in \R^{d\times d}_{\mathrm{SPD}} \ \forall \  p=1,...,N_p \right\}.
\end{align}
For the velocity vector $\pmbv$ and the pressure $p$, we use the $\calP_2$--$\calP_1$-Taylor--Hood element \cite{girault_raviart_2012} given by
\begin{align}
    \calV_h &\coloneqq \left\{ \pmbv_h \in C(\overline\Omega;\R^d) \cap H^1_0(\Omega;\R^d) \mid \pmbv_{h}|_{K} \in \calP_2(K;\R^d) \ \forall \  K\in\calT_h \right\},
\end{align}
for the discrete velocity and $\calS_h \cap L^2_0(\Omega)$ for the discrete pressure. 
We also introduce
\begin{align}
    \calV_{h,\mathrm{div}} &\coloneqq \left\{ \pmbv_h \in \calV_h \mid \int_\Omega \divergenz{\pmbv_h} q_h \dx = 0 \ \forall \  q_h \in \calS_h \right\},
\end{align}
which approximates the space $\mathbf{V}$.
\end{subequations}
It is well-known (cf.~\cite{girault_raviart_2012}) that this choice for the discrete velocity--pressure space satisfies the discrete Ladyzhenskaya--Babu{\v s}ka--Brezzi (LBB) stability condition
\begin{align}
    \label{eq:LBB}
    \inf_{q_h\in \calS_h} 
    \sup_{\pmbv_h\in \calV_h}
    \frac{\int_\Omega \divergenz{\pmbv_h}q_h \dx}{ \norm{q_h}_{L^2} \norm{\pmbv_h}_{H^1}  }
    \geq C > 0,
\end{align}
where, unless otherwise stated, $C>0$ always denotes a generic constant which is independent of $h,\Delta t, \alpha, \delta$.
At this point, let us mention that also other choices for the discrete velocity--pressure space can be used instead of the $\calP_2$--$\calP_1$-Taylor--Hood element as long as the discrete LBB stability condition \eqref{eq:LBB} is fulfilled. For example, the mini-element \cite{girault_raviart_2012} is also a suitable choice.

Moreover, we denote the standard nodal interpolation operator by $\calI_h: C(\overline{\Omega})\to \calS_h$ such that $(\calI_h \eta)(P_p) = \eta(P_p)$ for all $p\in\{1,...,N_p\}$ and $\eta\in C(\overline\Omega)$, which is naturally extended to $\calI_h: C(\overline\Omega;\R^{d\times d}_{\mathrm{S}}) \to \calW_h$. 
As we use \textit{mass lumping}, we introduce the following semi-inner products and the induced semi-norms on $C(\overline\Omega)$ and $C(\partial\Omega)$, respectively, by
\begin{alignat}{5}
    &\skp{\eta_1}{\eta_2}_h &&\coloneqq 
    \int_\Omega \calI_h \big[ \eta_1 \eta_2 \big] \dx,
    \quad\quad
    &&\norm{\eta_1}_h &&\coloneqq \sqrt{\skp{\eta_1}{\eta_1}_h},
    \quad\quad
    &&\forall \  \eta_1, \eta_2 \in C(\overline\Omega),
    \\
    &\skp{\eta_3}{\eta_4}_{h,\partial\Omega} &&\coloneqq 
    \int_{\partial\Omega} \calI_h \big[ \eta_3 \eta_4 \big] \dH^{d-1},
    \quad\quad
    &&\norm{\eta_3}_{h,{\partial\Omega}} &&\coloneqq \sqrt{\skp{\eta_3}{\eta_3}_{h,{\partial\Omega}}}
    \quad\quad
    &&\forall \  \eta_3, \eta_4 \in C({\partial\Omega}).
\end{alignat}

Below, we state some well-known properties concerning $\calS_h$ and the interpolant $\calI_h$. Let $K\in \calT_h$,  $0 \leq s \leq m \leq 1$ and $1 \leq r \leq p \leq \infty$. 
Then, as the family of triangulations is quasi-uniform, it holds for all $\eta\in H^2(\Omega)$ and all $q_h \in \calS_h$ that
\begin{alignat}{2}
    \label{eq:interp_H2}
    \norm{\eta - \calI_h \eta}_{L^2}
    + h \norm{\nabla(\eta-\calI_h\eta)}_{L^2} 
    &\leq C h^2 \abs{\eta}_{H^2},
    \\
    \label{eq:inverse_estimate}
    \abs{q_h}_{W^{m,p}(K)} 
    &\leq 
    C h^{s - m + \frac{d}{p} - \frac{d}{r}} \abs{q_h}_{W^{s,r}(K)},
\end{alignat}
see, e.g., \cite[Thm.~3.3]{bartels_2016} and \cite[Lem.~4.5.3]{brenner_scott_2008}, respectively.
It follows from an $L^\infty(\Omega)$-error estimate for $\calI_h$ (see \cite[Thm.~4.4.20]{brenner_scott_2008}) and an approximation argument by the Stone--Weierstrass theorem, that
\begin{align}
    \label{eq:interp_continuous}
    \lim_{h\to 0} \norm{\eta - \calI_h\eta}_{L^\infty}
    &= 0
    \quad\quad
    \forall \  \eta \in C(\overline\Omega).
\end{align}
As the basis functions associated with $\calS_h$ are non-negative and sum to one everywhere, it follows from a Cauchy--Schwarz inequality, that
\begin{align}
    \label{eq:interp_estimate}
    \abs{\calI_h \eta(x)}^2
    &\leq \calI_h\big[ \abs{\eta(x)}^2 \big]
    \quad\quad
    \forall \  x\in K, \ K \in \calT_h, \ \eta\in C(\overline\Omega).
\end{align}
We deduce from \eqref{eq:interp_estimate}, \eqref{eq:inverse_estimate} and H{\"o}lder's inequality, that, for all $q_h\in \calS_h$,
\begin{alignat}{3}
    \label{eq:norm_equiv}
    &\norm{q_h}_{L^2}^2 
    &&\leq \norm{q_h}_h^2 
    &&\leq C \norm{q_h}_{L^2}^2,
    \\
    \label{eq:norm_equiv_Gamma}
    &\norm{q_h}_{L^2({\partial\Omega})}^2 
    &&\leq \norm{q_h}_{h,{\partial\Omega}}^2 
    &&\leq C \norm{q_h}_{L^2({\partial\Omega})}^2.
\end{alignat}
Applying \eqref{eq:interp_H2} elementwise and then summing over all simplices yields the \textit{mass lumping} error estimate
\begin{align}
    \label{eq:lump_Sh_Sh}
    \abs{ \skp{q_h}{\zeta_h}_h 
    - \skp{q_h}{\zeta_h}_{L^2} } 
    &\leq
    C h^2 \norm{\nabla q_h}_{L^2} \norm{\nabla \zeta_h}_{L^2}
    \quad\quad
    \forall \   q_h, \zeta_h \in \calS_h.
\end{align}

Now, we provide a technical result concerning the \textit{mass lumping} errors on the boundary ${\partial\Omega}$. 
\begin{lemma}
Let $q_h, \zeta_h\in \calS_h$. Then, as $\{\calT_h\}_{h>0}$ is a conforming family of quasi-uniform partitionings, it holds 
\begin{align}
    \label{eq:lump_Gamma_Sh_Sh}
    \abs{ \skp{q_h}{\zeta_h}_{h,{\partial\Omega}} 
    - \skp{q_h}{\zeta_h}_{L^2({\partial\Omega})} } 
    &\leq
    C h \norm{\nabla q_h}_{L^2} \norm{\nabla \zeta_h}_{L^2}.
\end{align}
\end{lemma}

\begin{proof}
Let $E \in \partial\calT_h$ be a side simplex with diameter $h_E$ and let $K_E\in \calT_h$ such that $E\subset \partial K_E \cap {\partial\Omega}$. Then, it holds with H{\"o}lder's inequality and a local trace inequality (i.e.~\cite[Lem.~4.2]{bartels_2016}), that
\begin{align*}
\begin{split}
    \abs{ \int_E (\calI_h - \calI)[ q_h \zeta_h] \dH^{d-1} }
    &\leq
    C \abs{E}^\frac{1}{2} \norm{ (\calI_h - \calI)[ q_h \zeta_h] }_{L^2(E)}
    \\
    &\leq
    C \abs{E}^\frac{1}{2} \Big(
    h_E^{-1} \norm{ (\calI_h - \calI)[ q_h \zeta_h] }_{L^2(K_E)}^2
    + h_E \norm{ \nabla (\calI_h - \calI)[ q_h \zeta_h] }_{L^2(K_E)}^2 \Big)^\frac{1}{2},
\end{split}
\end{align*}
which gives us, on noting \eqref{eq:interp_H2} and as the family of triangulations is quasi-uniform, that
\begin{align*}
    \abs{ \int_E (\calI_h - \calI)[ q_h \zeta_h] \dH^{d-1} }
    &\leq C h^{\frac{d}{2} + 1} \abs{ q_h \zeta_h }_{H^2(K_E)}.
\end{align*}
As $q_h, \zeta_h \in \calS_h$, we obtain with a product rule and with \eqref{eq:inverse_estimate}, that
\begin{align*}
    \abs{ \int_E (\calI_h - \calI)[ q_h \zeta_h] \dH^{d-1} }
    &\leq C h^{\frac{d}{2} + 1} 
    \norm{\nabla q_h}_{L^\infty(K_E)}
    \norm{\nabla\zeta_h}_{L^2(K_E)}
    \leq C h \norm{\nabla q_h}_{L^2(K_E)}
    \norm{\nabla\zeta_h}_{L^2(K_E)}.
\end{align*}
Summing over all $E\in \partial\calT_h$ with $E\subset {\partial\Omega}$ and using the fact that each element $K_E$ occurs at most $d+1$ times imply \eqref{eq:lump_Gamma_Sh_Sh}.
\end{proof}

The results \eqref{eq:interp_H2}--\eqref{eq:lump_Gamma_Sh_Sh} can also be established with the corresponding matrix valued functions. The inverse inequality \eqref{eq:inverse_estimate} also holds for $\pmbv_h\in\calV_h$ instead of $q_h\in\calS_h$.


Furthermore, we recall the quasi-interpolation operator $\calI_h^{\mathrm{Cl}}: L^2(\Omega)\to \calS_h$ from Cl{\'e}ment \cite{clement_1975}, which is defined by local averages instead of nodal values. The following properties are taken from \cite[Chap.~3]{ciarlet}:
\begin{subequations}
\begin{alignat}{3}
    \label{eq:clement_error}
    \abs{\eta - \calI_h^{\mathrm{Cl}} \eta }_{W^{k,2}}
    &\leq C h^{m-k} \abs{\eta}_{W^{m,2}}
    \quad
    && \forall \  \eta \in W^{m,2}(\Omega),\ &&0 \leq k \leq m \leq 2,
    \\
    \label{eq:clement_conv}
    \lim\limits_{h\to 0} \norm{\eta-\calI_h^{\mathrm{Cl}} \eta}_{W^{k,2}} &= 0,
    \quad
    && \forall \  \eta \in W^{k,2}(\Omega), \ &&0 \leq k \leq 1.
\end{alignat}
Moreover, if only a finite number of patch shapes occur in the sequence of triangulations, then 
\begin{alignat}{2}
    \label{eq:clement_Gamma}
    \norm{\eta - \calI_h^{\mathrm{Cl}} \eta }_{L^2({\partial\Omega})}
    &\leq C h^{1/2} \norm{\nabla\eta}_{L^2}
    \quad
    && \forall \  \eta \in H^1(\Omega),
\end{alignat}
\end{subequations}
see \cite[Thm.~4.2]{bartels_2016}. In practice, this assumption seems to be not that restrictive. Hence, we suppose it to hold.

\subsection{Approximation of the initial and boundary values}

In this work, we require the following assumptions for the discrete initial and boundary values.

\begin{assumptions}
Suppose that the discrete initial data $( \phi_h^0, \sigma_h^0, \pmbv_h^0, \B_h^0 ) \in (\calS_h)^2 \times \calV_{h,\mathrm{div}} \times \calW_{h,\mathrm{PD}}$ and discrete boundary data $\sigma_{\infty,h}\in L^2(0,T;\calS_h)$ fulfill the following bounds uniformly in $h,\Delta t,\alpha,\delta$:
\begin{subequations}
\label{eq:init_bounds}
\begin{align}
    \label{eq:init_phi}
    \int_\Omega \calI_h\big[ \psi(\phi_h^0)\big] \dx 
    + \norm{\phi_h^0}_{H^1}^2
    + \Delta t \norm{\Delta_h \phi_h^0}_{L^2}^2
    &\leq C, 
    \\
    \label{eq:init_sig}
    \norm{\sigma_h^0}_{L^2}^2
    + \Delta t \norm{\nabla \sigma_h^0}_{L^2}^2
    + \Delta t \norm{\sigma_h^0}_{L^2(\partial\Omega)}^2
    &\leq C, 
    \\
    \label{eq:init_v}
    \norm{\pmbv_h^0}_{L^2}^2 
    + \Delta t \norm{\nabla\pmbv_h^0}_{L^2}^2 
    &\leq C,
    \\
    \label{eq:init_B}
    \norm{\B_h^0}_{L^2}^2
    + \Delta t \norm{\nabla\B_h^0}_{L^2}^2
    &\leq C,
    \\
    \label{eq:bc_sig}
    \norm{\sigma_{\infty,h}}_{L^2(0,T;L^2(\partial\Omega))} 
    &\leq C,
\end{align}
and, with constants $0<\tilde b^0_{\min}\leq \tilde b^0_{\max}$,
\begin{align}
    \label{eq:init_B_spd}
    \tilde b^0_{\min} \abs{\pmb\xi}^2 
    \leq \pmb\xi^T \B_h^0(P_p) \pmb\xi
    &\leq \tilde b^0_{\max} \abs{\pmb\xi}^2
    \quad\quad \forall \  \pmb\xi\in \R^d, \ \forall \  p\in\{1,...,N_p\}.
\end{align}
\end{subequations}
\end{assumptions}
Here, $\Delta_h: \calS_h \to \left\{z_h \in \calS_h \mid \int_\Omega z_h \dx = 0 \right\}$ denotes the discrete Neumann--Laplacian such that $\Delta_h q_h$ is the unique solution of
\begin{align}
    \label{eq:discr_laplace}
    \skp{\Delta_h q_h}{\zeta_h}_{h} 
    = \int_\Omega \calI_h\big[ (\Delta_h q_h) \zeta_h\big] \dx
    = - \int_\Omega \nabla q_h \cdot \nabla\zeta_h \dx
    = - \skp{\nabla q_h}{\nabla\zeta_h}_{L^2}
    \quad\quad \forall \  \zeta_h\in \calS_h.
\end{align}
We note for future reference, as $\{\calT_h\}_{h>0}$ is a quasi-uniform family of partitionings, and,
as the domain $\Omega$ is convex, that 
\begin{align}
    \label{eq:discr_laplace_bound}
    \abs{q_h}_{W^{1,s}} 
    \leq C \norm{\Delta_h q_h}_{L^2}
    \qquad \forall \  q_h\in \calS_h, \ \forall \  s\in\left[1,\tfrac{2d}{d-2}\right),
\end{align} 
see, e.g., \cite[Lem.~3.1]{barrett_langdon_nuernberg_2004} or \cite[Thm.~6.4]{gruen_2013}.

Moreover, we define for all $t\in[t^{n-1},t^n)$ and 
$n\in\{1,...,N_T\}$ the piecewise constant in time approximation of $\sigma_{\infty,h}$ by
\begin{align}
    \label{eq:def_bc}
    \sigma_{\infty,h}^{\Delta t, +} (t,\cdot) 
    \coloneqq \sigma_{\infty,h}^n(\cdot) 
    \coloneqq 
    \frac{1}{\Delta t} \int_{t^{n-1}}^{t^n} \sigma_{\infty,h}(\tilde t,\cdot) \dv{\tilde t} \ \in \calS_h,
\end{align}
which fulfills
\begin{subequations}
\begin{align}
    \label{eq:bounds_bc}
    \Delta t \sum_{n=1}^{N_T} \norm{\sigma_{\infty,h}^n}_{L^2({\partial\Omega})}^2
    = \norm{\sigma_{\infty,h}^{\Delta t, +}}_{L^2(0,T;L^2({\partial\Omega}))}^2 
    \leq \nnorm{\sigma_{\infty,h}}^2_{L^2(0,T;L^2(\partial\Omega))},
    \\
    \sigma_{\infty,h}^{\Delta t, +} \to \sigma_{\infty,h}
    \quad \text{strongly in } L^2(0,T;\calS_h), 
    \quad \text{as } \Delta t\to 0.
\end{align}
\end{subequations}

Furthermore, we make the following assumption on the discrete initial and boundary data which is needed for the limit process $(h,\Delta t)\to(0,0)$.
\begin{assumptions}
Let \ref{A5} hold true. Then, in the limit $(h,\Delta t)\to (0,0)$, we assume
\begin{subequations}
\label{eq:init_conv}
\begin{alignat}{3}
    \label{eq:init_phi_conv}
    \phi_h^0 &\to \phi_0    
    \quad &&\text{ weakly } 
    \quad &&\text{ in } L^2(\Omega),
    \\
    \label{eq:init_sig_conv}
    \sigma_h^0 &\to \sigma_0 
    \quad  &&\text{ weakly }
    \quad &&\text{ in } L^2(\Omega), 
    \\
    \label{eq:init_v_conv}
    \pmbv_h^0 &\to \pmbv_0 
    \quad  &&\text{ weakly }
    \quad &&\text{ in } \mathbf{H}, 
    \\
    \label{eq:init_B_conv}
    \B_h^0 &\to \B_0 
    \quad  &&\text{ weakly }
    \quad &&\text{ in } L^2(\Omega;\R^{d\times d}), 
    \\
    \label{eq:bc_conv}
    \sigma_{\infty,h}|_{\partial\Omega} &\to \sigma_\infty|_{\partial\Omega}
    \quad  && \text{ strongly }
    \quad &&\text{ in } L^2(0,T;L^2(\partial\Omega)).
\end{alignat}
\end{subequations}
\end{assumptions}

\begin{remark}
The assumptions \eqref{eq:init_bounds} and \eqref{eq:init_conv} are no severe constraints in practice.
For example, let \ref{A5} hold true and $\psi:\R\to\R$ be continuous. Then, the following choices for $\phi_h^0,\sigma_h^0,\pmbv_h^0,\B_h^0, \sigma_{\infty,h}$ are in accordance with \eqref{eq:init_bounds} and \eqref{eq:init_conv}:
\begin{subequations}
\label{eq:def_initial}
\begin{alignat}{2}
    \phi_h^0 &= \calI_h \phi_0, 
    \\
    \int_\Omega \calI_h\big[ \sigma_h^0 q_h \big] \dx
    + \Delta t \int_\Omega \nabla\sigma_h^0 \cdot \nabla q_h \dx
    + \Delta t \int_{\partial\Omega} \calI_h\big[ \sigma_h^0 q_h \big] \dH^{d-1}
    &= \int_\Omega \sigma_0 q_h \dx
    \qquad
    &&\forall \  q_h\in \calS_h,
    \\
    \int_\Omega \pmbv_h^0\cdot\pmbw_h \dx
    + \Delta t \int_\Omega \nabla\pmbv_h^0 : \nabla\pmbw_h \dx
    &= \int_\Omega \pmbv_0 \cdot \pmbw_h \dx
    \qquad
    &&\forall \  \pmbw_h \in \calV_{h,\mathrm{div}},
    \\
    \int_\Omega \calI_h\big[ \B_h^0 : \C_h \big] \dx
    + \Delta t \int_\Omega \nabla\B_h^0 : \nabla\C_h \dx
    &= \int_\Omega \B_0 : \C_h \dx
    \qquad
    &&\forall \  \C_h\in \calW_h,
    \\
    \sigma_{\infty,h} &= \calI_h^{\mathrm{Cl}} \sigma_\infty.
\end{alignat}
\end{subequations}
We note that \eqref{eq:init_phi} follows from \eqref{eq:interp_H2}, \cite[eq.~(3.16)]{barrett_nurnberg_styles_2004} and \ref{A5}. Moreover, \eqref{eq:init_sig}--\eqref{eq:init_B} are a direct consequence of Hölder's and Young's inequalities, \eqref{eq:norm_equiv}, \eqref{eq:norm_equiv_Gamma} and \ref{A5}. 
As we have a triangulation with non-obtuse simplices, \eqref{eq:init_B_spd} follows from \ref{A5} and \cite[Lem.~5.2]{barrett_boyaval_2009}. 
Furthermore, \eqref{eq:clement_Gamma} and \ref{A5} yield \eqref{eq:bc_sig}.

Moreover, \ref{A5} and the error estimates \eqref{eq:interp_H2} and \eqref{eq:clement_Gamma} imply \eqref{eq:init_phi_conv} and \eqref{eq:bc_conv}, respectively. 
Besides, \eqref{eq:init_sig_conv} follows from \ref{A5}, \eqref{eq:init_sig}, \eqref{eq:lump_Sh_Sh}, \eqref{eq:lump_Gamma_Sh_Sh}, the denseness of $H^1(\Omega)$ in $L^2(\Omega)$ and the fact that for all $q\in H^1(\Omega)$ there exists a sequence $\{q_h\}_{h>0}\subset \calS_h$ such that $\norm{q_h-q}_{H^1}\to 0$, as $h\to 0$.
Similarly, \eqref{eq:init_v_conv} follows from \ref{A5}, \eqref{eq:init_v}, the denseness of $\mathbf{V}$ in $\mathbf{H}$ and the fact that for all $\pmbw\in \mathbf{V}$ there exists a sequence $\{\pmbw_h\}_{h>0} \subset \calV_{h,\mathrm{div}}$ such that $\norm{\pmbw_h-\pmbw}_{H^1}\to 0$, as $h\to 0$, which is due to \eqref{eq:LBB}. The remaining identity \eqref{eq:init_B_conv} follows with similar arguments.

\end{remark}

\subsection{A regularized fully discrete finite element approximation}


Now, for given $\delta\in(0,\frac{1}{2}]$, we introduce a fully discrete approximation of \ref{P_alpha_delta}.
There are several difficulties on the fully discrete level which have to be taken into account.
One of the most important issues arises 
from the fact that $\B_h \in \calW_h$ only implies $\calI_h[G_\delta'(\B_h)] \in \calW_h$, as in general $G_\delta'(\B_h) \not\in \calW_h$. For that reason, it is not clear that the analogues of \eqref{eq:formal_delta_1}--\eqref{eq:delta_3} can be performed on the discrete level, especially controlling the convective term in \eqref{eq:B2_delta}.

Here, the approach of Barrett and Boyaval \cite[Sec.~5]{barrett_boyaval_2009} is very helpful.
We recall the fourth order tensorial function $\Lambda_\delta: \calW_h \to \R^{d^4}$, where the symmetric $(d\times d)$-matrix $\Lambda_{\delta,i,j}(\B_h)$ approximates $\delta_{i,j} \beta_\delta(\B_h)$ in a certain sense, where $i,j\in\{1,...,d\}$ and $\B_h\in\calW_h$ and $\delta_{i,j}$ denotes the Kronecker delta. The reason for introducing this nonlinear quantity is to control the discrete version of the convective term from \eqref{eq:B2_delta}, which is due to the following property:
\begin{align}
    \label{eq:Lambda1}
    \sum\limits_{j=1}^d 
    \Lambda_{\delta,i,j}(\B_h) : \partial_{x_j} \calI_h\big[ G_\delta'(\B_h) \big]
    = \partial_{x_i} \calI_h\big[ \trace\big( 
    H_\delta( G_\delta'(\B_h)) \big) \big]
    \quad \text{ on } K_k,
\end{align}
for $k\in\{1,...,N_k\}$ and $i\in\{1,...,d\}$, see \cite[eq.~(5.17)]{barrett_boyaval_2009},
which will make it possible to derive an \textit{a priori} estimate on the fully discrete level.
As the family of partitionings $\{\calT_h\}_{h>0}$ is quasi-uniform, 
it follows from the definition of $\Lambda_{\delta,i,j}$ (cf.~\cite[Sec.~5.1]{barrett_boyaval_2009}) that
\begin{align}
    \label{eq:Lambda2}
    \norm{ \Lambda_{\delta,i,j}(\B_h) }_{L^\infty(\Omega)} 
    \leq C \norm{\beta_\delta(\B_h)}_{L^\infty(\Omega)}
    \quad \forall \  \B_h\in \calW_h.
\end{align}
%
%
%

Next, we present an approximation of \ref{P_alpha_delta} for which we explain the motivation afterwards.

\subsubsection*{Problem \ref{P_alpha_delta_FE}:}
\mylabelHIDE{P_alpha_delta_FE}{$(\pmbP_{\alpha,\delta,h}^{\Delta t})$} 
Let $\delta\in(0,\frac{1}{2}]$. For given discrete initial and boundary data satisfying \eqref{eq:init_bounds} and $n\in\{1,...,N_T\}$, find the discrete solution $(\phi_h^{n}, \mu_h^n, \sigma_{h}^{n}, p_h^n, \pmbv_{h}^{n}, \B_{h}^{n}) \in (\calS_h)^4 \times \calV_h  \times \calW_h$ which satisfies, for any $(\zeta_h, \rho_h, \xi_h, q_h, \pmbw_h, \C_h) \in (\calS_h)^4 \times \calV_h \times \calW_h$:
%
%
\begin{subequations}
\begingroup
\allowdisplaybreaks
\begin{align}
    \label{eq:phi_FE_delta}
    0 &= \int_\Omega \calI_h \Big[ \Big(\frac{\phi_h^n-\phi_h^{n-1}}{\Delta t}
    - \Gamma_{\phi,h}^n \Big) \zeta_h \Big]
    + \calI_h[m(\phi_h^{n-1})] \nabla\mu_h^n \cdot \nabla \zeta_h
    - \phi_h^{n-1} \pmbv_h^{n} \cdot \nabla \zeta_h \dx,
    \\
    \label{eq:mu_FE_delta}
    0 &= \int_\Omega \calI_h \Big[ \Big( - \mu_h^n  
    + A \psi_1'(\phi_h^n) 
    + A \psi_2'(\phi_h^{n-1}) 
    - \chi_\phi \sigma_h^n \Big) \rho_h \Big] 
    + B \nabla\phi_h^n \cdot \nabla\rho_h\dx,
    \\
    \nonumber
    \label{eq:sigma_FE_delta}
    0 &= \int_\Omega \calI_h \Big[ \Big(\frac{\sigma_h^n-\sigma_h^{n-1}}{\Delta t}
    + \Gamma_{\sigma,h}^n \Big) \xi_h \Big]
    + \calI_h[n(\phi_h^{n-1})] 
    \nabla (\chi_\sigma \sigma_h^n - \chi_\phi \phi_h^n)  \cdot \nabla \xi_h  
    -  \sigma_h^{n-1} \pmbv_h^{n} \cdot \nabla\xi_h \dx
    \\
    &\qquad + \int_{\partial\Omega} \calI_h\Big[ K \big(\sigma_h^n - \sigma_{\infty,h}^n \big) \xi_h \Big] \dH^{d-1},
    \\
    \label{eq:div_v_FE_delta}
    0 &= \int_\Omega \divergenz{\pmbv_h^{n}} q_h \dx,
    \\
    \label{eq:v_FE_delta}
    \nonumber
    0 &= \int_\Omega \frac{\pmbv_h^n-\pmbv_h^{n-1}}{\Delta t} \cdot \pmbw_h
    + \frac{1}{2} \left( \left(\pmbv_h^{n-1}\cdot \nabla\right) \pmbv_h^n\right) \cdot \pmbw_h
    - \frac{1}{2} \pmbv_h^n \cdot \left(\left(\pmbv_h^{n-1} \cdot \nabla\right) \pmbw_h \right) 
    + 2\calI_h[\eta(\phi_h^{n-1})] \D(\pmbv_h^n) : \D(\pmbw_h)\dx 
    \\
    &\qquad + \int_\Omega 
    \kappa \calI_h\big[ \beta_\delta(\B_h^n) - \I \big] : \nabla\pmbw_h 
    - \divergenz{\pmbw_h} p_h^n 
    + \big( \phi_h^{n-1} \nabla\mu_h^n 
    +  \sigma_h^{n-1} \nabla (\chi_\sigma \sigma_h^n - \chi_\phi \phi_h^n) \big) \cdot \pmbw_h \dx,
    \\
    \label{eq:B_FE_delta}
    \nonumber
    0 &= \int_\Omega \calI_h \Big[ 
    \Big(\frac{\B_h^n - \B_h^{n-1}}{\Delta t}
    + \frac{\kappa}{\tau(\phi_h^{n-1})} (\B_h^n - \I) \Big): \C_h \Big]
    - 2 \nabla\pmbv^n_h : \calI_h\big[ \C_h \beta_\delta(\B_h^n) \big] 
    + \alpha \nabla\B_h^n : \nabla\C_h \dx
    \\
    &\qquad - \int_\Omega \sum\limits_{i,j=1}^d  
    [\pmbv_h^{n-1}]_i \Lambda_{\delta,i,j}(\B_h^n) : \partial_{x_j} \C_h \dx,
\end{align}
\endgroup
\end{subequations}
where we define
$\Gamma_{\phi,h}^n \coloneqq\Gamma_\phi(\phi_h^n,\sigma_h^n,\B_h^n)$ and $\Gamma_{\sigma,h}^n \coloneqq \Gamma_\sigma(\phi_h^n,\sigma_h^n)$.


Let us now motivate the idea for \ref{P_alpha_delta_FE} by explaining the derivation from the weak formulation of \ref{P_alpha} in the sense of Definition \ref{def:weak_solution}.
First, the $\delta$-regularization strategy from Section \ref{sec:regularization} is applied and, as we use finite element functions for the approximation in space, we also use the fourth order tensor $\Lambda_\delta$ to control the convection term for $\B$ on the discrete level in \eqref{eq:B_FE_delta}. Besides, a semi-implicit time discretization of first order is chosen where linear terms are treated fully implicitly and most of the nonlinear terms are treated explicitly. In \eqref{eq:mu_FE_delta}, a convex-concave splitting for the potential $\psi=\psi_1+\psi_2$ is chosen, which allows the inequality
\begin{align}
\label{eq:convex_concave}
    \left(\psi_1'(\phi_h^n) + \psi_2'(\phi_h^{n-1}) \right) (\phi_h^n- \phi_h^{n-1}) \geq \psi(\phi_h^n) - \psi(\phi_h^{n-1}).
\end{align}
Besides, the nonlinear source terms $\Gamma_\phi,\Gamma_\sigma$ are treated fully implicitly and the nonlinear functions $n, m, \eta, \tau$ are treated explicitly, but also a different time approximation can be chosen for these terms. 
The remaining terms are approximated in a way such that stability of the scheme \ref{P_alpha_delta_FE} can be shown, see Lemma \ref{lemma:bounds_FE_delta}.
Furthermore, we make use of numerical integration in terms of the nodal interpolation operator $\calI_h$. On the one hand, this can reduce the computational effort as the mass matrices are diagonal, whereas on the other hand, the nodal interpolation operator in \eqref{eq:B_FE_delta} and in the second line in \eqref{eq:v_FE_delta} is required for stability of the scheme.

\begin{remark}
In the literature, the velocity field in Navier--Stokes systems is sometimes approximated with finite element functions where the constraint \eqref{eq:div_v_FE_delta} is directly included in the finite element space $\calV_{h,\mathrm{div}}$. Hence, the velocity field and the test functions in equation \eqref{eq:v_FE_delta} would belong to the finite element space $\calV_{h,\mathrm{div}}$, and \eqref{eq:v_FE_delta} would be replaced by 
\begin{align}
    \label{eq:v_FE_delta_b}
    \nonumber
    0 &= \int_\Omega \frac{\pmbv_h^n-\pmbv_h^{n-1}}{\Delta t} \cdot \pmbw_h
    + \frac{1}{2} \left( \left(\pmbv_h^{n-1}\cdot \nabla\right) \pmbv_h^n\right) \cdot \pmbw_h
    - \frac{1}{2} \pmbv_h^n \cdot \left(\left(\pmbv_h^{n-1} \cdot \nabla\right) \pmbw_h \right) 
    + 2\calI_h[\eta(\phi_h^{n-1})] \D(\pmbv_h^n) : \D(\pmbw_h) \dx 
    \\
    &\qquad + \int_\Omega 
    \kappa \calI_h\big[ \beta_\delta(\B_h^n) - \I \big] : \nabla\pmbw_h 
    + \big( \phi_h^{n-1} \nabla\mu_h^n 
    + \sigma_h^{n-1} \nabla (\chi_\sigma \sigma_h^n - \chi_\phi \phi_h^n) \big) \cdot \pmbw_h \dx,
\end{align}
for all $\pmbw_h\in\calV_{h,\mathrm{div}}$, where $\pmbv_h^{n-1}, \pmbv_h^n \in\calV_{h,\mathrm{div}}$ are the solution from the previous time step and the unknown solution from the current time step, respectively.
The unknown pressure $p_h^n\in \calS_h$, which is unique up to an additive constant, can be reconstructed afterwards as the discrete LBB stability condition \eqref{eq:LBB} is fulfilled, see, e.g., \cite[Chap.~I, Lem.~4.1]{girault_raviart_2012} or \cite[Lem.~4.2]{braess_2007}. 
However, it is rather hard to construct test functions $\pmbw_h\in \calV_{h,\mathrm{div}}$ in practice. 
This is the reason why we use \eqref{eq:v_FE_delta} instead of \eqref{eq:v_FE_delta_b}.

\end{remark}

\subsection{Stability of the regularized discrete system}
\label{sec:stability}

We now introduce the discrete energy $\calF_{\delta,h}: \calS_h\times \calS_h \times \calV_h \times \calW_h \to \R$ of the problem \ref{P_alpha_delta_FE} given by 
\begin{align}
\label{eq:def_energy_FE_delta}
\begin{split}
    \calF_{\delta,h}(\phi_h,\sigma_h,\pmbv_h,\B_h) 
    &= 
    \int_\Omega 
    \calI_h\Big[ A \psi(\phi_h) 
    + \frac{\chi_\sigma}{2} \abs{\sigma_h}^2
    + \chi_\phi \sigma_h (1-\phi_h)
    + \frac{\kappa}{2} \trace\big( \B_h - G_\delta(\B_h) \big)\Big] \dx
    \\
    &\quad + \int_\Omega 
    \frac{B}{2} \abs{\nabla\phi_h}^2 
    + \frac{1}{2} \abs{\pmbv_h}^2 \dx,
\end{split}
\end{align}
for all $(\phi_h,\sigma_h,\pmbv_h,\B_h)\in \calS_h \times \calS_h \times \calV_h \times \calW_h$, where $\delta\in(0,\frac{1}{2}]$.
We remark that it is not guaranteed that $\calF_{\delta,h}$ is non-negative as the term $\sigma_h(1-\phi_h)$ can have a negative sign. This is one of the main difficulties we have to handle in the derivation of useful \textit{a priori} estimates.

For future reference, we note the elementary identity
\begin{align}
    \label{eq:elementary_identity}
    2x(x-y) = x^2-y^2 + (x-y)^2 \quad\quad \forall \  x,y\in\R.
\end{align}
Moreover, we recall the following discrete version of Gronwall's inequality, i.e.~Lemma \ref{lemma:gronwall}. For a proof, we refer to, e.g., \cite[pp.~401--402]{dahmen_reusken_numerik}. 

\begin{lemma}
\label{lemma:gronwall_discrete}
Assume that $e_n, a_n, b_n \geq 0$ for all $n\in \N_0$. Then 
\begin{align}
\label{eq:gronwall_discrete}
    e_n \leq a_n + \sum\limits_{i=0}^{n-1}  b_i e_i 
    \quad \forall \  n\in \N_0
    \quad
    \Longrightarrow \quad
    e_n  \leq a_n \cdot 
    \exp\Big( \sum\limits_{i=0}^{n-1} b_i \Big)
    \quad \forall \  n\in \N_0.
\end{align}
\end{lemma}



With the help of Lemma \ref{lemma:gronwall_discrete}, we derive stability bounds for solutions of \ref{P_alpha_delta_FE}.

\begin{lemma}[Stability]
\label{lemma:bounds_FE_delta}
Let \ref{A1}--\ref{A5} hold true and let $\delta\in(0,\frac{1}{2}]$. Suppose that the discrete initial and boundary data satisfy \eqref{eq:init_bounds} and assume that $\Delta t < \Delta t_*$, where the constant $\Delta t_*$ depends only on the model parameters and is defined in \eqref{eq:dt}. Then, for $n\in\{1,...,N_T\}$, a solution $(\phi_h^{n}$, $\mu_h^n$, $\sigma_{h}^{n}$, $p_h^n$, $\pmbv_{h}^{n}$, $\B_{h}^{n}) \in (\calS_h)^4 \times \calV_h \times \calW_h$ to the problem \ref{P_alpha_delta_FE}, if it exists, satisfies 
\begin{align}
\begin{split}
    \label{eq:bounds_FE_delta}
    &  \max_{n=1,...,N_T} \Big( 
    \norm{\phi_h^n}_{H^1}^2
    +  \norm{\sigma_h^n}_{L^2}^2
    +  \norm{\pmbv_h^n}_{L^2}^2
    +  \nnorm{ \calI_h\big[ \abs{\B_h^n} \big] }_{L^1}
    +  \frac{1}{\delta} \nnorm{ \calI_h\big[ \abs{ [\B_h^n]_- } \big] }_{L^1}  \Big)
    \\
    &\quad
    +  \sum_{n=1}^{N_T} \Big(
    \norm{\nabla\phi_h^n - \nabla\phi_h^{n-1}}_{L^2}^2 
    +  \norm{\sigma_h^n - \sigma_h^{n-1}}_{L^2}^2
    +  \norm{\pmbv_h^n - \pmbv_h^{n-1}}_{L^2}^2 \Big)
    \\
    &\quad
    + \Delta t \sum_{n=1}^{N_T} \Big( 
    \norm{\mu_h^n}_{H^1}^2 
    + \norm{\nabla \sigma_h^n}_{L^2}^2
    + \norm{\sigma_h^n}_{L^2({\partial\Omega})}^2 
    + \norm{\nabla \pmbv_h^n}_{L^2}^2   \Big)
    \\
    &\quad
    + \Delta t \sum_{n=1}^{N_T} \bigg( 
    \alpha\delta^2 \nnorm{ \nabla \calI_h\big[ G'_\delta(\B_h^n)\big] }_{L^2}^2
    + \int_\Omega
    \calI_h\Big[ \trace\big(\beta_\delta(\B_h^n) + [\beta_\delta(\B_h^n)]^{-1}-2\I \big) \Big] \dx  \bigg)
    \\
    &\leq 
    C(T) \Big( 1 + \abs{\calF_{h,\delta}(\phi_h^0, \sigma_h^0, \pmbv_h^0, \B_h^0)}
    + \Delta t \sum_{n=1}^{N_T} \norm{\sigma_{\infty,h}^n}_{L^2(\partial\Omega)}^2 \Big) 
    \leq C(T),
\end{split}
\end{align}
where the constants $C(T)$ are independent of $h, \Delta t, \alpha, \delta$, but depend exponentially on $T$.
\end{lemma}

\begin{proof} We now start with the testing procedure.
First, we choose $\zeta_h = \mu_h^n$ in \eqref{eq:phi_FE_delta}, $\rho_h = \frac{1}{\Delta t}(\phi_h^n-\phi_h^{n-1})$ in \eqref{eq:mu_FE_delta}, $\xi_h = \chi_\sigma \sigma_h^n + \chi_\phi (1-\phi_h^n)$ in \eqref{eq:sigma_FE_delta} and $\pmbw_h=\pmbv_h$ in \eqref{eq:v_FE_delta} and sum up the resulting equations. Then, we obtain on noting \ref{A3}, \ref{A4}, \eqref{eq:convex_concave} and \eqref{eq:elementary_identity}, that
\begin{align}
\begin{split}
\label{eq:energy_FE_1}
    &  \frac{B}{2\Delta t} \Big( \norm{\nabla\phi_h^n}_{L^2}^2 - \norm{\nabla\phi_h^{n-1}}_{L^2}^2  
    +  \norm{\nabla\phi_h^n - \nabla\phi_h^{n-1}}_{L^2}^2 \Big) 
    + \int_\Omega \frac{A}{\Delta t} \calI_h\Big[  \psi(\phi_h^n)-\psi(\phi_h^{n-1}) \Big]\dx
    \\
    &\quad 
    + \frac{\chi_\sigma}{2\Delta t}  \Big( 
    \norm{\sigma_h^n}_h^2 - \norm{\sigma_h^{n-1}}_h^2
    + \norm{\sigma_h^n - \sigma_h^{n-1}}_h^2
    \Big)
    + \frac{1}{2 \Delta t}
    \Big( \norm{\pmbv_h^n}_{L^2}^2 
    - \norm{\pmbv_h^{n-1}}_{L^2}^2 
    + \norm{\pmbv_h^n - \pmbv_h^{n-1}}_{L^2}^2 \Big)
    \\
    &\quad 
    + 
    m_0 \norm{\nabla\mu_h^n}_{L^2}^2 
    +  n_0  \nnorm{\chi_\sigma \nabla \sigma_h^n - \chi_\phi \nabla\phi_h^n}_{L^2}^2 
    + K \chi_\sigma \norm{\sigma_h^n}_{h,{\partial\Omega}}^2
    + 2 \eta_0 \norm{\D(\pmbv_h^n)}_{L^2}^2
    \\
    &\quad
    + \int_\Omega \calI_h \Big[ 
    \Gamma_{\sigma,h}^n \big(\chi_\sigma \sigma_h^n + \chi_\phi (1-\phi_h^n)\big) 
    - \mu_h^n \Gamma_{\phi,h}^n  
    + \chi_\phi (1-\phi_h^n) \Big(\frac{\sigma_h^n - \sigma_h^{n-1}}{\Delta t}\Big)  
    - \frac{\phi_h^n-\phi_h^{n-1}}{\Delta t} \chi_\phi \sigma_h^n \Big] \dx
    \\
    &\quad 
    + \int_\Omega 
    \kappa \calI_h\big[ \beta_\delta(\B_h^n) - \I \big] : \nabla\pmbv_h^n \dx
    + \int_{\partial\Omega} \calI_h\Big[
    K \chi_\phi \sigma_h^n (1-\phi_h^n)
    - K \sigma_{\infty,h}^n \big(\chi_\sigma \sigma_h^n + \chi_\phi (1-\phi_h^n)\big)  \Big] \dH^{d-1}
    \\
    &\leq 0.
\end{split}
\end{align}

Next, we test \eqref{eq:B_FE_delta} with $\C_h = \frac{\kappa}{2}\big( \I - \calI_h\big[ G'_\delta(\B_h^n)\big] \big)$. This gives
\begin{align*}
    0 &=
    \int_\Omega \calI_h \Big[ 
    \Big(\frac{\B_h^n - \B_h^{n-1}}{\Delta t}
    + \frac{\kappa}{\tau(\phi_h^{n-1})} (\B_h^n - \I) \Big)
    : \frac{\kappa}{2} \big( \I - G'_\delta(\B_h^n) \big) \Big] 
    - 2 \nabla\pmbv^n_h : \calI_h\left[ \frac{\kappa}{2} \big( \I -  G'_\delta(\B_h^n) \big)
    \beta_\delta(\B_h^n) \right] 
    \\
    &\quad
    + \int_\Omega \frac{\alpha\kappa}{2} \nabla\B_h^n : \nabla \calI_h\big[ G'_\delta(\B_h^n)\big] 
    +  \sum\limits_{i,j=1}^d  
    [\pmbv_h^{n-1}]_i \Lambda_{\delta,i,j}(\B_h^n) : \frac{\kappa}{2} \partial_{x_j} \calI_h\big[ G'_\delta(\B_h^n)\big] \dx.
\end{align*}
We now estimate the terms on the right-hand side.
Together with \eqref{eq:lemma_reg1e}, it follows that
\begin{align*}
    \Big(\frac{\B_h^n - \B_h^{n-1}}{\Delta t} \Big) : \big(\I-G'_\delta(\B_h^n)\big)
    \geq 
    \frac{ \trace\big(\B_h^n - G_\delta(\B_h^n)\big) 
    - \trace\big(\B_h^{n-1} - G_\delta(\B_h^{n-1})\big)}{\Delta t}.
\end{align*}
We have on noting \eqref{eq:lemma_reg1a}, \eqref{eq:lemma_reg1b} and \eqref{eq:lemma_reg1d} that
\begin{align*}
    &(\B_h^n - \I) : \big(\I - G_\delta'(\B_h^n) \big)
    \geq \big( \beta_\delta(\B_h^n) - \I \big) 
    : \big(\I - G_\delta'(\B_h^n)\big)
    = \trace\big(\beta_\delta(\B_h^n) + [\beta_\delta(\B_h^n)]^{-1}-2\I \big)
    \geq 0.
\end{align*}
Moreover, using \eqref{eq:lemma_reg1a} yields
\begin{align*}
    &\int_\Omega 2 \nabla\pmbv^n_h : \calI_h\left[ \frac{\kappa}{2} \big( \I -  G'_\delta(\B_h^n) \big) \beta_\delta(\B_h^n) \right] \dx
    = \int_\Omega \kappa (\nabla\pmbv_h^n) :
    \calI_h \big[  \beta_\delta(\B_h^n) - \I \big] \dx.
\end{align*}
On noting \cite[Lem.~5.1]{barrett_boyaval_2009} we obtain
\begin{align*}
    &- \int_\Omega \frac{\alpha\kappa}{2} \nabla\B_h^n : \nabla \calI_h\big[ G'_\delta(\B_h^n)\big] \dx
    \geq 
    \int_\Omega \frac{\alpha\kappa\delta^2}{2} \abs{ \nabla \calI_h\big[ G'_\delta(\B_h^n)\big] }^2 \dx.
\end{align*}
As $\pmbv_h^{n-1} \in \calV_{h,\text{div}}$ and $\calI_h\Big[ \trace\big( H_\delta(G_\delta'(\B_h^n)) \big) \Big] \in \calS_h$, we get with \eqref{eq:Lambda1} and integration by parts that
\begin{align*}
    \int_\Omega \sum\limits_{i,j=1}^d  
    [\pmbv_h^{n-1}]_i \Lambda_{\delta,i,j}(\B_h^n) : \frac{\kappa}{2} \partial_{x_j} \calI_h\big[ G'_\delta(\B_h^n)\big] \dx
    &= \int_\Omega \pmbv_h^{n-1} \cdot \nabla \calI_h\Big[ \trace\big( H_\delta(G_\delta'(\B_h^n)) \big) \Big] \dx
    = 0.
\end{align*}
Therefore, on noting \ref{A3}, we have the inequality
\begin{align}
\begin{split}
\label{eq:energy_FE_2}
    &\int_\Omega 
    \frac{\kappa}{2\Delta t} \calI_h\Big[ \trace\big(\B_h^n - G_\delta(\B_h^n)\big) 
    - \trace\big(\B_h^{n-1} - G_\delta(\B_h^{n-1}) \big) \Big] 
    + \frac{\kappa^2}{2\tau_1}
    \calI_h\Big[ \trace\big(\beta_\delta(\B_h^n) + [\beta_\delta(\B_h^n)]^{-1}-2\I \big) \Big]
    \dx
    \\
    &\quad 
    + \int_\Omega 
    \frac{\alpha\kappa\delta^2}{2} \abs{ \nabla \calI_h\big[ G'_\delta(\B_h^n)\big] }^2
    - \kappa \nabla\pmbv_h^n : \calI_h\big[ \beta_\delta(\B_h^n) \big] \dx
    \\&\leq 0.
\end{split}
\end{align}
We deduce from \eqref{eq:energy_FE_1} and \eqref{eq:energy_FE_2} that
\begin{align}
\begin{split}
\label{eq:energy_FE_3}
    &  \frac{B}{2\Delta t} \Big( \norm{\nabla\phi_h^n}_{L^2}^2 - \norm{\nabla\phi_h^{n-1}}_{L^2}^2  
    +  \norm{\nabla\phi_h^n - \nabla\phi_h^{n-1}}_{L^2}^2 \Big) 
    + \int_\Omega \frac{A}{\Delta t} \calI_h\Big[  \psi(\phi_h^n)-\psi(\phi_h^{n-1}) \Big]\dx
    \\
    &\quad 
    + \frac{\chi_\sigma}{2\Delta t}  \Big( 
    \norm{\sigma_h^n}_h^2 - \norm{\sigma_h^{n-1}}_h^2
    + \norm{\sigma_h^n - \sigma_h^{n-1}}_h^2
    \Big)
    + \frac{1}{2 \Delta t}
    \Big( \norm{\pmbv_h^n}_{L^2}^2 
    - \norm{\pmbv_h^{n-1}}_{L^2}^2 
    + \norm{\pmbv_h^n - \pmbv_h^{n-1}}_{L^2}^2 \Big)
    \\
    &\quad
    + \int_\Omega 
    \frac{\kappa}{2\Delta t} \calI_h\Big[ \trace\big(\B_h^n - G_\delta(\B_h^n)\big) 
    - \trace\big(\B_h^{n-1} - G_\delta(\B_h^{n-1}) \big) \Big] \dx
    \\
    &\quad 
    + m_0 \norm{\nabla\mu_h^n}_{L^2}^2 
    +  n_0  \nnorm{\chi_\sigma \nabla \sigma_h^n - \chi_\phi \nabla\phi_h^n}_{L^2}^2 
    + K \chi_\sigma \norm{\sigma_h^n}_{h,{\partial\Omega}}^2
    + 2 \eta_0 \norm{\D(\pmbv_h^n)}_{L^2}^2
    \\
    &\quad
    + \frac{\alpha\kappa\delta^2}{2} \nnorm{ \nabla \calI_h\big[ G'_\delta(\B_h^n)\big] }_{L^2}^2
    + \int_\Omega \frac{\kappa^2}{2\tau_1}
    \calI_h\Big[ \trace\big(\beta_\delta(\B_h^n) + [\beta_\delta(\B_h^n)]^{-1}-2\I \big) \Big] \dx
    \\
    &\quad
    + \int_\Omega \calI_h \Big[ 
    \Gamma_{\sigma,h}^n \big(\chi_\sigma \sigma_h^n + \chi_\phi (1-\phi_h^n)\big) 
    - \mu_h^n \Gamma_{\phi,h}^n 
    +
    \chi_\phi (1-\phi_h^n) \Big(\frac{\sigma_h^n - \sigma_h^{n-1}}{\Delta t}\Big)  
    - \frac{\phi_h^n-\phi_h^{n-1}}{\Delta t} \chi_\phi \sigma_h^n \Big] \dx
    \\
    &\quad 
    + \int_{\partial\Omega} \calI_h\Big[
    K \chi_\phi \sigma_h^n (1-\phi_h^n)
    - K \sigma_{\infty,h}^n \big(\chi_\sigma \sigma_h^n + \chi_\phi (1-\phi_h^n)\big)  \Big] \dH^{d-1}
    \\
    &\leq 0.
\end{split}
\end{align}


For the terms in \eqref{eq:energy_FE_3} involving the boundary integrals, we have by H{\"o}lder's and Young's inequalities, \eqref{eq:norm_equiv}, \eqref{eq:norm_equiv_Gamma} and the trace theorem, that
\begin{align}
\begin{split}
\label{eq:energy_FE_4}
    &\int_{\partial\Omega} \calI_h\Big[
    K \chi_\phi \sigma_h^n (1-\phi_h^n)
    - K \sigma_{\infty,h}^n \big(\chi_\sigma \sigma_h^n + \chi_\phi (1-\phi_h^n)\big)  \Big] \dH^{d-1}
    \\
    &\leq 
    K C_{\mathrm{tr}}^2 \left(\frac{\chi_\phi^2}{2\chi_\sigma} + 1 \right)  
    \Big( \norm{\phi_h^n}_h^2
    + \norm{\nabla\phi_h^n}_{L^2}^2 \Big)
    + \frac{3 \chi_\sigma K}{4}  \norm{\sigma_h^n}_{h,{\partial\Omega}}^2  
    + C(K, \chi_\phi, \chi_\sigma) \big( \abs{{\partial\Omega}}
    + \norm{\sigma_{\infty,h}^n}_{h,{\partial\Omega}}^2 \big) .
\end{split}
\end{align}
On noting H{\"o}lder's and Young's inequalities and \ref{A2}, we deal with the source terms in \eqref{eq:energy_FE_3} as follows:
\begin{align}
\begin{split}
\label{eq:energy_FE_5}
    &\int_\Omega \calI_h \Big[ 
    \Gamma_{\sigma,h}^n \big(\chi_\sigma \sigma_h^n + \chi_\phi (1-\phi_h^n)\big) 
    - \mu_h^n \Gamma_{\phi,h}^n  
    \Big] \dx
    \\
    &\leq
    \frac{1}{2} \norm{\mu_h}_h^2 
    + \frac{1}{2} \norm{\Gamma_{\phi,h}^n}_h^2
    + \frac{1}{2} \norm{\Gamma_{\sigma,h}^n}_h^2
    + \frac{3\chi_\sigma^2}{2} \norm{\sigma_h^n}_h^2
    + \frac{3\chi_\phi^2}{2} \norm{\phi_h^n}_h^2
    + \frac{3\chi_\phi^2}{2} \abs{\Omega}
    \\
    &\leq 
    \frac{1}{2} \norm{\mu_h}_h^2 
    + R_0^2 \nnorm{(1 + \abs{\phi_h^n} + \abs{\sigma_h^n}) }_h^2
    + \frac{3\chi_\sigma^2}{2} \norm{\sigma_h^n}_h^2
    + \frac{3\chi_\phi^2}{2} \norm{\phi_h^n}_h^2
    + \frac{3\chi_\phi^2}{2} \abs{\Omega}
    \\
    &\leq 
    \frac{1}{2} \norm{\mu_h}_h^2 
    + \left(3R_0^2+ \frac{3\chi_\sigma^2}{2}\right) \norm{\sigma_h^n}_h^2 
    + \left(3R_0^2+ \frac{3\chi_\phi^2}{2}\right)  \norm{\phi_h^n}_h^2
    + C(R_0, \chi_\phi, \Omega).
\end{split}
\end{align}
In order to control the source terms, we need to derive an estimate for the chemical potential.
Hence, on noting \eqref{A4_2} and Young's inequality, we receive by testing \eqref{eq:mu_FE_delta} with $\rho_h = \mu_h^n$ that
\begin{align*}
    &\int_\Omega \calI_h \big[ \abs{\mu_h^n}^2 \big] \dx 
    =
    \int_\Omega \calI_h \Big[
    \Big(A\psi_1'(\phi_h^n) + A\psi_2'(\phi_h^{n-1}) 
    - \chi_\phi \sigma_h^n\Big) \mu_h^n \Big] 
    + B \nabla\phi_h^n \cdot \nabla\mu_h^n \dx
    \\
    &\leq \int_\Omega \calI_h \Big[
    A R_3 (2 + \abs{\phi_h^{n-1}} 
    + \abs{\phi_h^n} \big) \abs{\mu_h^n} 
    + \chi_\phi \abs{\sigma_h^n} \abs{\mu_h^n} \Big]
    + B \abs{\nabla\phi_h^n} \abs{\nabla\mu_h^n} \dx
    \\
    &\leq \int_\Omega \calI_h \Big[
    \frac{1}{2} \abs{\mu_h^n}^2 
    + 2 A^2 R_3^2 \abs{\phi_h^n}^2  
    + 2 A^2 R_3^2 \abs{\phi_h^{n-1}}^2 
    + 2 \chi_\phi^2 \abs{\sigma_h^n}^2 \Big] \dx
    \\
    &\quad + \int_\Omega \frac{B^2}{m_0} \abs{\nabla\phi_h^n}^2 
    + \frac{m_0}{4} \abs{\nabla\mu_h^n}^2 \dx 
    + C(A, R_3, \Omega),
\end{align*}
which yields
\begin{align}
\begin{split}
    \label{eq:energy_FE_6}
    \norm{\mu_h^n}_h^2 
    &\leq 
    4 A^2 R_3^2 \Big(\norm{\phi_h^n}_h^2 
    +  \norm{\phi_h^{n-1}}_h^2 \Big)
    + 4 \chi_\phi^2 \norm{\sigma_h^n}_h^2
    + \frac{2B^2}{m_0} \norm{\nabla\phi_h^n}_{L^2}^2 
    + \frac{m_0}{2} \norm{\nabla\mu_h^n}_{L^2}^2 \dx 
    + C(A,R_3,\Omega).
\end{split}
\end{align}
Furthermore, we calculate
\begin{subequations}
\label{eq:energy_FE_7}
\begin{align}
\begin{split}
    & \int_\Omega \calI_h \Big[ \chi_\phi (1-\phi_h^n) \big(\sigma_h^n - \sigma_h^{n-1}\big)  
    -  \chi_\phi \big(\phi_h^n-\phi_h^{n-1}\big) \sigma_h^n \Big] \dx
    \\
    &= 
    \int_\Omega \calI_h \Big[ \chi_\phi  \sigma_h^n (1-\phi_h^n) 
    - \chi_\phi \sigma_h^{n-1} (1-\phi_h^{n-1})
    -  \chi_\phi (\phi_h^n-\phi_h^{n-1}) (\sigma_h^n- \sigma_h^{n-1}) \Big] \dx,
\end{split}
\end{align}
and
\begin{align}
\begin{split}
    & \abs{\int_\Omega \calI_h \Big[ \chi_\phi (\phi_h^n-\phi_h^{n-1}) (\sigma_h^n- \sigma_h^{n-1}) \Big] \dx }
    \leq \frac{\chi_\phi^2}{\chi_\sigma} \norm{\phi_h^n - \phi_h^{n-1}}_h^2
    + \frac{\chi_\sigma}{4} \norm{\sigma_h^n - \sigma_h^{n-1}}_h^2
    \\
    &\leq \frac{2\chi_\phi^2}{\chi_\sigma} \Big( \norm{\phi_h^n}_h^2 + \norm{\phi_h^{n-1}}_h^2 \Big)
    + \frac{\chi_\sigma}{4} \norm{\sigma_h^n - \sigma_h^{n-1}}_h^2.
\end{split}
\end{align}
\end{subequations}
Moreover, applying the triangle inequality and Young's inequality leads to
\begin{align}
\begin{split}
\label{eq:energy_FE_8}
    \chi_\sigma^2 \norm{\nabla \sigma_h^n}_{L^2}^2
    &\leq 2 \nnorm{\chi_\sigma \nabla \sigma_h^n - \chi_\phi \nabla\phi_h^n}_{L^2}^2
    + 2 \chi_\phi^2 \norm{\nabla \phi_h^n}_{L^2}^2.
\end{split}
\end{align}

For the reader's convenience, we define the constants
\begin{align}
\begin{split}
\label{eq:energy_FE_const}
    &c_1 \coloneqq \frac{m_0}{2}, \quad
    c_2 \coloneqq \frac{n_0 \chi_\sigma^2}{2}, \quad
    c_3 \coloneqq \frac{K \chi_\sigma}{4}, \quad
    c_4 \coloneqq 2 \eta_0, \quad
    c_5 \coloneqq \frac{\kappa^2}{2\tau_1}, \quad
    c_6 \coloneqq 3R_0^2+ \frac{3\chi_\sigma^2}{2} + 4 \chi_\phi^2, \quad
    c_7 \coloneqq 4 A^2 R_3^2,
    \\
    & 
    c_8 \coloneqq c_7
    + 3R_0^2+ \frac{3\chi_\phi^2}{2}
    + K C_{\mathrm{tr}}^2 \left(\frac{\chi_\phi^2}{2\chi_\sigma} + 1 \right), \quad
    c_9 \coloneqq \frac{2B^2}{m_0} + K C_{\mathrm{tr}}^2 \left(\frac{\chi_\phi^2}{2\chi_\sigma} + 1 \right) + n_0 \chi_\phi^2.
\end{split}
\end{align}
Hence, combining \eqref{eq:energy_FE_3}--\eqref{eq:energy_FE_8} and noting \eqref{eq:def_energy_FE_delta} and \eqref{eq:energy_FE_const} gives us
\begin{align}
\begin{split}
\label{eq:energy_FE_10}
    & \frac{1}{\Delta t} 
    \calF_{h,\delta}(\phi_h^n, \sigma_h^n, \pmbv_h^n, \B_h^n) 
    + \frac{B}{2\Delta t} \norm{\nabla\phi_h^n - \nabla\phi_h^{n-1}}_{L^2}^2 
    + \frac{\chi_\sigma}{4\Delta t} 
    \norm{\sigma_h^n - \sigma_h^{n-1}}_h^2
    + \frac{1}{2\Delta t} \norm{\pmbv_h^n - \pmbv_h^{n-1}}_{L^2}^2
    \\
    &\quad 
    + \frac{1}{2} \norm{\mu_h^n}_h^2 
    + c_1  \norm{\nabla\mu_h^n}_{L^2}^2 
    + c_2  \norm{\nabla \sigma_h^n}_{L^2}^2 \dx
    + c_3  \norm{\sigma_h^n}_{h,{\partial\Omega}}^2
    + c_4  \norm{\D(\pmbv_h^n)}_{L^2}^2
    \\
    &\quad
    + \frac{\alpha\kappa\delta^2}{2} \nnorm{ \nabla \calI_h\big[ G'_\delta(\B_h^n)\big] }_{L^2}^2
    + c_5 \int_\Omega
    \calI_h\Big[ \trace\big(\beta_\delta(\B_h^n) + [\beta_\delta(\B_h^n)]^{-1}-2\I \big) \Big] \dx
    \\
    &\leq 
    \frac{1}{\Delta t} \calF_{h,\delta}(\phi_h^{n-1}, \sigma_h^{n-1}, \pmbv_n^{n-1}, \B_h^{n-1})
    + C \big( 1
    + \norm{\sigma_{\infty,h}^n}_{h,{\partial\Omega}}^2 \big)
    \\
    &\quad
    + c_6 \norm{\sigma_h^n}_h^2
    +  \left( \frac{2\chi_\phi^2}{\chi_\sigma \Delta t}  
    + c_7 \right) \norm{\phi_h^{n-1}}_h^2
    + \left( \frac{2\chi_\phi^2}{\chi_\sigma \Delta t} 
    + c_8 \right) \norm{\phi_h^n}_h^2 
    + c_9 \norm{\nabla\phi_h^n}_{L^2}^2.
\end{split}
\end{align}
At this point, if the discrete energy is non-negative, a common strategy would be to absorb the terms from the right-hand side of \eqref{eq:energy_FE_10} with index $n$ to the left-hand side, supposed that the time step size $\Delta t$ is small enough. This would lead to a discrete energy inequality.
However, if $\chi_\phi \not = 0$, the discrete energy can be negative.
Therefore, we continue to bound the product $\chi_\phi \sigma_h^n (1-\phi_h^n)$ with Hölder's and Young's inequalities and hence absorb them with the help of the term $A \psi(\phi_h^n) + \tfrac{\chi_\sigma}{2} \abs{\sigma_h^n}^2$ from the discrete energy. This is where we make use of \ref{A4}, in particular \eqref{A4_1} and \eqref{A4_3}. By H{\"o}lder's and Young's inequalities, we have
\begin{align}
\label{eq:energy_FE_11}
    &\int_\Omega \calI_h \Big[ \chi_\phi \sigma_h^n (1-\phi_h^n) \Big] \dx
    \leq \frac{\chi_\sigma}{4} \norm{\sigma_h^n}_h^2
    + \frac{2 \chi_\phi^2}{\chi_\sigma } \norm{\phi_h^n}_h^2
    + C(\chi_\sigma, \chi_\phi, \Omega).
\end{align}
Moreover, we obtain from \eqref{A4_1} that
\begin{align}
\label{eq:energy_FE_12}
    & R_1 \norm{\phi_h^n}_h^2
    \leq  \int_\Omega \calI_h \big[ \psi(\phi_h^n) \big] \dx 
    + R_2 \abs{\Omega}.
\end{align}
Multiplying both sides of \eqref{eq:energy_FE_10} with $\Delta t$, summing from $n=1,...,m$ for $m\in \{1,...,N_T\}$, noting \eqref{eq:energy_FE_11}--\eqref{eq:energy_FE_12} and absorbing the terms on the right-hand side with index $m=n$ yields
\begin{align}
\begin{split}
\label{eq:energy_FE_14}
    & \Big( A R_1 - \frac{4 \chi_\phi^2}{\chi_\sigma}
    - c_8 \Delta t \Big) 
    \norm{\phi_h^m}_h^2
    + \Big( \frac{B}{2} - c_9 \Delta t \Big)
    \norm{\nabla\phi_h^m}_{L^2}^2
    + \Big( \frac{\chi_\sigma}{4} - c_6 \Delta t \Big) 
    \norm{\sigma_h^m}_h^2
    \\
    &\quad
    + \frac{1}{2} \norm{\pmbv_h^m}_{L^2}^2
    + \frac{\kappa}{2} \int_\Omega \calI_h\big[ 
    \trace(\B_h^m - G_\delta(\B_h^m)) \big] \dx
    \\
    &\quad
    +  \sum_{n=1}^m \Big( 
    \frac{B}{2} \norm{\nabla\phi_h^n - \nabla\phi_h^{n-1}}_{L^2}^2 
    + \frac{\chi_\sigma}{4}  \norm{\sigma_h^n - \sigma_h^{n-1}}_h^2
    + \frac{1}{2} \norm{\pmbv_h^n - \pmbv_h^{n-1}}_{L^2}^2 \Big)
    \\
    &\quad 
    + \Delta t \sum_{n=1}^m \Big( \frac{1}{2} \norm{\mu_h^n}_h^2 
    + c_1 \norm{\nabla\mu_h^n}_{L^2}^2 
    +  c_2  \norm{\nabla \sigma_h^n}_{L^2}^2
    + c_3 \norm{\sigma_h^n}_{h,{\partial\Omega}}^2 
    + c_4  \norm{\D(\pmbv_h^n)}_{L^2}^2 \Big) 
    \\
    &\quad
    + \Delta t \sum_{n=1}^m \bigg( \frac{\alpha\kappa\delta^2}{2} \nnorm{ \nabla \calI_h\big[ G'_\delta(\B_h^n)\big] }_{L^2}^2
    + c_5 \int_\Omega
    \calI_h\Big[ \trace\big(\beta_\delta(\B_h^n) + [\beta_\delta(\B_h^n)]^{-1}-2\I \big) \Big] \dx \bigg)
   \\
    &\leq 
    \abs{\calF_{h,\delta}(\phi_h^0, \sigma_h^0, \pmbv_h^0, \B_h^0)} + C \left( T
    + \Delta t \sum_{n=1}^{N_T} \norm{\sigma_{\infty,h}^n}_{h,{\partial\Omega}}^2 \right)
    + \sum_{n=0}^{m-1} 
    \left( \frac{4\chi_\phi^2}{\chi_\sigma } 
    + (c_7 + c_8) \Delta t \right) \norm{\phi_h^n}_h^2
    \\
    &\quad 
    + \Delta t \sum_{n=1}^{m-1} \Big( c_6 \norm{\sigma_h^n}_h^2 
    + c_9  \norm{\nabla\phi_h^n}_{L^2}^2 \Big).
\end{split}
\end{align}
The coefficients on the left-hand side are positive supposed that the time step size $\Delta t$ fulfills \eqref{eq:dt}.
\begin{align}
\label{eq:dt}
    \Delta t &< \Delta t_* \coloneqq 
    \min \left\{ \frac{B}{2 c_9}, \
    \frac{\chi_\sigma}{4 c_6}, \
    \frac{A R_1 - \frac{4 \chi_\phi^2}{\chi_\sigma}}{c_8}
    \right\},
\end{align}
where the constants $c_6, c_8, c_9>0$ are defined by \eqref{eq:energy_FE_const} and $A R_1 - \frac{4 \chi_\phi^2}{\chi_\sigma}>0$ due to \eqref{A4_3}. 
Next, we obtain from a discrete Gronwall argument (i.e.~Lemma \ref{lemma:gronwall_discrete}), that
\begin{align}
\begin{split}
\label{eq:energy_FE_15}
    &  \norm{\phi_h^m}_h^2
    +  \norm{\nabla\phi_h^m}_{L^2}^2
    +  \norm{\sigma_h^m}_h^2
    +  \norm{\pmbv_h^m}_{L^2}^2
    +  \int_\Omega \calI_h\big[ 
    \trace(\B_h^m - G_\delta(\B_h^m)) \big] \dx
    \\
    &\quad
    +  \sum_{n=1}^m \Big(
    \norm{\nabla\phi_h^n - \nabla\phi_h^{n-1}}_{L^2}^2 
    +  \norm{\sigma_h^n - \sigma_h^{n-1}}_h^2
    +  \norm{\pmbv_h^n - \pmbv_h^{n-1}}_{L^2}^2 \Big)
    \\
    &\quad
    + \Delta t \sum_{n=1}^m \Big( \norm{\mu_h^n}_h^2 
    + \norm{\nabla\mu_h^n}_{L^2}^2 
    + \norm{\nabla \sigma_h^n}_{L^2}^2
    + \norm{\sigma_h^n}_{h,{\partial\Omega}}^2 
    + \norm{\D(\pmbv_h^n)}_{L^2}^2  \Big)
    \\
    &\quad
    + \Delta t \sum_{n=1}^m \bigg( 
    \alpha\delta^2 \nnorm{ \nabla \calI_h\big[ G'_\delta(\B_h^n)\big] }_{L^2}^2
    + \int_\Omega
    \calI_h\Big[ \trace\big(\beta_\delta(\B_h^n) + [\beta_\delta(\B_h^n)]^{-1}-2\I \big) \Big] \dx \bigg)
    \\
    &\leq 
    C \Big( T + \abs{\calF_{h,\delta}(\phi_h^0, \sigma_h^0, \pmbv_h^0, \B_h^0)}
    + \Delta t \sum_{n=1}^{N_T} \norm{\sigma_{\infty,h}^n}_{h,{\partial\Omega}}^2 \Big)  \exp( CT )
\end{split}
\end{align}
for some constants $C>0$ that are independent of $h, \Delta t, \alpha, \delta$.
Noting Korn's inequality, \eqref{eq:lemma_reg1b}, \eqref{eq:lemma_reg1g}, \eqref{eq:norm_equiv}, \eqref{eq:norm_equiv_Gamma} and taking the maximum over $m=1,...,N_T$ on the left-hand side of \eqref{eq:energy_FE_15} yield
\begin{align}
\begin{split}
\label{eq:energy_FE_16}
    &  \max_{n=1,...,N_T} \Big( 
    \norm{\phi_h^n}_{H^1}^2
    +  \norm{\sigma_h^n}_{L^2}^2
    +  \norm{\pmbv_h^n}_{L^2}^2
    +  \nnorm{ \calI_h\big[ \abs{\B_h^n} \big] }_{L^1}
    +  \frac{1}{\delta} \nnorm{ \calI_h\big[ \abs{ [\B_h^n]_- } \big] }_{L^1}  \Big)
    \\
    &\quad
    +  \sum_{n=1}^{N_T} \Big(
    \norm{\nabla\phi_h^n - \nabla\phi_h^{n-1}}_{L^2}^2 
    +  \norm{\sigma_h^n - \sigma_h^{n-1}}_{L^2}^2
    +  \norm{\pmbv_h^n - \pmbv_h^{n-1}}_{L^2}^2 \Big)
    \\
    &\quad
    + \Delta t \sum_{n=1}^{N_T} \Big( 
    \norm{\mu_h^n}_{H^1}^2 
    + \norm{\nabla \sigma_h^n}_{L^2}^2
    + \norm{\sigma_h^n}_{L^2({\partial\Omega})}^2 
    + \norm{\nabla \pmbv_h^n}_{L^2}^2   \Big)
    \\
    &\quad
    + \Delta t \sum_{n=1}^{N_T} \bigg( 
    \alpha\delta^2 \nnorm{ \nabla \calI_h\big[ G'_\delta(\B_h^n)\big] }_{L^2}^2
    + \int_\Omega
    \calI_h\Big[ \trace\big(\beta_\delta(\B_h^n) + [\beta_\delta(\B_h^n)]^{-1}-2\I \big) \Big] \dx  \bigg)
    \\
    &\leq 
    C(T) \Big( 1 + \abs{\calF_{h,\delta}(\phi_h^0, \sigma_h^0, \pmbv_h^0, \B_h^0)}
    + \Delta t \sum_{n=1}^{N_T} \norm{\sigma_{\infty,h}^n}_{L^2(\partial\Omega)}^2 \Big),
\end{split}
\end{align}
for a constant $C(T)>0$ that is independent of $h, \Delta t, \alpha, \delta$ but depends exponentially on $T$.
On noting \eqref{eq:init_bounds}, the right-hand side of \eqref{eq:energy_FE_16} is bounded uniformly in $h,\Delta t, \alpha, \delta$. This proves the result.
\end{proof}






\subsection{Existence of regularized discrete solutions}
\label{sec:existence}

In the next theorem, we apply a strategy based on Brouwer's fixed point theorem \cite[Chap.~8.1.4, Thm.~3]{evans_2010} in order to prove existence of discrete solutions to \ref{P_alpha_delta_FE}. 
Here, one of the main difficulties is to construct specific mappings on a finite dimensional Hilbert space such that Brouwer's fixed point theorem can be applied in the right way.
It turns out that the testing procedure of Lemma \ref{lemma:bounds_FE_delta}
is very helpful. However, we need to deal with similar difficulties as in Lemma \ref{lemma:bounds_FE_delta}, which explains the minor constraint on the time step size. 

\begin{theorem}[Existence]
\label{theorem:existence_FE_delta}
Let \ref{A1}--\ref{A5} hold true and let $\delta\in(0,\frac{1}{2}]$. Suppose that the discrete initial and boundary data satisfy \eqref{eq:init_bounds} and assume that $\Delta t < \Delta t_*$, where $\Delta t_*$ is defined in \eqref{eq:dt}.
Then, for all $n\in\{1,...,N_T\}$, there exists at least one solution $(\phi_h^{n},\mu_h^n,\sigma_{h}^{n},p_h^n,\pmbv_{h}^{n},\B_{h}^{n}) \in (\calS_h)^4 \times \calV_h \times \calW_h$ to the problem \ref{P_alpha_delta_FE} which is stable in the sense of \eqref{eq:bounds_FE_delta}.
\end{theorem}

\begin{proof}
We prove existence of solutions to the discrete problem \ref{P_alpha_delta_FE} with the combination of the stability result \eqref{eq:bounds_FE_delta} and a fixed point argument. 
However, we can not directly show existence for solutions $(\phi_h^{n}$, $\mu_h^n$, $ \sigma_{h}^{n}$, $p_h^n$, $\pmbv_{h}^{n}$, $ \B_{h}^{n}) \in (\calS_h)^4 \times  \calV_{h} \times \calW_h$ as we have no control over the pressure $p_h^n$. 
Therefore, we first prove existence of functions $(\phi_h^{n}$, $\mu_h^n$, $ \sigma_{h}^{n}$, $ \pmbv_{h}^{n}$, $ \B_{h}^{n}) \in (\calS_h)^3 \times \calV_{h,\text{div}} \times \calW_h$ which solve \ref{P_alpha_delta_FE} with \eqref{eq:v_FE_delta} replaced by \eqref{eq:v_FE_delta_b}.
Afterwards, we reconstruct the discrete pressure $p_h^n$ and show that $\phi_h^{n}$, $\mu_h^n$, $ \sigma_{h}^{n}$, $p_h^n$, $\pmbv_{h}^{n}$, $ \B_{h}^{n}$ solve the problem \ref{P_alpha_delta_FE}.

First, we define the following inner product on the Hilbert space $(\calS_h)^3 \times \calV_{h,\text{div}} \times \calW_h$
\begin{align*}
    \skp{(\phi_h,\mu_h,\sigma_h,\pmbv_h,\B_h)}{(\zeta_h,\rho_h,\xi_h,\pmbw_h,\C_h)}
    &\coloneqq \int_\Omega \calI_h\Big[ 
    \phi_h\zeta_h + \mu_h\rho_h + \sigma_h\xi_h 
    + \B_h:\C_h \Big] 
    + \pmbv_h\cdot \pmbw_h \dx,
\end{align*}
for all $(\phi_h,\mu_h,\sigma_h,\pmbv_h,\B_h), (\zeta_h,\rho_h,\xi_h,\pmbw_h,\C_h) \in (\calS_h)^3 \times \calV_{h,\text{div}} \times \calW_h$. 

For some given $(\phi_h^{n-1},\sigma_h^{n-1},\pmbv_h^{n-1}, \B_h^{n-1})\in (\calS_h)^2 \times \calV_{h,\text{div}} \times \calW_h$, let the mapping
\begin{align*}
    \calH^h: (\calS_h)^3 \times \calV_{h,\text{div}} \times \calW_h \longrightarrow (\calS_h)^3 \times \calV_{h,\text{div}} \times \calW_h
\end{align*}
be such that for any $(\phi_h,\mu_h,\sigma_h,\pmbv_h,\B_h) \in (\calS_h)^3 \times \calV_{h,\text{div}} \times \calW_h$
\begin{align*}
    &\skp{\calH^h (\phi_h,\mu_h,\sigma_h,\pmbv_h,\B_h)}{(\zeta_h,\rho_h,\xi_h,\pmbw_h,\C_h)}
    \\
    &\coloneqq 
    \int_\Omega \calI_h \Big[ \Big(\frac{\phi_h-\phi_h^{n-1}}{\Delta t}
    - \Gamma_{\phi,h}^n \Big) \zeta_h \Big]
    + \calI_h[m(\phi_h^{n-1})] \nabla\mu_h \cdot \nabla \zeta_h
    - \phi_h^{n-1} \pmbv_h \cdot\nabla\zeta_h \dx
    \\
    &+ \int_\Omega \calI_h \Big[ \Big( - \mu_h  
    + A \psi_1'(\phi_h) + A \psi_2'(\phi_h^{n-1}) 
    - \chi_\phi \sigma_h \Big) \rho_h \Big] 
    + B \nabla\phi_h \cdot \nabla\rho_h\dx
    + \int_{\partial\Omega} \calI_h\Big[ K \big(\sigma_h - \sigma_{\infty,h}^n \big) \xi_h \Big] \dH^{d-1}
    \\
    &+ \int_\Omega \calI_h \Big[ \Big(\frac{\sigma_h-\sigma_h^{n-1}}{\Delta t}
    + \Gamma_{\sigma,h}^n \Big) \xi_h \Big]
    + \calI_h[n(\phi_h^{n-1})] \nabla (\chi_\sigma\sigma_h - \chi_\phi\phi_h)   \cdot \nabla \xi_h  
    - \sigma_h^{n-1} \pmbv_h \cdot\nabla\xi_h \dx
    \\
    &+ \int_\Omega \frac{\pmbv_h-\pmbv_h^{n-1}}{\Delta t} \cdot \pmbw_h
    + \frac{1}{2} \left( \left(\pmbv_h^{n-1}\cdot \nabla\right) \pmbv_h\right) \cdot \pmbw_h
    - \frac{1}{2} \pmbv_h \cdot \left(\left(\pmbv_h^{n-1} \cdot \nabla \right) \pmbw_h \right)  \dx 
    \quad\quad
    \\
    &+ \int_\Omega 2\calI_h[\eta(\phi_h^{n-1})] \D(\pmbv_h) : \D(\pmbw_h)
    + \kappa \calI_h\big[ \beta_\delta(\B_h) - \I \big] : \nabla\pmbw_h  
    + \big( \phi_h^{n-1}  \nabla\mu_h
    +  \sigma_h^{n-1} \nabla(\chi_\sigma\sigma_h - \chi_\phi\phi_h) \big) \cdot \pmbw_h \dx
    \\
    &+ \int_\Omega \calI_h \Big[ 
    \Big(\frac{\B_h - \B_h^{n-1}}{\Delta t}
    + \frac{\kappa}{\tau(\phi_h^{n-1})} (\B_h - \I) \Big): \C_h \Big]
    - 2 \nabla\pmbv_h : \calI_h\big[ \C_h \beta_\delta(\B_h) \big] \dx
    \quad\quad
    \\
    &+ \int_\Omega \alpha \nabla\B_h : \nabla\C_h
    - \sum\limits_{i,j=1}^d  
    [\pmbv_h^{n-1}]_i \Lambda_{\delta,i,j}(\B_h) : \partial_{x_j} \C_h \dx,
\end{align*}
for all $(\zeta_h,\rho_h,\xi_h,\pmbw_h,\C_h) \in (\calS_h)^3 \times \calV_{h,\text{div}} \times \calW_h $.
A wanted solution to our problem, if it exists, corresponds to a zero of $\calH^h$. On noting the definition of $\Lambda_{\delta,i,j}$, it follows that the mapping $\calH^h$ is continuous.

Let $R>0$ be given. Let us assume that the continuous mapping $\calH^h \circ f^{-1}$ has no zero which lies in the ball 
\begin{align*}
    B_R^h \coloneqq \{ (\zeta_h,\rho_h,\xi_h,\pmbw_h,\C_h) \in (\calS_h)^3 \times \calV_{h,\text{div}} \times \calW_h
    \ : \ 
    ||| (\zeta_h,\rho_h,\xi_h,\pmbw_h,\C_h)||| \leq R \},
\end{align*}
where
\begin{align*}
    |||(\zeta_h,\rho_h,\xi_h,\pmbw_h,\C_h)|||^2 
    \coloneqq 
    \skp{(\zeta_h,\rho_h,\xi_h,\pmbw_h,\C_h)}{(\zeta_h,\rho_h,\xi_h,\pmbw_h,\C_h)},
\end{align*}
and where the linear transformation $f: (\calS_h)^3 \times \calV_{h,\text{div}} \times \calW_h \longrightarrow (\calS_h)^3 \times \calV_{h,\text{div}} \times \calW_h$ and its inverse are given by
\begin{align*}
    f: (\phi_h, \mu_h, \sigma_h, \pmbv_h, \B_h) 
    &\mapsto 
    \Big( \mu_h, \
    \frac{\phi_h}{\Delta t} - 2 \mu_h, \
    \chi_\sigma \sigma_h - \chi_\phi \phi_h, \
    \pmbv_h, \ 
    \B_h \Big),
    \\
    f^{-1}: (\zeta_h, \rho_h, \xi_h, \pmbw_h, \C_h) 
    &\mapsto 
    \Big(\Delta t (\rho_h + 2 \zeta_h), \
    \zeta_h, \
    \frac{\chi_\phi}{\chi_\sigma} \Delta t (\rho_h + 2 \zeta_h) 
    + \frac{1}{\chi_\sigma} \xi_h, \
    \pmbw_h, \ 
    \C_h \Big).
\end{align*}
Then for such $R>0$, we define a continuous mapping $\calG_R^h: B_R^h \to \partial B_R^h$ by
\begin{align*}
    \calG_R^h (\zeta_h,\rho_h,\xi_h,\pmbw_h,\C_h)
    \coloneqq 
    -R \frac{(\calH^h  \circ f^{-1}) (\zeta_h,\rho_h,\xi_h,\pmbw_h,\C_h)}{ 
    ||| (\calH^h \circ f^{-1}) (\zeta_h,\rho_h,\xi_h,\pmbw_h,\C_h) ||| }
    \quad\quad \forall \  (\zeta_h,\rho_h,\xi_h,\pmbw_h,\C_h)\in B_R^h,
\end{align*}
We deduce from Brouwer's fixed point theorem \cite[Chap.~8.1.4, Thm.~3]{evans_2010} that there exists at least one fixed point $(\zeta_h,\rho_h,\xi_h,\pmbw_h,\C_h) = f(\phi_h, \mu_h,\sigma_h, \pmbv_h, \B_h) \in B_R^h$ of the mapping $\calG_R^h$ satisfying
\begin{align}
\label{eq:existence_R}
    ||| f(\phi_h, \mu_h,\sigma_h, \pmbv_h, \B_h) |||
    = ||| (\zeta_h,\rho_h,\xi_h,\pmbw_h,\C_h) |||
    = ||| \calG_R^h (\zeta_h,\rho_h,\xi_h,\pmbw_h,\C_h) |||
    = R.
\end{align}


On noting \eqref{eq:interp_estimate}, \eqref{eq:norm_equiv}, \eqref{eq:inverse_estimate} and \eqref{eq:existence_R}, we have
\begin{align*}
    \nnorm{\calI_h\big[ \abs{\B_h} \big] }_{L^\infty}^2
    \leq 
    \nnorm{\calI_h\big[ \abs{\B_h}^2 \big] }_{L^\infty}
    \leq 
    C h^{-d} \norm{\calI_h\big[ \abs{\B_h}^2 \big]}_{L^1}
    =
    C h^{-d} \norm{\B_h}_h^2
    \leq C h^{-d} R^2,
\end{align*}
which, on noting \eqref{eq:lemma_reg1g}, leads to
\begin{align}
\label{eq:existence_1}
    \int_\Omega \calI_h\Big[ \trace\big( \B_h - G_\delta(\B_h) \big) \Big] \dx
    \geq C h^{\frac{d}{2}} R^{-1} 
    \nnorm{ \calI_h\big[ \abs{\B_h} \big] }_{L^1} \nnorm{ \calI_h\big[ \abs{\B_h} \big] }_{L^\infty}
    \geq C h^{\frac{d}{2}} R^{-1}  \norm{ \B_h  }_h^2 .
\end{align}
Hence, analogously to the proof of \eqref{eq:bounds_FE_delta}, we get together with \eqref{eq:existence_1} that
\begin{align*}
    &\skp{ \calH^h(\phi_h,\mu_h, \sigma_h, \pmbv_h, \B_h)}{ \Big(\mu_h, 
    \frac{\phi_h}{\Delta t} - 2\mu_h,
    \chi_\sigma\sigma_h - \chi_\phi \phi_h, 
    \pmbv_h, 
    \frac{\kappa}{2}\big( \I-G_\delta'(\B_h)\big)\Big)}
    \\
    &\geq C \Big(\frac{1}{\Delta t} - \frac{1}{\Delta t_*}\Big) \Big( \norm{\phi_h}_{H^1}^2 
    + \norm{\sigma_h}_{L^2}^2 
    + \norm{\pmbv_h}_{L^2}^2 
    + \int_\Omega \calI_h\Big[ \trace\big( \B_h - G_\delta(\B_h) \big) \Big] \dx \Big)
    \\
    &\quad
    + C \Big( \norm{\mu_h}_{H^1}^2
    + \norm{\nabla\sigma_h}_{L^2}^2
    + \norm{\sigma_h}_{L^2(\partial\Omega)}^2
    + \norm{\D(\pmbv_h)}_{L^2}^2 \Big)
    - C(\phi_h^{n-1},\sigma_h^{n-1},\pmbv_h^{n-1},\B_h^{n-1}, \sigma_{\infty,h}^n)
    \\
    &\geq 
    C \min\Big\{1,  \ h^\frac{d}{2} R^{-1} \Big( \frac{1}{\Delta t} - \frac{1}{\Delta t_*} \Big) \Big\}
    ||| (\phi_h, \mu_h, \sigma_h, \pmbv_h, \B_h) |||^2
    - C(\phi_h^{n-1},\sigma_h^{n-1},\pmbv_h^{n-1},\B_h^{n-1}, \sigma_{\infty,h}^n),
\end{align*}
where $C(\phi_h^{n-1},\sigma_h^{n-1},\pmbv_h^{n-1},\B_h^{n-1}, \sigma_{\infty,h}^n)$ denotes a constant that depends on $\phi_h^{n-1}$, $\sigma_h^{n-1}$, $\pmbv_h^{n-1}$, $\B_h^{n-1}$, $\sigma_{\infty,h}^n$ but not on $\phi_h,\mu_h,\sigma_h,\pmbv_h,\B_h$.
We remark that both $|||\cdot|||$ and $||| f(\cdot)|||$ define norms on $(\calS_h)^3 \times \calV_{h,\text{div}} \times \calW_h$. Hence, due to norm equivalence in finite dimensions, there exist constants $c_1(h),c_2(h)>0$ (which can depend on $h$ in general) such that 
\begin{align*}
    c_1(h) ||| (\phi_h, \mu_h, \sigma_h, \pmbv_h, \B_h) |||
    \leq ||| f(\phi_h, \mu_h, \sigma_h, \pmbv_h, \B_h) |||
    \leq c_2(h) ||| (\phi_h, \mu_h, \sigma_h, \pmbv_h, \B_h) |||.
\end{align*}
Hence, with \eqref{eq:existence_R} and with $R>0$ large enough, we obtain
\begin{align*}
    \skp{\calH^h  (\phi_h,\mu_h, \sigma_h, \pmbv_h, \B_h)}{ \Big(\mu_h, 
    \frac{\phi_h}{\Delta t} - 2\mu_h,
    \chi_\sigma\sigma_h - \chi_\phi\phi_h, 
    \pmbv_h, 
    \frac{\kappa}{2}\big( \I-G_\delta'(\B_h)\big)\Big)}
    > 0.
\end{align*}
On the other side, as $(\zeta_h,\rho_h,\xi_h,\pmbw_h,\C_h) = f(\phi_h, \mu_h,\sigma_h, \pmbv_h, \B_h) \in B_R^h$ is a fixed point of $\calG_R^h$, it holds
\begin{align*}
    &\skp{\calH^h (\phi_h,\mu_h, \sigma_h, \pmbv_h, \B_h)}{ \Big(\mu_h, 
    \frac{\phi_h}{\Delta t} - 2\mu_h,
    \chi_\sigma\sigma_h - \chi_\phi\phi_h, \pmbv_h, 
    \frac{\kappa}{2}\big( \I-G_\delta'(\B_h)\big)\Big)}
    \\
    &= - \frac{||| \calH^h (\phi_h,\mu_h, \sigma_h, \pmbv_h, \B_h)  ||| }{R}
    \\
    &\quad\quad\cdot 
    \skp{(\calG_R^h \circ f)(\phi_h,\mu_h, \sigma_h, \pmbv_h, \B_h)}{\Big(\mu_h, 
    \frac{\phi_h}{\Delta t} - 2\mu_h,
    \chi_\sigma\sigma_h - \chi_\phi\phi_h, \pmbv_h, 
    \frac{\kappa}{2}\big( \I-G_\delta'(\B_h)\big)\Big)}
    \\
    &= - \frac{||| \calH^h (\phi_h,\mu_h, \sigma_h, \pmbv_h, \B_h)  ||| }{R}
    \\
    &\quad\quad\cdot 
    \skp{f(\phi_h,\mu_h, \sigma_h, \pmbv_h, \B_h)}{\Big(\mu_h, 
    \frac{\phi_h}{\Delta t} - 2\mu_h,
    \chi_\sigma\sigma_h - \chi_\phi\phi_h, \pmbv_h, 
    \frac{\kappa}{2}\big( \I-G_\delta'(\B_h)\big)\Big)}
    \\
    &= - \frac{||| \calH^h (\phi_h,\mu_h, \sigma_h, \pmbv_h, \B_h)  ||| }{R}
    \\
    &\quad\quad\cdot \Big(\norm{\mu_h}_h^2 
    + \nnorm{\frac{\phi_h}{\Delta t} - 2 \mu_h}_h^2
    + \norm{\chi_\sigma \sigma_h - \chi_\phi\phi_h}_h^2
    + \norm{\pmbv_h}_{L^2}^2
    + \frac{\kappa}{2} \int_\Omega \calI_h \big[ 
    \B_h : (\I - G_\delta'(\B_h))] \dx \Big)
    \\
    &\leq - C \frac{||| \calH^h (\phi_h,\mu_h, \sigma_h, \pmbv_h, \B_h)  ||| }{R}  \Big( 
    \min\big\{1, h^\frac{d}{2} R^{-1}\big\}  
    |||f(\phi_h,\mu_h,\sigma_h,\pmbv_h,\B_h)|||^2 
    - 1 \Big),
\end{align*}
where, in the last step, we used \eqref{eq:lemma_reg1h} and argued like in \eqref{eq:existence_1}. 
Therefore, with \eqref{eq:existence_R} and with $R>0$ large enough, we obtain
\begin{align*}
    &\skp{\calH^h (\phi_h,\mu_h, \sigma_h, \pmbv_h, \B_h)}{ \Big(\mu_h, 
    \frac{\phi_h}{\Delta t} - 2\mu_h,
    \chi_\sigma\sigma_h - \chi_\sigma\sigma_h, \pmbv_h, 
    \frac{\kappa}{2}\big( \I-G_\delta'(\B_h)\big)\Big)}
    < 0,
\end{align*}
which yields a contradiction. Hence, if $R>0$ is large enough, the mapping $\calH^h \circ f^{-1}$ possesses a zero $(\zeta_h,\rho_h,\xi_h,\pmbw_h,\C_h)$ in $B_R^h$. Moreover, $(\phi_h, \mu_h,\sigma_h, \pmbv_h, \B_h) = f^{-1}(\zeta_h,\rho_h,\xi_h,\pmbw_h,\C_h)$ corresponds to a solution of our problem.

To obtain the existence of a discrete pressure $p_h^n\in \calS_h$ and therefore to justify \eqref{eq:v_FE_delta}, we proceed as follows. Equation \eqref{eq:v_FE_delta} defines a linear functional $\calV_h \to \R$ which vanishes on $\calV_{h,\text{div}}$. The existence of a unique pressure $p_h^n\in\calS_h \cap L^2_0(\Omega)$ follows directly from, e.g., \cite[Chap. I, Lem~4.1]{girault_raviart_2012} or \cite[Lem.~4.2]{braess_2007} on noting the discrete LBB stability condition \eqref{eq:LBB} of the discrete velocity--pressure spaces.

The stability result \eqref{eq:bounds_FE_delta} follows from Lemma \ref{lemma:bounds_FE_delta}. This proves the theorem.
\end{proof}


\subsection{Existence of unregularized discrete solutions}
\label{sec:delta_to_zero}

Now let us consider a finite element approximation of \ref{P_alpha} without the regularization parameter $\delta$ and with a positive definite discrete Cauchy--Green tensor. 

\subsubsection*{Problem \ref{P_alpha_FE}:}
\mylabelHIDE{P_alpha_FE}{$(\pmbP_{\alpha,h}^{\Delta t})$} 
For given discrete initial and boundary data satisfying \eqref{eq:init_bounds} and $n\in\{1,...,N_T\}$, find the discrete solution $(\phi_h^{n}, \mu_h^n, \sigma_{h}^{n}, p_h^n, \pmbv_{h}^{n}, \B_{h}^{n}) \in (\calS_h)^4 \times \calV_h  \times \calW_{h,\mathrm{PD}}$ which satisfies, for any $(\zeta_h, \rho_h, \xi_h, q_h, \pmbw_h, \C_h) \in (\calS_h)^4 \times \calV_h \times \calW_h$:
\begin{subequations}
\begingroup
\allowdisplaybreaks
\begin{align}
    \label{eq:phi_FE}
    0 &= \int_\Omega \calI_h \Big[ \Big(\frac{\phi_h^n-\phi_h^{n-1}}{\Delta t}
    - \Gamma_{\phi,h}^n \Big) \zeta_h \Big]
    + \calI_h[m(\phi_h^{n-1})] \nabla\mu_h^n \cdot \nabla \zeta_h
    - \phi_h^{n-1} \pmbv_h^{n} \cdot\nabla \zeta_h \dx,
    \\
    \label{eq:mu_FE}
    0 &= \int_\Omega \calI_h \Big[ \Big( - \mu_h^n  
    + A \psi_1'(\phi_h^n) + A \psi_2'(\phi_h^{n-1}) 
    - \chi_\phi \sigma_h^n \Big) \rho_h \Big] 
    + B \nabla\phi_h^n \cdot \nabla\rho_h\dx,
    \\
    \nonumber
    \label{eq:sigma_FE}
    0 &= \int_\Omega \calI_h \Big[ \Big(\frac{\sigma_h^n-\sigma_h^{n-1}}{\Delta t}
    + \Gamma_{\sigma,h}^n \Big) \xi_h \Big]
    + \calI_h[n(\phi_h^{n-1})] \nabla(\chi_\sigma\sigma_h^n - \chi_\phi\phi_h^n)   \cdot \nabla \xi_h  
    - \sigma_h^{n-1} \pmbv_h^{n} \cdot\nabla\xi_h \dx
    \\
    & \qquad + \int_{\partial\Omega} \calI_h\big[ K (\sigma_h^n - \sigma_{\infty,h}^n) \xi_h \big] \dH^{d-1},
    \\
    \label{eq:div_v_FE}
    0 &= \int_\Omega \divergenz{\pmbv_h^{n}} q_h \dx,
    \\
    \label{eq:v_FE}
    \nonumber
    0 &= \int_\Omega \frac{\pmbv_h^n-\pmbv_h^{n-1}}{\Delta t} \cdot \pmbw_h
    + \frac{1}{2} \left( \left(\pmbv_h^{n-1}\cdot \nabla\right) \pmbv_h^n\right) \cdot \pmbw_h
    - \frac{1}{2} \pmbv_h^n \cdot \left(\left(\pmbv_h^{n-1} \cdot \nabla\right) \pmbw_h \right)  
    + 2\calI_h[\eta(\phi_h^{n-1})] \D(\pmbv_h^n) : \D(\pmbw_h)\dx 
    \\
    &\qquad + \int_\Omega 
    \kappa ( \B_h^n - \I ) : \nabla\pmbw_h 
    - \divergenz{\pmbw_h} p_h^n 
    + \big( \phi_h^{n-1}  \nabla\mu_h^n
    + \sigma_h^{n-1} \nabla(\chi_\sigma\sigma_h^n - \chi_\phi\phi_h^n) \big) \cdot \pmbw_h \dx,
    \\
    \label{eq:B_FE}
    \nonumber
    0 &= \int_\Omega \calI_h \Big[ 
    \Big(\frac{\B_h^n - \B_h^{n-1}}{\Delta t}
    + \frac{\kappa}{\tau(\phi_h^{n-1})} (\B_h^n - \I) \Big): \C_h \Big]
    - 2 \nabla\pmbv^n_h : \calI_h\big[ \C_h \B_h^n \big] 
    + \alpha \nabla\B_h^n : \nabla\C_h \dx
    \\
    &\qquad - \int_\Omega \sum\limits_{i,j=1}^d  
    [\pmbv_h^{n-1}]_i \Lambda_{i,j}(\B_h^n) : \partial_{x_j} \C_h \dx.
\end{align}
\endgroup
\end{subequations}

Here, the nonlinear function $\Lambda_{i,j}(\C_h)$ for $\C_h\in\calW_{h,\mathrm{PD}}$ is defined similarly to $\Lambda_{\delta,i,j}(\tilde\C_h)$ for $\tilde\C_h\in\calW_h$, see \cite[Rem.~5.1]{barrett_boyaval_2009}.
Moreover, the analogues of \eqref{eq:Lambda1},\eqref{eq:Lambda2} without $\delta$-regularization follow with the same arguments. In particular, for $k\in\{1,...,N_k\}$, it holds
\begin{alignat}{2}
    \label{eq:Lambda3}
    \sum\limits_{j=1}^d 
    \Lambda_{i,j}(\B_h) : \partial_{x_j} \calI_h\big[ \B_h^{-1} \big]
    &= - \partial_{x_i} \calI_h\big[ \trace ( 
    \ln \B_h ) \big]
    \qquad &&\text{on } K_k,
    \\
    \label{eq:Lambda4}
    \norm{ \Lambda_{i,j}(\B_h) }_{L^\infty(\Omega)} 
    &\leq C \norm{\B_h}_{L^\infty(\Omega)}
    \qquad &&\forall \  \B_h\in \calW_{h,\mathrm{PD}}.
\end{alignat}
Here we note that $H_\delta(G_\delta'(s)) \to \ln(s^{-1}) = - \ln(s)$ and $\beta_\delta(s) \to s$ for all $s>0$, as $\delta\to0$.

Now we define the unregularized energy $\calF_{h}: \calS_h\times \calS_h \times \calV_h \times \calW_{h,\mathrm{PD}} \to \R$ of the problem \ref{P_alpha_FE} by 
\begin{align}
\label{eq:def_energy_FE}
\begin{split}
    \calF_h(\phi_h,\sigma_h,\pmbv_h,\B_h) 
    &= 
    \int_\Omega 
    \calI_h\Big[ A \psi(\phi_h) 
    + \frac{\chi_\sigma}{2} \abs{\sigma_h}^2
    + \chi_\phi \sigma_h (1-\phi_h)
    + \frac{\kappa}{2} \trace\big( \B_h - \ln(\B_h)  \big)\Big] \dx
    \\
    &\quad + \int_\Omega 
    \frac{B}{2} \abs{\nabla\phi_h}^2 
    + \frac{1}{2} \abs{\pmbv_h}^2 \dx,
\end{split}
\end{align}
for all $(\phi_h,\sigma_h,\pmbv_h,\B_h)\in \calS_h \times \calS_h \times \calV_{h} \times \calW_{h,\mathrm{PD}}$.

Next, we obtain existence and stability of solutions to the problem \ref{P_alpha_FE} by passing to the limit $\delta\to0$ in the regularized discrete problem \ref{P_alpha_delta_FE} and in the \textit{a priori} bounds \eqref{eq:bounds_FE_delta}. 
This can be achieved analogously to \cite[Thm.~5.2]{barrett_boyaval_2009}.
We remark that the positive definiteness of the discrete Cauchy--Green tensor for the problem \ref{P_alpha_FE} is guaranteed as we can control the negative eigenvalues and the inverse of the discrete Cauchy--Green tensor from the $\delta$-regularized problem \ref{P_alpha_delta_FE}, which is due to \eqref{eq:bounds_FE_delta}. 
Moreover, as we have no control over the pressure of the regularized problem \ref{P_alpha_delta_FE}, the existence of a pressure for the problem \ref{P_alpha_FE} can still be established with the discrete LBB stability condition \eqref{eq:LBB}.

\begin{theorem}[Solutions to the unregularized discrete problem]
\label{theorem:existence_FE}
Let \ref{A1}--\ref{A5} hold. Suppose that the discrete initial and boundary data satisfy \eqref{eq:init_bounds} and assume that $\Delta t < \Delta t_*$, where $\Delta t_*$ is defined in \eqref{eq:dt}.
Then, for all $n\in\{1,...,N_T\}$, there exists at least one solution $(\phi_h^{n},\mu_h^n,\sigma_{h}^{n},p_h^n,\pmbv_{h}^{n},\B_{h}^{n}) \in (\calS_h)^4 \times \calV_h \times \calW_{h,\mathrm{PD}}$ to the unregularized discrete problem \ref{P_alpha_FE} with $\B_{h}^{n}$ being positive definite.
Moreover, all solutions of \ref{P_alpha_FE} are stable in the sense that
\begin{align}
\begin{split}
    \label{eq:bounds_FE}
    &  \max_{n=1,...,N_T} \Big( 
    \norm{\phi_h^n}_{H^1}^2
    +  \norm{\sigma_h^n}_{L^2}^2
    +  \norm{\pmbv_h^n}_{L^2}^2
    +  \nnorm{ \calI_h\big[ \abs{\B_h^n} \big] }_{L^1}\Big)
    \\
    &\quad
    +  \sum_{n=1}^{N_T} \Big(
    \norm{\nabla\phi_h^n - \nabla\phi_h^{n-1}}_{L^2}^2 
    +  \norm{\sigma_h^n - \sigma_h^{n-1}}_{L^2}^2
    +  \norm{\pmbv_h^n - \pmbv_h^{n-1}}_{L^2}^2 \Big)
    \\
    &\quad
    + \Delta t \sum_{n=1}^{N_T} \Big( 
    \norm{\mu_h^n}_{H^1}^2 
    + \norm{\nabla \sigma_h^n}_{L^2}^2
    + \norm{\sigma_h^n}_{L^2({\partial\Omega})}^2 
    + \norm{\nabla \pmbv_h^n}_{L^2}^2  
    + \int_\Omega
    \calI_h\Big[ \trace\big(\B_h^n + [\B_h^n]^{-1} -2\I \big) \Big] \dx  \Big)
    \\
    &\leq 
    C(T) \Big( 1 + \abs{\calF_h(\phi_h^0, \sigma_h^0, \pmbv_h^0, \B_h^0)}
    + \Delta t \sum_{n=1}^{N_T} \norm{\sigma_{\infty,h}^n}_{L^2(\partial\Omega)}^2 \Big) 
    \leq C(T),
\end{split}
\end{align}
where the constants $C(T)>0$ are independent of $h, \Delta t, \alpha$ but depend exponentially on $T$.

\end{theorem}




\subsection{Improving the regularity results in arbitrary dimensions}
\label{sec:regularity}

%
%

In the following, we derive higher order estimates for discrete solution of \ref{P_alpha_FE} in arbitrary dimensions $d\in\{2,3\}$. 
For the next steps, we require the $L^2$ projectors $\calP_h: \mathbf{V}\to \calV_{h,\mathrm{div}}$ and $\calQ_h: H^1(\Omega) \to \calS_h$ defined by
\begin{alignat}{2}
    \label{eq:projector_Ph_def}
    \int_\Omega \calP_h \pmbv \cdot \pmbw_h \dx
    &= \int_\Omega \pmbv\cdot \pmbw_h \dx
    \quad\quad &&\forall \  \pmbw_h \in \calV_{h,\mathrm{div}},
    \\
    \label{eq:projector_Qh_def}
    \int_\Omega \calI_h \big[\calQ_h \rho \zeta_h\big] \dx
    &= \int_\Omega \rho \zeta_h \dx
    \quad\quad &&\forall \  \zeta_h \in \calS_h,
\end{alignat}
which fulfill, as $\Omega$ is convex and the family $\{\calT_h\}_{h>0}$ is quasi-uniform, that
\begin{align}
    \label{eq:projector_Ph_bound}
    \norm{\calP_h\pmbv}_{H^1} &\leq C \norm{\pmbv}_{H^1} 
    \quad\quad \forall \  \pmbv\in \mathbf{V},
    \\
    \label{eq:projector_Qh_bound}
    \norm{\calQ_h\zeta}_{H^1} &\leq C \norm{\zeta}_{H^1} 
    \quad\quad \forall \  \zeta\in H^1(\Omega),
\end{align}
see, e.g., \cite{barrett_boyaval_2009} and references therein. Analogously to \eqref{eq:projector_Qh_def}, we also introduce a matrix valued projection operator $\calQ_h: H^1(\Omega;\R^{d\times d}_{\mathrm{S}}) \to \calW_h$ which fulfills a stability estimate corresponding to \eqref{eq:projector_Qh_bound}, see \cite{barrett_boyaval_2009}.


Now, we improve the regularity for the order parameter and the nutrient. 

\begin{lemma}
Let \ref{A1}--\ref{A5} hold. Suppose that the discrete initial and boundary data satisfy \eqref{eq:init_bounds} and let $\Delta t < \Delta t_*$, where $\Delta t_*$ is defined in \eqref{eq:dt}. Then, in addition to \eqref{eq:bounds_FE}, all solutions of \ref{P_alpha_FE} fulfill for any $l \in \{1,...,N_T\}$, 
\begin{subequations}
\begin{align}
\label{eq:bounds_FE_phi_dtphi_dtsigma}
    \Delta t \sum\limits_{n=1}^{N_T} \left( 
    \norm{\Delta_h \phi_h^n}_{L^2}^2 
    + \nnorm{\frac{\phi_h^n - \phi_h^{n-1}}{\Delta t}  }_{(H^1)'}^2
    + \nnorm{\frac{\sigma_h^n - \sigma_h^{n-1}}{\Delta t} }_{(H^1)'}^{4/d}
    \right)
    &\leq C(T),
    \\
    \label{eq:bounds_FE_phi_translation}
    \Delta t 
    \sum_{n=0}^{N_T - l}
    \norm{ \phi_h^{n+l}
    - \phi_h^n}_{L^2}^2 
    &\leq C(T) l \Delta t,   
\end{align}
\end{subequations}
where the constants $C(T)$ are independent of $\alpha, h,\Delta t$, but depend exponentially on $T$.
\end{lemma}


\begin{proof}
The first estimate in \eqref{eq:bounds_FE_phi_dtphi_dtsigma} can easily be shown by choosing $\rho_h = \Delta_h \phi_h^n$ in \eqref{eq:mu_FE} and using \eqref{eq:discr_laplace}, \ref{A4}, \eqref{eq:init_bounds}, \eqref{eq:bounds_FE}, \eqref{eq:norm_equiv} together with Hölder's and Young's inequalities.

For the reader's convenience, we show the third estimate in \eqref{eq:bounds_FE_phi_dtphi_dtsigma} and we note that the second estimate in \eqref{eq:bounds_FE_phi_dtphi_dtsigma} follows with similar arguments.
Let $\xi\in H^1(\Omega)$. Then, on choosing $\xi_h = \calQ_h \xi \in \calS_h$ in \eqref{eq:sigma_FE} and noting \eqref{eq:projector_Qh_def}, we obtain
\begin{align*}
    & \int_\Omega \Big( \frac{\sigma_h^n - \sigma_h^{n-1}}{\Delta t} \Big) \xi \dx 
    =
    \int_\Omega \calI_h \Big[ \Big(\frac{\sigma_h^n - \sigma_h^{n-1}}{\Delta t} \Big)
    \calQ_h \xi \Big] \dx
    = 
    \int_\Omega \sigma_h^{n-1} \pmbv_h^n \cdot \nabla \calQ_h \xi 
    - \calI_h\Big[ \Gamma_\sigma(\phi_h^n,\sigma_h^n) \calQ_h \xi \Big] \dx
    \\
    &\quad 
    - \int_\Omega \calI_h\big[ n(\phi_h^{n-1}) \big] \big( \chi_\sigma \nabla\sigma_h^n - \chi_\phi \nabla\phi_h^n \big) \cdot \nabla \calQ_h \xi \dx
    + \int_{\partial\Omega} K \calI_h\Big[ (\sigma_{\infty,h}^n - \sigma_h^n) 
    \calQ_h \xi \Big] \dH^{d-1} .
\end{align*}
Hence, on noting \eqref{eq:projector_Qh_bound}, \eqref{eq:norm_equiv}, \eqref{eq:norm_equiv_Gamma}, \ref{A2}, H{\"o}lder's inequality and the trace theorem, we get
\begin{align*}
    \abs{\int_\Omega \Big( \frac{\sigma_h^n - \sigma_h^{n-1}}{\Delta t} \Big) \xi \dx}
    &\leq 
    C \Big(1 +  \norm{\phi_h^n}_{H^1}
    + \norm{\sigma_h^n}_{H^1}
    + \norm{\pmbv_h^n \sigma_h^{n-1} }_{L^2}
    + \norm{\sigma_{\infty,h}^n}_{L^2(\partial\Omega)}\Big) 
    \norm{\xi}_{H^1}.
\end{align*}
By H{\"o}lder's inequality, a Gagliardo--Nirenberg inequality 
and the Sobolev embedding $H^1(\Omega) \hookrightarrow L^q(\Omega)$ with $1\leq q \leq 6$ for $d\in\{2,3\}$, we receive 
\begin{align*}
    \norm{\pmbv_h^n \sigma_h^{n-1} }_{L^2}
    \begin{cases}
    \leq
    C \norm{\sigma_h^{n-1}}_{L^4} 
    \norm{\pmbv_h^n}_{L^4}
    \leq 
    C \norm{\sigma_h^{n-1}}_{L^2}^{1/2}
    \norm{\sigma_h^{n-1}}_{H^1}^{1/2}
    \norm{\pmbv_h^n}_{L^2}^{1/2}
    \norm{\pmbv_h^n}_{H^1}^{1/2}, 
    & \text{ if } d=2,
    \\
    \leq
    C \norm{\sigma_h^{n-1}}_{L^3} 
    \norm{\pmbv_h^n}_{L^6}
    \leq 
    C \norm{\sigma_h^{n-1}}_{L^2}^{1/2}
    \norm{\sigma_h^{n-1}}_{H^1}^{1/2}
    \norm{\pmbv_h^n}_{H^1},
    & \text{ if } d=3.
    \end{cases}
\end{align*}
This leads to
\begin{align*}
    &\Delta t \sum\limits_{n=1}^{N_T}
    \nnorm{ \frac{\sigma_h^n - \sigma_h^{n-1}}{\Delta t}  }_{(H^1)'}^{\frac{4}{d}}
    \leq C 
    \Delta t \sum\limits_{n=1}^{N_T} \Big(1 
    + \norm{\phi_h^n}_{H^1} 
    + \norm{\sigma_h^n}_{H^1} 
    + \norm{\pmbv_h^n \sigma_h^{n-1} }_{L^2}
    + \norm{\sigma_{\infty,h}^n}_{L^2(\partial\Omega)} 
    \Big)^{\frac{4}{d}},
\end{align*}
where the right-hand side is bounded due to \eqref{eq:bounds_FE}, \eqref{eq:init_bounds} and H{\"o}lder's inequality. This shows the third estimate in \eqref{eq:bounds_FE_phi_dtphi_dtsigma}.


Next, we set $\zeta_h = \Delta t (\phi_h^{m+l} - \phi_h^m)$ in \eqref{eq:phi_FE}, where $m\in\{0,...,N_T - l\}$ and $l\in\{1,...,N_T\}$, to obtain
\begin{align*}
    0 &= \int_\Omega \calI_h \Big[ \Big(\phi_h^n-\phi_h^{n-1}
    - \Delta t \Gamma_{\phi,h}^n \Big)  (\phi_h^{m+l} - \phi_h^m) \Big]
    + \Delta t \Big( 
    \calI_h[m(\phi_h^{n-1})] \nabla\mu_h^n 
    - \phi_h^{n-1} \pmbv_h^{n} \Big)
    \cdot \nabla  (\phi_h^{m+l} - \phi_h^m) \dx.
\end{align*}
Summing from $n=m+1, ... , m+l$ gives
\begin{align*}
    0 &= \int_\Omega \calI_h \Big[ \abs{\phi_h^{m+l} - \phi_h^m}^2 \Big] \dx
    - \Delta t \sum_{n=m+1}^{m+l} \int_\Omega  
    \calI_h\Big[ (\phi_h^{m+l} - \phi_h^m) \Gamma_{\phi,h}^n  \Big] \dx
    \\
    &\quad
    + \Delta t \sum_{n=m+1}^{m+l} \int_\Omega \Big( 
    \calI_h\big[m(\phi_h^{n-1})\big] \nabla\mu_h^n \cdot \nabla  (\phi_h^{m+l} - \phi_h^m) 
    - \phi_h^{n-1} \pmbv_h^{n} \cdot \nabla  (\phi_h^{m+l} - \phi_h^m) \Big) \dx,
\end{align*}
which yields on noting \eqref{eq:norm_equiv}, \ref{A2}, \ref{A3}, H{\"o}lder's inequality, Sobolev embedding $H^1(\Omega) \hookrightarrow L^q(\Omega)$ with $1\leq q \leq 6$ for $d\in\{2,3\}$, \eqref{eq:bounds_FE} and \eqref{eq:init_bounds}, that
\begin{align*}
    &\norm{\phi_h^{m+l} - \phi_h^m}_{L^2}^2
    \leq
    C \Delta t \sum_{n=m+1}^{m+l} 
    \Big( \norm{\Gamma_\phi(\phi_h^n, \sigma_h^n, \B_h^n)}_{L^2} 
    + \norm{\nabla\mu_h^n}_{L^2}
    + \norm{\phi_h^{n-1}}_{L^6} \norm{\pmbv_h^n}_{L^3} 
    \Big) \norm{\phi_h^{m+l} - \phi_h^m}_{H^1}
    \\
    &\leq
    C \Delta t \sum_{n=m+1}^{m+l} 
    \Big( 1 + \norm{\phi_h^n}_{L^2} 
    + \norm{\sigma_h^n}_{L^2} 
    + \norm{\nabla\mu_h^n}_{L^2}
    + \norm{\phi_h^{n-1}}_{H^1} \norm{\pmbv_h^n}_{H^1} 
    \Big) \norm{\phi_h^{m+l} - \phi_h^m}_{H^1}
    \\
    &\leq 
    C(T) \Delta t \sum_{k=1}^{l} 
    \Big( 1 + \norm{\nabla\mu_h^{m+k}}_{L^2}
    + \norm{\pmbv_h^{m+k}}_{H^1} 
    \Big) \norm{\phi_h^{m+l} - \phi_h^m}_{H^1}.
\end{align*}
Multiplying both sides by $\Delta t$, summing from $m=0,...,N_T - l$, applying H{\"o}lder's inequality and noting \eqref{eq:bounds_FE} and \eqref{eq:init_bounds} leads to
\begin{align*}
    &\Delta t \sum_{m=0}^{N_T - l}
    \norm{\phi_h^{m+l} - \phi_h^m}_{L^2}^2
    \leq 
    C(T) (\Delta t)^2 \sum_{k=1}^l \sum_{m=0}^{N_T - l} 
    \Big( 1 + \norm{\nabla\mu_h^{m+k}}_{L^2}
    + \norm{\pmbv_h^{m+k}}_{H^1} 
    \Big) \norm{\phi_h^{m+l} - \phi_h^m}_{H^1}
    \\
    &\leq C(T) \Delta t  \sum_{k=1}^l
    \left( 1 + 
    \left( \Delta t \sum_{m=0}^{N_T - l} \norm{\nabla\mu_h^{m+k}}_{L^2}^2 \right)^\frac{1}{2}
    + \left( \Delta t 
    \sum_{m=0}^{N_T - l} \norm{\pmbv_h^{m+k}}_{H^1}^2 \right)^\frac{1}{2} 
    \right)
    \left(\Delta t \sum_{m=0}^{N_T - l} \norm{\phi_h^{m+l} - \phi_h^m}_{H^1}^2
    \right)^\frac{1}{2}
    \\
    &\leq 
    C(T) l \Delta t.
\end{align*}
This proves the lemma.
\end{proof}

\subsection{Improving the regularity results in two dimensions}
\label{sec:regularity_2D}

The next result contains ideas of \cite[Thm.~7.1]{barrett_boyaval_2009}.
We provide a regularity result for the left Cauchy--Green tensor in two space dimensions, supposed that a CFL condition for the time step size is fulfilled. 
The restriction to two space dimensions is due to a Gagliardo--Nirenberg inequality for $d=2$.

\begin{lemma}
\label{lemma:regul_B}
Let \ref{A1}--\ref{A6} hold true. Suppose that the discrete initial and boundary data satisfy \eqref{eq:init_bounds} and assume
\begin{align}
\label{eq:dt2}
    \Delta t \leq \min\{ \Delta t_*, \ c_*(T) \alpha^2 h^2 \},
\end{align}
where $\Delta t_*$ is defined in \eqref{eq:dt} and $c_*(T)>0$ is a (probably very small) constant which is independent of $\alpha, h, \Delta t$ but can depend on $T$. Then, in addition to \eqref{eq:bounds_FE}, \eqref{eq:bounds_FE_phi_dtphi_dtsigma}, \eqref{eq:bounds_FE_phi_translation}, the following bound holds for all solutions of \ref{P_alpha_FE}:
\begin{align}
\label{eq:bounds_FE_B_dtB}
\begin{split}
    &\max\limits_{n=1,...,N_T} \norm{\B_h^n}_{L^2}^2
    +  \sum\limits_{n=1}^{N_T} \left(
    \norm{\B_h^n - \B_h^{n-1}}_{L^2}^2 
    + \Delta t \norm{\nabla \B_h^n}_{L^2}^2 
    + \Delta t \nnorm{ \frac{\B_h^n - \B_h^{n-1}}{\Delta t} }_{(H^1)'}^{4/3}
    \right)
    \leq C(T, \alpha^{-1}),
\end{split}
\end{align}
where the constant $C(T,\alpha^{-1})>0$ is independent of $h,\Delta t$, but depends exponentially on $T, \alpha^{-1}$.
\end{lemma}

\begin{proof}

On choosing $\C_h=\B_h^n$ in \eqref{eq:B_FE}, it follows from \eqref{eq:elementary_identity}, \ref{A3}, H{\"o}lder's and Young's inequalities that
\begin{align}
\begin{split}
\label{eq:reg_B_1}
    & \frac{1}{2} \norm{\B_h^n}_h^2 
    + \frac{1}{2} \norm{\B_h^n - \B_h^{n-1}}_h^2 
    + \Delta t \frac{\kappa}{2\tau_1} \norm{\B_h^n}_h^2
    + \Delta t \alpha \norm{\nabla\B_h^n}_{L^2}^2
    \\
    &\leq \frac{1}{2} \norm{\B_h^{n-1}}_h^2
    + \Delta t \frac{\kappa}{4\tau_1}  \norm{\B_h^n}_h^2
    + C \Delta t \Big( 
    1 + \norm{\nabla\pmbv_h^n}_{L^2} \nnorm{\calI_h\big[ \abs{ \B_h^n }^2 \big]}_{L^2}
    +  \norm{\pmbv_h^{n-1}}_{L^4}
    \max\limits_{i,j=1,2} \norm{\Lambda_{i,j}(\B_h^n)}_{L^4} \norm{\nabla\B_h^n}_{L^2} \Big).
\end{split}
\end{align}
A calculation from \cite[Thm.~7.1]{barrett_boyaval_2009} and a Gagliardo--Nirenberg inequality for $d=2$ yield
\begin{align}
\label{eq:reg_B_2}
    \nnorm{\calI_h\big[ \abs{ \B_h^n }^2 \big]}^2_{L^2}
    + \max\limits_{i,j=1,2}  \norm{\Lambda_{i,j}(\B_h^n)}_{L^4}^4
    \leq C \norm{\B_h^n}_{L^4}^4 
    \leq C 
    \norm{\B_h^n}^2_{L^2} \norm{\B_h^n}^2_{H^1}.
\end{align}
It follows from a Gagliardo--Nirenberg inequality for $d=2$, \eqref{eq:bounds_FE}, \eqref{eq:init_v} and the Poincaré inequality that
\begin{align}
\label{eq:reg_B_3}
    \norm{\pmbv_h^{n-1}}_{L^4} 
    \leq C \norm{\pmbv_h^{n-1}}_{L^2}^{1/2} \norm{\pmbv_h^{n-1}}_{H^1}^{1/2}
    \leq C(T) \norm{\nabla\pmbv_h^{n-1}}_{L^2}^{1/2}.
\end{align}
Combining \eqref{eq:reg_B_1}--\eqref{eq:reg_B_3} gives together with \eqref{eq:norm_equiv}, \eqref{eq:bounds_FE} and a (generalized) Young's inequality that
\begin{align*}
    & \norm{\B_h^n}_{L^2}^2
    + \norm{\B_h^n - \B_h^{n-1}}_{L^2}^2 
    + \Delta t \frac{\kappa}{2\tau_1} \norm{\B_h^n}_{L^2}^2 
    + 2 \Delta t \alpha \norm{\nabla\B_h^n}_{L^2}^2 
    \\
    &\leq C \norm{\B_h^{n-1}}_{L^2}^2 
    + C(T) \Delta t \Big( 1 
    + \norm{\nabla\pmbv_h^n}_{L^2} \norm{\B_h^n}_{L^2} \norm{\B_h^n}_{H^1}
    + \norm{\nabla\pmbv_h^{n-1}}_{L^2}^{1/2} \norm{\B_h^n}_{L^2}^{1/2} \norm{\B_h^n}_{H^1}^{3/2} \Big)
    \\
    &\leq C \norm{\B_h^{n-1}}_{L^2}^2 
    +  \Delta t \min\left\{ \frac{\kappa}{4\tau_1}, \alpha \right\} \norm{\B_h^n}_{H^1}^2 
    + C(T) \Delta t + C(T)  \alpha^{-2} \Delta t \Big( 
    \norm{\nabla\pmbv_h^n}_{L^2}^2 
    + \norm{\nabla\pmbv_h^{n-1}}_{L^2}^2 \Big)
    \norm{\B_h^n}_{L^2}^2. 
\end{align*}
Summing from $n=1,...,m$, where $m\in\{1,...,N_T\}$ and absorbing the terms with index $n=m$ to the left-hand side yields
\begin{align}
\begin{split}
\label{eq:reg_B_4}
    &\Big( 1 - C(T) \alpha^{-2} \Delta t \big( \norm{\nabla\pmbv_h^m}_{L^2}^2 
    + \norm{\nabla\pmbv_h^{m-1}}_{L^2}^2 \big) \Big)
    \norm{\B_h^m}_{L^2}^2 
    + \sum\limits_{n=1}^m \Big( \norm{\B_h^n - \B_h^{n-1}}_{L^2}^2 
    + \Delta t \min\left\{ \frac{\kappa}{4\tau_1}, \alpha \right\}  \norm{\B_h^n}_{H^1}^2 \Big)
    \\
    &\leq C(T)  \Big( \norm{\B_h^0}_{L^2}^2 + 1 \Big)
    + C(T) \alpha^{-2} \Delta t  \sum\limits_{n=1}^{m-1} \Big(
    \norm{\nabla\pmbv_h^n}_{L^2}^2 
    + \norm{\nabla\pmbv_h^{n-1}}_{L^2}^2 \Big)
    \norm{\B_h^n}_{L^2}^2.
\end{split}
\end{align}
On noting \eqref{eq:inverse_estimate}, \eqref{eq:bounds_FE} and \eqref{eq:init_v}, we obtain
\begin{align*}
    \norm{\nabla\pmbv_h^m}_{L^2}^2 
    + \norm{\nabla\pmbv_h^{m-1}}_{L^2}^2 
    \leq C h^{-2} \big( \norm{\pmbv_h^m}_{L^2}^2 + \norm{\pmbv_h^{m-1}}_{L^2}^2 \big)
    \leq C(T) h^{-2}.
\end{align*}
Hence, if $\Delta t \leq \min\{ \Delta t_*, c_*(T) \alpha^2 h^2 \}$ for a (probably very small) constant $c_*(T)>0$ which depends on $T$ but not on $h,\Delta t, \alpha$, then the coefficient of $\norm{\B_h^m}_{L^2}^2$ on the left-hand side of \eqref{eq:reg_B_4} is positive. Then, we deduce from a discrete Gronwall argument (i.e.~Lemma \ref{lemma:gronwall_discrete}) that
\begin{align*}
    &\norm{\B_h^m}_{L^2}^2 
    + \sum\limits_{n=1}^m \Big( \norm{\B_h^n - \B_h^{n-1}}_{L^2}^2 
    + \Delta t \min\left\{ \frac{\kappa}{4\tau}, \alpha \right\}  \norm{\B_h^n}_{H^1}^2 \Big)
    \\
    &\leq C(T) \Big( \norm{\B_h^0}_{L^2}^2 + 1 \Big) 
    \exp\Big(  C(T) \alpha^{-2} \Delta t
    \sum\limits_{n=1}^{N_T} \Big( \norm{\nabla\pmbv_h^n}_{L^2}^2 
    + \norm{\nabla\pmbv_h^{n-1}}_{L^2}^2\Big)  \Big).
\end{align*}
Applying \eqref{eq:init_v}, \eqref{eq:init_B} and \eqref{eq:bounds_FE} proves the first three bounds in \eqref{eq:bounds_FE_B_dtB}.

Let $\C_h \in \calW_h$. Then, a straightforward calculation yields on noting
H{\"o}lder's inequality, a Gagliardo--Nirenberg inequality and \eqref{eq:bounds_FE_B_dtB} that
%
%
\begin{align}
\begin{split}
\label{eq:reg_dtB_1}
    \abs{ \int_\Omega \nabla\pmbv^n_h : \calI_h\big[ \C_h \B_h^n \big] \dx}
    \leq C \norm{\C_h}_{L^4}
    \norm{\B_h^n}_{L^4}
    \norm{\nabla\pmbv_h^n}_{L^2} 
    &\leq 
    C  \norm{\C_h}_{H^1}
    \norm{\B_h^n}_{L^2}^{1/2}
    \norm{\B_h^n}_{H^1}^{1/2}
    \norm{\pmbv_h^n}_{H^1}
    \\
    &\leq 
    C(T, \alpha^{-1})  \norm{\C_h}_{H^1}
    \norm{\B_h^n}_{H^1}^{1/2}
    \norm{\pmbv_h^n}_{H^1}.
\end{split}
\end{align}
Further, it holds with H{\"o}lder's inequality, \eqref{eq:reg_B_2}, \eqref{eq:bounds_FE}, \eqref{eq:init_v}, \eqref{eq:bounds_FE_B_dtB} and a Gagliardo--Nirenberg inequality that
\begin{align}
\begin{split}
\label{eq:reg_dtB_2}
    \abs{ \int_\Omega \sum\limits_{i,j=1}^2  
    [\pmbv_h^{n-1}]_i \Lambda_{i,j}(\B_h^n) : \partial_{x_j} \C_h \dx }
    &\leq C 
    \norm{\pmbv_h^{n-1}}_{L^4}
    \max_{i,j=1,2}
    \norm{\Lambda_{i,j}(\B_h^n)}_{L^4}
    \norm{\nabla\C_h}_{L^2}
    \\
    &\leq
    C \norm{\pmbv_h^{n-1}}_{L^2}^{1/2}
    \norm{\pmbv_h^{n-1}}_{H^1}^{1/2}
    \norm{\B_h^n}_{L^2}^{1/2}
    \norm{\B_h^n}_{H^1}^{1/2} 
    \norm{\C_h}_{H^1}
    \\
    &\leq 
    C(T,\alpha^{-1}) \norm{\pmbv_h^{n-1}}_{H^1}^{1/2}
    \norm{\B_h^n}_{H^1}^{1/2}
    \norm{\C_h}_{H^1}.
\end{split}
\end{align}
Let $\C\in H^1(\Omega; \R^{d\times d}_{\text{S}})$. On choosing $\C_h = \calQ_h \C \in \calW_h$ in \eqref{eq:B_FE}, we obtain on noting \eqref{eq:projector_Qh_def} and H{\"o}lder's inequality, that
\begin{align*}
    & \int_\Omega \Big(  \frac{\B_h^n - \B_h^{n-1}}{\Delta t} \Big) : \C \dx
    =
    \int_\Omega \calI_h \Big[ \Big( \frac{\B_h^n - \B_h^{n-1}}{\Delta t}\Big)  : 
    \calQ_h \C \Big] \dx
    = 
    \int_\Omega \calI_h \left[ 
    \frac{\kappa}{\tau(\phi_h^{n-1})} (\I - \B_h^n ): \calQ_h \C \right] \dx
    \\
    &\quad
    + \int_\Omega 2 \nabla\pmbv^n_h : \calI_h\Big[ \calQ_h \C \B_h^n \Big] 
    - \alpha \nabla\B_h^n : \nabla \calQ_h \C 
    + \sum\limits_{i,j=1}^d  
    [\pmbv_h^{n-1}]_i \Lambda_{i,j}(\B_h^n) : \partial_{x_j} \calQ_h \C \dx,
\end{align*}
which, on noting \eqref{eq:projector_Qh_bound}, \eqref{eq:reg_dtB_1} and \eqref{eq:reg_dtB_2}, yields
\begin{align*}
    \abs{ \int_\Omega \Big(  \frac{\B_h^n - \B_h^{n-1}}{\Delta t} \Big) : \C \dx }
    &\leq 
    C(T,\alpha^{-1}) \Big( 1
    + \norm{\B_h^n}_{H^1}
    + \norm{\B_h^n}_{H^1}^{1/2}
    \norm{\pmbv_h^n}_{H^1} 
    + \norm{\pmbv_h^{n-1}}_{H^1}^{1/2}
    \norm{\B_h^n}_{H^1}^{1/2} \Big) 
    \norm{\C}_{H^1}.
\end{align*}
This yields
\begin{align*}
    \Delta t \sum\limits_{n=1}^{N_T} 
    \nnorm{ \frac{\B_h^n - \B_h^{n-1}}{\Delta t}  }_{(H^1)'}^{4/3} 
    &\leq C(T,\alpha^{-1}) \Delta t \sum\limits_{n=1}^{N_T} \Big( 1 + 
    \norm{\B_h^n}_{H^1}^{4/3}  
    + \norm{\B_h^n}_{H^1}^{2/3}
    \norm{\pmbv_h^n}_{H^1}^{4/3} 
    + \norm{\pmbv_h^{n-1}}_{H^1}^{2/3}
    \norm{\B_h^n}_{H^1}^{2/3} \Big)
\end{align*}
On noting \eqref{eq:bounds_FE}, \eqref{eq:init_v}, the first and the third bounds in \eqref{eq:bounds_FE_B_dtB} and a H{\"o}lder inequality, we obtain the last bound in \eqref{eq:bounds_FE_B_dtB}.
\end{proof}


Now we have more control over the left Cauchy--Green tensor. This makes it possible to prove a regularity result for the discrete time derivative for the velocity.

First, we introduce the Helmholtz--Stokes operator $\calS: \mathbf{V}' \to \mathbf{V}$ such that $\calS\pmb{u}$ is the unique solution to the Helmholtz--Stokes problem
\begin{align}
    \label{eq:helmholtz-stokes}
    \int_\Omega (\calS\pmb{u})\cdot \pmbw 
    + \nabla (\calS\pmb{u}) : \nabla\pmbw \dx
    = \dualp{\pmb{u}}{\pmbw}_{\mathbf{V}} 
    \quad\quad \forall \  \pmbw\in \mathbf{V},
\end{align}
where $\dualp{\cdot}{\cdot}_{\mathbf{V}}$ denotes the duality pairing between $\mathbf{V}'$ and $\mathbf{V}$. We remark that $\norm{\calS \cdot}_{H^1}$ and $\norm{\cdot}_{\mathbf{V}'}$ are equivalent norms on $\mathbf{V}'$, see, e.g., \cite{barrett_sueli_2007}.

\begin{lemma}
\label{lemma:regul_v}
Let \ref{A1}--\ref{A6} hold. Suppose that the discrete initial and boundary data satisfy \eqref{eq:init_bounds} and that the CFL constraint \eqref{eq:dt2} holds. Then, in addition to \eqref{eq:bounds_FE}, \eqref{eq:bounds_FE_phi_dtphi_dtsigma}, \eqref{eq:bounds_FE_phi_translation}, \eqref{eq:bounds_FE_B_dtB}, all solutions of \ref{P_alpha_FE} fulfill
\begin{align}
\label{eq:bounds_FE_dtv_p}
    & \Delta t \sum\limits_{n=1}^{N_T}
    \nnorm{\calS\left( \frac{\pmbv_h^n - \pmbv_h^{n-1}}{\Delta t}  \right)}_{H^1}^{\frac{4}{3}}
    \leq C(T, \alpha^{-1}),
\end{align}
where the constant $C(T,\alpha^{-1})>0$ is depending exponentially on $T, \alpha^{-1}$, but is independent of $h, \Delta t$.
\end{lemma}

\begin{proof}

On choosing $\pmbw_h = \calP_h \Big[ \calS\big( \frac{\pmbv_h^n - \pmbv_h^{n-1}}{\Delta t} \big) \Big] \in \calV_{h,\text{div}} \subset \calV_h$ in \eqref{eq:v_FE}, we obtain on noting \eqref{eq:helmholtz-stokes}, \eqref{eq:projector_Ph_def}, \eqref{eq:projector_Ph_bound}, Hölder's inequality and Young's inequality, that
\begin{align*}
    &\nnorm{\calS \left( \frac{\pmbv_h^n - \pmbv_h^{n-1}}{\Delta t} \right) }_{H^1}^2
    =
    \int_\Omega \frac{\pmbv_h^n - \pmbv_h^{n-1}}{\Delta t} \cdot 
    \calP_h \left[ \calS\left( \frac{\pmbv_h^n - \pmbv_h^{n-1}}{\Delta t} \right) \right] \dx
    \\
    &= 
    \int_\Omega - \frac{1}{2} \left( \left(\pmbv_h^{n-1}\cdot \nabla \right) \pmbv_h^n\right) \cdot \calP_h \left[ \calS\left( \frac{\pmbv_h^n - \pmbv_h^{n-1}}{\Delta t} \right) \right] 
    + \frac{1}{2} \pmbv_h^n \cdot \left(\left(\pmbv_h^{n-1} \cdot \nabla\right) 
    \calP_h \left[ \calS\left( \frac{\pmbv_h^n - \pmbv_h^{n-1}}{\Delta t} \right) \right] \right)  \dx 
    \\
    &\quad
    - \int_\Omega 2\calI_h[\eta(\phi_h^{n-1})] \D(\pmbv_h^n) :
    \D\left( \calP_h \left[ \calS\left( \frac{\pmbv_h^n - \pmbv_h^{n-1}}{\Delta t} \right) \right] \right) 
    + \kappa \calI_h\big[ \B_h^n - \I \big] : \nabla \calP_h \left[ \calS\left( \frac{\pmbv_h^n - \pmbv_h^{n-1}}{\Delta t} \right) \right] \dx
    \\
    &\quad
    - \int_\Omega \big( \phi_h^{n-1}  \nabla\mu_h^n
    + \sigma_h^{n-1} \nabla (\chi_\sigma\sigma_h^n - \chi_\phi\phi_h^n) \big) \cdot \calP_h \left[ \calS\left( \frac{\pmbv_h^n - \pmbv_h^{n-1}}{\Delta t} \right) \right] \dx
    \\
    &\leq \frac{1}{2} \nnorm{\calS\left( \frac{\pmbv_h^n - \pmbv_h^{n-1}}{\Delta t} \right) }_{H^1}^2
    +  \Big( 1 + \norm{\B_h^n}_{L^2}^2
    + \norm{\nabla\pmbv_h^n}_{L^2}^2 
    + \nnorm{ \abs{\pmbv_h^{n-1}} \abs{\pmbv_h^n} }_{L^2}^2
    + \nnorm{ \abs{\pmbv_h^{n-1}} \abs{\nabla\pmbv_h^n} }_{L^{4/3}}^2
    \\
    &\quad
    + \norm{\phi_h^{n-1}  \nabla\mu_h^n
    + \sigma_h^{n-1} \nabla (\chi_\sigma\sigma_h^n - \chi_\phi\phi_h^n) }_{(H^1)'}^2 \Big).
\end{align*}
With Hölder's inequality, a Gagliardo--Nirenberg inequality for $d=2$, \eqref{eq:bounds_FE}, \eqref{eq:init_v}, the Poincaré inequality and Young's inequality, we obtain
\begin{align*}
    \nnorm{ \abs{\pmbv_h^{n-1}} \abs{\pmbv_h^n} }_{L^2}^2
    \leq C 
    \norm{\pmbv_h^{n-1}}_{L^4}^2  
    \norm{\pmbv_h^n}_{L^4}^2
    &\leq C 
    \norm{\pmbv_h^{n-1}}_{L^2}  
    \norm{\pmbv_h^{n-1}}_{H^1}  
    \norm{\pmbv_h^n}_{L^2}
    \norm{\pmbv_h^n}_{H^1}
    \\
    &\leq C(T) 
    \Big( \norm{\nabla\pmbv_h^n}_{L^2}^2 
    + \norm{\nabla\pmbv_h^{n-1}}_{L^2}^2 \Big),
\end{align*}
and, with similar arguments,
\begin{align*}
    \nnorm{ \abs{\pmbv_h^{n-1}} \abs{\nabla\pmbv_h^n} }_{L^{4/3}}^2
    &\leq C(T)
    \Big( \norm{\nabla\pmbv_h^{n-1} }_{L^2}^{3}
    + \norm{\nabla\pmbv_h^n }_{L^2}^{3} \Big),
\end{align*}
which leads to
\begin{align}
\begin{split}
\label{eq:bounds_FE_dtv_p_1}
    \nnorm{\calS\left( \frac{\pmbv_h^n - \pmbv_h^{n-1}}{\Delta t} \right) }_{H^1}^2
    &\leq 
    C(T) \Big( 1 + \norm{\B_h^n}_{L^2}^2
    + \norm{\nabla\pmbv_h^n}_{L^2}^2 
    + \norm{\nabla\pmbv_h^{n-1}}_{L^2}^2
    + \norm{\nabla\pmbv_h^n }_{L^2}^3
    \\
    &\quad
    + \norm{\nabla\pmbv_h^{n-1} }_{L^2}^3
    + \norm{\phi_h^{n-1}  \nabla\mu_h^n
    + \sigma_h^{n-1} \nabla(\chi_\sigma\sigma_h^n - \chi_\phi\phi_h^n)}_{(H^1)'}^2 \Big).
\end{split}
\end{align}

Next, we need an estimate for the last term in \eqref{eq:bounds_FE_dtv_p_1}.
Let $\pmbw\in H^1(\Omega;\R^d)$.
On noting H{\"o}lder's inequality, Sobolev embedding along with a Gagliardo--Nirenberg inequality, \eqref{eq:bounds_FE} and \eqref{eq:init_sig}, we obtain
\begin{align*}
    \abs{ \int_\Omega  \sigma_h^{n-1} \nabla (\chi_\sigma\sigma_h^n - \chi_\phi\phi_h^n)
    \cdot \pmbw \dx }
    &\leq 
    C \norm{\sigma_h^{n-1}}_{L^4}
    \big( \norm{\nabla\sigma_h^n}_{L^2}  
    + \norm{\nabla\phi_h^n}_{L^2}\big)
    \norm{\pmbw}_{L^4}
    \\
    &\leq 
    C \norm{\sigma_h^{n-1}}_{L^2}^{1/2} \norm{\sigma_h^{n-1}}_{H^1}^{1/2} 
    \big( \norm{\sigma_h^n}_{H^1} 
    + \norm{\phi_h^n}_{H^1}\big)
    \norm{\pmbw}_{H^1}
    \\
    &\leq 
    C(T) \norm{\sigma_h^{n-1}}_{H^1}^{1/2} 
    \big( \norm{\sigma_h^n}_{H^1} 
    + 1 \big)
    \norm{\pmbw}_{H^1}.
\end{align*}
This yields
\begin{align*}
    \Delta t \sum\limits_{n=1}^{N_T} 
    \nnorm{\sigma_h^{n-1} \nabla(\chi_\sigma\sigma_h^n - \chi_\phi\phi_h^n)}_{(H^1)'}^{4/3}
    &\leq C(T) \Delta t \sum\limits_{n=1}^{N_T} 
    \big( \norm{\sigma_h^n}_{H^1} 
    + 1 \big)^{2/3}
    \norm{\sigma_h^{n-1}}_{H^1}^{4/3},
\end{align*}
where the right-hand side is bounded due to a H{\"o}lder inequality, \eqref{eq:bounds_FE} and \eqref{eq:init_sig}.
From similar arguments, we deduce
\begin{align*}
    \Delta t \sum\limits_{n=1}^{N_T} \norm{\phi_h^{n-1}  \nabla\mu_h^n}_{(H^1)'}^2 \leq C(T).
\end{align*}
Taking the $\frac{2}{3}$ power on both sides of \eqref{eq:bounds_FE_dtv_p_1}, multiplying by $\Delta t$ and summing from $n=1,...,N_T$, we then get
\begin{align*}
    \Delta t \sum\limits_{n=1}^{N_T} \nnorm{\calS\left( \frac{\pmbv_h^n - \pmbv_h^{n-1}}{\Delta t} \right) }_{H^1}^{4/3} 
    &\leq C(T) \Delta t \sum\limits_{n=1}^{N_T} \Big( 1 
    + \norm{\B_h^n}_{L^2}^{4/3}
    + \norm{\nabla\pmbv_h^n}_{L^2}^{4/3} 
    + \norm{\nabla\pmbv_h^{n-1}}_{L^2}^{4/3}
    + \norm{\nabla\pmbv_h^n }_{L^2}^2
    \\
    &\quad
    + \norm{\nabla\pmbv_h^{n-1} }_{L^2}^2
    + \norm{\phi_h^{n-1}  \nabla\mu_h^n
    + \sigma_h^{n-1} \nabla(\chi_\sigma\sigma_h^n - \chi_\phi\phi_h^n)}_{(H^1)'}^{4/3} \Big).
\end{align*}
This leads to \eqref{eq:bounds_FE_dtv_p}, on noting \eqref{eq:bounds_FE}, \eqref{eq:bounds_FE_B_dtB}, \eqref{eq:init_bounds}, H{\"o}lder's inequality and the calculations from above.
\end{proof}

At this point, we note that the bound \eqref{eq:bounds_FE_dtv_p} is not useful to apply common compactness techniques based on Aubin--Lions. 
The problem is that the discrete velocity belongs to $\calV_{h,\mathrm{div}} \subset H^1_0(\Omega;\R^d)$, but $\calV_{h,\mathrm{div}}$ is no subspace of $\mathbf{V}$, as the discrete velocity is only divergence-free with respect to the ansatz space $\calS_h$. 
Moreover, $H^1_0(\Omega;\R^d)$ is compactly embedded in $L^2(\Omega;\R^d)$, but there exists no injective mapping from $L^2(\Omega;\R^d)$ into $\mathbf{V}'$. This is why Aubin--Lions cannot be applied here.

However, there are other techniques one can use for the velocity, see, e.g., \cite{barrett_boyaval_2009, barrett_2018_fene-p, guillen_2013_liquid_crystal}. 
In this work, we follow the strategy of Metzger \cite{metzger_2018} which is based on \cite{azerad_guillen_2001} and we introduce the orthogonal Stokes projector $\calR_h: \calV_{h,\mathrm{div}} \to \mathbf{V}$ by
\begin{align}
\label{eq:projector_Stokes_def}
    \int_\Omega \nabla \calR_h \pmbv_h : \nabla \pmbw \dx
    = \int_\Omega \nabla \pmbv_h : \nabla\pmbw \dx
    \quad\quad \forall \  \pmbw \in \mathbf{V}.
\end{align}
For any $\pmbv_h \in \calV_{h,\mathrm{div}}$, it holds (cf.~\cite{guillen_2013_liquid_crystal})
\begin{subequations}
\begin{align}
    \label{eq:projector_Stokes_H1}
    \norm{\calR_h \pmbv_h}_{H^1} 
    &\leq C \norm{\pmbv_h}_{H^1},
    \\
    \label{eq:projector_Stokes_L2}
    \norm{\calR_h \pmbv_h - \pmbv_h}_{L^2} 
    &\leq C h \norm{\divergenz{\pmbv_h}}_{L^2},
    \\
    \label{eq:projector_Stokes_H1'}
    \norm{\calR_h \pmbv_h}_{\mathbf{V}'} 
    &\leq C \big( h \norm{\divergenz{\pmbv_h}}_{L^2}
    + \norm{\pmbv_h}_{\mathbf{V}'} \big).
\end{align}
\end{subequations}

The following result is based on the bounds \eqref{eq:projector_Stokes_H1'}, \eqref{eq:bounds_FE} and \eqref{eq:bounds_FE_dtv_p}.

\begin{lemma}
Let \ref{A1}--\ref{A6} hold. Suppose that the discrete initial and boundary data satisfy \eqref{eq:init_bounds} and that the CFL constraint \eqref{eq:dt2} holds. Then, there exists a constant $C(T,\alpha^{-1})>0$ depending exponentially on $T, \alpha^{-1}$ but not on $h,\Delta t$, such that, in addition to \eqref{eq:bounds_FE}, \eqref{eq:bounds_FE_phi_dtphi_dtsigma}, \eqref{eq:bounds_FE_phi_translation}, \eqref{eq:bounds_FE_B_dtB}, \eqref{eq:bounds_FE_dtv_p}, all solutions of \ref{P_alpha_FE} satisfy for all $l\in\{1,...,N_T\}$,
\begin{align}
    \label{eq:bounds_FE_v_translation}
    \Delta t \sum_{n=0}^{N_T - l}
    \norm{\calR_h \pmbv_h^{n+l} - \calR_h \pmbv_h^{n}}_{\mathbf{V}'}^2
    \leq C(T,\alpha^{-1})  \big( l^\frac{3}{4} \Delta t  +  h^2 \big).
\end{align}
\end{lemma}

\begin{proof}
It follows from \eqref{eq:projector_Stokes_H1'} that
\begin{align*}
    &\Delta t \sum_{n=0}^{N_T - l}
    \norm{\calR_h \pmbv_h^{n+l} - \calR_h \pmbv_h^{n}}_{\mathbf{V}'}^2
    \leq 
    C \Delta t \sum_{n=0}^{N_T - l} \Big(
    \norm{\pmbv_h^{n+l} -  \pmbv_h^{n}}_{\mathbf{V}'}^2
    + h^2 \norm{\divergenz{\pmbv_h^{n+l} -  \pmbv_h^{n}}}_{L^2}^2 \Big)
    \eqqcolon I + II.
\end{align*}
On applying a H{\"o}lder's inequality, we obtain for the first term
\begin{align*}
    I \leq 
    C \bigg( \Delta t \sum_{n=0}^{N_T - l} 
    \norm{\pmbv_h^{n+l} -  \pmbv_h^{n}}_{\mathbf{V}'}^\frac{4}{3} \bigg)^\frac{3}{4}
    \bigg( \Delta t \sum_{n=0}^{N_T - l}
    \norm{\pmbv_h^{n+l} -  \pmbv_h^{n}}_{L^2}^4 \bigg)^\frac{1}{4},
\end{align*}
where the second term on the right-hand side is bounded due to \eqref{eq:bounds_FE}. By a triangle inequality and \eqref{eq:bounds_FE_dtv_p}, we get
\begin{align*}
    \Delta t \sum_{n=0}^{N_T - l} 
    \norm{\pmbv_h^{n+l} -  \pmbv_h^{n}}_{\mathbf{V}'}^\frac{4}{3} 
    &\leq 
    C \Delta t \sum_{m=1}^l \sum_{n=0}^{N_T - l} 
    (\Delta t)^\frac{4}{3}
    \nnorm{\frac{\pmbv_h^{n+m} -  \pmbv_h^{n+m-1}}{\Delta t}}_{\mathbf{V}'}^\frac{4}{3} 
    \leq 
    C(T,\alpha^{-1}) l (\Delta t)^\frac{4}{3},
\end{align*}
which yields $I \leq C(T,\alpha^{-1}) l^\frac{3}{4} \Delta t$.
Moreover, it holds by \eqref{eq:bounds_FE} that $II \leq C(T) h^2$.
This shows the result.
\end{proof}


\subsection{Passage to the limit and convergence to a weak solution}
\label{sec:convergence}

In the following, we prove that there exists a weak solution of \ref{P_alpha} in the sense of Definition \ref{def:weak_solution} which is obtained from converging subsequences of discrete solutions of \ref{P_alpha_FE} by passing to the limit $(h,\Delta t) \to (0,0)$.


For future reference, we recall the following compactness results from \cite[Sect.~8, Cor.~4 and Thm.~5]{simon_1986}.
Let $X, Y, Z$ be Banach spaces with a compact embedding $X \hookrightarrow \hookrightarrow Y$ and a continuous embedding $Y \hookrightarrow Z$. Let $1\leq p < \infty$ and $r>1$. Then, the following embeddings are compact:
\begin{subequations}
\begin{alignat}{3}
    \label{eq:compact_Lp}
    &\{\eta \in L^p(0,T;X) 
    &&\mid \partial_t\eta \in L^1(0,T;Z) \}
    &&\hookrightarrow \hookrightarrow L^p(0,T;Y),
    \\
    \label{eq:compact_C}
    &\{\eta \in L^\infty(0,T;X)  
    &&\mid \partial_t\eta \in L^r(0,T;Z) \}
    &&\hookrightarrow \hookrightarrow  C([0,T];Y).
\end{alignat}
Moreover, let $F$ be a bounded subset in $L^p(0,T;X)$ with
\begin{align}
    \label{eq:compact_translation}
    \lim_{\theta\to 0} \norm{\eta(\cdot+\theta) - \eta(\cdot)}_{L^p(0,T;Z)} = 0 
    \quad
    \text{ uniformly for } \eta\in F.
\end{align}
Then $F$ is relatively compact in $L^p(0,T;Y)$ if $1\leq p < \infty$ and in $C([0,T];Y)$ if $p=\infty$, respectively.

Furthermore, we recall the following ``compactness by perturbation'' result from Azérad and Guillén-González \cite{azerad_guillen_2001} which provides strong convergence for subsequences of the discrete velocity.
Let $X,Y,Z$ be like before. Let $\{f_\epsilon\}_{\epsilon>0}$ be a family of functions which is bounded in $L^p(0,T;X)$ with $1\leq p < \infty$ such that
\begin{align}
    \label{eq:compact_perturbation}
    \norm{f_\epsilon(\cdot+\theta) - f_\epsilon(\cdot)}_{L^p(0,T;Z)}
    \leq g_1(\theta) + g_2(\epsilon),
\end{align}
with $g_1(\theta)\to 0$ as $\theta\to 0$ and $g_2(\epsilon)\to 0$ as $\epsilon\to 0$. Then, the family $\{f_\epsilon\}_{\epsilon>0}$ possesses a cluster point in $L^p(0,T;Y)$ as $\epsilon\to 0$.

\end{subequations}

\bigskip
Let us introduce the following notation for affine linear and piecewise constant extensions of time discrete function $a^n(\cdot)$, $n=0,...,N_T$:
\begin{alignat}{2}
    \label{def:fun_Delta_t}
    a^{\Delta t}(\cdot, t) 
    &\coloneqq 
    \frac{t - t^{n-1}}{\Delta t} a^n(\cdot)
    + \frac{t^n - t}{\Delta t} a^{n-1}(\cdot)
    \quad\quad 
    && t\in [t^{n-1},t^n], n\in \{1,...,N_T\},
    \\
    \label{def:fun_Delta_t_pm}
    a^{\Delta t,+}(\cdot, t) 
    &\coloneqq a^n(\cdot),
    \quad\quad 
    a^{\Delta t,-}(\cdot, t) 
    \coloneqq a^{n-1}(\cdot)
    \quad\quad 
    && t\in (t^{n-1},t^n], n\in \{1,...,N_T\}.
\end{alignat}
Let us note that we write $a^{\Delta t,\pm}$ for results that hold true for both $a^{\Delta t,+}$ and $a^{\Delta t,-}$, and we write $a^{\Delta t(,\pm)}$ for results that hold true for $a^{\Delta t}$, $a^{\Delta t,+}$ and $a^{\Delta t,-}$, respectively.

Using this notation, we reformulate the problem \ref{P_alpha_FE} continuously in time. Multiplying each equation by $\Delta t$ and summing from $n=1,...,N_T$, we obtain for any test functions $(\zeta_h$, $\rho_h$, $ \xi_h$, $ q_h$, $ \pmbw_h$, $ \C_h) \in (L^2(0,T;\calS_h))^4 \times L^2(0,T;\calV_h) \times L^2(0,T;\calW_h)$ that
\begin{subequations}
\begingroup
\allowdisplaybreaks
\begin{align}
    \label{eq:phi_FE_time}
    0 &= \int_{\Omega_T} \calI_h \Big[ \big(\partial_t \phi_h^{\Delta t}
    - \Gamma_{\phi,h}^{\Delta t,+} \big) \zeta_h \Big]
    + \calI_h[m(\phi_h^{\Delta t,-})] \nabla\mu_h^{\Delta t,+} \cdot \nabla \zeta_h
    - \phi_h^{\Delta t,-} \pmbv_h^{\Delta t,+} \cdot\nabla \zeta_h \dx \dt,
    \\
    \label{eq:mu_FE_time}
    0 &= \int_{\Omega_T} \calI_h \Big[ \Big( - \mu_h^{\Delta t,+}  
    + A \psi_1'(\phi_h^{\Delta t,+}) 
    + A \psi_2'(\phi_h^{\Delta t,-}) 
    - \chi_\phi \sigma_h^{\Delta t,+} \Big) \rho_h \Big] 
    + B \nabla\phi_h^{\Delta t,+} \cdot \nabla\rho_h\dx \dt,
    \\
    \nonumber
    \label{eq:sigma_FE_time}
    0 &= \int_{\Omega_T} \calI_h \Big[ \big( \partial_t \sigma_h^{\Delta t}
    + \Gamma_{\sigma,h}^{\Delta t,+} \big) \xi_h \Big]
    + \calI_h[n(\phi_h^{\Delta t,-})] \nabla \big(\chi_\sigma\sigma_h^{\Delta t,+} - \chi_\phi\phi_h^{\Delta t,+} \big)   \cdot \nabla \xi_h  
    - \sigma_h^{\Delta t,-} \pmbv_h^{\Delta t,+} \cdot\nabla \xi_h \dx \dt
    \\
    &\qquad + \int_0^T \int_{\partial\Omega} \calI_h\Big[ K \big( \sigma_h^{\Delta t,+} - \sigma_{\infty,h}^{\Delta t,+} \big) \xi_h \Big] \dH^{d-1} \dt,
    \\
    \label{eq:div_v_FE_time}
    0 &= \int_{\Omega_T} \divergenz{\pmbv_h^{\Delta t,+}} q_h \dx \dt,
    \\
    \label{eq:v_FE_time}
    \nonumber
    0 &= \int_{\Omega_T} \partial_t \pmbv_h^{\Delta t} \cdot \pmbw_h
    + \frac{1}{2} \left( \left(\pmbv_h^{\Delta t,-}\cdot \nabla\right) \pmbv_h^{\Delta t,+}\right) \cdot \pmbw_h
    - \frac{1}{2} \pmbv_h^{\Delta t,+} \cdot \left(\left(\pmbv_h^{\Delta t,-} \cdot \nabla\right) \pmbw_h \right)   \dx \dt 
    \\
    \nonumber
    &\qquad + \int_{\Omega_T} 2\calI_h[\eta(\phi_h^{\Delta t,-})] \D(\pmbv_h^{\Delta t,+}) : \D(\pmbw_h)
    + \kappa ( \B_h^{\Delta t,+} - \I ) : \nabla\pmbw_h 
    - \divergenz{\pmbw_h} p_h^{\Delta t,+} \dx \dt
    \\
    &\qquad + \int_{\Omega_T} \Big( \phi_h^{\Delta t,-}  \nabla\mu_h^{\Delta t,+}
    + \sigma_h^{\Delta t,-} \nabla \big(\chi_\sigma\sigma_h^{\Delta t,+} - \chi_\phi\phi_h^{\Delta t,+}\big) \Big) \cdot \pmbw_h \dx \dt,
    \\
    \label{eq:B_FE_time}
    \nonumber
    0 &= \int_{\Omega_T} \calI_h \Big[ 
    \Big( \partial_t \B_h^{\Delta t}
    + \frac{\kappa}{\tau(\phi_h^{\Delta t,-})} (\B_h^{\Delta t,+} - \I) \Big): \C_h \Big]
    - 2 \nabla\pmbv_h^{\Delta t,+} : \calI_h\big[ \C_h \B_h^{\Delta t,+} \big] \dx \dt
    \\
    &\qquad + \int_{\Omega_T} \alpha \nabla\B_h^{\Delta t,+} : \nabla\C_h
    - \sum\limits_{i,j=1}^d  
    [\pmbv_h^{\Delta t,-}]_i \Lambda_{i,j}(\B_h^{\Delta t,+}) : \partial_{x_j} \C_h \dx \dt,
\end{align}
\endgroup
\end{subequations}
subject to the initial conditions $\phi_h^{\Delta t}(0) = \phi_h^0$, $\sigma_h^{\Delta t}(0) = \sigma_h^0$, $\pmbv_h^{\Delta t}(0) = \pmbv_h^0$ and $\B_h^{\Delta t}(0) = \B_h^0$, where we write $\Gamma_{\phi,h}^{\Delta t,+} = \Gamma_{\phi}(\phi_h^{\Delta t,+}, \sigma_h^{\Delta t,+}, \B_h^{\Delta t,+})$ and similarly for $\Gamma_{\sigma,h}^{\Delta t,+}$.

The following result is a direct consequence of \eqref{def:fun_Delta_t}, \eqref{def:fun_Delta_t_pm}, Theorem \ref{theorem:existence_FE}, \eqref{eq:bounds_FE}, \eqref{eq:bounds_FE_phi_dtphi_dtsigma}, \eqref{eq:bounds_FE_phi_translation}, \eqref{eq:bounds_FE_B_dtB}, \eqref{eq:bounds_FE_dtv_p}, \eqref{eq:bounds_FE_v_translation} and \eqref{eq:init_bounds}.

\begin{corollary}
Let \ref{A1}--\ref{A5} hold. Suppose that the discrete initial and boundary data satisfy \eqref{eq:init_bounds}. Moreover, assume that $\Delta t < \Delta t_*$, where $\Delta t_*$ is defined in \eqref{eq:dt}. Then, there exist functions $\phi_h^{\Delta t(,\pm)}$, 
$\mu_h^{\Delta t,+}$, 
$\sigma_h^{\Delta t(,\pm)}$, 
$p_h^{\Delta t,+}$, 
$\pmbv_h^{\Delta t(,\pm)}$, 
$\B_h^{\Delta t(,\pm)}$ solving \eqref{eq:phi_FE_time}--\eqref{eq:B_FE_time} and constants $C(T)>0$ depending on $T$ but not on $\alpha,h,\Delta t$, such that
\begin{subequations}
\begin{align}
\label{eq:bounds_time1}
    \nonumber
    &\norm{\phi_h^{\Delta t(,\pm)}}_{L^\infty(0,T;H^1)}
    + \norm{\Delta_h \phi_h^{\Delta t(,\pm)}}_{L^2(0,T;L^2)} 
    + \norm{\partial_t \phi_h^{\Delta t} }_{L^2(0,T;(H^1)')}
    + \tfrac{1}{\sqrt{\Delta t}} \norm{\phi_h^{\Delta t} - \phi_h^{\Delta t,\pm}}_{L^2(0,T;H^1)}
    \\
    \nonumber
    &\quad
    + \norm{\mu_h^{\Delta t,+}}_{L^2(0,T;H^1)}
    + \norm{\sigma_h^{\Delta t(,\pm)}}_{L^\infty(0,T;L^2)}
    + \norm{\sigma_h^{\Delta t(,\pm)}}_{L^2(0,T;H^1)} 
    + \norm{\partial_t \sigma_h^{\Delta t} }_{L^{4/d}(0,T;(H^1)')}
    \\
    \nonumber
    &\quad
    + \tfrac{1}{\sqrt{\Delta t}} \norm{\sigma_h^{\Delta t} - \sigma_h^{\Delta t,\pm}}_{L^2(0,T;L^2)}
    + \norm{\sigma_h^{\Delta t(,\pm)}}_{L^2(0,T;L^2({\partial\Omega}))}
    + \norm{\pmbv_h^{\Delta t(,\pm)}}_{L^\infty(0,T;L^2)} 
    + \norm{\pmbv_h^{\Delta t(,\pm)}}_{L^2(0,T;H^1)}
    \\
    \nonumber
    &\quad 
    + \tfrac{1}{\sqrt{\Delta t}} \norm{\pmbv_h^{\Delta t} - \pmbv_h^{\Delta t,\pm}}_{L^2(0,T;L^2)}
    \\
    &\leq C(T),
\end{align}
and, for any $l\in\{1,...,N_T\}$,
\begin{align}
    \label{eq:bounds_time_phi_translation}
    \int_0^{T- l \Delta t} 
    \nnorm{\phi_h^{\Delta t(,\pm)}(\cdot,t+ l \Delta t) 
    - \phi_h^{\Delta t(,\pm)}(\cdot,t) }_{L^2}^2 \dt
    &\leq C(T)  l\Delta t.
\end{align}
Further, if in addition \ref{A6} and the CFL constraint \eqref{eq:dt2} hold true, then there exist constants $C(T,\alpha^{-1})>0$ depending on $T,\alpha^{-1}$ but not on $h,\Delta t$, such that
\begin{align}
\label{eq:bounds_time2}
    \nonumber
    & \norm{\calS \partial_t \pmbv_h^{\Delta t} }_{L^{4/3}(0,T;H^1)}
    + \norm{\B_h^{\Delta t(,\pm)}}_{L^\infty(0,T;L^2)} 
    + \norm{\B_h^{\Delta t(,\pm)}}_{L^2(0,T;H^1)} 
    + \norm{\partial_t \B_h^{\Delta t} }_{L^{4/3}(0,T;(H^1)')}
    \\
    &\quad 
    + \tfrac{1}{\sqrt{\Delta t}} \norm{\B_h^{\Delta t} - \B_h^{\Delta t,\pm}}_{L^2(0,T;L^2)}
    \leq C(T,\alpha^{-1}),
\end{align}
and, for any $l\in\{1,...,N_T\}$,
\begin{align}
    \label{eq:bounds_time_v_translation}
    \int_0^{T- l \Delta t} 
    \nnorm{\calR_h \pmbv_h^{\Delta t(,\pm)}(\cdot,t+ l \Delta t) 
    - \calR_h \pmbv_h^{\Delta t(,\pm)}(\cdot,t) }_{\mathbf{V}'}^2 \dt
    &\leq C(T,\alpha^{-1})  \Big( l^\frac{3}{4} \Delta t +  h^2 \Big).
\end{align}
\end{subequations}
\end{corollary}


We now show that there exist subsequences of discrete solutions which converge to some limit functions, as $(h,\Delta t)\to 0$. 

\begin{lemma}[Converging subsequences]
\label{lemma:convergence}
Let \ref{A1}--\ref{A5} hold. Suppose that the discrete initial and boundary data satisfy \eqref{eq:init_bounds} and \eqref{eq:init_conv}. Moreover, assume that $\Delta t < \Delta t_*$, where $\Delta t_*$ is defined in \eqref{eq:dt}.
Then, there exists a (non-relabeled) subsequence of $\left\{  \phi_h^{\Delta t(,\pm)}, 
\mu_h^{\Delta t,+}, 
\sigma_h^{\Delta t(,\pm)}, 
p_h^{\Delta t,+}, 
\pmbv_h^{\Delta t(,\pm)}, 
\B_h^{\Delta t(,\pm)} \right\}_{h,\Delta t>0}$, such that \eqref{eq:phi_FE_time}--\eqref{eq:B_FE_time} is fulfilled, and functions 
\begin{align*}
    \phi &\in L^\infty(0,T;H^1) 
    \cap L^2(0,T;H^2)
    \cap H^1(0,T; (H^1)'),
    \quad 
    \mu \in L^2(0,T; H^1),
    \\
    \sigma &\in L^\infty(0,T; L^2) 
    \cap L^2(0,T; H^1) 
    \cap W^{1,\frac{4}{d}}(0,T; (H^1)'),
    \quad 
    \pmbv \in L^\infty(0,T; \mathbf{H}) \cap L^2(0,T; \mathbf{V}),
\end{align*}
with $\phi(0) = \phi_0 \text{ in } L^2(\Omega)$ and $\sigma(0) = \sigma_0 \text{ in } L^2(\Omega)$ exist,
%
%
such that, as $(h,\Delta t) \to (0,0)$,
\begin{alignat}{3}
    \newsubeqblock
    \mysubeq
    \label{eq:conv_phi_Linf_H1}
    \phi_h^{\Delta t(,\pm)} &\to \phi \quad &&\text{weakly-$*$} \quad && \text{in } L^\infty(0,T;H^1),
    \\
    \mysubeq
    \label{eq:conv_dtphi_L2_H1'}
    \partial_t \phi_h^{\Delta t} &\to \partial_t\phi \quad
    &&\text{weakly} \quad && \text{in } L^2(0,T;(H^1)'),
    \\
    \mysubeq
    \label{eq:conv_phi_L2_H2}
    \Delta_h \phi_h^{\Delta t(,\pm)} &\to \Delta \phi \quad &&\text{weakly} \quad && \text{in } L^2(0,T;L^2),
    \\
    \mysubeq
    \label{eq:conv_phi_L2_W1s}
    \phi_h^{\Delta t(,\pm)} &\to \phi \quad &&\text{weakly} \quad && \text{in } L^2(0,T;W^{1,s}),
    \\
    \mysubeq
    \label{eq:conv_phi_strong}
    \phi_h^{\Delta t(,\pm)} &\to \phi \quad &&\text{strongly} \quad && \text{in } L^2(0,T;C^{0,\gamma}(\overline\Omega)),
    \\[2ex]
    \label{eq:conv_mu_L2_H1}
    \mu_h^{\Delta t,+} &\to \mu \quad &&\text{weakly} \quad && \text{in } L^2(0,T;H^1),
    \\[2ex]
    \newsubeqblock
    \mysubeq
    \label{eq:conv_sigma_Linf_L2}
    \sigma_h^{\Delta t(,\pm)} &\to \sigma \quad &&\text{weakly-$*$} \quad && \text{in } L^\infty(0,T; L^2),
    \\
    \mysubeq
    \label{eq:conv_sigma_L2_H1}
    \sigma_h^{\Delta t(,\pm)} &\to \sigma \quad &&\text{weakly} \quad && \text{in }  L^2(0,T;H^1),
    \\
    \mysubeq
    \label{eq:conv_dtsigma_L4/3_H1'}
    \partial_t \sigma_h^{\Delta t} &\to \partial_t\sigma \quad
    &&\text{weakly} \quad && \text{in } L^{4/d}(0,T;(H^1)'),
    \\
    \mysubeq
    \label{eq:conv_sigma_strong}
    \sigma_h^{\Delta t(,\pm)} &\to \sigma \quad &&\text{strongly} \quad && \text{in } L^2(0,T; L^q),
    \\[2ex]
    \newsubeqblock
    \mysubeq
    \label{eq:conv_v_Linf_L2}
    \pmbv_h^{\Delta t(,\pm)} &\to \pmbv \quad
    && \text{weakly-$*$} \quad &&\text{in } 
    L^\infty(0,T; \mathbf{H}), 
    \\
    \mysubeq
    \label{eq:conv_v_L2_H1}
    \pmbv_h^{\Delta t(,\pm)} &\to \pmbv \quad
    && \text{weakly} \quad &&\text{in } 
    L^2(0,T;\mathbf{V}), 
\end{alignat}
where $s\in\left[2, \frac{2d}{d-2} \right)$, $\gamma\in \left(0, \frac{4-d}{2}\right)$ and $q\in \left[1,\frac{2d}{d-2}\right)$, respectively.
Moreover, if in addition \ref{A6} and the CFL constraint \eqref{eq:dt2} holds, then additionally $\pmbv \in W^{1,\frac{4}{3}}(0,T; \mathbf{V}')$ with $\pmbv(0) = \pmbv_0 \text{ in } \mathbf{H}$ and there exists a function
\begin{align*}
    \B &\in L^\infty\left(0,T; L^2(\Omega;\R^{2\times2}_{\mathrm{SPD}})\right)
    \cap L^2\left(0,T; H^1(\Omega;\R^{2\times2}_{\mathrm{S}})\right)
    \cap W^{1,\frac{4}{3}}\left(0,T; (H^1(\Omega;\R^{2\times2}_{\mathrm{S}}))'\right)
\end{align*}
with $\B(0) = \B_0 \text{ in } L^2(\Omega;\R^{2\times 2}_{\mathrm{S}}) \text{ and } \B \text{ positive definite a.e.~in } \Omega\times(0,T)$, such that, as $(h,\Delta t)\to (0,0)$,
\begin{alignat}{3}
    \mysubeq
    \label{eq:conv_dtv_L4/3_H1'}
    \partial_t \pmbv_h^{\Delta t} &\to \partial_t \pmbv \quad
    && \text{weakly} \quad &&\text{in } 
    L^{4/3}(0,T;\mathbf{V}'), 
    \\
    \mysubeq
    \label{eq:conv_v_strong}
    \pmbv_h^{\Delta t(,\pm)} &\to \pmbv \quad
    && \text{strongly} \quad &&\text{in } 
    L^2(0,T;L^r(\Omega;\R^2)),
    \\[2ex]
    \newsubeqblock
    \mysubeq
    \label{eq:conv_B_Linf_L2}
    \B_h^{\Delta t(,\pm)} &\to \B \quad
    && \text{weakly-$*$} \quad &&\text{in } 
    L^\infty \left(0,T;L^2(\Omega;\R^{2\times2}_{\mathrm{S}})\right), 
    \\
    \mysubeq
    \label{eq:conv_B_L2_H1}
    \B_h^{\Delta t(,\pm)} &\to \B \quad
    && \text{weakly} \quad &&\text{in } 
    L^2 \left(0,T;H^1(\Omega;\R^{2\times2}_{\mathrm{S}})\right), 
    \\
    \mysubeq
    \label{eq:conv_dtB_L4/3_H1'}
    \partial_t \B_h^{\Delta t} &\to \partial_t \B \quad
    && \text{weakly} \quad &&\text{in } 
    L^{4/3}\left(0,T;(H^1(\Omega;\R^{2\times2}_{\mathrm{S}}))'\right), 
    \\
    \mysubeq
    \label{eq:conv_B_strong}
    \B_h^{\Delta t(,\pm)} &\to \B \quad
    && \text{strongly} \quad &&\text{in } 
    L^2\left(0,T;L^p(\Omega;\R^{2\times2}_{\mathrm{S}})\right),
\end{alignat}
where $r,p\in [1,\infty)$, respectively.
\end{lemma}

\begin{proof}
In the first step, we prove the weak(-$*$) convergence results. As it is not clear if subsequences of $\phi_h^{\Delta t}$, $\phi_h^{\Delta t,+}$ and $\phi_h^{\Delta t,-}$ converge to the same limit function (and similarly for the other discrete functions), we note that it follows from \eqref{eq:bounds_time1} that
\begin{align}
\label{eq:convergence_1}
    \norm{\phi_h^{\Delta t} - \phi_h^{\Delta t,\pm} }_{L^2(0,T;H^1)}^2
    + \norm{\sigma_h^{\Delta t} - \sigma_h^{\Delta t, \pm}}_{L^2(0,T;L^2)}^2
    + \norm{\pmbv_h^{\Delta t} - \pmbv_h^{\Delta t, \pm}}_{L^2(0,T;L^2)}^2 
    &\leq C(T) \Delta t,
\end{align}
and, if additionally \ref{A6} and \eqref{eq:dt2} are satisfied, it follows from \eqref{eq:bounds_time2} that
\begin{align}
\label{eq:convergence_1b}
    \norm{\B_h^{\Delta t} - \B_h^{\Delta t, \pm}}_{L^2(0,T;L^2)}^2
    &\leq C(T,\alpha^{-1}) \Delta t.
\end{align}
Therefore, by the denseness of $\bigcup_{h>0} \calS_h$ in $L^2(\Omega)$ and on noting \eqref{eq:div_v_FE_time}, \eqref{eq:bounds_time1}, \eqref{eq:bounds_time2} and \eqref{eq:convergence_1}, \eqref{eq:convergence_1b}, we can choose a (non-relabeled) subsequence of $\left\{  \phi_h^{\Delta t(,\pm)}, 
\mu_h^{\Delta t,+}, 
\sigma_h^{\Delta t(,\pm)}, 
p_h^{\Delta t,+}, 
\pmbv_h^{\Delta t(,\pm)}, 
\B_h^{\Delta t(,\pm)} \right\}_{h,\Delta t>0}$ such that there exist limit functions 
such that the convergence results \eqref{eq:conv_phi_Linf_H1}, \eqref{eq:conv_dtphi_L2_H1'}, \eqref{eq:conv_mu_L2_H1}, \eqref{eq:conv_sigma_Linf_L2}, \eqref{eq:conv_sigma_L2_H1}, \eqref{eq:conv_dtsigma_L4/3_H1'}, 
\eqref{eq:conv_v_Linf_L2}, \eqref{eq:conv_v_L2_H1}, \eqref{eq:conv_dtv_L4/3_H1'}, \eqref{eq:conv_B_Linf_L2}, \eqref{eq:conv_B_L2_H1} and \eqref{eq:conv_dtB_L4/3_H1'} hold under the respective assumptions.

It follows from \eqref{eq:init_phi}, \eqref{eq:bounds_time1} and \eqref{eq:norm_equiv} that
\begin{align}
\label{eq:convergence_2}
    \norm{\Delta_h \phi_h^{\Delta t(,\pm)}}_{L^2(0,T;L^2)} 
    \leq C(T).
\end{align}
Then, the result \eqref{eq:conv_phi_L2_H2} follows analogously to \cite[Lem.~3.1]{barrett_nurnberg_styles_2004}, where the main argument combines \eqref{eq:convergence_1} and \eqref{eq:discr_laplace} to show that $\Delta_h \phi_h^{\Delta t}$, $\Delta_h \phi_h^{\Delta t,+}$ and $\Delta_h \phi_h^{\Delta t,-}$ have the same limit function. Together with elliptic regularity, as the domain $\Omega$ is convex and polygonal, we obtain in addition that $\phi\in L^2(0,T;H^2)$.
Moreover, on extracting a further subsequence, it follows from \eqref{eq:conv_phi_L2_H2} and \eqref{eq:discr_laplace_bound} that \eqref{eq:conv_phi_L2_W1s} holds.

In the next step, we prove the strong convergence results where we apply the compactness results \eqref{eq:compact_Lp}--\eqref{eq:compact_perturbation} which were stated in the beginning of the section.
We deduce from \eqref{eq:compact_translation}, \eqref{eq:bounds_time_phi_translation} and \eqref{eq:conv_phi_L2_W1s} that \eqref{eq:conv_phi_strong} holds, 
as the embedding $W^{1,s}(\Omega) \hookrightarrow \hookrightarrow C^{0,\gamma}(\overline\Omega)$ is compact and $C^{0,\gamma}(\overline\Omega) \hookrightarrow L^2(\Omega)$ is continuous, where $s\in\left[2,\frac{2d}{d-2}\right)$ and $\gamma\in\left(0,\frac{4-d}{2}\right)$. 

The strong convergence result \eqref{eq:conv_sigma_strong} for a subsequence of $\sigma_h^{\Delta t}$ holds on noting \eqref{eq:conv_dtsigma_L4/3_H1'}, \eqref{eq:conv_sigma_L2_H1} and \eqref{eq:compact_Lp}, as the embedding $H^1(\Omega) \hookrightarrow \hookrightarrow L^q(\Omega)$ is compact for $q\in\left[1,\frac{2d}{d-2}\right)$. Combining this with \eqref{eq:convergence_1}, \eqref{eq:bounds_time1} and a Gagliardo--Nirenberg inequality
yields the result \eqref{eq:conv_sigma_strong} for a subsequence of $\sigma_h^{\Delta t,\pm}$.
The convergence result \eqref{eq:conv_B_strong} for $\B$ follows with similar arguments under the respective assumptions.

To prove \eqref{eq:conv_v_strong}, we use a similar strategy as in \cite{metzger_2018}. By a triangle inequality and with $r\in[1,\infty)$, it holds 
\begin{align}
\begin{split}
\label{eq:convergence_3}
    \norm{\pmbv_h^{\Delta t(,\pm)} - \pmbv}_{L^2(0,T;L^r)}
    &\leq 
    \norm{\pmbv_h^{\Delta t(,\pm)} - \pmbv_h^{\Delta t,+}}_{L^2(0,T;L^r)}
    + \norm{\pmbv_h^{\Delta t,+} - \calR_h \pmbv_h^{\Delta t,+}}_{L^2(0,T;L^r)}
    \\
    &\quad+ \norm{\calR_h\pmbv_h^{\Delta t,+} - \pmbv}_{L^2(0,T;L^r)}.
\end{split}
\end{align}
The first two terms in \eqref{eq:convergence_3} converge to zero in the limit $(h,\Delta t)\to (0,0)$ on noting \eqref{eq:convergence_1}, \eqref{eq:projector_Stokes_L2}, \eqref{eq:bounds_time1} and a Gagliardo--Nirenberg inequality.
The third term in \eqref{eq:convergence_3} vanishes due to \eqref{eq:compact_perturbation}, \eqref{eq:bounds_time_v_translation}, \eqref{eq:bounds_time1}, \eqref{eq:bounds_time2} and a Gagliardo--Nirenberg inequality.
Hence, we have proved \eqref{eq:conv_phi_Linf_H1}--\eqref{eq:conv_B_strong}.



Next, we show that the initial conditions are satisfied.
We have from \eqref{eq:conv_phi} and \eqref{eq:compact_C} that $\phi: [0,T] \to L^2(\Omega)$ is weakly continuous, which together with \eqref{eq:init_phi_conv} yields $\phi(0) = \phi_0$ in the required sense.
Similarly, on noting \eqref{eq:conv_v}, \eqref{eq:compact_C} and as $\mathbf{H} \hookrightarrow \mathbf{V}'$ is a continuous injection, it follows from \cite[Chap.~3, Lem.~1.4]{temam_2001} that $\pmbv: [0,T] \to \mathbf{H}$ is weakly continuous. Hence, on noting \eqref{eq:init_v_conv}, $\pmbv(0) = \pmbv_0$ holds in the stated sense.
The results for $\sigma$ and $\B$ follow analogously.

Finally, it remains to prove the positive definiteness of the left Cauchy--Green tensor $\B$. Since $\B_h^{\Delta t(,\pm)} \in L^2(0,T; \calW_{h,\text{PD}})$, it follows from \eqref{eq:conv_B_strong} that $\B$ is symmetric and positive semi-definite a.e. in $\Omega \times (0,T)$. Together with a contradiction argument, i.e.~\cite[eq.~(6.53)]{barrett_boyaval_2009}, it follows that $\B$ is positive definite a.e.~in $\Omega \times (0,T)$. This proves the lemma.
\end{proof}


For the main result, we recall the following technical result taken from \cite[Lem.~5.3]{barrett_boyaval_2009}.

\begin{lemma}
For all $K_k\in\calT_h$, and for all $\C_h\in \calW_{h,\mathrm{PD}}$, it holds 
\begin{align}
\label{eq:error_Lambda}
    \max\limits_{i,j=1,...,d} 
    \int_{K_k} \abs{ \Lambda_{i,j}(\C_h) - \C_h \delta_{ij}  }^2 \dx
    \leq C h^2 \int_{K_k} \abs{\nabla\C_h}^2 \dx.
\end{align}
\end{lemma}

We also recall the following result from \cite[Lem.~6.8]{barrett_2018_fene-p}.
\begin{lemma}
Let $g\in C^{0,1}(\R)$ with Lipschitz constant $L_g$. For all $K_k \in \calT_h$ and for all $q_h\in \calS_h$, $\C_h\in \calW_h$, it holds 
\begin{align}
\label{eq:interp_Lipschitz}
\begin{split}
    \int_{K_k} \abs{\calI_h\big[ g(q_h) \big] - g(q_h) }^2 \dx
    &\leq C L_g^2 h^2 \int_{K_k} \abs{\nabla q_h}^2 \dx,
    \\
    \int_{K_k} \abs{\calI_h\big[ g(\C_h) \big] - g(\C_h) }^2 \dx
    &\leq C L_g^2 h^2 \int_{K_k} \abs{\nabla \C_h}^2 \dx.
\end{split}
\end{align}
\end{lemma}


We now pass to the limit $(h,\Delta t) \to (0,0)$ in \ref{P_alpha_FE} and show that the function tuple $( \phi, \mu, \sigma, \pmbv, \B )$ from Lemma \ref{lemma:convergence} forms a weak solution to \ref{P_alpha} in the sense of Definition \ref{def:weak_solution}, which finally proves Theorem \ref{theorem:weak_solution}. 
For the reader's convenience, we state this result in the following theorem.

Note that in comparison to Lemma \ref{lemma:convergence} we need the additional assumption \ref{A7} on the source terms $\Gamma_\phi,\Gamma_\sigma$, as otherwise the limit passing would not be possible in presence of $\calI_h$. However, this assumption is not strict in practice and can be dropped if, e.g., mass lumping is not considered.

\begin{theorem}[Limit passing in \ref{P_alpha_FE}]
\label{theorem:convergence}
Let \ref{A1}--\ref{A7} hold true. Suppose that the discrete initial and boundary data satisfy \eqref{eq:init_bounds} and \eqref{eq:init_conv}, and that the CFL constraint \eqref{eq:dt2} holds.
Then the functions $\phi,\mu,\sigma, \pmbv, \B$ from Lemma \ref{lemma:convergence} form a weak solution to \ref{P_alpha} in the sense of Definition \ref{def:weak_solution}.
\end{theorem}

%
%
\begin{proof}

Let $\zeta_h = \calI_h \zeta$, where $\zeta \in C_0^\infty\big(0,T;C^\infty(\overline\Omega)\big)$. The limit passing in the linear terms in \eqref{eq:phi_FE_time} can be established with a straightforward calculation from \cite[Thm.~6.3]{trautwein_2021}, which is based on \eqref{eq:lump_Sh_Sh}, \eqref{eq:interp_H2}, \eqref{eq:bounds_time1} and the convergence properties from Lemma \ref{lemma:convergence}.
%
%
%

We now pass to the limit in the nonlinear terms in \eqref{eq:phi_FE_time}.
On noting the continuity of $m(\cdot)$, \eqref{eq:conv_phi_strong}, \eqref{eq:interp_H2}, \eqref{eq:interp_continuous} and applying the generalised Lebesgue dominated convergence theorem \cite[Chap.~3]{alt_2016}, we have
\begin{align*}
    \nnorm{\calI_h\big[m(\phi_h^{\Delta t,-})\big]  \nabla\zeta_h - m(\phi) \nabla\zeta}_{L^2(0,T;L^2)} \to 0,
\end{align*}
as $(h,\Delta t)\to (0,0)$. Then, together with \eqref{eq:conv_mu_L2_H1}, we obtain by the product of weak-strong convergence \cite[Chap.~8]{alt_2016}, that
\begin{align*}
    \int_0^T \int_\Omega \calI_h\big[m(\phi_h^{\Delta t,-})\big]  \nabla\zeta_h \cdot \nabla\mu_h^{\Delta t,+} \dx\dt
    \to \int_0^T \int_\Omega m(\phi)  \nabla\zeta \cdot \nabla\mu \dx\dt,
\end{align*}
as $(h,\Delta t)\to (0,0)$.

In order to pass to the limit in the source term in \eqref{eq:phi_FE_time}, we proceed as follows:
\begin{align*}
    &\abs{\int_0^T \skp{\Gamma_\phi(\phi_h^{\Delta t,+},\sigma_h^{\Delta t,+}, \B_h^{\Delta t,+})}{\zeta_h}_h
    - \skp{\Gamma_\phi(\phi, \sigma, \B)}{\zeta}_{L^2} \dt}
    \\
    &\leq
    \abs{\int_0^T \skp{\calI_h\big[\Gamma_\phi(\phi_h^{\Delta t,+}, \sigma_h^{\Delta t,+}, \B_h^{\Delta t,+})\big]}{\zeta_h}_h
    - \skp{\calI_h\big[\Gamma_\phi(\phi_h^{\Delta t,+}, \sigma_h^{\Delta t,+}, \B_h^{\Delta t,+})\big]}{\zeta_h}_{L^2} \dt}
    \\
    &\quad
    + \abs{\int_0^T \skp{\calI_h\big[\Gamma_\phi(\phi_h^{\Delta t,+}, \sigma_h^{\Delta t,+}, \B_h^{\Delta t,+})\big] - \Gamma_\phi(\phi_h^{\Delta t,+}, \sigma_h^{\Delta t,+}, \B_h^{\Delta t,+})}{\zeta_h}_{L^2} \dt}
    \\
    &\quad
    + \abs{\int_0^T \skp{\Gamma_\phi(\phi_h^{\Delta t,+}, \sigma_h^{\Delta t,+}, \B_h^{\Delta t,+}) 
    - \Gamma_\phi(\phi, \sigma, \B) }{\zeta_h}_{L^2}\dt}
    + \abs{\int_0^T \skp{\Gamma_\phi(\phi, \sigma, \B) }{\zeta_h - \zeta}_{L^2}\dt}
    \\
    &\eqqcolon Ia + Ib + Ic + Id.
\end{align*}
On noting H{\"o}lder's inequality, \eqref{eq:lump_Sh_Sh}, \eqref{eq:inverse_estimate}, \eqref{eq:interp_H2}, \eqref{eq:norm_equiv} and \ref{A2}, it holds
\begin{align}
\begin{split}
\label{eq:theorem_convergence_Ia}
    Ia &\leq 
    C h \nnorm{\calI_h\big[\Gamma_\phi(\phi_h^{\Delta t,+}, \sigma_h^{\Delta t,+}, \B_h^{\Delta t,+}) \big]}_{L^2(0,T;L^2)} \norm{\zeta_h}_{L^2(0,T;H^1)}
    \\
    &\leq 
    C h \Big( 1 
    + \norm{ \phi_h^{\Delta t,+} }_{L^2(0,T;L^2)} 
    + \norm{ \sigma_h^{\Delta t,+} }_{L^2(0,T;L^2)} \Big)
    \norm{\zeta}_{L^2(0,T;H^2)}.
\end{split}
\end{align}
Let us remark that an analogue of \eqref{eq:interp_Lipschitz} also holds for $\Gamma_\phi$ instead of $g(\cdot)$.
Then, we receive on noting H{\"o}lder's inequality, \eqref{eq:interp_Lipschitz}, \ref{A7} and \eqref{eq:interp_H2} that
\begin{align}
\begin{split}
\label{eq:theorem_convergence_Ib}
    Ib &\leq 
    C \nnorm{\calI_h\big[\Gamma_\phi(\phi_h^{\Delta t,+}, \sigma_h^{\Delta t,+}, \B_h^{\Delta t,+})\big] - \Gamma_\phi(\phi_h^{\Delta t,+}, \sigma_h^{\Delta t,+}, \B_h^{\Delta t,+})}_{L^2(0,T;L^2)} \norm{\zeta_h}_{L^2(0,T;L^2)}
    \\
    &\leq 
    C h \left( 
    \norm{ \nabla\phi_h^{\Delta t,+} }_{L^2(0,T;L^2)} 
    + \norm{ \nabla\sigma_h^{\Delta t,+} }_{L^2(0,T;L^2)} 
    + \norm{ \nabla\B_h^{\Delta t,+} }_{L^2(0,T;L^2)} \right)
    \norm{\zeta}_{L^2(0,T;H^2)}.
\end{split}
\end{align}
Hence, on noting \eqref{eq:bounds_time1} we obtain that $Ia, Ib\to 0$, as $(h,\Delta t)\to (0,0)$. 
Moreover, by \ref{A2}, \eqref{eq:conv_phi_strong}, \eqref{eq:conv_sigma_strong}, \eqref{eq:interp_H2} and the generalized Lebesgue dominated convergence theorem \cite[Chap.~3]{alt_2016}, we obtain that $Ic\to 0$, as $(h,\Delta t)\to (0,0)$. Further, it holds that $Id\to 0$, as $(h,\Delta t) \to (0,0)$ by noting \eqref{eq:interp_H2}, \eqref{eq:conv_phi}, \eqref{eq:conv_sigma} and \ref{A2}.
This leads to
\begin{align*}
    &\abs{\int_0^T \skp{\Gamma_\phi(\phi_h^{\Delta t,+},\sigma_h^{\Delta t,+}, \B_h^{\Delta t,+})}{\zeta_h}_h
    - \skp{\Gamma_\phi(\phi, \sigma, \B)}{\zeta}_{L^2} \dt}
    \to 0,
\end{align*}
as $(h,\Delta t)\to (0,0)$.

It remains to pass to the limit in the convective term in \eqref{eq:phi_FE_time}.
It holds
\begin{align*}
    &\abs{\int_0^T \int_\Omega \phi_h^{\Delta t,-}  \pmbv_h^{\Delta t,+} \cdot \nabla \zeta_h
    - \phi \pmbv\cdot \nabla\zeta \dx\dt }
    \\
    &\leq 
    \abs{\int_0^T \int_\Omega 
    \big(\phi_h^{\Delta t,-} - \phi \big) 
    \pmbv_h^{\Delta t,+} \cdot \nabla\zeta_h 
    \dx\dt }
    + \abs{\int_0^T \int_\Omega \phi \big(\pmbv_h^{\Delta t,+} - \pmbv\big) \cdot \nabla\zeta_h \dx\dt }
    \\
    &\quad
    + \abs{\int_0^T \int_\Omega \phi \pmbv \cdot \big( \nabla\zeta_h - \nabla\zeta \big) \dx\dt }
    \\
    &\leq \norm{\phi_h^{\Delta t,-} - \phi}_{L^2(0,T;L^4)}
    \norm{\pmbv_h^{\Delta t,+}}_{L^2(0,T;L^4)}
    \norm{\zeta_h}_{L^\infty(0,T;H^1)}
    \\
    &\quad 
    + \norm{\phi}_{L^2(0,T;L^4)} 
    \norm{\pmbv_h^{\Delta t,+} - \pmbv}_{L^2(0,T;L^4)}
    \norm{\zeta_h}_{L^\infty(0,T;H^1)}
    \\
    &\quad
    + \norm{\phi}_{L^2(0,T;L^4)} 
    \norm{\pmbv}_{L^2(0,T;L^4)}
    \norm{\zeta_h - \zeta}_{L^\infty(0,T;H^1)},
\end{align*}
where the terms on the right-hand vanish in the limit $(h,\Delta t)\to (0,0)$ due to Sobolev embeddings, \eqref{eq:bounds_time1}, \eqref{eq:conv_phi_strong}, \eqref{eq:conv_v_strong}, \eqref{eq:conv_phi}, \eqref{eq:conv_v} and \eqref{eq:interp_H2}.

Hence, by the denseness of $C_0^\infty(0,T;C^\infty(\overline\Omega))$ in $L^2(0,T;H^1)$, we finally obtain that \eqref{eq:phi_weak} is fulfilled in the stated sense.

The passage to the limit in \eqref{eq:mu_FE_time}, \eqref{eq:sigma_FE_time}, \eqref{eq:v_FE_time} and \eqref{eq:B_FE_time} can be established in a similar way. For example, the passage to the limit in the boundary terms in \eqref{eq:sigma_FE_time} follows with \eqref{eq:def_bc}, \eqref{eq:bc_conv}, \eqref{eq:lump_Gamma_Sh_Sh}, \eqref{eq:bounds_time1}, \eqref{eq:interp_H2} and \eqref{eq:conv_sigma_L2_H1}.
Besides, the nonlinear terms containing $\psi_1'(\cdot)$, $\psi_2'(\cdot)$ can be treated similarly to the $m(\cdot)$-term in \eqref{eq:phi_FE_time}, and the $\Gamma_\sigma$-term analogously to the $\Gamma_\phi$-term, respectively. For the term containing $\Lambda_{i,j}(\cdot)$, the limit passing is based on \eqref{eq:bounds_time2}, \eqref{eq:interp_H2}, \eqref{eq:conv_v_strong}, \eqref{eq:conv_B_strong} and the error estimate \eqref{eq:error_Lambda}.

For the reader's convenience, we show the passage to the limit in the second and the third term of \eqref{eq:B_FE_time}. Let us define $\C_h = \calI_h \C$, where $\C \in C_0^\infty\big(0,T;C^\infty(\overline\Omega;\R^{2\times 2}_{\text{S}})\big)$.
It holds
\begin{align*}
    &\abs{ \int_0^T \int_\Omega \calI_h \bigg[ 
    \frac{1}{\tau(\phi_h^{\Delta t,-})} \B_h^{\Delta t,+} : \C_h \bigg] 
    - \frac{1}{\tau(\phi)} \B:\C \dx \dt }
    \\
    &\leq 
    \abs{ \int_0^T \int_\Omega
    \big(\calI_h-\calI\big) \bigg[ 
    \B_h^{\Delta t,+} : 
    \calI_h \bigg[ 
    \frac{1}{\tau(\phi_h^{\Delta t,-})} \C_h \bigg]
    \bigg] \dx \dt }
    \\
    &\quad
    + \abs{ \int_0^T \int_\Omega  
    \B_h^{\Delta t,+} : 
    \calI_h \bigg[
    \frac{1}{\tau(\phi_h^{\Delta t,-})} \C_h \bigg]
    - \frac{1}{\tau(\phi)} \B:\C \dx \dt }
    \eqqcolon IIa + IIb,
\end{align*}
where $\calI$ denotes the identity.
For the first term, it holds with \eqref{eq:lump_Sh_Sh}, \eqref{eq:inverse_estimate}, \ref{A3}, H{\"o}lder's inequality, \eqref{eq:interp_estimate} and \eqref{eq:norm_equiv} that
\begin{align*}
    IIa
    &\leq 
    C h \int_0^T
    \norm{\B_h^{\Delta t,+}}_{H^1}
    \nnorm{\calI_h \bigg[ 
    \frac{1}{\tau(\phi_h^{\Delta t,-})} \C_h \bigg]}_{L^2} \dt
    \leq 
    C h \norm{\B_h^{\Delta t,+}}_{L^2(0,T;H^1)} \norm{\C_h}_{L^2(0,T;L^2)},
\end{align*}
which yields, on noting \eqref{eq:bounds_time2} and \eqref{eq:interp_H2}, that $IIa$ vanishes as $(h,\Delta t)\to (0,0)$. 
Moreover, it holds by the continuity of $\tau(\cdot)$, \eqref{eq:conv_phi_strong}, \eqref{eq:interp_continuous}, a triangle inequality and as $\calI_h \C_h = \calI_h\C$, that
\begin{align*}
    &\nnorm{ \calI_h\left[ \frac{1}{\tau(\phi_h^{\Delta t,-})} \C_h \right]
    - \frac{1}{\tau(\phi)} \C}_{L^2(0,T;L^2)}
    = \nnorm{ \calI_h\left[ \frac{1}{\tau(\phi_h^{\Delta t,-})} \C \right]
    - \frac{1}{\tau(\phi)} \C}_{L^2(0,T;L^2)}
    \\
    &\quad \leq 
    \nnorm{ \calI_h\left[ \left( \frac{1}{\tau(\phi_h^{\Delta t,-})}  
    - \frac{1}{\tau(\phi)} \right) \C \right]}_{L^2(0,T;L^2)}
    + \nnorm{ \calI_h\left[ \frac{1}{\tau(\phi)} \C \right]
    - \frac{1}{\tau(\phi)} \C}_{L^2(0,T;L^2)}
    \\
    &\quad \leq 
    C \nnorm{ \left(\frac{1}{\tau(\phi_h^{\Delta t,-})}
    - \frac{1}{\tau(\phi)} \right) \C }_{L^2(0,T;L^\infty)}
    + C \nnorm{ \calI_h\left[ \frac{1}{\tau(\phi)} \C \right]
    - \frac{1}{\tau(\phi)} \C}_{L^2(0,T;L^\infty)}
    \\
    &\quad \to 0,
\end{align*}
as $(h,\Delta t)\to (0,0)$.
We finally obtain $IIb\to 0$, as $(h,\Delta t)\to (0,0)$, by using \eqref{eq:conv_B_strong} and the product of weak-strong convergence \cite[Chap.~8]{alt_2016}.

Moreover, we have
\begin{align*}
    &\abs{ \int_0^T \int_\Omega 
    \nabla\pmbv_h^{\Delta t,+} : \calI_h\big[ \C_h \B_h^{\Delta t,+} \big] 
    - \nabla\pmbv : (\C\B) \dx \dt }
    \\
    &\leq 
    \abs{ \int_0^T \int_\Omega 
    \nabla\pmbv_h^{\Delta t,+} : (\calI_h - \calI) \big[ \C_h \B_h^{\Delta t,+} \big] \dx\dt }
    +\abs{ \int_0^T \int_\Omega 
    \nabla\pmbv_h^{\Delta t,+} : (\C_h \B_h^{\Delta t,+} ) 
    - \nabla\pmbv : (\C\B) \dx \dt }
    \\
    &\eqqcolon IIIa + IIIb.
\end{align*}
Similarly to \eqref{eq:interp_Lipschitz}, it holds for all $k\in\{1,...,N_k\}$ and $K_k\in\calT_h$, that
\begin{align}
    \nnorm{ (\calI_h - \calI) 
    \big[ \C_h \B_h^{\Delta t,+}\big] }_{L^2(K_k)}^2 
    &\leq C h^2 
    \Big(
    \norm{\C_h}_{L^\infty(K_k)}^2 
    \norm{\nabla \B_h^{\Delta t,+}}_{L^2(K_k)}^2
    + \norm{\B_h^{\Delta t,+}}_{L^\infty(K_k)}^2
    \norm{\nabla \C_h}_{L^2(K_k)}^2 \Big).
\end{align}
It follows with \eqref{eq:inverse_estimate} and the Sobolev embedding $H^1(\Omega) \hookrightarrow L^s(\Omega)$, where $s\in(2, \infty)$ for $d=2$, that
\begin{align}
\begin{split}
    \nnorm{ (\calI_h - \calI) 
    \big[ \C_h \B_h^{\Delta t,+}\big] }_{L^2(K_k)}^2 
    &\leq C h^{2-\frac{4}{s}}
    \Big(
    \norm{\C_h}_{L^s(K_k)}^2 
    \norm{\nabla \B_h^{\Delta t,+}}_{L^2(K_k)}^2
    + \norm{\B_h^{\Delta t,+}}_{L^s(K_k)}^2
    \norm{\nabla \C_h}_{L^2(K_k)}^2 \Big)
    \\
    &\leq C h^{2-\frac{4}{s}}
    \norm{\C_h}_{H^1(K_k)}^2 
    \norm{\B_h^{\Delta t,+}}_{H^1(K_k)}^2.
\end{split}
\end{align}
Summing over all $k=1,...,N_k$ and integrating with respect to the time variable yield
%
%
\begin{align}
\label{eq:convergence_product}
    \nnorm{ (\calI_h - \calI) 
    \big[ \C_h\B_h^{\Delta t,+}\big] }_{L^2(0,T;L^2)}^2 
    \leq C h^{2-\frac{4}{s}} 
    \norm{\C_h}_{L^\infty(0,T;H^1)}^2
    \norm{\B_h^{\Delta t,+}}_{L^2(0,T;H^1)}^2.
\end{align}
Then, with H{\"o}lder's inequality, \eqref{eq:convergence_product}, \eqref{eq:bounds_time2} and \eqref{eq:interp_H2}, we obtain that $IIIa \to 0$, as $(h,\Delta t)\to (0,0)$. 
Furthermore, on noting \eqref{eq:interp_H2} and \eqref{eq:conv_B_strong} we obtain that $\C_h\B_h^{\Delta t,+} \to \C\B$ strongly in $L^2(0,T;L^2(\Omega;\R^{2\times 2}))$, as $(h,\Delta t)\to (0,0)$. Hence, by the product of weak-strong convergence \cite[Chap.~8]{alt_2016} and \eqref{eq:conv_v_L2_H1}, we receive that $IIIb \to 0$, as $(h,\Delta t)\to (0,0)$. 

By the denseness of $C_0^\infty(0,T;C^\infty(\overline\Omega;\R^{2\times 2}_{\text{S}}))$ in $L^4(0,T;H^1(\Omega;\R^{2\times 2}_{\text{S}}))$, we finally obtain that \eqref{eq:B_weak} is fulfilled in the stated sense. This proves the theorem.
\end{proof}






\section{Numerical results}
\label{sec:numeric}

In this section, we present numerical results for the scheme \ref{P_alpha_FE} that was analyzed in Section \ref{sec:fem}.

\subsection{Computational aspects}
\subsubsection{Description of the solution algorithm}
Before presenting the numerical results, we first discuss the solution strategy.
On one side, one could think of applying Newton's method to solve the nonlinear system of equations \ref{P_alpha_FE} as it can provide good error reduction rates.
However, Newton's method would be too expensive and would require too much memory, as the coupled system of equations \ref{P_alpha_FE} is very large, and is hence not useful in practice. 
On the other side, making use of a fixed point iteration allows to decouple the system of equation \ref{P_alpha_FE} into linear subsystems \eqref{eq:phi_FE}--\eqref{eq:mu_FE}, \eqref{eq:sigma_FE}, \eqref{eq:div_v_FE}--\eqref{eq:v_FE} and \eqref{eq:B_FE} that can be solved separately. But here, a very small time step size $\Delta t$ is required such that the fixed point iteration can converge. Therefore, a fixed point iteration would need too much computing time in practice.
%
%
Moreover, numerical experiments indicate that the subsystem \eqref{eq:phi_FE}--\eqref{eq:mu_FE} requires additional consideration and the most precision which is because of the scaling with $B = \beta\epsilon$ and $A= \frac{\beta}{\epsilon}$ with $\epsilon>0$ very small. 

For these reasons, we propose an inner-outer type algorithm to solve the nonlinear coupled scheme \ref{P_alpha_FE}. 
For the outer iteration, we apply a fixed point-like strategy, where \ref{P_alpha_FE} is decoupled into the subsystems \eqref{eq:phi_FE}--\eqref{eq:mu_FE}, \eqref{eq:sigma_FE}, \eqref{eq:div_v_FE}--\eqref{eq:v_FE} and \eqref{eq:B_FE}, where all nonlinear terms are treated explicitly except of $A \psi_1'(\phi_h^n)$ in \eqref{eq:mu_FE} which is treated implicitly. 
Hence, we first solve the nonlinear subsystem \eqref{eq:phi_FE}--\eqref{eq:mu_FE} with Newton's method, where the resulting linear systems are solved with a preconditioned BICGSTAB-method. After that, we solve the linear subsystems \eqref{eq:sigma_FE}, \eqref{eq:div_v_FE}--\eqref{eq:v_FE} and \eqref{eq:B_FE} separately with an AMG-preconditioned MINRES-solver, an AMG-preconditioned GMRES-method and an AMG-preconditioned MINRES-solver, respectively. The algorithm is implemented with the finite element toolbox FEniCS \cite{fenics_book_2012} which also provides the iterative linear solvers and the preconditioners.
%
%

However, due to limited computational possibilities with the finite element toolbox FEniCS, we consider \eqref{eq:B_FE} with $\sum_{i,j=1}^d   [\pmbv_h^{n-1}]_i \Lambda_{i,j}(\B_h^n) : \partial_{x_j} \C_h$ replaced by $\B_h^{n} : ((\pmbv_h^{n-1}\cdot\nabla) \C_h)$, which however is a good approximation due to \eqref{eq:error_Lambda}, and, we replace $\nabla\pmbv_h^n : \calI_h\big[\C_h\B_h^n\big]$ by $\nabla\pmbv_h^n : (\C_h\B_h^n)$. 
To increase the accuracy of our numerical solutions, we make use of a mesh refinement strategy, similarly to \cite{trautwein_2021}, where the mesh is locally refined near the interface where $\abs{\phi_h^{n-1}} \leq 1-\delta$, for a small $\delta>0$, where the local mesh size corresponds to a uniform $N_f \times N_f$ grid. Away from the interface, where $\abs{\phi_h^{n-1}} > 1-\delta$, a coarse mesh is used with a local mesh size corresponding to a uniform $N_c \times N_c$ grid.

\subsubsection{Specification of the parameters, model functions and initial data}

Now we specify the parameters, model functions and initial data, where our choices are motivated by, e.g., \cite{ebenbeck_garcke_nurnberg_2020, GarckeLSS_2016}.
We perform the calculations on the domain $\Omega = (-5,5)^2 \subset \R^2$ and we use the model functions
\begin{alignat*}{3}
    &\Gamma_\phi(\phi,\sigma,\B) = h(1.1\phi) \big( \calP \sigma f(\B) -\calA \big) ,
    \qquad &&\Gamma_\sigma(\phi,\sigma) = \calC h(\phi) \sigma,
    \qquad &&m(\phi) = 2 \left( h(\phi) \right)^2 + m_0,
    \\
    &\eta(\phi) = \eta_{-1} h(-\phi) + \eta_{1} h(\phi),
    \qquad
    &&\tau(\phi) =  \tau_{-1} h(-\phi) + \tau_{1} h(\phi),
    \qquad
    &&n(\phi) = n_0,
\end{alignat*}
where the interpolation function with cut-offs $h(\cdot)$ is defined in \eqref{eq:h_phi} and $f(\cdot)$ is defined in \eqref{eq:f_B}. 
Unless otherwise stated, we choose the parameters
{\small
\begin{equation}
\begin{aligned}
\label{eq:num_parameters}
    &  N_c = 32,\quad
    && N_f = 1024,\quad
    &&\delta = 0.075,\quad
    &&\Delta t= 5 \cdot 10^{-4}, \quad
    &&\epsilon=0.01,\quad
    &&\beta=0.1,\quad
    \\
    & \calP = 2, \quad              
    &&\calA = 0, \quad
    &&\calC = 10, \quad
    &&\chi_\sigma = 500, \quad
    &&\chi_\phi = 10, \quad
    &&\alpha=10^{-3},\quad
    \\
    & m_0 = 10^{-12}, \quad         
    && n_0 = 0.002,\quad
    && K = 1000,\quad
    &&\eta_{1}=5000,\quad
    &&\eta_{-1}=5000,\quad
    &&\kappa=10^{4},\quad           
    \\
    &\frac{\tau_{1}}{\kappa} = 1,\quad
    && \frac{\tau_{-1}}{\kappa} = 1.
\end{aligned}
\end{equation}
}
Therefore, $h_{\max} = 10\cdot 2^{-5} = 0.3125$ is the maximal diameter and $h_{\min} = 10 \cdot 2^{-10} \approx 0.009766$ is the minimal diameter of all triangular elements. 
Moreover, note that the assumptions \ref{A2}--\ref{A3} are fulfilled with these choices. 
Actually, the (modified) potential with quadratic growth from \eqref{eq:psi_modified} should be used such that the growth assumptions in \ref{A4} are fulfilled. 
However, for simplicity, we use the (unmodified) potential
\begin{align*}
    &\psi(\phi) = \tfrac{1}{4} (1-\phi^2)^2,
    \quad \text{where} \quad
    \psi'(\phi) = \psi_1'(\phi) + \psi_2'(\phi) = \phi^3 - \phi,
\end{align*}
as the order parameter $\phi$ always stays very close to the interval $[-1,1]$ in our numerical experiments. Also note the scaling with $1.1\phi$ in the source term $\Gamma_\phi$, as $\phi\approx{-1}$ in practice in the pure healthy phase and we want to exclude proliferation effects there.

For the initial tumour profile, we set $\phi_h^0 = \calI_h \phi_0$, where $\phi_0 \in  C(\overline\Omega)$ is a slightly perturbed sphere given by
\begin{align}
    \label{eq:num_phi0}
    \phi_0(x) &= 
    - \tanh \Big(\frac{r(x)}{\sqrt{2}\epsilon} \Big),
    \quad \text{where} \quad
    r(x) = \abs{x} - \frac{5}{12}(2 + 0.2\cos(2\theta)),
\end{align}
where $x = \abs{x} ( \cos(\theta),\sin(\theta) )^T \in \Omega$. 
We choose $\sigma_h^0 \in\calS_h$ as the solution of the quasi-static equation
\begin{align}
    \label{eq:num_sigma_FE_quasistatic}
    \int_\Omega  
    n_0 \big(\chi_\sigma \nabla\sigma_h^0 - \chi_\phi \nabla \phi_h^0 \big)   \cdot \nabla \xi_h  
    + \calI_h \big[ \Gamma_\sigma(\phi_h^0,\sigma_h^0) \xi_h \big]\dx
    + \int_{\partial\Omega}  \calI_h\Big[ K(\sigma_h^0 - \sigma_\infty) \xi_h \Big] \dH^{d-1}
    = 0,
\end{align}
for all $\xi\in\calS_h$, where $\sigma_\infty \coloneqq 1$ in the numerical experiments unless otherwise stated.
Moreover, we assume no initial velocity and no initial elastic stresses.
More precisely, we start with 
\begin{align}
\label{eq:num_v0_B0}
    \pmbv_0^h(x) = (0,0)^T, \qquad \B_0^h(x) &= \text{diag}(1, 1), 
\end{align}
where $x\in\Omega$. The initial profile $\phi_h^0$, the initial nutrient $\sigma_h^0$ and the initial mesh are shown in Figure \ref{fig:init}.
It is easy to verify that the initial and boundary values satisfy the assumptions \eqref{eq:init_bounds}.

\begin{figure}[H]
    \centering
    \subfloat
	{\includegraphics[width=0.24\textwidth]{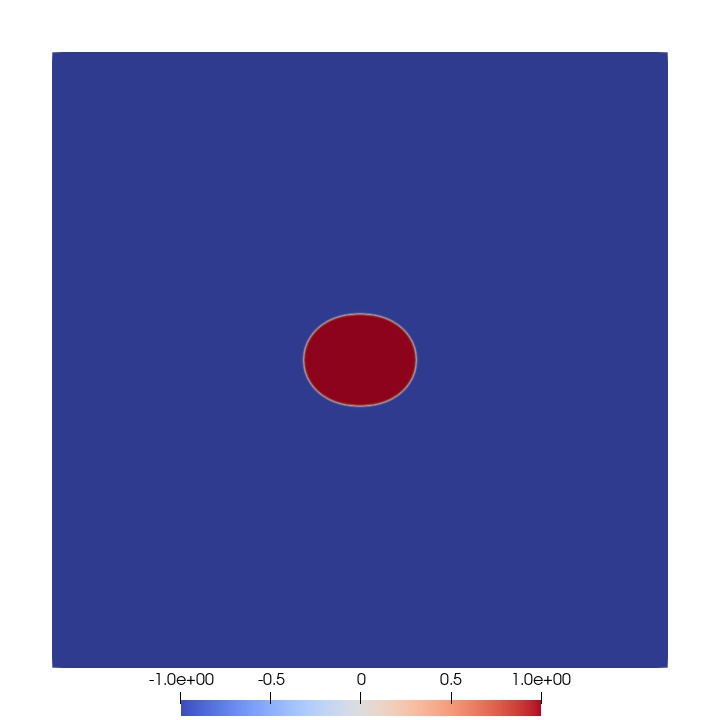}}
	\hspace{-0.8em}
	\subfloat
	{\includegraphics[width=0.24\textwidth]{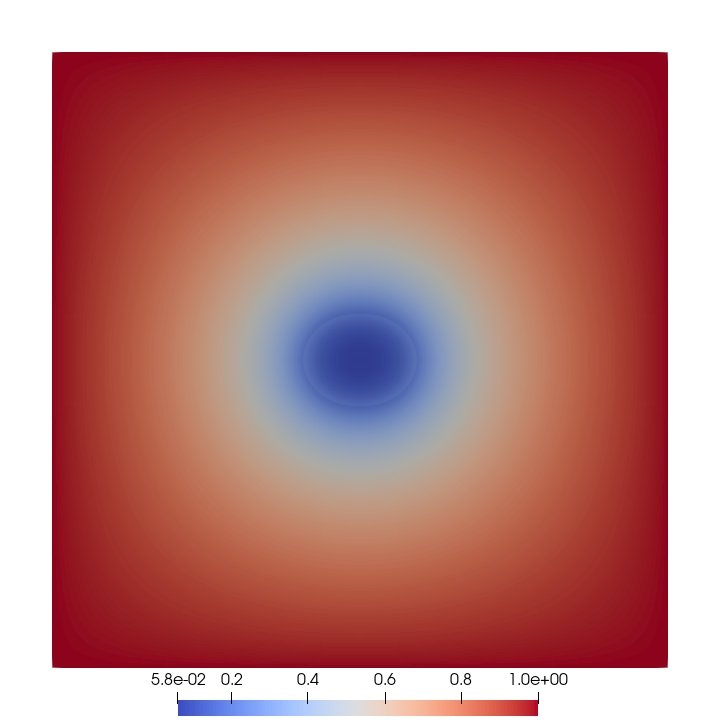}}
	\hspace{-0.8em}
	\subfloat
	{\includegraphics[width=0.24\textwidth]{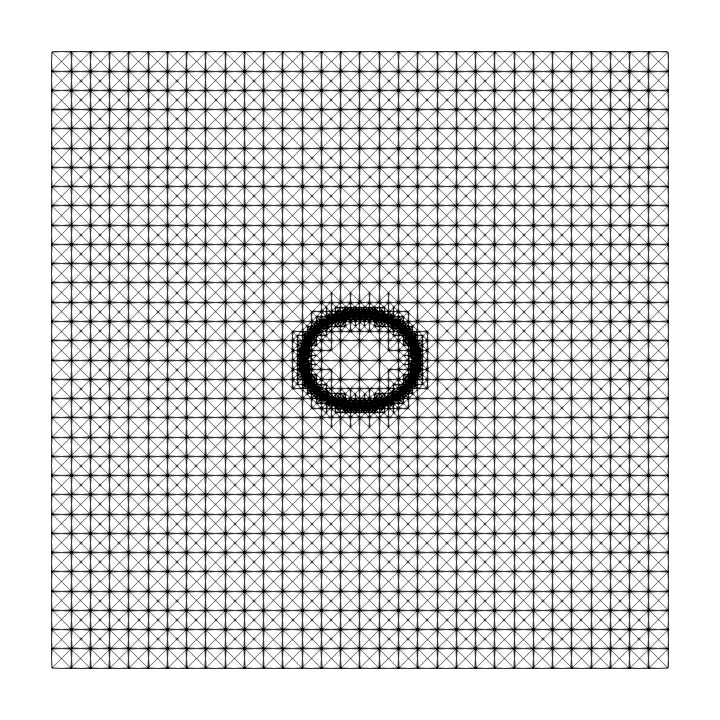}}
    \caption{Initial tumour (left), initial nutrient (center) and initial mesh (right).} 
    \label{fig:init}
\end{figure}

In the following, we will systematically interpret the influence of different parameters in our model.
The influence of the parameters $\epsilon, \beta, \calP, \calA, \calC, \chi_\phi, \chi_\sigma$ and the mobility $m(\cdot)$ in related models has been extensively studied, see, e.g., \cite{ebenbeck_garcke_nurnberg_2020, GarckeLSS_2016, trautwein_2021}, and we observed similar behaviour for our model. For that reason, we focus the presentation of the numerical tests on the effects arising from viscoelasticity.

The main difficulty is to find a good choice of parameters. To observe an unstable growth, i.e.~the development of fingers, the chemotaxis parameter $\chi_\phi$ has to be in the same scale as $\frac{\beta}{\epsilon}$. 
Choosing $\chi_\phi$ too large or $\beta$ too small reduces the forming of the pure phases $\phi=\pm1$. 
On the other side, we observe a jump of the nutrient $\sigma$ along the interface which is proportional to $\chi_\sigma^{-1} \chi_\phi$, hence $\chi_\sigma$ is chosen large compared to $\chi_\phi$. 
However, this can result in very large velocities if the viscosities $\eta_1, \eta_{-1}$ are not large enough.

\subsection{Comparison with the fully viscous model}
We now investigate the time evolution in the viscoelastic model and compare it to the fully viscous model where $\B=\I$. 
The parameters are chosen as in \eqref{eq:num_parameters} and the goal is to observe an unstable growth. 
In absence of initial elastic stresses, i.e.~$\B_0=\I$, any changes in the left Cauchy--Green tensor $\B$ are induced by the velocity field $\pmbv$, see \eqref{eq:B_FE}. As the viscosities $\eta_1, \eta_{-1}$ are chosen very large, we can expect small velocities and hence $\B\approx\I$, such that the elastic stress tensor is approximately zero, i.e.~$\T_{\mathrm{el}}(\B)=\kappa (\B-\I)\approx 0$.
Therefore, we expect that the qualitative behaviour of both models is very similar, which can be observed in Figure \ref{fig:2}. Here, we show the numerical solutions for both models at time $t=2$. In the first row from left to right, we show the order parameter $\phi$, the nutrient $\sigma$, the velocity magnitude $\abs\pmbv$ and the final mesh for the fully viscous model. In the second row from left to right, the order parameter $\phi$, the nutrient $\sigma$, the velocity magnitude $\abs\pmbv$ and the magnitude of the elastic stress tensor $\abs{\T_{\mathrm{el}}(\B)}$ of the viscoelastic model are visualized.
Indeed, the qualitative behaviour for both models is very similar, as $\abs{\T_{\mathrm{el}}(\B)} \approx 0$ is close to machine precision for the viscoelastic model. 
In both cases, the tumour has developed fingers showing towards regions with higher concentration of the nutrient which can be interpreted as the chemotaxis effect, i.e.~the cell movement in response to an extracellular chemical gradient. This behaviour has also been observed for other models \cite{ebenbeck_garcke_nurnberg_2020, GarckeLSS_2016, trautwein_2021}.
After that, in Figure \ref{fig:2b}, we show the time evolution of the tumour for the viscoelastic model at the times $t\in\{0.5, 1, 1.5, 2\}$.

\begin{figure}[H]
    \centering
    \subfloat
	{\includegraphics[width=0.22\textwidth]{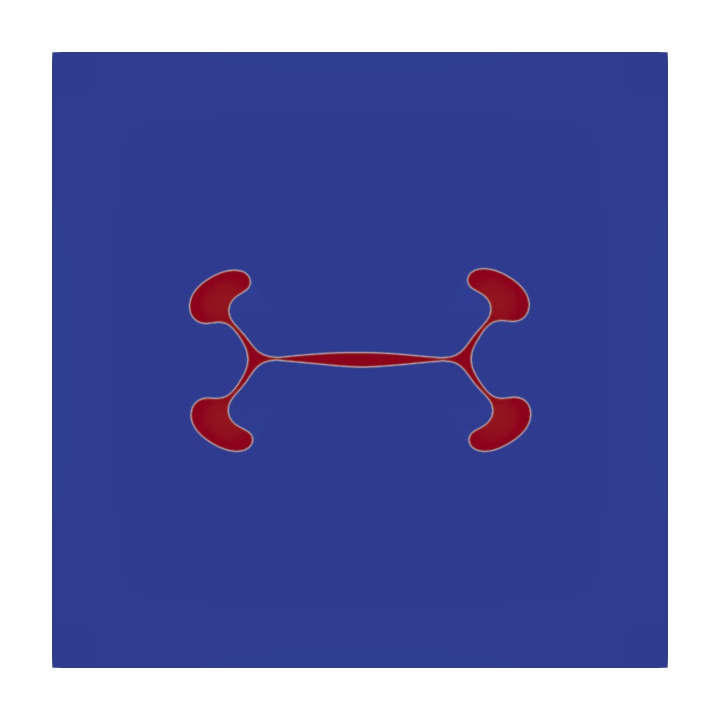}}
	\hspace{-0.8em} 
	\subfloat
	{\includegraphics[width=0.22\textwidth]{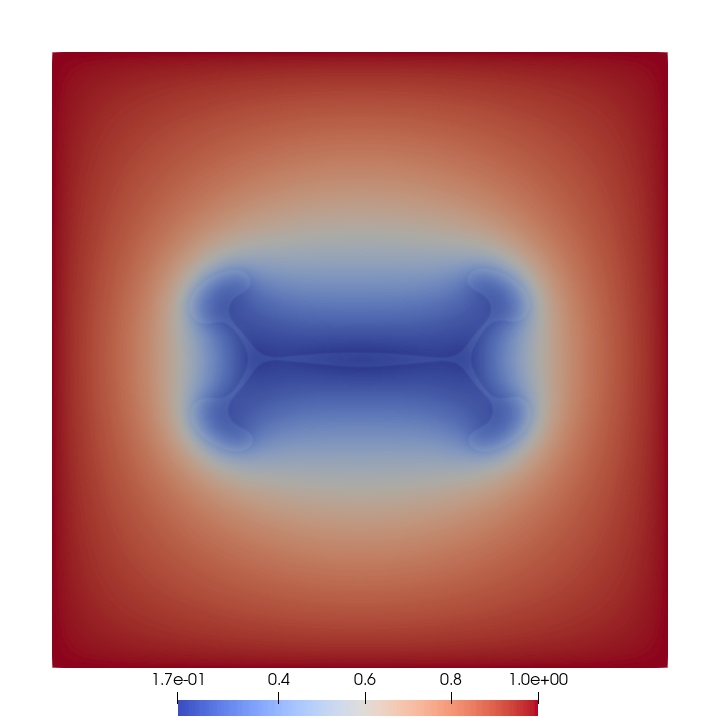}}
	\hspace{-0.8em}
	\subfloat
	{\includegraphics[width=0.22\textwidth]{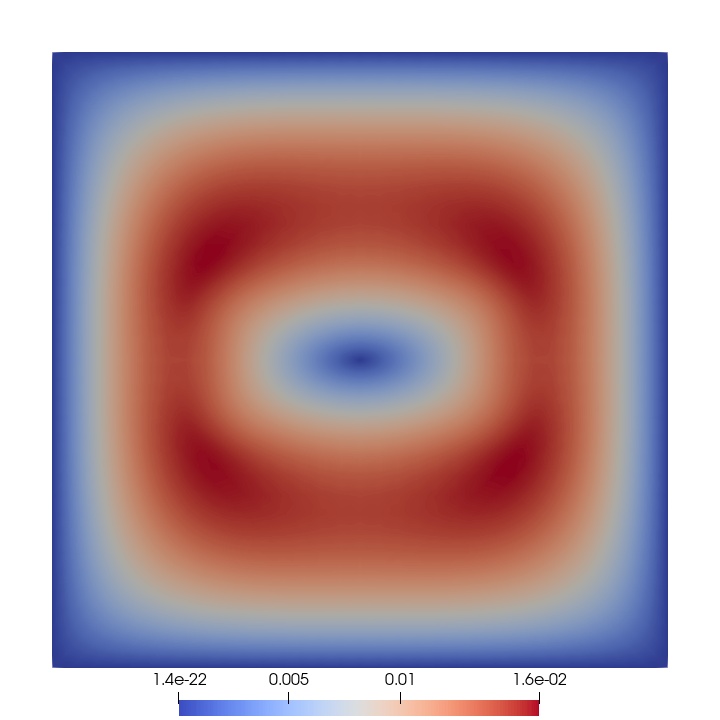}}
	\hspace{-0.8em}
	\subfloat
	{\includegraphics[width=0.22\textwidth]{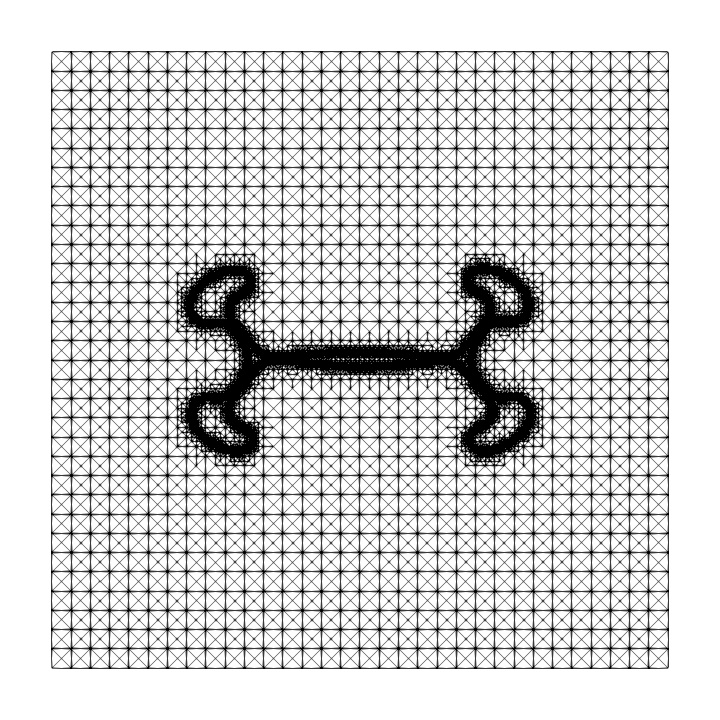}}
	\\[-2.7ex] 
	\subfloat
	{\includegraphics[width=0.22\textwidth]{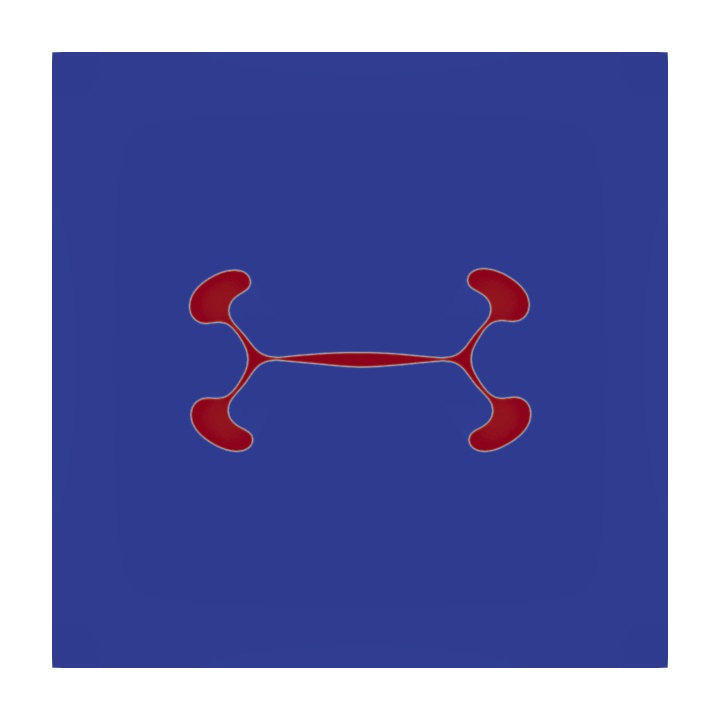}}
	\hspace{-0.8em} 
	\subfloat
	{\includegraphics[width=0.22\textwidth]{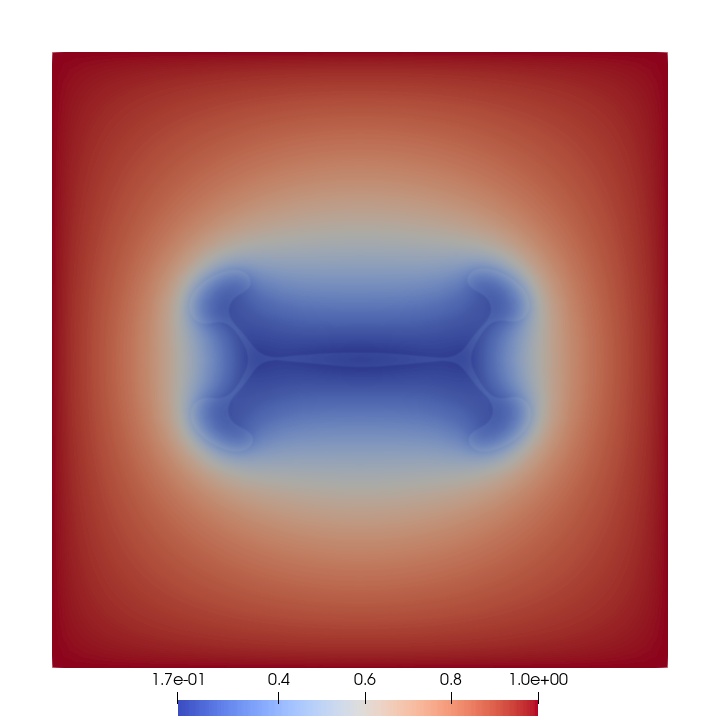}}
	\hspace{-0.8em}
	\subfloat
	{\includegraphics[width=0.22\textwidth]{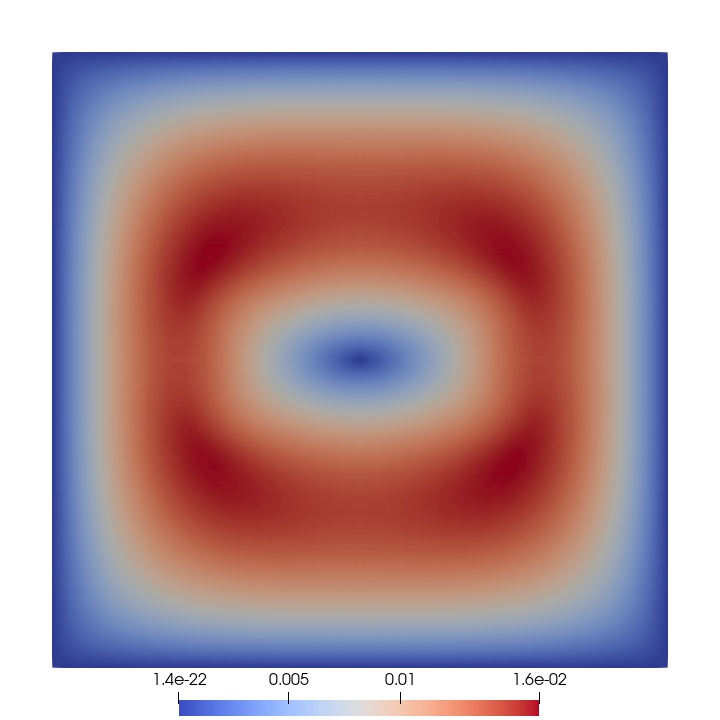}}
	\hspace{-0.8em}
	\subfloat
	{\includegraphics[width=0.22\textwidth]{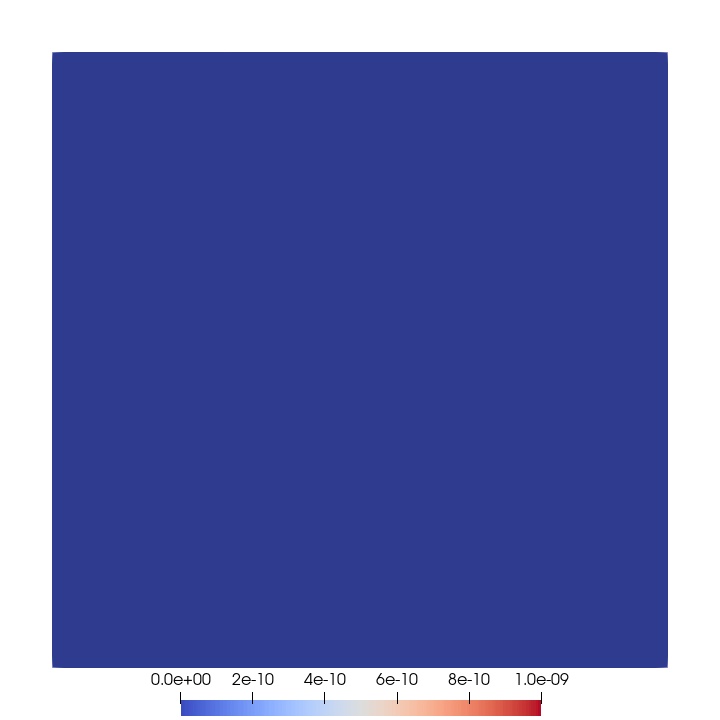}}
    \caption{Comparison of the fully viscous model (first row) to the viscoelastic model (second row) at time $t=2$. In the first three columns, $\phi$, $\sigma$ and $\abs\pmbv$ are visualized. In the last column, the final mesh and $\abs{\T_{\mathrm{el}}(\B)}$ are shown.} 
    \label{fig:2}
\end{figure}

\begin{figure}[H]
    \centering
    \subfloat
	{\includegraphics[width=0.22\textwidth]{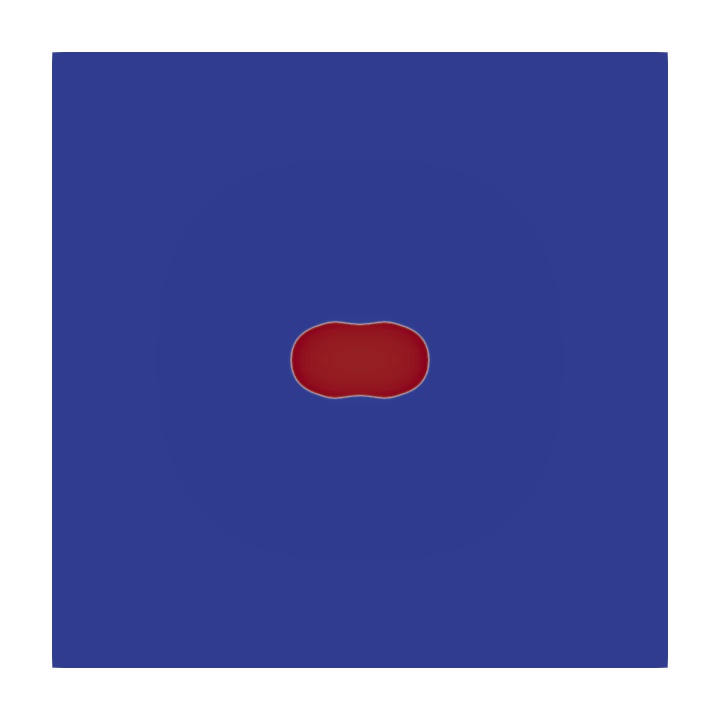}}
	\hspace{-0.8em} 
	\subfloat
	{\includegraphics[width=0.22\textwidth]{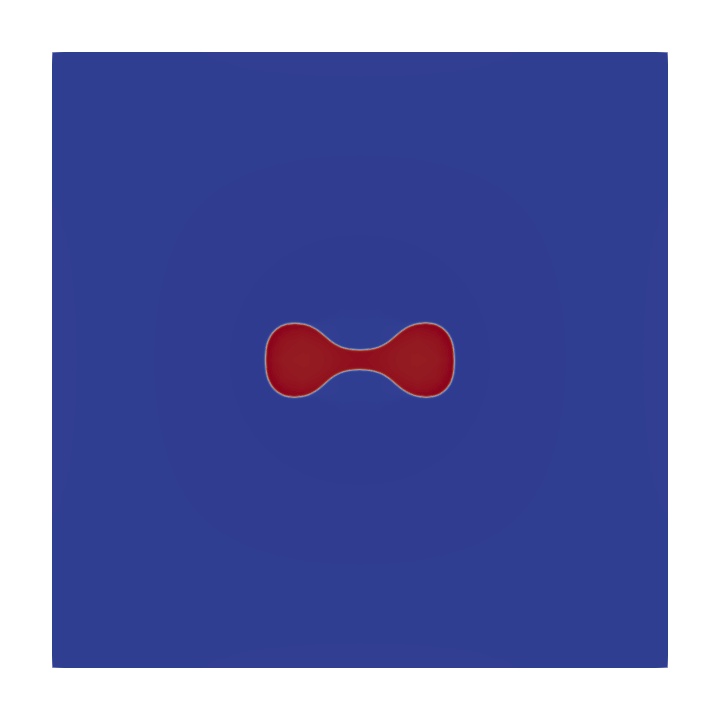}}
	\hspace{-0.8em}
	\subfloat
	{\includegraphics[width=0.22\textwidth]{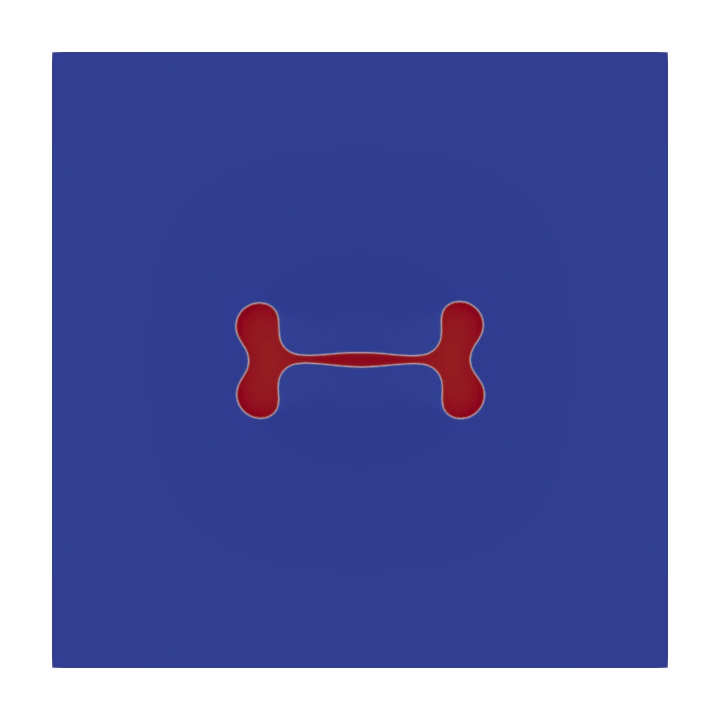}}
	\hspace{-0.8em}
	\subfloat
	{\includegraphics[width=0.22\textwidth]{figures/viscous_viscoelastic/3b1_phi.jpeg}}
    \caption{Time evolution of the tumour $\phi$ for the viscoelastic model at times $t\in\{0.5, 1, 1.5, 2\}$.} 
    \label{fig:2b}
\end{figure}


\subsection{Influence of the viscosity}

In the next experiment, we illustrate how the choice of the viscosity function $\eta(\cdot)$ can affect the elastic stress tensor and hence the evolution of the tumour. We increase the proliferation rate, i.e.~$\calP=5$, and decrease the chemotactic sensitivity, i.e.~$\chi_\phi=7.5$. Moreover, we increase the relaxation times such that $\frac{\tau_{1}}{\kappa}=\frac{\tau_{-1}}{\kappa} = 100$ and choose the viscosities $\eta_1 = 2000$ and $\eta_{-1}\in\{500, 1000, 1500\}$.
The numerical solutions $\phi$ (first row) and $\abs{\T_{\mathrm{el}}(\B)}$ (second row) at time $t=1.6$ are visualized in Figure \ref{fig:large_stress} for the cases $\eta_{-1}\in\{500, 1000, 1500\}$. From left to right, the elastic stresses decrease and the size of the tumours increase with increasing viscosity $\eta_{-1}$. However, note that the elasticity parameter $\kappa=10^{4}$ is chosen very large in order to see the influence of the elastic stresses.

\begin{figure}[H]
    \centering
    \subfloat
	{\includegraphics[width=0.24\textwidth]{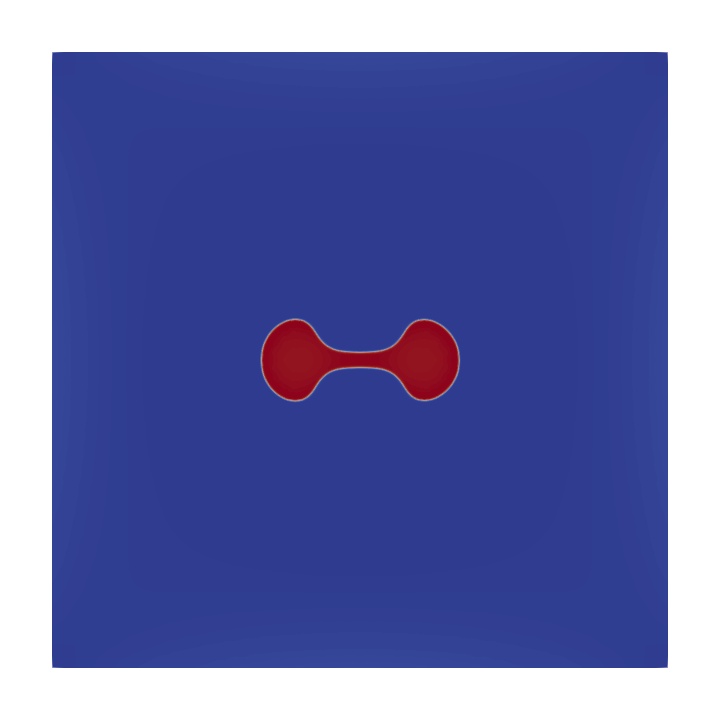}}
	\hspace{-0.8em} 
	\subfloat
	{\includegraphics[width=0.24\textwidth]{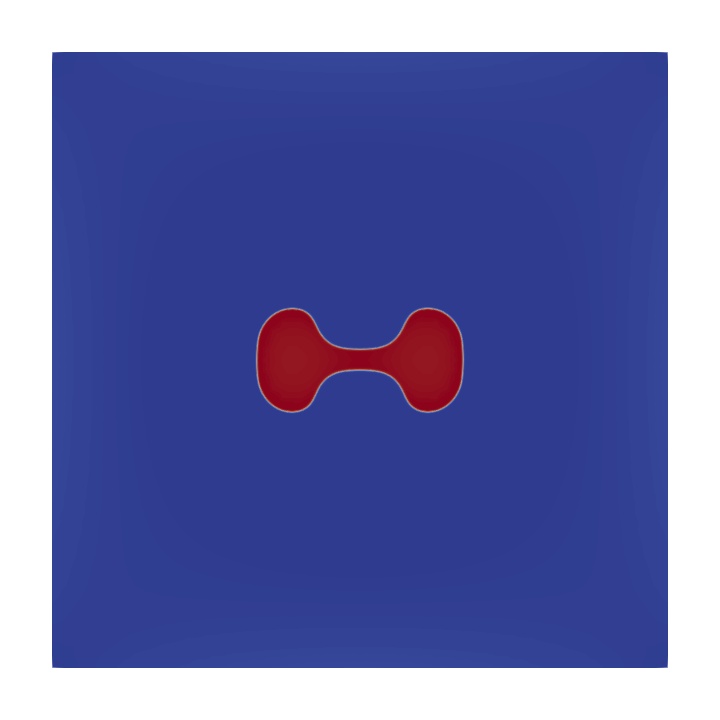}}
	\hspace{-0.8em}
	\subfloat
	{\includegraphics[width=0.24\textwidth]{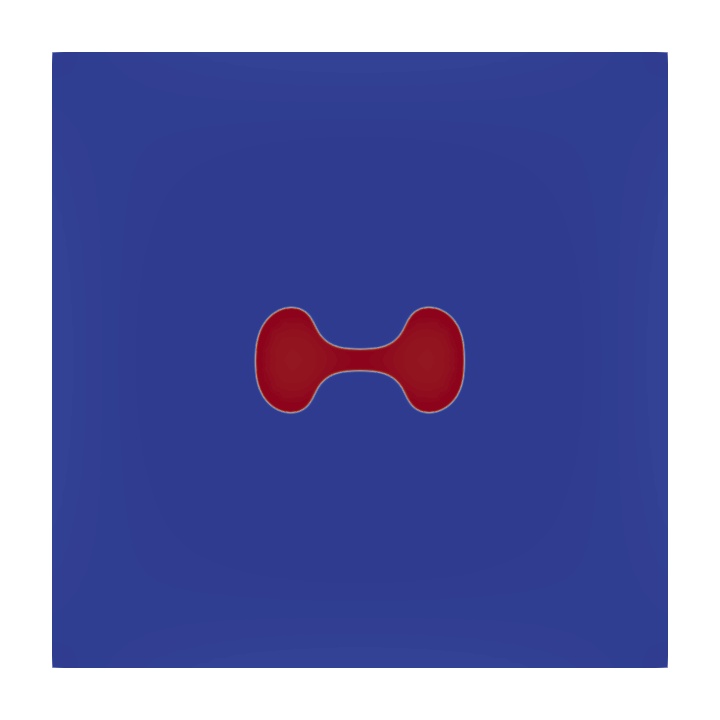}}
	\\[-4.5ex] 
	\subfloat
	{\includegraphics[width=0.24\textwidth]{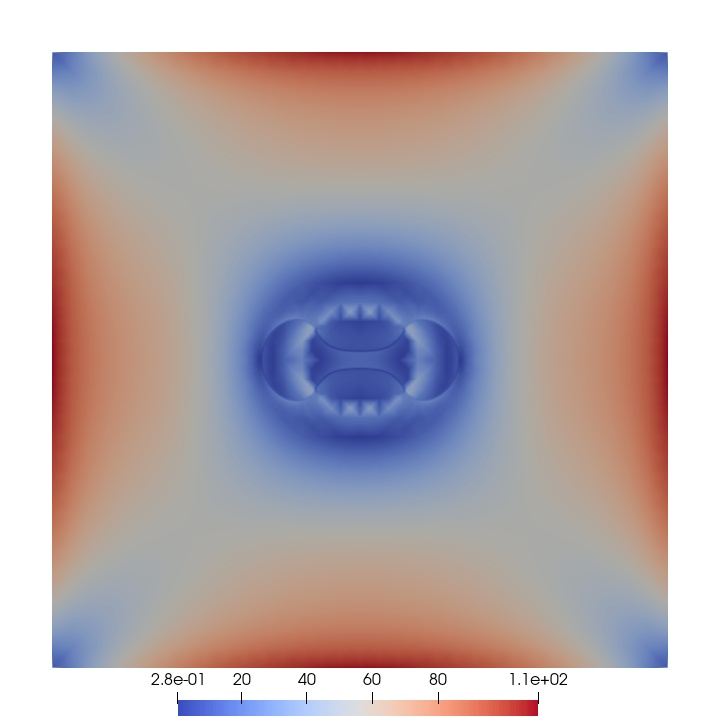}}
	\hspace{-0.8em} 
	\subfloat
	{\includegraphics[width=0.24\textwidth]{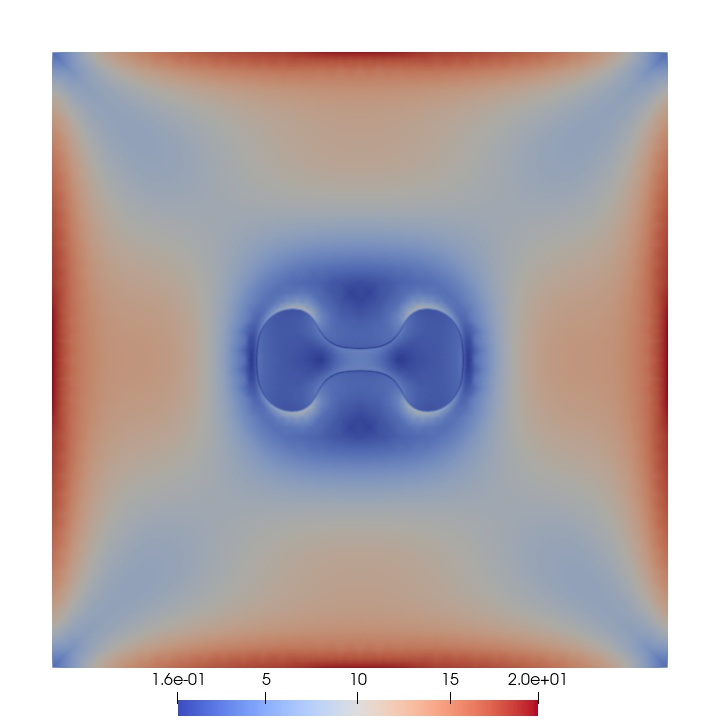}}
	\hspace{-0.8em}
	\subfloat
	{\includegraphics[width=0.24\textwidth]{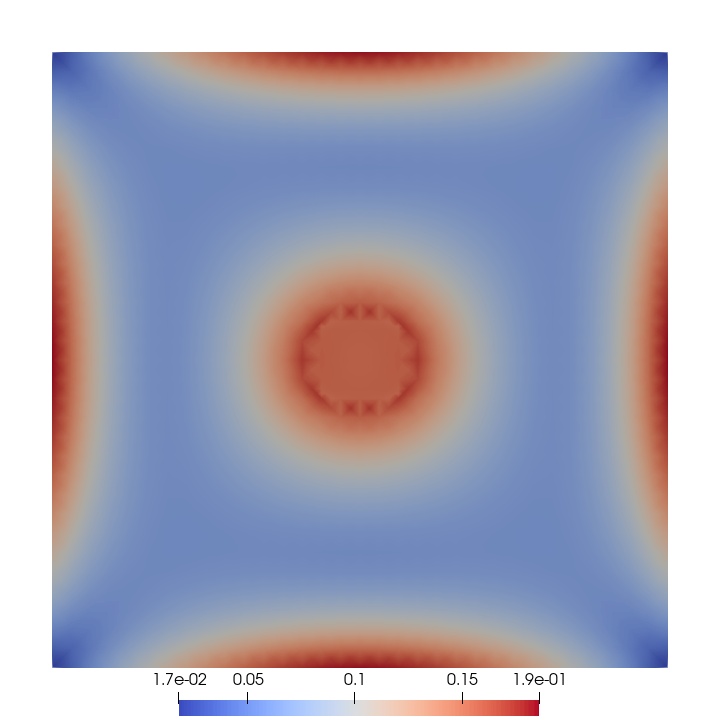}}
    \caption{From left to right: $\eta_{-1}\in\{500, 1000, 1500\}$, first row: $\phi$, second row: $\abs{\T_{\mathrm{el}}}$ at time $t=1.6$.} 
    \label{fig:large_stress}
\end{figure}


\subsection{Mechanical stresses generated by growth}
Now we consider a variant of the model with an additional source term in the equation of $\B$ like in \eqref{eq:B_growth}. For this reason we replace \eqref{eq:B_FE} by 
\begin{align}
\label{eq:B_FE_growth}
    \nonumber
    0 &= \int_\Omega \calI_h \Big[ 
    \Big(\frac{\B_h^n - \B_h^{n-1}}{\Delta t}
    + \frac{\kappa}{\tau(\phi_h^{n-1})} (\B_h^n - \I) \Big): \C_h \Big]
    - 2 \nabla\pmbv_h^n : (\C_h\B_h^n) 
    + \alpha \nabla\B_h^n : \nabla\C_h \dx
    \\
    &\qquad + \int_\Omega 
    \calI_h\Big[  \gamma(\phi_h^{n-1},\sigma_h^{n-1}) \B_h^n : \C_h \Big]
    - \B_h^n : ((\pmbv_h^{n-1}\cdot\nabla) \C_h) 
    \dx,
\end{align}
where $\gamma(\phi,\sigma) = \calG \sigma h(\phi)$ with a constant $\calG\geq0$. 

For the first numerical test, the parameters are chosen as in \eqref{eq:num_parameters} but with 
$\calP=5$, $\chi_\phi=7.5$, $\eta_1 = 2000$, $\eta_{-1}=1500$, $\frac{\tau_{1}}{\kappa}=\frac{\tau_{-1}}{\kappa} = 0.01$ and $\calG\in\{0, 0.1, 0.5, 1\}$.
In Figure \ref{fig:stress_source}, we visualize the tumour (upper row) and the magnitude of the elastic stress tensor (lower row) where $\calG\in\{0, 0.1, 0.2, 0.5\}$ from left to right at time $t=2.25$. From left to right, the size of the tumours decrease as the elastic stresses become larger with increasing $\calG$. Moreover, we observe that the elastic stresses are particularly large in the fingers of the tumour.

\begin{figure}[H]
    \centering
    \subfloat
	{\includegraphics[width=0.24\textwidth]{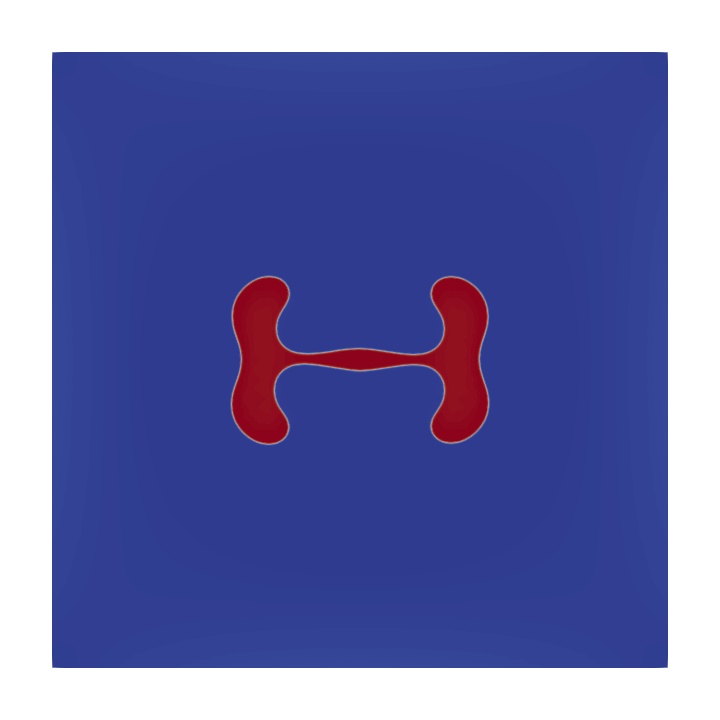}}
	\hspace{-0.8em} 
	\subfloat
	{\includegraphics[width=0.24\textwidth]{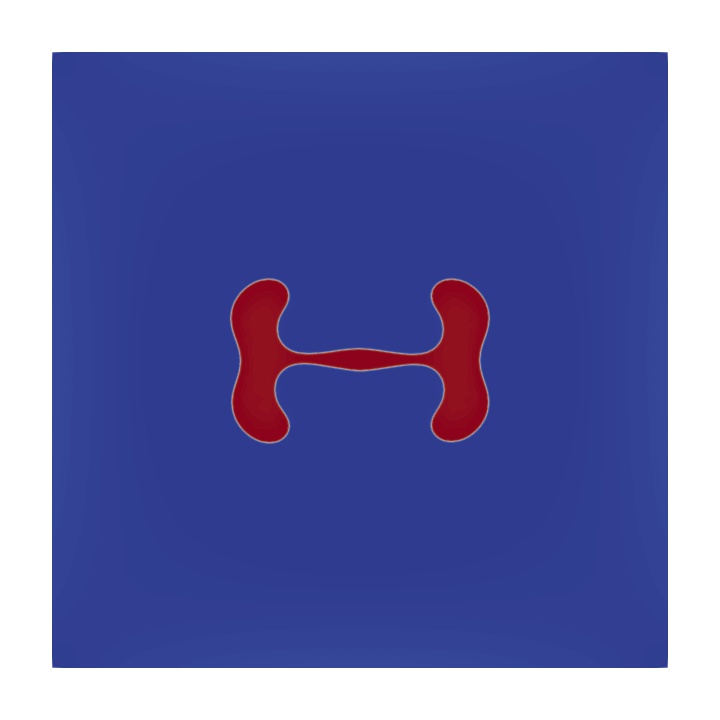}}
	\hspace{-0.8em}
	\subfloat
	{\includegraphics[width=0.24\textwidth]{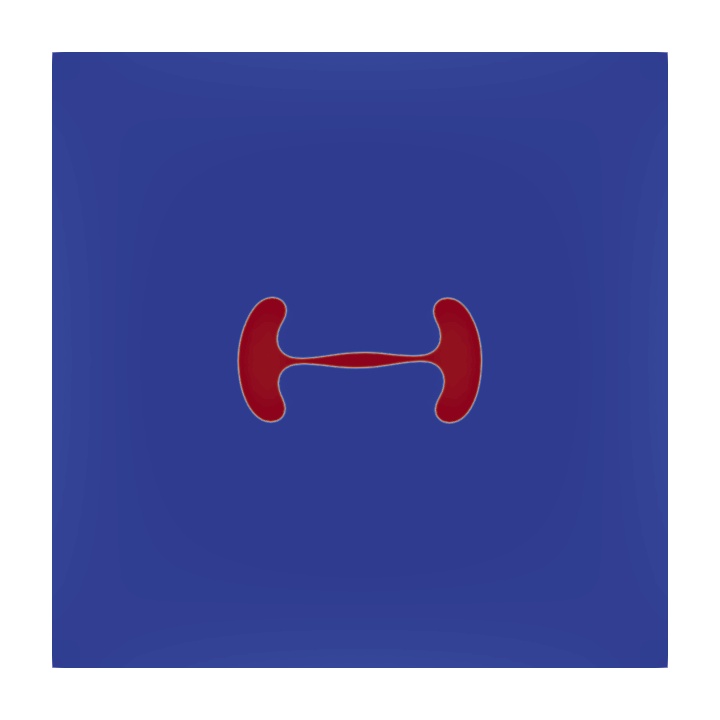}}
	\hspace{-0.8em}
	\subfloat
	{\includegraphics[width=0.24\textwidth]{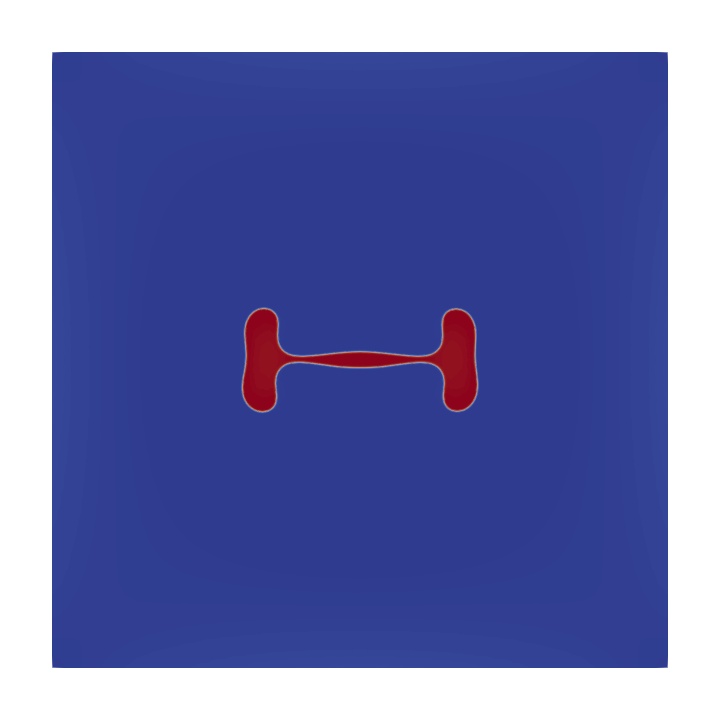}}
	\\[-4.5ex] 
	\subfloat
	{\includegraphics[width=0.24\textwidth]{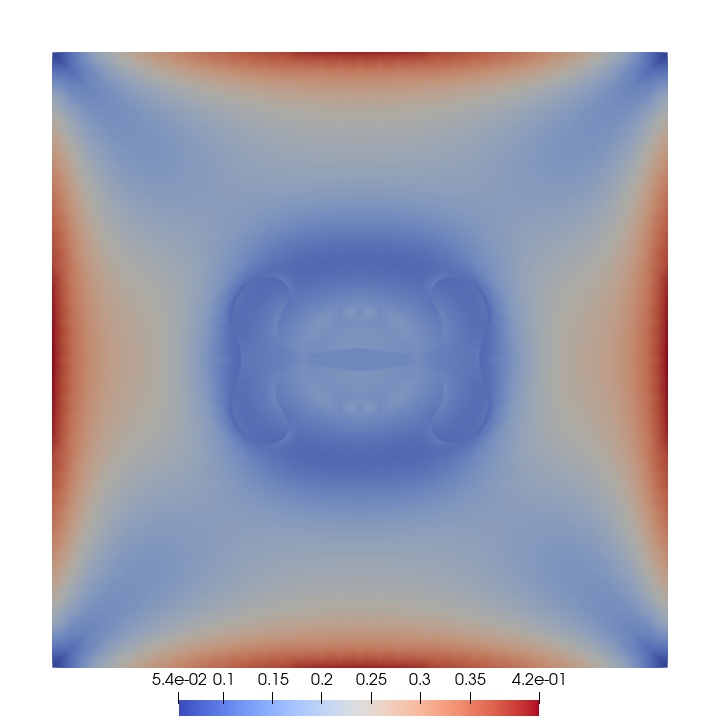}}
	\hspace{-0.8em} 
	\subfloat
	{\includegraphics[width=0.24\textwidth]{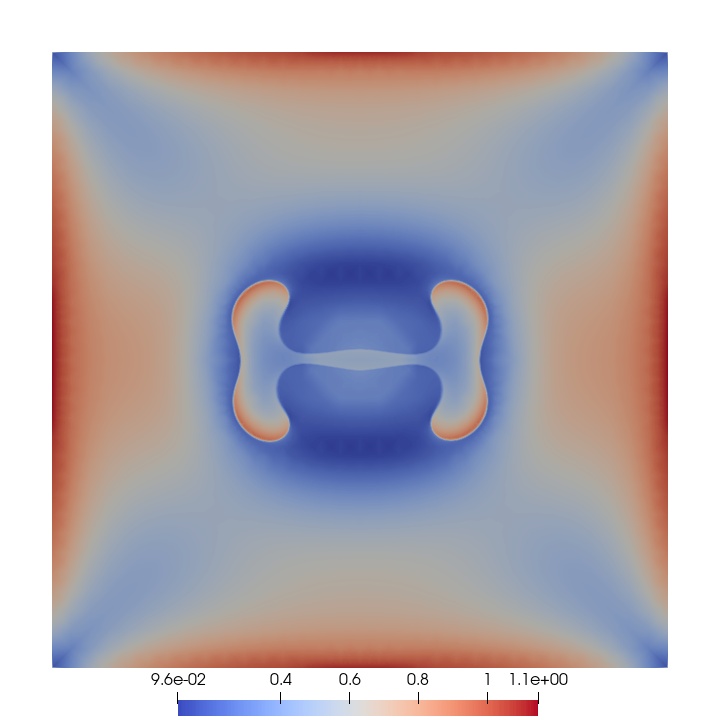}}
	\hspace{-0.8em}
	\subfloat
	{\includegraphics[width=0.24\textwidth]{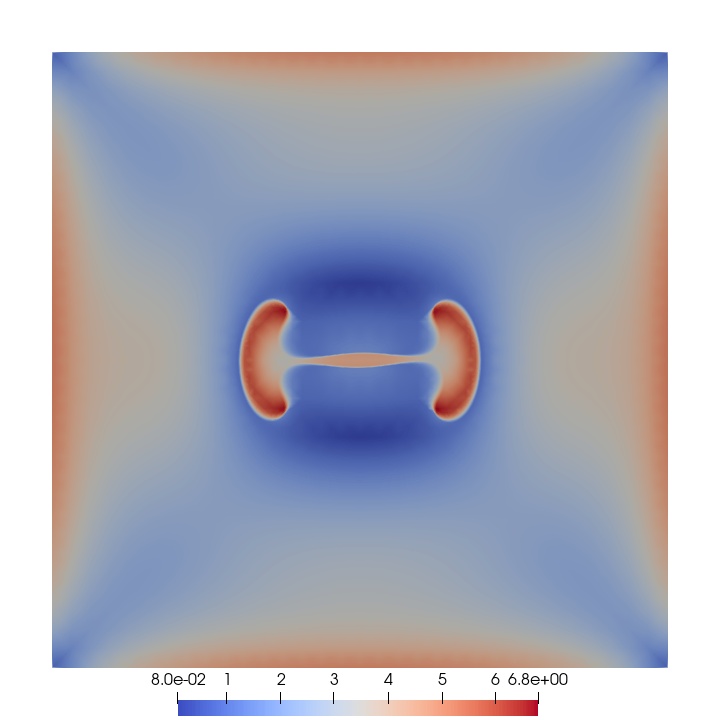}}
	\hspace{-0.8em}
	\subfloat
	{\includegraphics[width=0.24\textwidth]{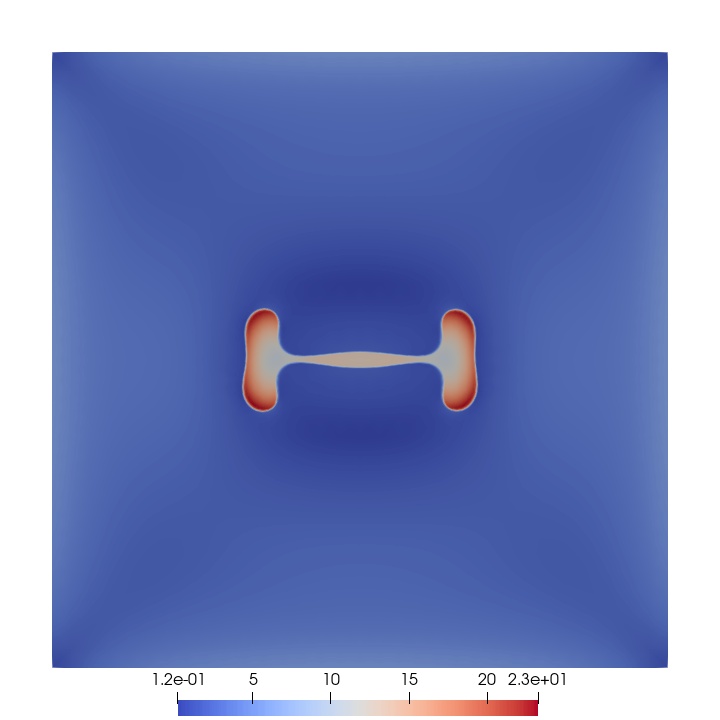}}
    \caption{From left to right: $\calG\in\{0, 0.1, 0.2, 0.5\}$. First row: $\phi$, second row: $\abs{\T_{\mathrm{el}}}$ at time $t=2.25$.} 
    \label{fig:stress_source}
\end{figure}

%


Next, we present the influence of the relaxation times $\tau_{-1},\tau_{1}$ on the growth behaviour in presence of source terms for $\B$.
On one side, we expect the elastic stresses to vanish if the relaxation time is small enough. On the other side, large elastic stresses can build up if the relaxation time is large, which then reduces the proliferation effect.
Therefore, the parameters are chosen as in \eqref{eq:num_parameters} but with 
$\calP=5$, $\chi_\phi=7.5$, $\calG=0.2$ and with $\frac{\tau_{1}}{\kappa} = \frac{\tau_{-1}}{\kappa} \in \{10^{-4}, 10^{-2}, 1\}$.
In Figure \ref{fig:relaxation}, we visualize the tumour (upper row) and the magnitude of the elastic stress tensor (lower row) at time $t=2.4$, where $\frac{\tau_{1}}{\kappa} = \frac{\tau_{-1}}{\kappa} \in \{10^{-4}, 10^{-2}, 1\}$ from left to right. Here, no elastic stresses occur if the relaxation time is very small, and the elastic stresses can become very large if the relaxation time is large.

\begin{figure}[H]
    \centering
    \subfloat
	{\includegraphics[width=0.24\textwidth]{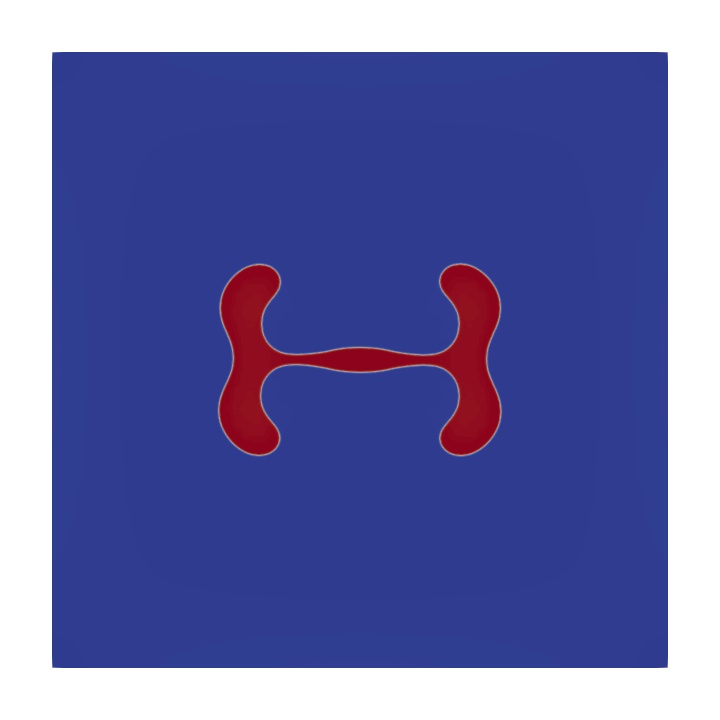}}
	\hspace{-0.8em} 
	\subfloat
	{\includegraphics[width=0.24\textwidth]{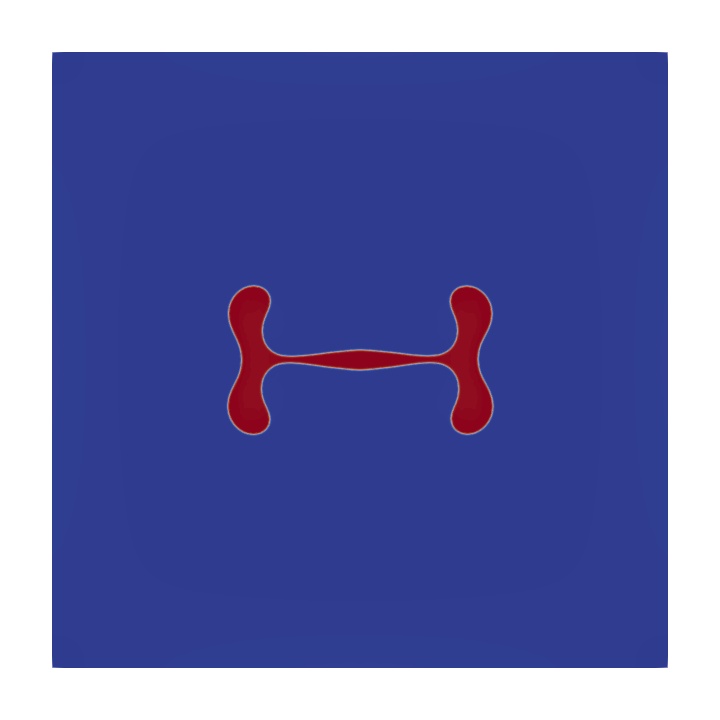}}
	\hspace{-0.8em}
	\subfloat
	{\includegraphics[width=0.24\textwidth]{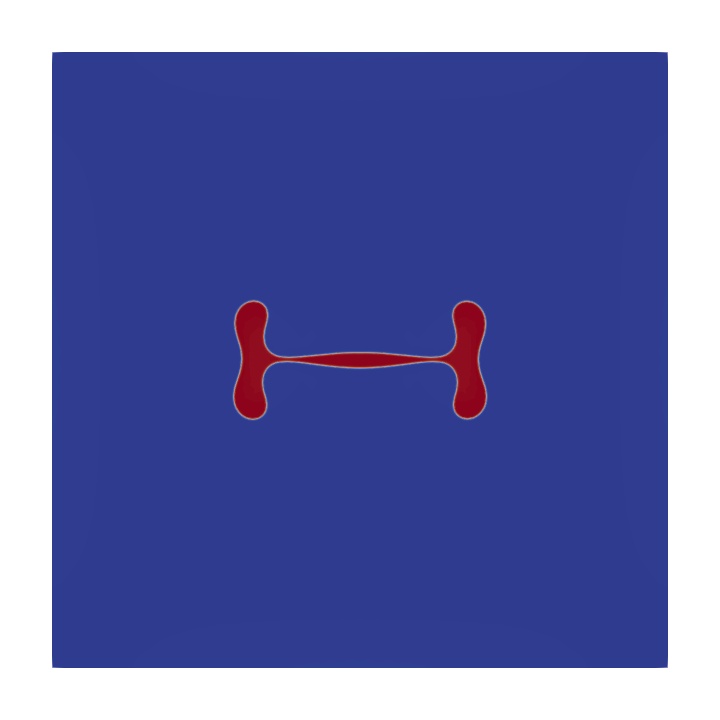}}
	\\[-4.5ex] 
	\subfloat
	{\includegraphics[width=0.24\textwidth]{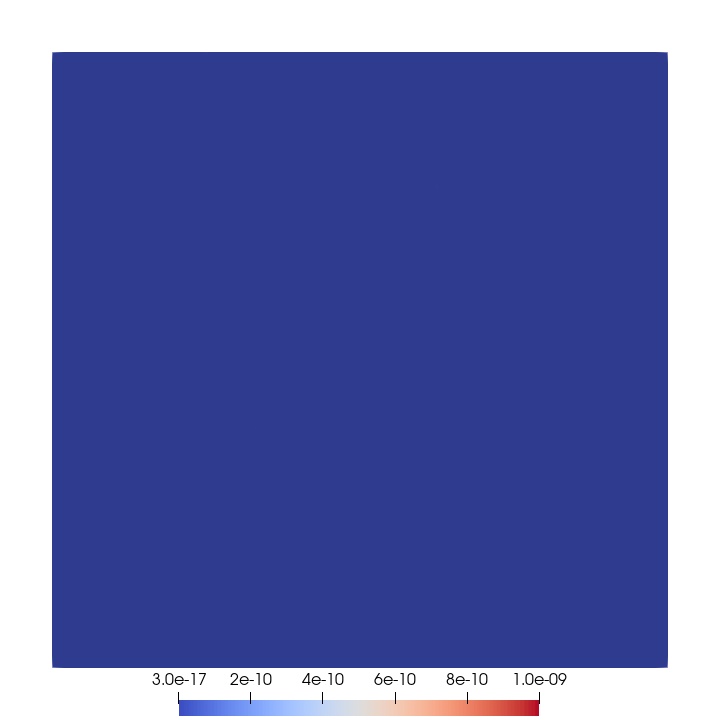}}
	\hspace{-0.8em} 
	\subfloat
	{\includegraphics[width=0.24\textwidth]{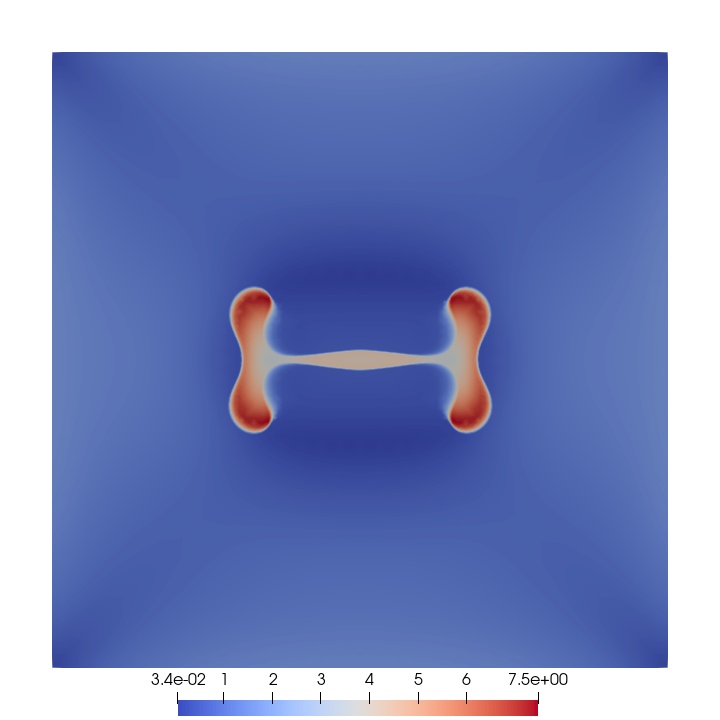}}
	\hspace{-0.8em}
	\subfloat
	{\includegraphics[width=0.24\textwidth]{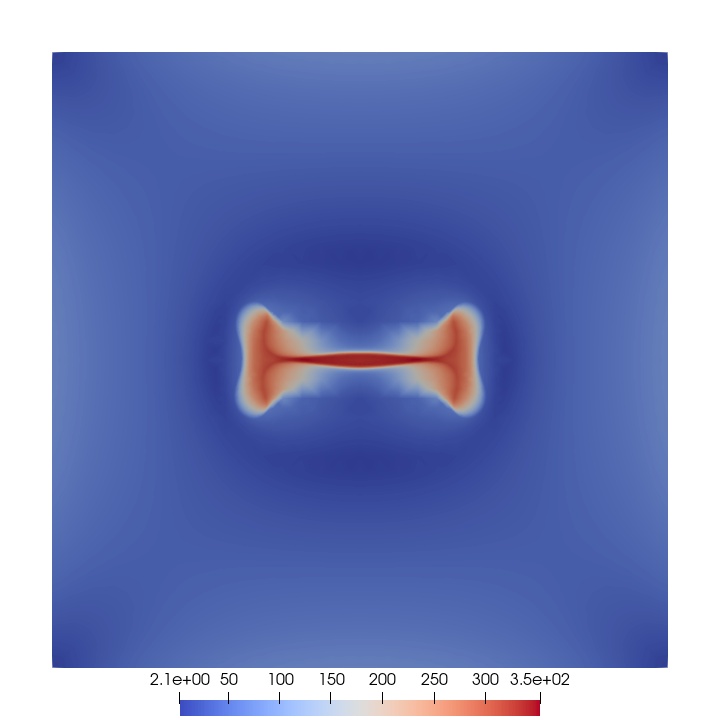}}
    \caption{From left to right: $\frac{\tau_{1}}{\kappa} = \frac{\tau_{-1}}{\kappa} \in \{10^{-4}, 10^{-2}, 1\}$. First row: $\phi$, second row: $\abs{\T_{\mathrm{el}}}$ at time $t=2.4$.} 
    \label{fig:relaxation}
\end{figure}


\subsection{Numerical results for a phase-dependent elastic energy density}
In the following, we want to illustrate the impact of a phase-dependent elastic energy density on the evolution of the tumour. In particular, we now consider the phase-dependent elasticity parameter function $\kappa(\phi) = \tfrac{1}{2} \kappa_1 (1+\phi) + \tfrac{1}{2} \kappa_{-1} (1-\phi)$, which leads to the elastic energy density $W(\phi,\B) = \frac{1}{2} \kappa(\phi) \trace (\B - \ln \B)$.
Hence, we adapt the system of equations \eqref{eq:phi_FE}--\eqref{eq:v_FE}, \eqref{eq:B_FE_growth} as follows. 
First, we replace $\kappa$ with $\kappa(\phi_h^n)$ in \eqref{eq:v_FE}, \eqref{eq:B_FE_growth} and in the source term $\Gamma_\phi$.
To be consistent with \eqref{eq:mu}, we add the term $+ \int_\Omega \frac{1}{4} (\kappa_1 - \kappa_{-1}) \calI_h\left[\trace(\B_h^{n} - \ln\B_h^{n}) \rho_h\right] \dx$ to the right-hand side of \eqref{eq:mu_FE}. 
Lastly, due to \eqref{eq:v0b}, we include the term $- \int_\Omega \frac{1}{4} (\kappa_1 - \kappa_{-1}) \phi_h^{n-1} \nabla \trace(\B_h^{n}-\ln\B_h^n) \cdot \pmbw_h \dx$ on the right-hand side of \eqref{eq:v_FE}.

The goal is now to study the growth behaviour of the tumour. In particular, it is of main interest whether the chemotactic development of fingers is intensified, weakened or completely changed.
Now, as the term $W_{,\phi} = \frac{1}{4}(\kappa_{1}-\kappa_{-1}) \trace (\B - \ln\B)$ enters the equation for the chemical potential $\mu$, we vary the elasticity parameters $\kappa_{1},\kappa_{-1}$ and we make sure that $\trace (\B - \ln\B)$ changes near the tumour region.
For these reasons, the parameters in the following experiments are chosen as in \eqref{eq:num_parameters} but with $\chi_\phi=7.5$, $\calG=10$, $\frac{\tau_{1}}{\kappa_{1}}=10, \frac{\tau_{-1}}{\kappa_{-1}}=0.1$ and with varying $\kappa_{1},\kappa_{-1}$.

First, we show the time evolution of the tumour with matched elasticity parameters $\kappa_{1}=\kappa_{-1}=1$ at times $t\in\{1, 1.5, 2.3\}$ and the magnitude of the elastic stress tensor at the final time. Here, the tumour growths and develops four thick fingers.

\begin{figure}[H]
    \centering
    \subfloat
	{\includegraphics[width=0.24\textwidth]{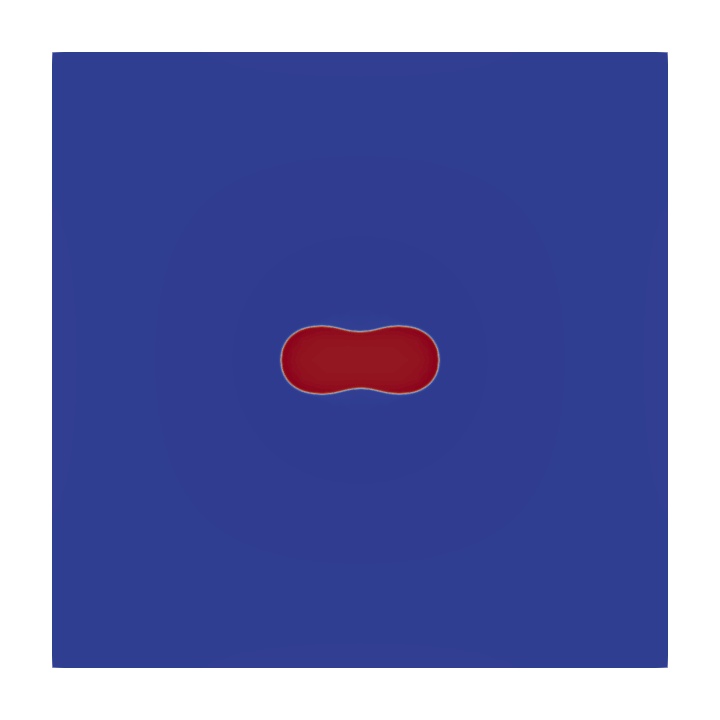} }
	\hspace{-0.8em} 
	\subfloat
	{\includegraphics[width=0.24\textwidth]{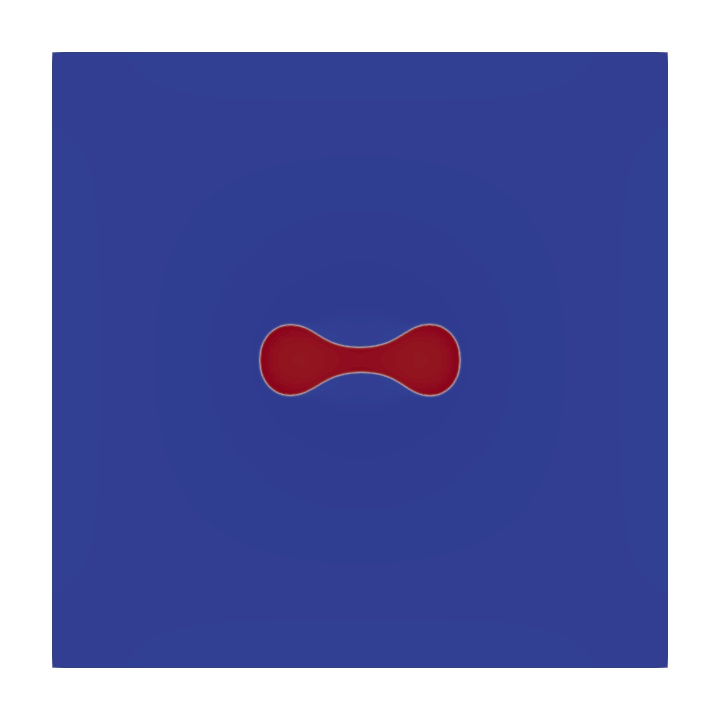}}
	\hspace{-0.8em}
	\subfloat
	{\includegraphics[width=0.24\textwidth]{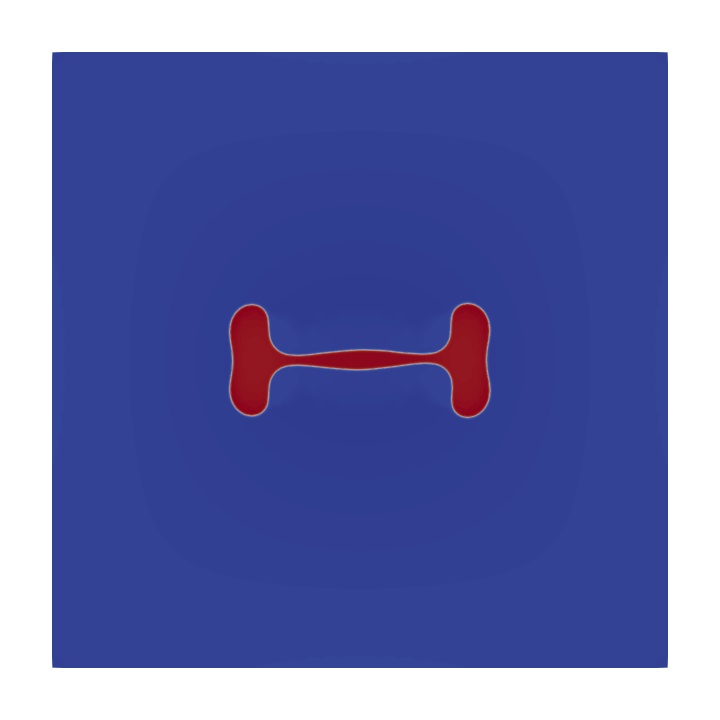}}
	\hspace{-0.8em}
	\subfloat
	{\includegraphics[width=0.24\textwidth]{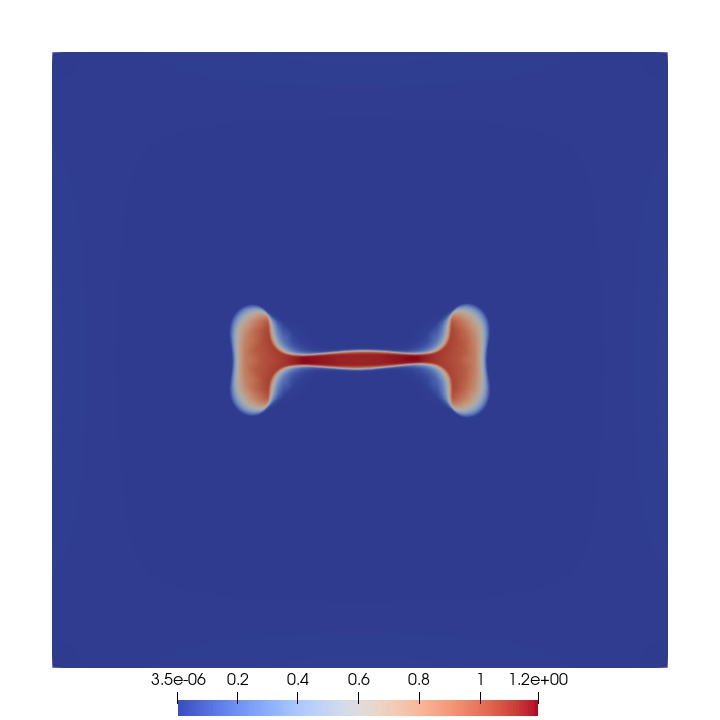}}
    \caption{Numerical solution $\phi$ with $\kappa_{1}=\kappa_{-1}=1$ at times $t\in\{1, 1.5, 2.3\}$, and $\abs{\T_{\mathrm{el}}}$ at the final time.} 
    \label{fig:3a}
\end{figure}

Now, we visualize the tumour with unmatched elasticity parameters, where $\kappa_{1}=1$, $\kappa_{-1}=2$ and $t\in\{1,2,3\}$ in Figure \ref{fig:3b}, and $\kappa_{1}=1$, $\kappa_{-1}=5$ and $t\in\{1,2,3\}$ in Figure \ref{fig:3c}. In comparison to the case with matched elasticity parameters, we observe that the invasive growth of the tumour needs more time when $\kappa_{-1}$ is large. In addition, the shapes of the tumours are more elongated and the development of fingers is barely recognizable.

\begin{figure}[H]
    \centering
    \subfloat
	{\includegraphics[width=0.24\textwidth]{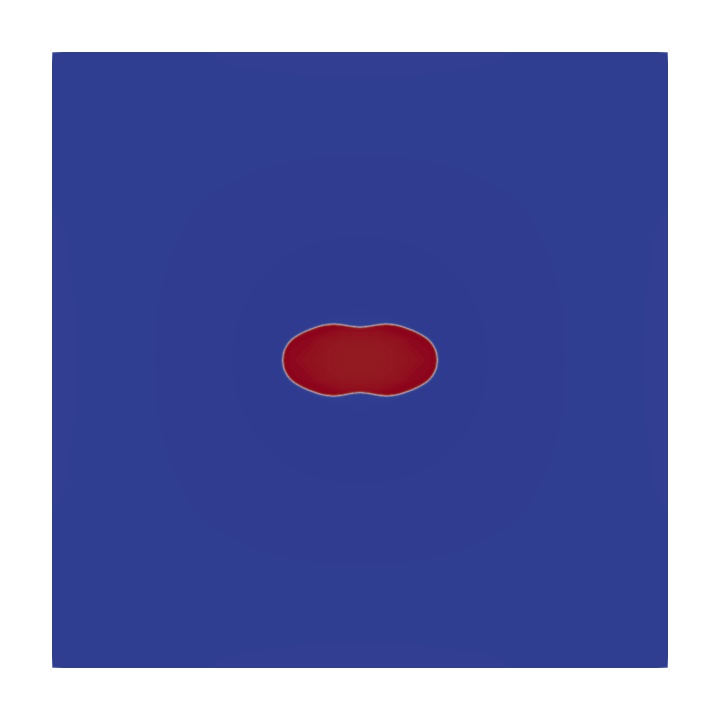}}
	\hspace{-0.8em} 
	\subfloat
	{\includegraphics[width=0.24\textwidth]{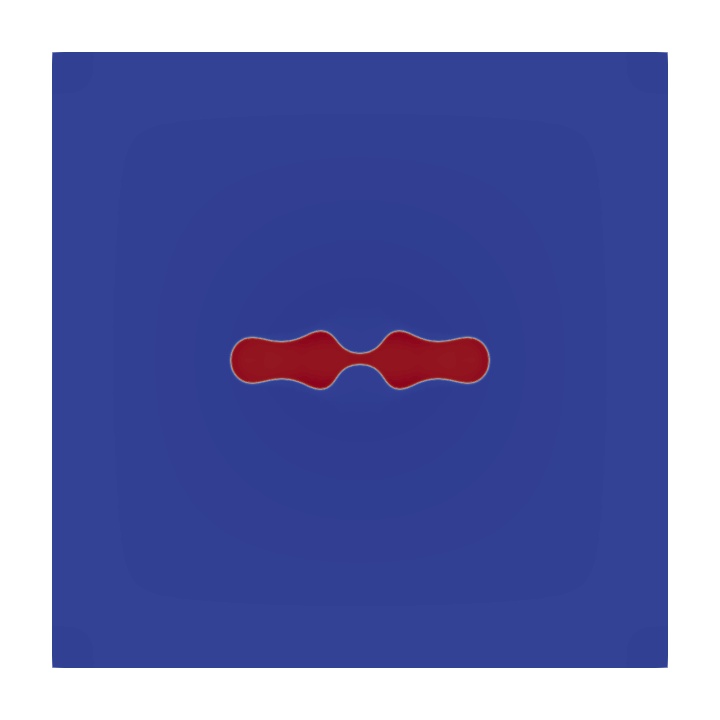}}
	\hspace{-0.8em}
	\subfloat
	{\includegraphics[width=0.24\textwidth]{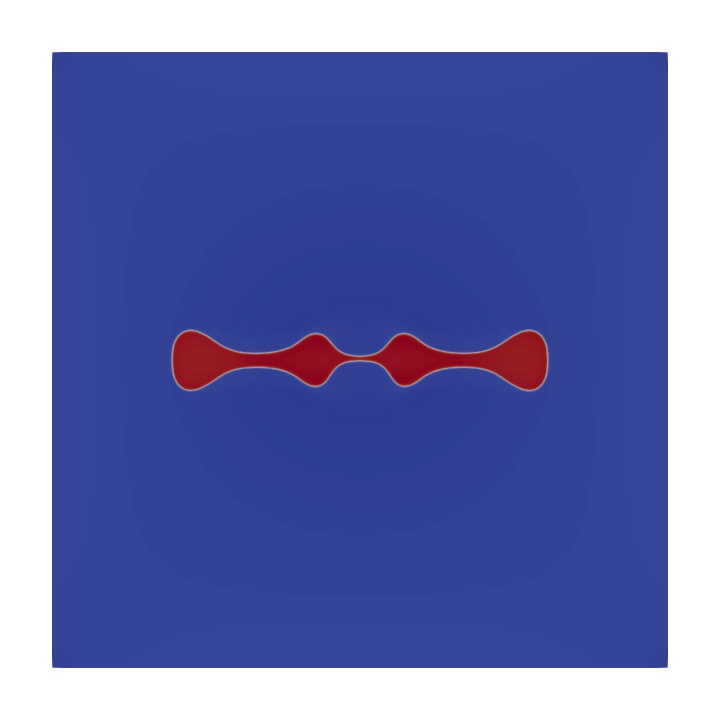}}
	\hspace{-0.8em}
	\subfloat
	{\includegraphics[width=0.24\textwidth]{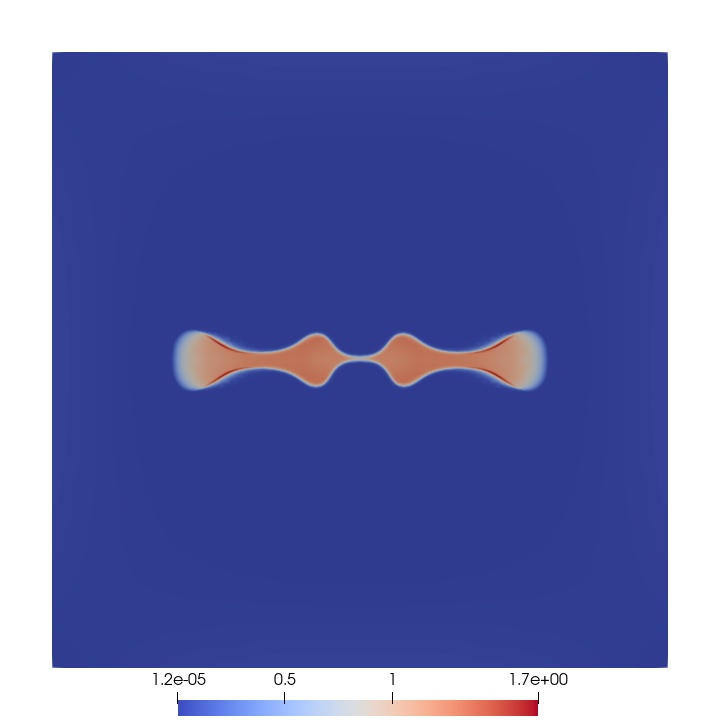}}
    \caption{Numerical solution $\phi$ with $\kappa_{1}=1, \kappa_{-1}=2$ at times $t\in\{1,2,3\}$, and $\abs{\T_{\mathrm{el}}}$ at the final time.} 
    \label{fig:3b}
\end{figure}

\begin{figure}[H]
    \centering
    \subfloat
	{\includegraphics[width=0.24\textwidth]{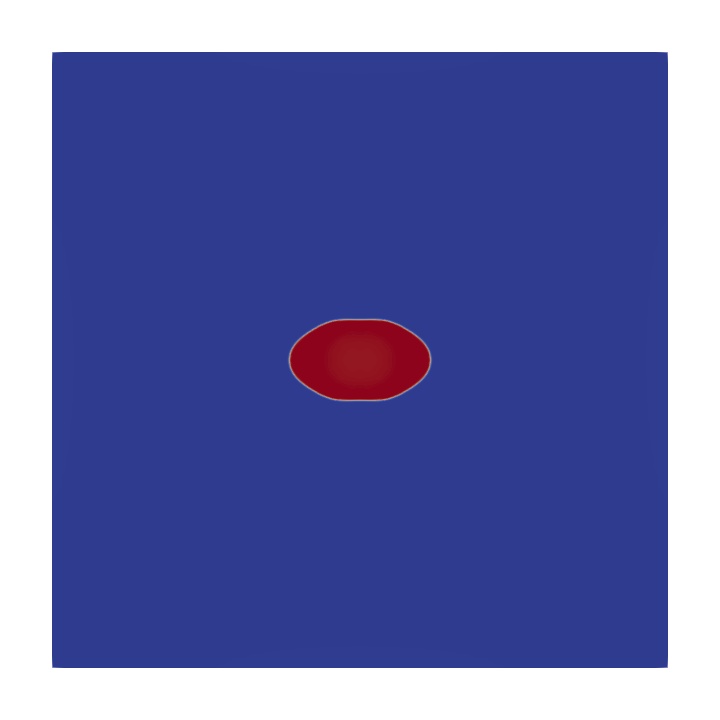}}
	\hspace{-0.8em} 
	\subfloat
	{\includegraphics[width=0.24\textwidth]{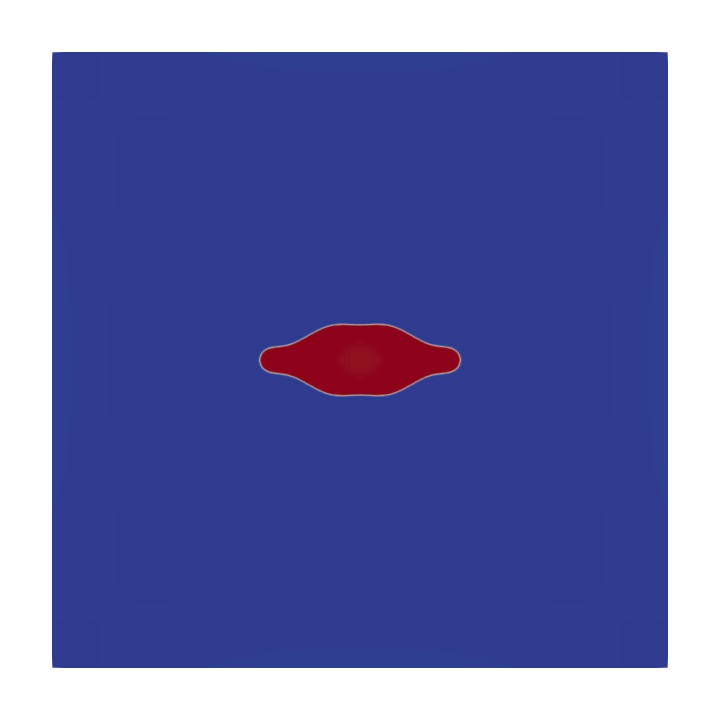}}
	\hspace{-0.8em}
	\subfloat
	{\includegraphics[width=0.24\textwidth]{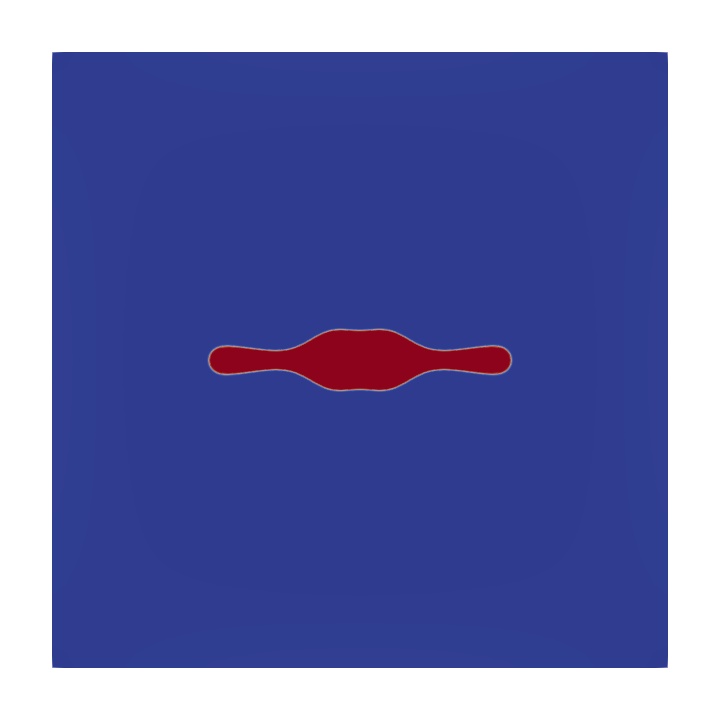}}
	\hspace{-0.8em}
	\subfloat
	{\includegraphics[width=0.24\textwidth]{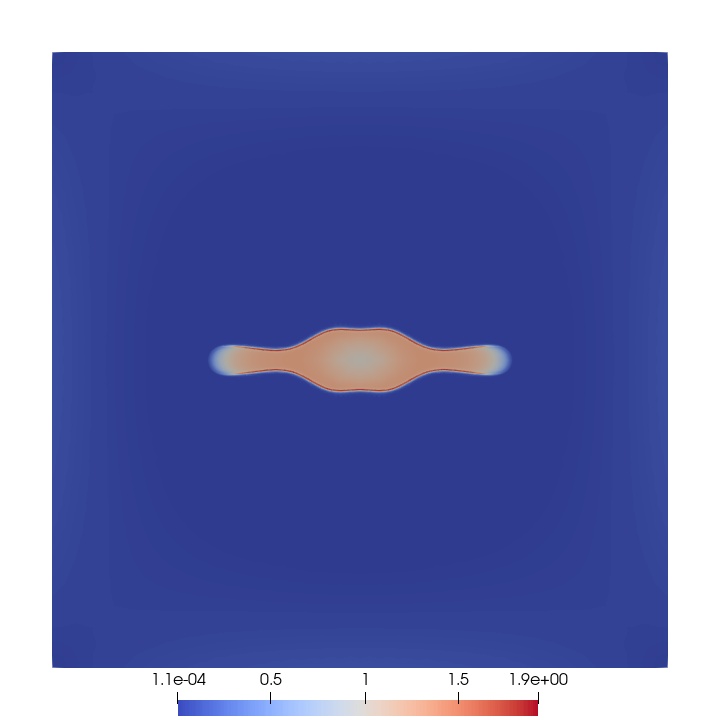}}
    \caption{Numerical solution $\phi$ with $\kappa_{1}=1, \kappa_{-1}=5$ at times $t\in\{1,2,3\}$, and $\abs{\T_{\mathrm{el}}}$ at the final time.} 
    \label{fig:3c}
\end{figure}

Next, we show the tumour with unmatched elasticity parameters, where $\kappa_{1}=2$, $\kappa_{-1}=1$ and $t\in\{0.5,1,1.5\}$ in Figure \ref{fig:3d}, and $\kappa_{1}=5$, $\kappa_{-1}=1$ and $t\in\{0.5,0.75,1\}$ in Figure \ref{fig:3e}. Here we observe that the invasive growth of the tumour takes less time when $\kappa_{1}$ is large and the development of fingers increases.

\begin{figure}[H]
    \centering
    \subfloat
	{\includegraphics[width=0.24\textwidth]{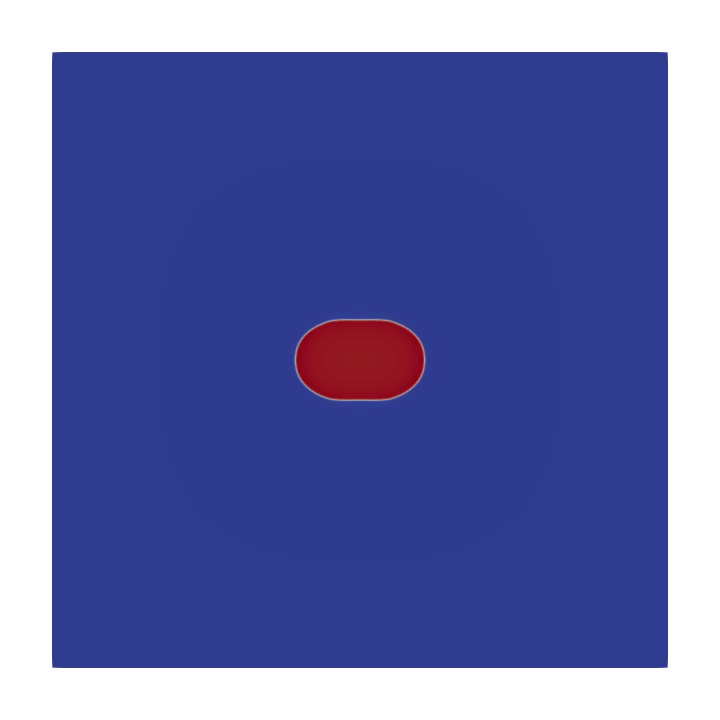}}
	\hspace{-0.8em} 
	\subfloat
	{\includegraphics[width=0.24\textwidth]{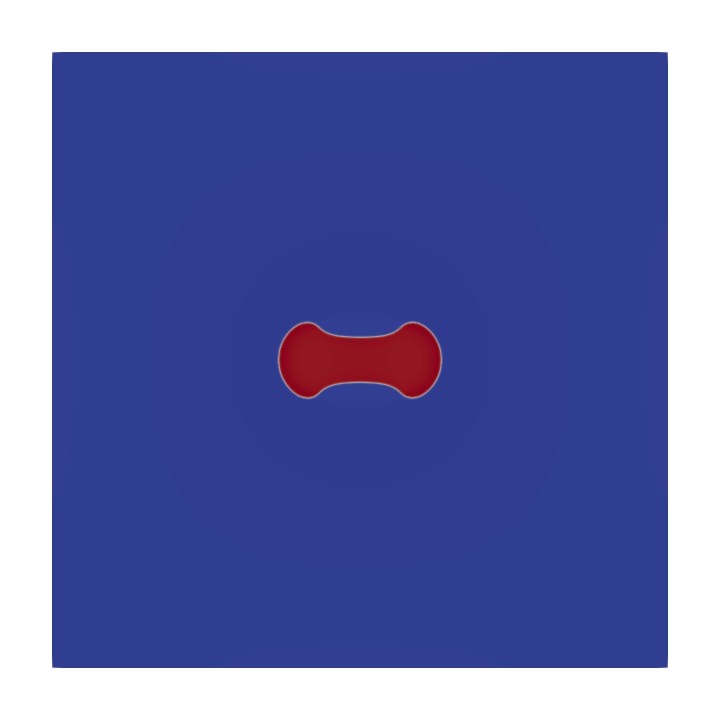}}
	\hspace{-0.8em}
	\subfloat
	{\includegraphics[width=0.24\textwidth]{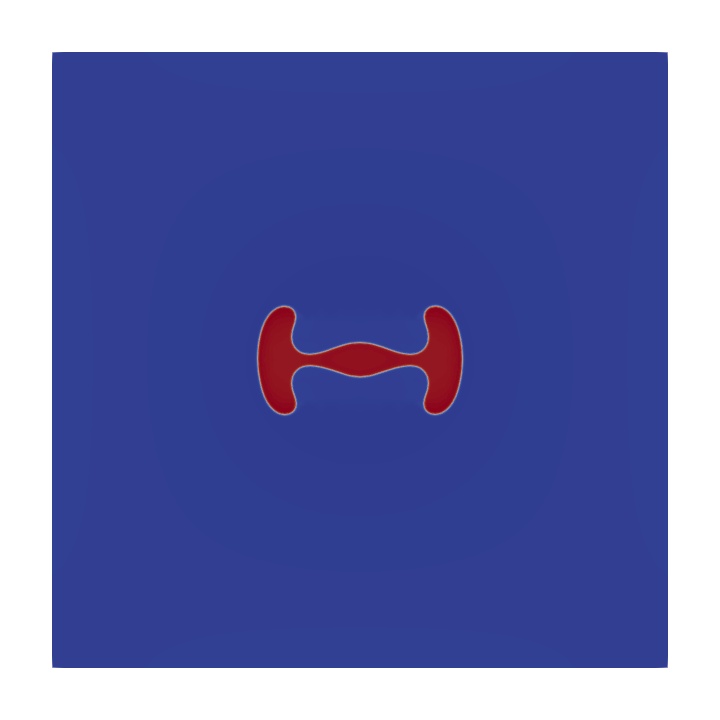}}
	\hspace{-0.8em}
	\subfloat
	{\includegraphics[width=0.24\textwidth]{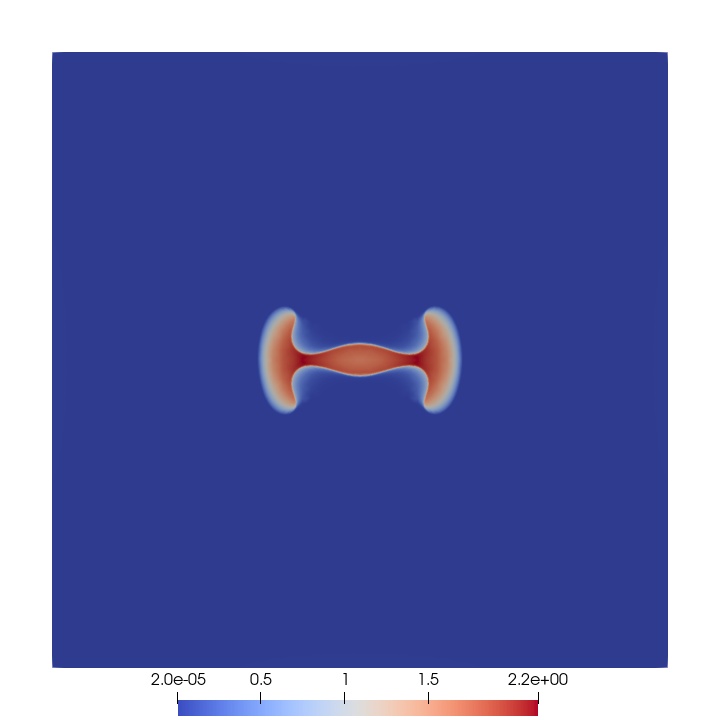}}
    \caption{Numerical solution $\phi$ with $\kappa_{1}=2, \kappa_{-1}=1$ at times $t\in\{0.5, 1, 1.5\}$, and $\abs{\T_{\mathrm{el}}}$ at the final time.} 
    \label{fig:3d}
\end{figure}

\begin{figure}[H]
    \centering
    \subfloat
	{\includegraphics[width=0.24\textwidth]{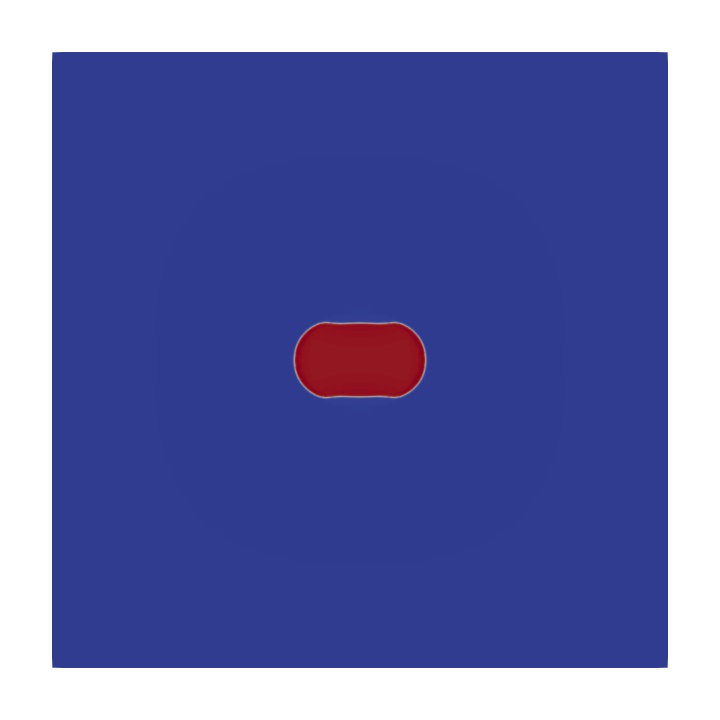}}
	\hspace{-0.8em} 
	\subfloat
	{\includegraphics[width=0.24\textwidth]{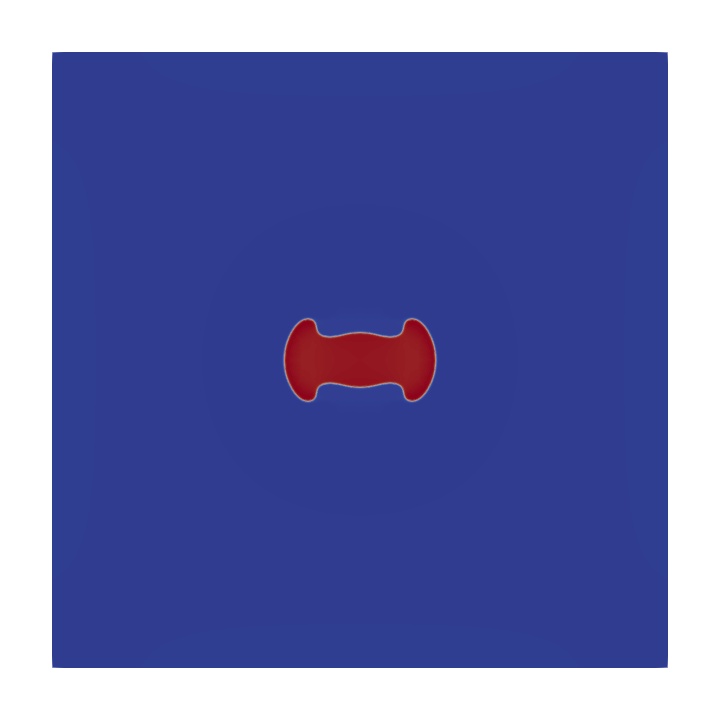}}
	\hspace{-0.8em}
	\subfloat
	{\includegraphics[width=0.24\textwidth]{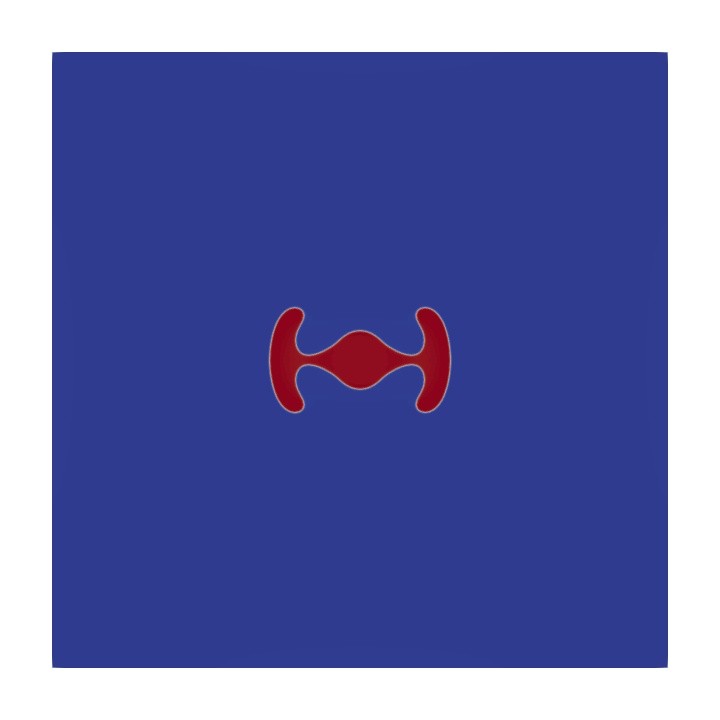}}
	\hspace{-0.8em}
	\subfloat
	{\includegraphics[width=0.24\textwidth]{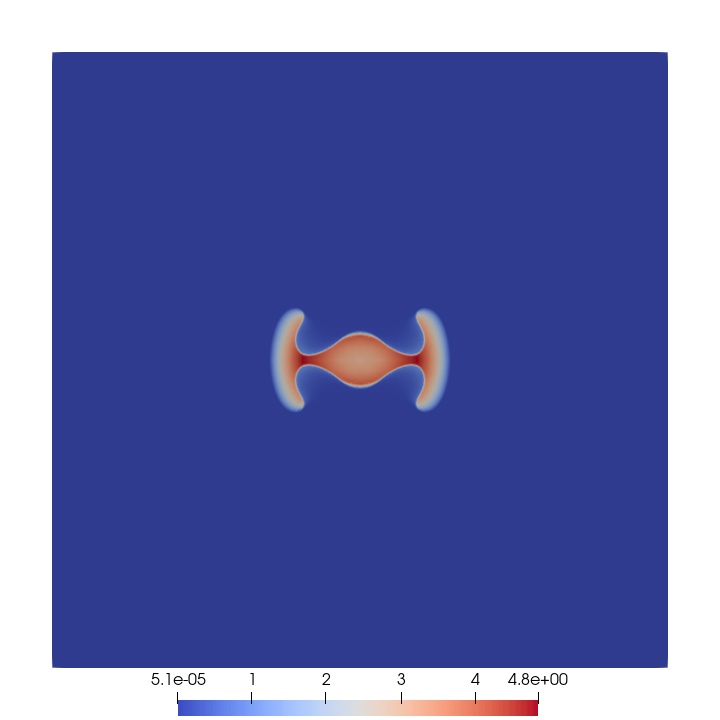}}
    \caption{Numerical solution $\phi$ with $\kappa_{1}=5, \kappa_{-1}=1$ at times $t\in\{0.5, 0.75, 1\}$, and $\abs{\T_{\mathrm{el}}}$ at the final time.} 
    \label{fig:3e}
\end{figure}

\section*{Acknowledgments}
The authors gratefully acknowledge the support by the Graduiertenkolleg 2339 IntComSin of the Deutsche Forschungsgemeinschaft (DFG, German Research Foundation) -- Project-ID 321821685.
The work of Bal\'azs Kov\'acs is funded by the Heisenberg Programme of the DFG -- Project-ID 446431602.


\addcontentsline{toc}{section}{References}
\printbibliography

@article{garcke_lam_2017, 
    title={{Well--posedness of a Cahn--Hilliard system modelling tumour growth with chemotaxis and active transport}}, 
    volume={28}, 
    DOI={10.1017/S0956792516000292}, 
    number={2}, 
    journal={European J.~Appl.~Math.},
    publisher={Cambridge University Press}, 
    author={Garcke, H. and Lam, K. F.}, 
    year={2017}, 
    pages={284–316}
}

@book{evans_2010,
  address = {Providence, R.I.},
  author = {Evans, L. C.},
  publisher = {American Mathematical Society},
  refid = {465190110},
  title = {Partial Differential Equations},
  year = {2010},
  edition = {2}
}

@book{ciarlet,
  title={The Finite Element Method for Elliptic Problems},
  author={Ciarlet, P. G.},
  lccn={77024477},
  series={Classics in Applied Mathematics},
  year={2002},
  publisher={Society for Industrial and Applied Mathematics},
  doi={10.1137/1.9780898719208}
}

@article{GarckeLSS_2016,
author = {Garcke, H. and Lam, K. F. and Sitka, E. and Styles, V.},
title = {{A Cahn--Hilliard--Darcy model for tumour growth with chemotaxis and active transport}},
journal = {Math.~Models Methods Appl.~Sci.},
volume = {26},
number = {06},
pages = {1095-1148},
year = {2016},
doi = {10.1142/S0218202516500263},
}

@book{eck-garcke-knabner,
  title={Mathematical Modeling},
  author={Eck, C. and Garcke, H. and Knabner, P.},
  year={2017},
  publisher={Springer},
  doi={10.1007/978-3-319-55161-6}
}

@article{clement_1975,
     author = {Cl{\'e}ment, P.},
     title = {Approximation by finite element functions using local regularization},
     journal = {ESAIM: Math. Model. Numer. Anal.}, 
     publisher = {Dunod},
     volume = {9},
     number = {R2},
     year = {1975},
     pages = {77-84},
     zbl = {0368.65008},
     mrnumber = {400739},
     doi= {10.1051/m2an/197509R200771},
}

@book{braess_2007, 
place={Cambridge}, 
edition={3}, 
title={Finite Elements: Theory, Fast Solvers, and Applications in Solid Mechanics},
DOI={10.1017/CBO9780511618635}, 
publisher={Cambridge University Press},
author={Braess, D.}, 
year={2007}}

@book{bartels_2016,
  author = {Bartels, S.},
  title = {Numerical Approximation of Partial Differential Equations},
  series = {Texts in Applied Mathematics},
  volume = {64},
  publisher = {Springer, [Cham]},
  year = {2016},
  pages = {xv+535},
  mrclass = {65-01 (65M60 65N30)},
  mrnumber = {3496530},
  doi = {10.1007/978-3-319-32354-1},
}

@book{alt_2016,
  title={Linear Functional Analysis: An Application-Oriented Introduction},
  author={Alt, H. W.},
  year={2016},
  publisher={Springer},
  doi={10.1007/978-1-4471-7280-2}
}

@book{fenics_book_2012,
  title = {Automated Solution of Differential Equations by the Finite Element Method},
  author = {A. Logg and K. A. Mardal and G. N. Wells and others},
  year = {2012},
  publisher = {Springer},
  doi = {10.1007/978-3-642-23099-8},
}

@article{simon_1986,
author = {Simon, J.},
year = {1986},
pages = {65--96},
title = {{Compact sets in the space $L^p(0,T; B)$}},
volume = {146},
JOURNAL = {Ann. Mat. Pura Appl. (4)},
FJOURNAL = {Annali di Matematica Pura ed Applicata. Serie Quarta},
doi = {10.1007/BF01762360}
}

@article{barrett_nurnberg_styles_2004,
author = {Barrett, J. W. and N{\"u}rnberg, R. and Styles, V.},
title = {Finite element approximation of a phase field model for void electromigration},
journal = {SIAM J.~Num.~Anal.},
volume = {42},
number = {2},
pages = {738-772},
year = {2004},
doi = {10.1137/S0036142902413421},
}

@article{ebenbeck_garcke_nurnberg_2020,
title = {{Cahn--Hilliard--Brinkman systems for tumour growth}},
journal = {Discrete Contin.~Dyn.~Syst.~Ser.~S},
volume = {14},
number = {11},
pages = {3989--4033},
DOI={10.3934/dcdss.2021034},
year = {2021},
author = {M. Ebenbeck and H. Garcke and R. N{\"u}rnberg},
publisher={American Institute of Mathematical Sciences},
}

@article{AbelsGG_2012,
author = {Abels, H. and Garcke, H. and Gr{\"u}n, G.},
title = {Thermodynamically consistent, frame indifferent diffuse interface models for incompressible two-phase flows with different densities},
journal = {Math.~Models Methods Appl.~Sci.}, 
volume = {22},
number = {03},
pages = {1150013, 40},
year = {2012},
doi = {10.1142/S0218202511500138},
}

@inbook{HuLinLiu_2018,
title = "Equations for viscoelastic fluids",
author = "X. Hu and Lin, {F. H.} and C. Liu",
year = "2018",
doi = "10.1007/978-3-319-13344-7_25",
pages = "1045--1073",
booktitle = "Handbook of Mathematical Analysis in Mechanics of Viscous Fluids",
publisher = "Springer International Publishing",
}

@article{bresch_2009,
author = {Bresch, D. and Colin, T. and Grenier, E. and Ribba, B. and Saut, O.},
title = {A viscoelastic model for avascular tumor growth},
JOURNAL = {Discrete Contin. Dyn. Syst.},
FJOURNAL = {Discrete and Continuous Dynamical Systems. Series A},
NUMBER = {Dynamical systems, differential equations and applications.
              7th AIMS Conference, suppl.},
PAGES = {101--108},
year = {2009},
doi={10.3934/proc.2009.2009.101}
}

@inbook{malek_prusa_2018,
author = {M{\'a}lek, J. and Pr{\r u}{\v s}a, V.},
year = {2018},
pages = {3-72},
title = {Derivation of equations for continuum mechanics and thermodynamics of fluids},
booktitle = {Handbook of Mathematical Analysis in Mechanics of Viscous Fluids},
publisher = {Springer International Publishing},
doi = {10.1007/978-3-319-13344-7_1}
}

@article{barrett_boyaval_2009,
author = {Barrett, J. W. and Boyaval, S.},
year = {2011},
pages = {1783-1837},
title = {Existence and approximation of a (regularized) {Oldroyd-B} model},
volume = {21},
number = {9},
journal = {Math.~Models Methods Appl.~Sci.},
doi = {10.1142/S0218202511005581}
}

@book{dahmen_reusken_numerik,
  author={Dahmen, W. and Reusken, A.},
  title={Numerik f{\"u}r Ingenieure und Naturwissenschaftler},
  edition={2}, 
  address={Berlin, Heidelberg},
  year={2008},
  publisher = {Springer},
  doi={10.1007/978-3-540-76493-9},
}

@article{malek_2018_Oldroyd_diffusive,
  title={Thermodynamics of viscoelastic rate-type fluids with stress diffusion},
  author={M{\'a}lek, J. and Pr{\r u}{\v s}a, V. and Sk{\v r}ivan, T. and S{\"u}li, E.},
  journal={Physics of Fluids},
  volume={30},
  number={2},
  pages={023101},
  year={2018},
  doi = {10.1063/1.5018172},
}

@book{girault_raviart_2012,
  title={{Finite Element Methods for Navier--Stokes Equations: Theory and Algorithms}},
  author={Girault, V. and Raviart, P. A.},
  volume={5},
  year={2012},
  publisher={Springer Science \& Business Media}
}

@article{gruen_2013,
  title={On convergent schemes for diffuse interface models for two-phase flow of incompressible fluids with general mass densities},
  author={Gr{\"u}n, G.},
  journal={SIAM J.~Num.~Anal.},
  volume={51},
  number={6},
  pages={3036--3061},
  year={2013},
  publisher={SIAM},
  doi={10.1137/130908208}
}

@article{barrett_lu_sueli_2017,
  title={{Existence of large-data finite-energy global weak solutions to a compressible Oldroyd-B model}},
  author={Barrett, J. W. and Lu, Y. and S{\"u}li, E.},
  journal={Commun.~Math.~Sci.}, 
  volume={15},
  number={5},
  year={2017},
  pages={1265--1323},
  publisher={International Press},
  doi={10.4310/CMS.2017.v15.n5.a5}
}

@article{barrett_2018_fene-p,
  title={{Finite element approximation of the FENE-P model}},
  author={Barrett, J. W. and Boyaval, S.},
  journal={IMA J.~Numer.~Anal.},
  volume={38},
  number={4},
  pages={1599--1660},
  year={2018},
  publisher={Oxford University Press},
  doi = {10.1093/imanum/drx061},
}

@article{barrett_langdon_nuernberg_2004,
  title={Finite element approximation of a sixth order nonlinear degenerate parabolic equation},
  author={Barrett, J. W. and Langdon, S. and N{\"u}rnberg, R.},
  journal={Numer.~Math.},
  volume={96},
  number={3},
  pages={401--434},
  year={2004},
  publisher={Springer},
  doi ={ 10.1007/s00211-003-0479-4}
}

@article{guillen_2013_liquid_crystal,
  title={{A linear mixed finite element scheme for a nematic Ericksen--Leslie liquid crystal model}},
  author={Guill{\'e}n-Gonz{\'a}lez, F. M. and Guti{\'e}rrez-Santacreu, J. V.},
  journal={ESAIM: Math.~Model.~Numer.~Anal.}, 
  volume={47},
  number={5},
  pages={1433--1464},
  year={2013},
  doi={10.1051/m2an/2013076}
}

@article{azerad_guillen_2001,
  title={Mathematical justification of the hydrostatic approximation in the primitive equations of geophysical fluid dynamics},
  author={Az{\'e}rad, P. and Guill{\'e}n-Gonz{\'a}lez, F. M.},
  journal={SIAM J.~Math.~Anal.},
  volume={33},
  number={4},
  pages={847-859},
  year={2001},
  publisher={SIAM},
  doi = {10.1137/S0036141000375962}
}

@article{metzger_2018,
  title={On convergent schemes for two-phase flow of dilute polymeric solutions},
  author={Metzger, S.},
  journal={ESAIM: Math. Model. Numer. Anal.}, 
  volume={52},
  number={6},
  pages={2357--2408},
  year={2018},
  publisher={EDP Sciences},
  doi={10.1051/m2an/2018042}
}

@article{horgan_2004_constitutive,
  title={Constitutive models for compressible nonlinearly elastic materials with limiting chain extensibility},
  author={Horgan, C. O. and Saccomandi, G.},
  journal={Journal of Elasticity},
  volume={77},
  number={2},
  pages={123--138},
  year={2004},
  publisher={Springer},
  doi={10.1007/s10659-005-4408-x}
}

@article{liu_2008_viscoelastic_incompr,
  title={Global solutions for incompressible viscoelastic fluids},
  author={Lei, Z. and Liu, C. and Zhou, Y.},
  JOURNAL = {Arch. Ration. Mech. Anal.},
  FJOURNAL = {Archive for Rational Mechanics and Analysis},
  volume={188},
  number={3},
  pages={371--398},
  year={2008},
  publisher={Springer},
  doi={10.1007/s00205-007-0089-x}
}

@article{liu_1972,
  title={{Method of Lagrange multipliers for exploitation of the entropy principle}},
  author={Liu, I S.},
  journal={Archive for Rational Mechanics and Analysis},
  volume={46},
  number={2},
  pages={131--148},
  year={1972},
  publisher={Springer-Verlag},
  doi={10.1007/BF00250688}
}

@article{giesekus_1982,
title = {A simple constitutive equation for polymer fluids based on the concept of deformation-dependent tensorial mobility},
journal = {Journal of Non-Newtonian Fluid Mechanics},
volume = {11},
number = {1},
pages = {69-109},
year = {1982},
doi = {10.1016/0377-0257(82)85016-7},
author = {H. Giesekus},
}

@article{barrett_sueli_2007,
author = {Barrett, J. W. and S{\"u}li, E.},
title = {Existence of global weak solutions to some regularized kinetic models for dilute polymers},
journal = {Multiscale Modeling \& Simulation},
volume = {6},
number = {2},
pages = {506-546},
year = {2007},
doi = {10.1137/060666810}
}

@book{temam_2001,
  title={Navier--Stokes Equations: Theory and Numerical Analysis},
  author={Temam, R.},
  series={AMS/Chelsea publication},
  year={2001},
  publisher={AMS Chelsea Pub.}
}

@article{cahn_hilliard_1958,
author = {Cahn, J. W.  and Hilliard, J. E. },
title = {{Free energy of a nonuniform system. I. Interfacial free energy}},
journal = {The Journal of Chemical Physics},
volume = {28},
number = {2},
pages = {258-267},
year = {1958},
doi = {10.1063/1.1744102}
}

@article{LinLiuZhang_2005,
author = {Lin, F. H. and Liu, C. and Zhang, P.},
title = {On hydrodynamics of viscoelastic fluids},
JOURNAL = {Comm. Pure Appl. Math.},
FJOURNAL = {Communications on Pure and Applied Mathematics},
volume = {58},
number = {11},
pages = {1437--1471},
doi = {10.1002/cpa.20074},
year = {2005}
}

@article{Hu_Lelievre_2007,
  title={{New entropy estimates for the Oldroyd-B model and related models}},
  author={D. Hu and T. Leli{\`e}vre},
  journal={Commun.~Math.~Sci.}, 
  year={2007},
  volume={5},
  number=4,
  pages={909--916},
  doi={10.4310/CMS.2007.V5.N4.A9}
}

@article{Lukacova_2017,
author = {Luk{\'a}{\v c}ov{\'a}-Medvid’ov{\'a}, M. and Mizerov{\'a}, H. and Ne{\v c}asov{\'a}, S. and Renardy, M.},
title = {{Global existence result for the generalized Peterlin viscoelastic model}},
journal = {SIAM J.~Math.~Anal.},
volume = {49},
number = {4},
pages = {2950-2964},
year = {2017},
doi = {10.1137/16M1068505},
}

@book{brenner_scott_2008,
  title={The Mathematical Theory of Finite Element Methods},
  author={Brenner, S. C. and Scott, L. R.},
  volume={3},
  year={2008},
  publisher={Springer}
}

@article{hawkins_2012,
author = {Hawkins-Daarud, A. and van der Zee, K. G. and Oden, J. T.},
title = {Numerical simulation of a thermodynamically consistent four-species tumor growth model},
journal = {Int.~J.~Numer.~Methods
Biomed.~Eng.},
volume = {28},
number = {1},
pages = {3-24},
doi = {https://doi.org/10.1002/cnm.1467},
year = {2012}
}

@article{Oldroyd_1950,
author = {Oldroyd, J. G.},
title = {On the formulation of rheological equations of state},
journal = {Proc.~R.~Soc.~Lond.~A}, 
volume = {200},
number = {1063},
pages = {523-541},
year = {1950},
doi = {10.1098/rspa.1950.0035},
}

@article{ebenbeck_2019_analysis,
title = {Analysis of a {Cahn--Hilliard--Brinkman} model for tumour growth with chemotaxis},
journal = {J.~Differential Equations},
volume = {266},
number = {9},
pages = {5998-6036},
year = {2019},
doi = {10.1016/j.jde.2018.10.045},
author = {M. Ebenbeck and H. Garcke},
}

@article{ebenbeck_2019_singular_limit,
author = {Ebenbeck, M. and Garcke, H.},
title = {On a {Cahn--Hilliard--Brinkman} model for tumor growth and its singular limits},
journal = {SIAM J.~Math.~Anal.},
volume = {51},
number = {3},
pages = {1868-1912},
year = {2019},
doi = {10.1137/18M1228104},
}

@article{ebenbeck_2020_optimal_control,
	author = {Ebenbeck, M. and Knopf, P.},
	title = {Optimal control theory and advanced optimality conditions for a diffuse interface model of tumor growth},
	DOI= "10.1051/cocv/2019059",
	journal = {ESAIM: COCV},
	year = 2020,
	volume = 26,
	pages = "71",
}

@article{ebenbeck_2019_medication,
author = {Ebenbeck, M. and Knopf, P.},
year = {2019},
title = {{Optimal medication for tumors modeled by a Cahn--Hilliard--Brinkman equation}},
journal = {Calc.~Var.~Partial Differential Equations}, 
doi = {10.1007/s00526-019-1579-z},
volume = {58},
number = {4},
pages={1--31},
publisher={Springer}
}

@article{ebenbeck_2021_singular_potentials,
author = {Ebenbeck, M. and Lam, K. F.},
year = {2021},
title = {{Weak and stationary solutions to a Cahn--Hilliard--Brinkman model with singular potentials and source terms}},
journal = {Adv.~Nonlinear Anal.},  
doi = {10.1515/anona-2020-0100},
volume = {10},
number = {1},
pages={24--65},
}

@article{chemotaxis_in_cancer_2011,
title={Chemotaxis in cancer},
author={E. Roussos and J. Condeelis and A. Patsialou},
year={2011},
journal={Nat.~Rev.~Cancer},
volume={11},
pages={573--587},
doi={10.1038/nrc3078}
}

@article{ambrosi_2009,
title={Cell adhesion mechanisms and stress relaxation in the mechanics of tumours},
author={D. Ambrosi and L. Preziosi},
year={2009},
journal={Biomech. Model. Mechanobiol.},
volume={8},
number={5},
pages={397--413},
doi={10.1007/s10237-008-0145-y}
}

@article{mokbel_abels_aland_2018,
title = {A phase-field model for fluid--structure interaction},
JOURNAL = {J. Comput. Phys.},
FJOURNAL = {Journal of Computational Physics}, 
volume = {372},
pages = {823--840},
year = {2018},
doi = {10.1016/j.jcp.2018.06.063},
author = {D. Mokbel and H. Abels and S. Aland},
}

@article{Brunk_Lukacova_2021,
	doi = {10.1088/1361-648x/abeb13},
	year = 2021,
	publisher = {{IOP} Publishing},
	volume = {33},
	number = {23},
	pages = {234002},
	author = {A. Brunk and B. D{\"u}nweg and H. Egger and O. Habrich and M. Luk{\'a}{\v c}ov{\'a}-Medvid’ov{\'a} and D. Spiller},
	title = {Analysis of a viscoelastic phase separation model},
	journal = {J.~Phys.:~Condens.~Matter}
}

@article{lima_2016,
author = {Lima, E. and Oden, J. and Hormuth, D. and Yankeelov, T. and Almeida, R.},
title = {Selection, calibration, and validation of models of tumor growth},
journal = {Math.~Models Methods Appl.~Sci.}, 
volume = {26},
number = {12},
pages = {2341--2368},
year = {2016},
doi = {10.1142/S021820251650055X},
}

@article {garcke_lam_signori_2021_optimal_control,
    AUTHOR = {Garcke, H. and Lam, K. F. and Signori, A.},
     TITLE = {Sparse optimal control of a phase field tumor model with
              mechanical effects},
   JOURNAL = {SIAM J.~Control Optim.},
  FJOURNAL = {SIAM Journal on Control and Optimization},
    VOLUME = {59},
      YEAR = {2021},
    NUMBER = {2},
     PAGES = {1555--1580},
       DOI = {10.1137/20M1372093}
}

@article{lowengrub_2021_viscoelastic,
  title={Stress generation, relaxation and size control in confined tumor growth},
  author={Yan, H. and Ramirez-Guerrero, D. and Lowengrub, J. and Wu, M.},
  journal = {PLoS. Comput. Biol},
  volume={17},
  number={12},
  year={2021},
  pages={e1009701},
  doi={10.1371/journal.pcbi.1009701}
}

@unpublished{garcke_2022_viscoelastic,
  title={A {Cahn--Hilliard} model coupled to viscoelasticity with large deformations},
  author={Agosti, A. and Colli, P. and Garcke, H. and Rocca, E.},
  note = {In preparation},
}

@article{trautwein_2021,
      title={Numerical analysis for a {Cahn--Hilliard} system modelling tumour growth with chemotaxis and active transport}, 
      author={H. Garcke and D. Trautwein},
      year={2022},
      journal={J.~Numer.~Math.},
      fjournal={Journal of Numerical Mathematics},
      pages={to apperar}
}

\end{document}